\documentclass{amsart}
\usepackage[usenames,dvipsnames]{pstricks}
\usepackage{epsfig, framed, xcolor}
\usepackage{amsmath,amsthm, eucal, amscd, mathabx, accents}
 
\newtheorem{Theorem}{Theorem}[section]
\newtheorem{Proposition}[Theorem]{Proposition}
\newtheorem{Lemma}[Theorem]{Lemma}
\newtheorem{Corollary}[Theorem]{Corollary}
\newtheorem{Sublemma}[Theorem]{Sublemma}	
\newtheorem{Claim}{Claim}

\theoremstyle{definition}
\newtheorem{Definition}[Theorem]{Definition}

\theoremstyle{remark}
\newtheorem{Remark}[Theorem]{Remark}
\newtheorem*{Conjecture}{Conjecture}

\numberwithin{equation}{section}

\begin{document}



\title[The semiclassical zeta function]{The semiclassical zeta function for geodesic flows on negatively curved manifolds}

\author{Fr\'ed\'eric FAURE}

\author{ Masato TSUJII}

\address[Fr\'ed\'eric FAURE]{Institut Fourier, UMR 5582 100 rue des Maths, BP74 FR-38402 St Martin d'H\`eres, FRANCE}

\address[Masato TSUJII]{Department of Mathematics, Kyushu University, 744 Motooka, Nishi-ku, Fukuoka, 819-0395, JAPAN}

\email[Fr\'ed\'eric FAURE]{frederic.faure@ujf-grenoble.fr}
\email[Masato TSUJII]{tsujii@math.kyushu-u.ac.jp}


\date{\today}

\thanks{The authors thank Colin Guillarmou and Semyon Dyatlov for helpful discussions and useful comments during they are writing this paper. They also thank the anonymous referee for many (indeed more than 100!) comments to the previous version of this paper based on very precise reading, which was indispensable in correcting errors and making the paper more readable.
F.~Faure thanks the ANR agency support ANR-09-JCJC-0099-01. M.~Tsujii thanks the support by JSPS KAKENHI Grant Number 22340035.}

\begin{abstract}
We consider the semi-classical (or Gutzwiller-Voros) zeta functions for  $C^\infty$ contact Anosov flows. 
Analyzing the spectra of the generators of some transfer operators associated to the flow, we prove that, for arbitrarily small $\tau>0$, its zeros are contained in the union of the $\tau$-neighborhood of the imaginary axis, $|\Re(s)|<\tau$,  and the half-plane $\Re(s)<-\chi_0+\tau$, up to finitely many exceptions, where $\chi_0>0$ is the hyperbolicity exponent of the flow. 
Further we show that the density of the zeros along the imaginary axis satisfy an analogue of the Weyl law. 
\end{abstract}

\maketitle

\tableofcontents


\newcommand{\real}{\mathbb{R}}
\newcommand{\integer}{\mathbb{Z}}
\newcommand{\complex}{\mathbb{C}}
\renewcommand{\Re}{\mathrm{Re}}
\renewcommand{\Im}{\mathrm{Im}}

\newcommand{\supp}{\mathrm{supp}\,}
\newcommand{\trace}{\mathrm{Tr}}
\newcommand{\flatTr}{\mathrm{Tr}^\flat}
\newcommand{\disk}{\mathbb{D}}
\newcommand{\vol}{\mathbf{m}}


\newcommand{\cL}{\mathcal{L}}
\newcommand{\wL}{\widetilde{\L}}
\newcommand{\bL}{\mathbf{L}}


\newcommand{\Bargmann}{{\mathcal{B}}}
\newcommand{\BargmannP}{\mathcal{P}}
\newcommand{\pBargmann}{{\mathfrak{B}}}  
\newcommand{\pBargmannP}{{\mathfrak{P}}} 

\newcommand{\cT}{\mathcal{T}}
\newcommand{\bT}{\mathbf{T}}

\newcommand{\mult}{\mathcal{M}} 
\newcommand{\cH}{\mathcal{H}}
\newcommand{\bH}{\mathbf{H}}
\newcommand{\wK}{{\widetilde{\mathcal{K}}}}
\newcommand{\cK}{\mathcal{K}}
\newcommand{\bK}{\mathbf{K}}
\newcommand{\cW}{\mathcal{W}}

\renewcommand{\j}{\mathbf{j}} 
\newcommand{\cJ}{\mathcal{J}}

\newcommand{\wA}{\widehat{A}}
\newcommand{\cone}{\mathbf{C}}
\newcommand{\cA}{\mathcal{A} }

\newcommand{\cX}{\mathcal{X}}
\newcommand{\lift}{\mathrm{lift}}

\newcommand{\cQ}{\mathcal{Q}}
\newcommand{\cG}{\mathcal{G}}
\newcommand{\cN}{\mathcal{N}}
\newcommand{\cD}{\mathcal{D}}
\newcommand{\cS}{\mathcal{S}}

\newcommand{\bu}{\mathbf{u}}
\newcommand{\bv}{\mathbf{v}}

\newcommand{\bI}{\mathbf{I}} 

\newcommand{\proj}{\mathfrak{p}} 

\newcommand{\bPi}{\mathbf{\Pi}}
\newcommand{\bP}{\boldsymbol{\pi}}

\newcommand{\cR}{\mathcal{R}}

\newpage

\newpage
\section{Introduction}\label{sec:intro}
The dynamical zeta functions for  flows are introduced by S.~Smale in the monumental paper ``{\it Differentiable dynamical systems}" \cite{Smale}. 
In the former part of the paper, he  discussed  about the Artin-Masur zeta function for discrete dynamical systems among others and showed that it is a rational function for any Anosov diffeomorphism. 
Then, in the latter part, he considered  a parallel object for continuous dynamical systems (or flows). 
He defined  the dynamical zeta function for a (non-singular) flow  by the formula
\begin{equation}\label{eq:Smale_definition_of_zeta}
Z(s):=\prod_{\gamma\in \Gamma}\prod_{k=0}^\infty \left(1-e^{-(s+k)|\gamma|}\right)=\exp\left(-\sum_{\gamma\in \Gamma}\sum_{k=0}^\infty \sum_{m=1}^\infty  \frac{e^{-(s+k)m\cdot |\gamma|}}{m}  \right),
\end{equation}
where $\Gamma$ denotes the set of prime periodic orbits for the flow and $|\gamma|$ denotes the period of $\gamma \in \Gamma$.
This  definition, seemingly rather complicated, is motivated by a famous result of Selberg  \cite{Selberg56}. For the geodesic flow on a  closed hyperbolic surface, {\it i.e.}\ a closed surface with negative constant curvature ($\equiv -1$), $Z(s)$ coincides with the  Selberg zeta function and the result of Selberg gives\footnote{
The paper \cite{Selberg56} treats much more general setting and the results are  stated in terms of geometry. Since the closed geodesics correspond to the periodic orbits of the geodesic flow, we may interpret the results in terms of dynamical systems. For the result mentioned here, we refer  \cite{McKean72, Sunada}.} the following analytic properties of $Z(s)$:  (See Figure \ref{fig1}.)

\begin{itemize}
\item[(a)] The infinite product and sum on the right-hand side of  (\ref{eq:Smale_definition_of_zeta}) converge absolutely when  $\Re(s)>1$. Hence $Z(s)$ is initially defined as an analytic function without zeros on the region $\{\Re(s)>1\}$. 
\item[(b)] The function $Z(s)$ thus defined extends analytically to the whole complex plane $\complex$. 
\item[(c)] The analytic extension of $Z(s)$ has zeros at $s=-n$ for $n=0,1,2,\cdots$ and the order of the zero $s=-n$ is $(2n+1)(g-1)$, where $g\ge 2$ is the genus of the surface.  The other zeros are exactly
\[
s=\frac{1}{2}\pm \sqrt{\frac{1}{4}-\lambda_i}, \quad i=0,1,2,\cdots
\]
where $\lambda_0=0\le \lambda_1\le \lambda_2\le \cdots$ are the eigenvalues of the Laplacian on the surface. In particular, all of the zeros of the latter kind (called non-trivial zeros)  are located on the line $\Re(s)=1/2$ with finitely many exceptions.
\item[(d)] The analytic extension of $Z(s)$ satisfies the functional equation\footnote{The line integral on the right-hand side may change its value by an integer multiple of $2\pi i$ when we consider a different path for integration. But this ambiguity is cancelled when it is put in the exponential function  and therefore the factor  $\exp(\cdot)$ on the right-hand side is well-defined.}
\[
Z(1-s)=Z(s)\cdot \exp\left(2(g-1) \int_{0}^{s-1/2} \pi x \tan (\pi x) dx\right).
\]
\end{itemize}
\begin{figure}
\scalebox{1} 
{
\begin{pspicture}(0,-3.62)(7.4818945,3.62)
\psline[linewidth=0.04cm](2.2,3.6)(2.2,-3.6)
\psline[linewidth=0.04cm](-0.3,0.2)(6.6,0.2)
\usefont{T1}{ptm}{m}{n}
\rput(3.65,0.605){$\displaystyle\frac{1}{2}$}
\usefont{T1}{ptm}{m}{n}
\rput(1.95,0.505){$0$}
\rput(-0.1,0.505){$-1$}
\psdots[dotsize=0.2,dotstyle=x](2.2,0.2)
\psdots[dotsize=0.2,dotstyle=x](3.4,0.2)
\psdots[dotsize=0.2,dotstyle=x](3.4,0.6)
\psdots[dotsize=0.2,dotstyle=x](3.4,1.0)
\psdots[dotsize=0.2,dotstyle=x](3.4,1.4)
\psdots[dotsize=0.2,dotstyle=x](3.4,1.8)
\psdots[dotsize=0.2,dotstyle=x](3.4,2.0)
\psdots[dotsize=0.2,dotstyle=x](3.4,2.2)
\psdots[dotsize=0.2,dotstyle=x](3.4,2.4)
\psdots[dotsize=0.2,dotstyle=x](3.4,2.6)
\psdots[dotsize=0.2,dotstyle=x](3.4,3.0)
\psdots[dotsize=0.2,dotstyle=x](3.4,3.2)
\psdots[dotsize=0.2,dotstyle=x](3.4,3.4)
\psdots[dotsize=0.2,dotstyle=x](3.4,-0.2)
\psdots[dotsize=0.2,dotstyle=x](3.4,-0.6)
\psdots[dotsize=0.2,dotstyle=x](3.4,-1.0)
\psdots[dotsize=0.2,dotstyle=x](3.4,-1.4)
\psdots[dotsize=0.2,dotstyle=x](3.4,-1.6)
\psdots[dotsize=0.2,dotstyle=x](3.4,-1.8)
\psdots[dotsize=0.2,dotstyle=x](3.4,-2.0)
\psdots[dotsize=0.2,dotstyle=x](3.4,-2.2)
\psdots[dotsize=0.2,dotstyle=x](3.4,-2.4)
\psdots[dotsize=0.2,dotstyle=x](3.4,-2.8)
\psdots[dotsize=0.2,dotstyle=x](3.4,-3.0)
\psdots[dotsize=0.2,dotstyle=x](3.4,-3.2)
\psdots[dotsize=0.2,dotstyle=x](2.8,0.2)
\psdots[dotsize=0.2,dotstyle=x](4.0,0.2)
\psdots[dotsize=0.2,dotstyle=x](4.6,0.2)
\psdots[dotsize=0.2,dotstyle=x](-0.1,0.2)
\psline[linewidth=0.012cm](3.4,3.6)(3.4,-3.6)
\usefont{T1}{ptm}{m}{n}
\rput(4.651455,0.505){$1$}
\usefont{T1}{ptm}{m}{n}
\rput(6.231455,0.525){$\mathrm{Re}(s)$}
\usefont{T1}{ptm}{m}{n}
\rput(1.551455,3.305){$\mathrm{Im}(s)$}
\end{pspicture} 
}
\label{fig1}
\caption{Zeros of the Selberg zeta function $Z(s)$ }
\end{figure}
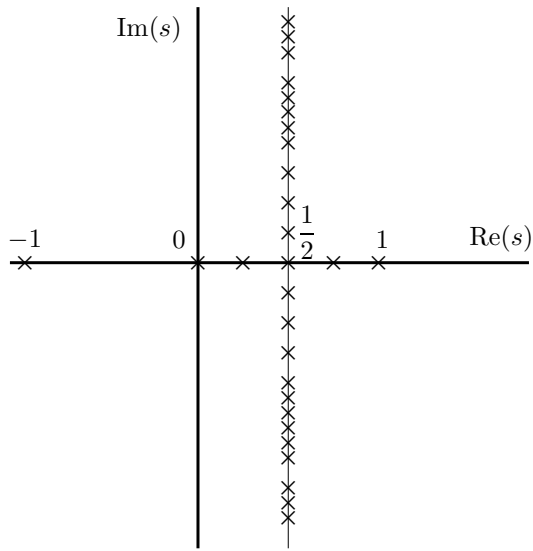

Smale's idea was to study the dynamical zeta function $Z(s)$ defined as above in more general context. 
The main question\footnote{There are many other related problems. For instance the relation of special values of the dynamical zeta function to the geometric properties of the underlying manifolds is an interesting problem. See \cite{1606.04560, Fried88,  Morita}.} ought to have been whether the properties (a)-(d) above hold for more general types of flow, such as the geodesic flows on manifolds with negative variable  curvature or, more generally, to general  Anosov flows. 
But it was not clear whether this idea was reasonable, since the results of Selberg were based on the  Selberg trace formula for the heat kernel on the surface and depended crucially on the fact that the surface was of  negative constant curvature. This must be the reason why Smale described his idea ``wild".  In \cite{Smale}, he showed that $Z(s)$ has meromorphic extension to the whole complex plane if the flow is a suspension flow  of an Anosov diffeomorphism by a constant roof function. However the main part of the ``wild" idea was left as a question. 

Later the dynamical zeta function $Z(s)$ is  generalized and studied extensively by many people not only in dynamical system theory but also in the fields of mathematical physics related to ``quantum chaos". In  dynamical system theory, the dynamical zeta function $Z(s)$ and its variants are related to semi-groups of transfer operators associated to the flow through the  Atiyah-Bott-Guillemin trace formula, as we will explain later.
We refer the papers  \cite{ParryPollicott, Ruelle76flow} for the development in the early stage and the paper \cite{GLP12} (and the references therein) for the recent state of related researches. 

For the extensions of the claim (a) and (b) above, we already have satisfactory results: for instance, the dynamical zeta function 
$Z(s)$ for a $C^\infty$ Anosov flow  is known to have meromorphic extension to the whole complex plane $\complex$. (See  \cite{GLP12, DZ13}. Note that the arguments in these papers are applicable to more general dynamical zeta functions.) However, to the best of authors' understanding, not much is known about the extension of the claims (c) or (d), or more generally on the distributions of singularities of the (generalized) dynamical zeta functions. In this paper, we consider an extension of the claim (c) in the case of geodesic flows on negatively curved manifolds or, more generally, contact Anosov flows.

Before proceeding with the problem, we would like to pose a question whether the zeta function $Z(s)$ introduced by Smale is the ``right" candidate to be studied. In fact, there are variety of generalized dynamical zeta functions which coincide with $Z(s)$ in the cases of geodesic flows on closed hyperbolic surfaces, because so do some dynamical exponents. Each of such generalized dynamical zeta functions may be regarded as an extension of Selberg zeta function in their own rights and their analytic property will be different when we consider them in more general cases. And there is no clear evidence that $Z(s)$ is better or more natural than the others. This is actually one of the question that the authors would like to address  in  this paper.  
For the geodesic flows on a negatively curved closed manifold $N$ (or more generally non-singular flows with some hyperbolicity),  the  ``semi-classical" or ``Gutzwiller-Voros" zeta function $Z_{sc}(s)$ is defined by
\begin{equation}\label{def:semiclassicalzeta}
Z_{sc}(s)=\exp\left(-\sum_{\gamma\in \Gamma}\sum_{m=1}^\infty  \frac{e^{-s\cdot m\cdot |\gamma|}}{m\cdot {|\det(\mathrm{Id}-D_\gamma^m)|^{1/2}}} \right)
\end{equation}
where $D_\gamma$ is the transversal Jacobian matrix\footnote{This is the Jacobian matrix of the Poincar\'e map for the orbit $\gamma$ at the intersection.} along a prime periodic orbit $\gamma$.
(See \cite{Rugh} for instance.) As we will see, this is a variant of the dynamical zeta function $Z(s)$ and coincides with the dynamical zeta function $Z(s)$ if $N$ is a closed surface with constant negative ($\equiv -1$) curvature\footnote{For the case of a surface with  constant negative curvature ($\equiv -1$), the eigenvalues of $D_\gamma$ are
$\exp(\pm|\gamma|)$. Hence we can check the equality $Z_{sc}(s)=Z(s+1/2)$ by simple calculation.},  
with shift of the variable $s$ by $1/2$. Hence we may regard   $Z_{sc}(s)$ as a different generalization of  Selberg zeta function than $Z(s)$. As the main result of this paper, we show that an extension of the claim (c) holds for $Z_{sc}(s)$ in the case of the geodesic flows on manifolds with negative variable curvature (and more generally for contact Anosov flows), \textit{that is},  countably many zeros of the analytic extension of $Z_{sc}(s)$ concentrate along the imaginary axis and there are regions on the both sides of the imaginary axis with only finitely many zeros. (See Figure \ref{fig2} and compare it with Figure \ref{fig1}.)
\begin{figure}
\scalebox{1} 
{
\begin{pspicture}(0,-3.62)(7.2818947,3.62)
\definecolor{color588b}{rgb}{0.8,0.8,0.8}
\psframe[linewidth=0.01,linestyle=dotted,dotsep=0.16cm,dimen=outer,fillstyle=solid,fillcolor=color588b](1.0,3.6)(0.0,-3.6)
\psframe[linewidth=0.01,linestyle=dotted,dotsep=0.16cm,dimen=outer,fillstyle=solid,fillcolor=color588b](3.4,3.6)(3.0,-3.6)
\psline[linewidth=0.04cm](3.2,3.6)(3.2,-3.6)
\psline[linewidth=0.04cm](0.0,0.2)(6.6,0.2)
\psdots[dotsize=0.12](3.3,3.0)
\psdots[dotsize=0.12](3.16,2.48)
\psdots[dotsize=0.12](3.1,1.76)
\psdots[dotsize=0.12](3.28,1.56)
\psdots[dotsize=0.12](3.12,0.8)
\psdots[dotsize=0.12](2.72,0.64)
\psdots[dotsize=0.12](1.9,0.48)
\psdots[dotsize=0.12](2.9,1.2)
\psdots[dotsize=0.12](3.32,2.68)
\psdots[dotsize=0.12](3.12,3.28)
\psdots[dotsize=0.12](3.28,3.46)
\psdots[dotsize=0.12](3.3,2.08)
\psdots[dotsize=0.12](3.72,0.94)
\psdots[dotsize=0.12](4.02,0.38)
\psdots[dotsize=0.12](1.88,-0.04)
\psdots[dotsize=0.12](2.72,-0.22)
\psdots[dotsize=0.12](3.12,-0.4)
\psdots[dotsize=0.12](2.88,-0.96)
\psdots[dotsize=0.12](3.12,-1.36)
\psdots[dotsize=0.12](3.32,-1.16)
\psdots[dotsize=0.12](3.28,-1.68)
\psdots[dotsize=0.12](3.18,-1.96)
\psdots[dotsize=0.12](3.32,-2.16)
\psdots[dotsize=0.12](3.32,-2.42)
\psdots[dotsize=0.12](3.1,-2.6)
\psdots[dotsize=0.12](3.32,-2.8)
\psdots[dotsize=0.12](4.0,0.04)
\psdots[dotsize=0.12](3.74,-0.5)
\psdots[dotsize=0.12](0.72,2.56)
\psdots[dotsize=0.12](0.72,1.26)
\psdots[dotsize=0.12](0.68,-0.4)
\psdots[dotsize=0.12](0.68,-1.94)
\psdots[dotsize=0.12](0.78,0.66)
\psdots[dotsize=0.12](0.8,1.88)
\psdots[dotsize=0.12](0.54,-1.0)
\psdots[dotsize=0.12](0.64,-2.52)
\psdots[dotsize=0.12](0.92,3.24)
\psdots[dotsize=0.12](0.9,-3.22)
\usefont{T1}{ptm}{m}{n}
\rput(1.411455,0.445){$-\chi_0$}
\usefont{T1}{ptm}{m}{n}
\rput(3.49708,0.425){0}
\usefont{T1}{ptm}{m}{n}
\rput(6.031455,0.525){$\mathrm{Re}(s)$}
\usefont{T1}{ptm}{m}{n}
\rput(3.9314551,3.285){$\mathrm{Im}(s)$}
\end{pspicture} 
}
\caption{The zeros of the semi-classical zeta function $Z_{sc}(s)$. The zeros are symmetric with respect to the complex conjugation.}
\label{fig2}
\end{figure}
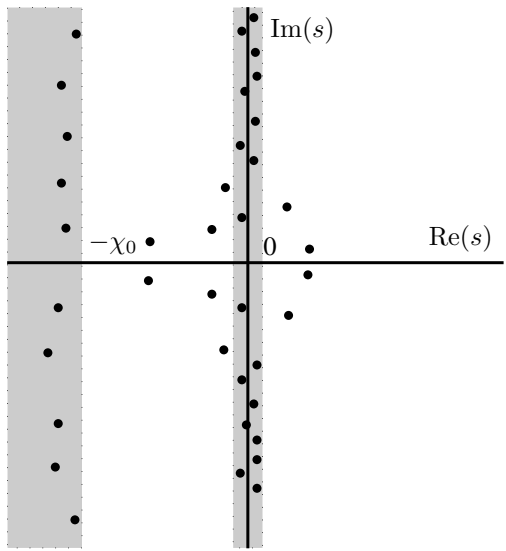
It seems  that this result and the argument in the proof are suggesting that the semi-classical zeta function $Z_{sc}(\cdot)$ is the ``right" generalization of the Selberg zeta function when we consider the extension of the claim (c) (and (d)).

Below we describe our result more precisely. 
Let $M$ be a closed $C^\infty$ manifold of odd dimension, say, $2d+1$. 
We consider a $C^\infty$ contact Anosov flow $f^t:M\to M$. 
By definition, the flow $f^t$ preserves a contact form $\alpha$ on $M$, that is, a differential $1$-form for which $\alpha\wedge (d\alpha)^{\wedge d}$ vanishes nowhere. 
(We may and do assume $\alpha(V)\equiv 1$ for the generating vector field $V$ of the flow$f^t$ by multiplying $\alpha$ by a $C^\infty$ function.)
Also there exist constants $\chi_0>0$, $C>0$ and a $Df^t$-invariant continuous  decomposition  
\[
TM=E_0\oplus E_s\oplus E_u 
\] 
of the tangent bundle $TM$   such that $E_0$ is the one-dimensional subbundle spanned by the generating vector field $V$ of the flow~$f^t$ and that 
\begin{equation}\label{eq:Anosov}
\|Df^t|_{E_s}\|\le Ce^{-\chi_0 t}\quad \mbox{and} \quad \|Df^{-t}|_{E_u}\|\le Ce^{-\chi_0 t}\quad \mbox{for $t\ge 0$.}
\end{equation}

\begin{Remark}
The geodesic flow $f^t:T^*_1N\to T^*_1N$ on a closed negatively curved manifold $N$ is a contact Anosov flow, where $T^*_1N$ is the unit cotangent bundle of $N$ and the contact form $\alpha$ preserved by the flow is the restriction of the canonical one form on $T^*N$.
\end{Remark}

From the definitions, it is not difficult to see that 
\[
E_s\oplus E_u=\ker \alpha\quad \mbox{and}\quad \dim E_u=\dim E_s=d.
\]
We henceforth fix $\chi_0>0$ satisfying (\ref{eq:Anosov}) and call it the hyperbolicity exponent of the flow $f^t$. Note that the subbundles $E_s$ and $E_u$ are in general not smooth but only H\"older continuous. Below we suppose that the subbundles  $E_s$ and $E_u$ are $\beta$-H\"older continuous with exponent
\begin{equation}\label{beta}
0<\beta<1.
\end{equation}

The main result of this paper is the following theorem. 
\begin{Theorem}
\label{mainth:spec}If $f^t:M\to M$ is a contact Anosov flow, its semi-classical zeta function $Z_{sc}(s)$, which is initially defined\/\footnote{Since the factor $|\det(\mathrm{Id}-D^m_\gamma)|$ in the definition of $Z_{sc}(s)$ is positive and proportional to $|\det (Df^{m|\gamma|}|_{E_u}(x_\gamma))|$ for  $x_\gamma\in \gamma$, the sum in the definition of $Z_{sc}(s)$ converges absolutely if and only if  (\ref{eq:region_convergence_sc}) holds. (See \cite[Theorem C]{Lu} for the definition of topological pressure $P_{top}(\cdot)$ and its expression in terms of periodic orbits in the case of Anosov flow.) Further $Z_{sc}(s)$ has its rightmost zero at  $P_{top}(f^t, -(1/2)\log |Df^t|_{E_u}|)$. For the proof, see \cite{Ruelle76flow} and the expression (\ref{eq:Expression_of_Zsc}).} by (\ref{def:semiclassicalzeta}) as a holomorphic function without zeros on the half-plane 
\begin{equation}\label{eq:region_convergence_sc}
\Re(s)>P_{top}(f^t, -(1/2)\log |Df^t|_{E_u}|)>0,
\end{equation}
extends to a meromorphic function on the whole complex plane $\complex$. 
For arbitrarily small $\tau>0$, the zeros of the meromorphic extension of $Z_{sc}(s)$ are contained in the region 
\begin{equation}\label{eq:Uchiepsilon}
U(\chi_0, \tau):=\{z\in \complex\mid |\Re(z)|<\tau \mbox{ or }\;\Re(z)<-\chi_0+\tau\}
\end{equation}
up to finitely many exceptions\footnote{The number of exceptional zeros may increase as $\tau$ becomes smaller. }, while there are at most finitely many poles on the region $\Re(s)>-\chi_0+\tau$. There do exist infinitely many zeros on the strip
\begin{equation}\label{eq:Uzero}
U_0(\tau)=\{z\in \complex\mid |\Re(z)|<\tau\}
\end{equation}
and a (weak) analogue of Weyl law holds for the distribution of the imaginary part of the zeros in $U_0(\tau)$, that is, for any $\delta>0$, there exists a constant $C>1$ such that, for arbitrarily small $0<\tau<\chi_0$,  the estimate    
\begin{equation}\label{eq:estimate_density}
\frac{|\omega|^{d}}{C}<
\frac{\#\{ \mbox{ zeros of $Z_{sc}(s)$ }\mid\mbox{ $|\Re(s)|<\tau$, 
$\omega\le \Im(s)\le \omega+  |\omega|^\delta$ }\}}{|\omega|^\delta}<C|\omega|^{d}
\end{equation}
holds for any real number $\omega$ with sufficiently large absolute value. 
\end{Theorem}

The last claim implies in particular that 
\[
\frac{|\omega|^{d+1}}{C'}<
\#\left\{ \mbox{ zeros of $Z_{sc}(s)$ }\mid \mbox{ $|\Re(s)|<\tau$,  $|\Im(s)|\le  \omega $ }\right\}<C'|\omega|^{d+1}
\]
for some constant $C'>0$ and for sufficiently large $\omega>0$. 

\begin{Remark}
In (\ref{eq:estimate_density}) above, the estimate \emph{from below} is the main assertion. In similar problems (such as density of Ruelle-Pollicott resonances and resonances in the scattering problems), reasonable estimates from below are usually much more difficult to obtain compared with those from above. 
(For estimates from above, we refer \cite{DDZ13, FaureSjostrand11}.) It will be possible to make the estimate (\ref{eq:estimate_density}) more precise by replacing the factors $|\omega|^\delta$ in it by smaller factor such as $\log |\omega|$ or even by some fixed large constant. 
Also, by analogy with the Weyl law for the Laplacians,  it is natural to expect that the ratio
\[
\frac{\#\{ \mbox{zeros of $Z_{sc}(s)$ }\mid\mbox{ $|\Re(s)|<\tau$, 
$\omega\le \Im(s)\le \omega+  |\omega|^\delta$ }\}}{|\omega|^{d+\delta}}
\] 
converges to $(2\pi)^{-d-1} \mathrm{Vol}(M)$ as $\omega\to \pm \infty$ where $\mathrm{Vol}$ denotes the contact volume, \textit{i.e.}\ $\mathrm{Vol}=\alpha\wedge (d\alpha)^{\wedge d}$. But we do not go farther into these problems in this paper.
\end{Remark}

We deduce the theorem above from spectral properties of some transfer operators associated to the flow $f^t$. Let us recall an idea due to Ruelle.
Let $\pi_V:V\to M$ be a complex vector bundle\footnote{We always assume that each vector bundle is complexified and equipped with a Hermitian inner product on it. The choice of  the Hermitian inner product is not essential. But we need it for some expressions, e.g. the definition (\ref{eq:ltk}). }  over $M$ and write $\Gamma^0(V)$ for the set of continuous sections of $V$. Let $F^t:V\to V$ be a one-parameter group of vector bundle maps which makes the following diagram commutes:
\[
\begin{CD} V@>{F^t}>> V\\
@V{\pi_V}VV @V{\pi_V}VV\\
M@>{f^t}>> M
\end{CD}
\]
We consider the one-parameter group of vector-valued transfer operators
\[
\cL^t:\Gamma^0(V)\to\Gamma^0(V), \quad \cL^t v(x)= F^{t}(v(f^{-t}(x))).
\]
The flat (or Atiyah-Bott-Guillemin) trace of $\cL^t$ is calculated as   
\begin{equation}\label{def:dynTr}
\flatTr\,\cL^t=\sum_{\gamma\in \Gamma}\sum_{m=1}^\infty
\frac{|\gamma|\cdot \trace\, E_\gamma^m}{|\det(\mathrm{Id}-D_\gamma^{-m})|}\cdot \delta(t-m\cdot |\gamma|),
\end{equation}
where $E_\gamma$ is the linear transformation $F^{|\gamma|}:\pi_V^{-1}(x_\gamma)\to \pi_V^{-1}(x_\gamma)$ at a point $x_\gamma$ on the orbit~$\gamma$. (See the remark below.) Notice that $\flatTr\,\cL^t$ is not a function of $t$ in the usual sense but is a distribution. 

\begin{Remark}\label{rm:dyntrace} 
The flat trace in  (\ref{def:dynTr}) is defined as the integral of the Schwartz kernel of $\cL^t$ on the diagonal set in $M\times M$. The integral is well-defined as a distribution (see \cite[Ch. 8]{Hormander1}, \cite[Th.8]{Guillemin}), because hyperbolicity of the flow $f^t$ ensures that the graph of $(t,x)\mapsto f^t(x)$ is transversal to the diagonal set $x=y$ in $\real_{+}\times M\times M$.
But, for our argument, it will be more convenient to interpret the definition as follows.  
First we consider the case where $V$ is one-dimensional and trivial, so that $\cL^t$ may be regarded as scalar-valued. Let $K(x,y;t)$ be the Schwartz kernel of the operator $\cL^t$ and let $K^{\delta}(x,y;t)$ for $\delta\ge 0$ be a one-parameter family of smoothings of $K(x,y;t)$ by using mollifier, which converges to $K(x,y;t)$ as $\delta\to +0$. We define the distribution  $\flatTr\,\cL^t$ on $(0,\infty)$ by the relation
\[
\langle \flatTr\,\cL^t,\varphi\rangle:=\lim_{\delta\to +0} \int K^{\delta}(x,x;t) \varphi(t) dx
\]
for $\varphi\in C^\infty_0(\real)$ supported on $\{t\in \real\mid t>0\}$. It is not difficult to check that the limit on the right-hand side exists and does not depend on the choice of the smoothing $K^{\delta}(x,y;t)$. 
When $V$ is higher dimensional or non-trivial, we write the transfer operator $\cL^t$ as a matrix of scalar-valued transfer operators $(\cL^t_{ij})_{1\le i,j\le N}$ by using a system of local trivializations of $V$ and an associated partition of unity. We define the flat trace as 
$\flatTr\,\cL^t=\sum_{i=1}^N \flatTr\,\cL^t_{ii}$. We can check that this definition does not depend on the matrix expression $(\cL^t_{ij})_{1\le i,j\le N}$ and gives  (\ref{def:dynTr}). 
\end{Remark}

For $0\le k\le d$, let $\pi:(E_u^*)^{\wedge k}\to M$ be the $k$-th exterior product of the dual $E_u^*$ of the unstable sub-bundle $E_u$ and let $F_k^t:(E_u^*)^{\wedge k}\to (E_u^*)^{\wedge k}$ be the vector bundle map defined by\footnote{As careful readers may have realized, the sub-bundles $E_u$ (and $E_s$) are not smooth in general and this will cause many technical difficulties in the argument. Indeed this is the main issue of this paper in technical sense. We will address this problem in the next section. For a while, we assume that $E_u$ is smooth or just ignore the problem. } 
\begin{equation}\label{eq:Fk}
F_k^t(v)=|\det Df^t|_{E_u}(\pi(v))|^{1/2}\cdot ((Df^{-t})^*)^{\wedge k}(v).
\end{equation}
Note that the action of $(Df^{-t})^*$ on $E_u^*$ is contracting when $t\ge 0$.
The corresponding one-parameter family of vector-valued transfer operators is 
\begin{align}\label{eq:ltk}
\cL_k^t u(x)&=F_k^t(u(f^{-t}(x)))\\
&=|\det Df^t|_{E_u}(f^{-t}(x))|^{1/2}\cdot ((Df^{-t})^*)^{\wedge k} (u(f^{-t}(x)))\notag
\end{align}
and its flat trace is 
\[
\flatTr\,\cL_k^t=\sum_{\gamma\in \Gamma}\sum_{m=1}^\infty
\frac{|\gamma|\cdot|\det D_\gamma^u|^{m/2}\cdot  \trace\, ((D_\gamma^u)^{-m})^{\wedge k}}{|\det(\mathrm{Id}-D_\gamma^{-m})|}\cdot \delta(t-m\cdot |\gamma|)
\]
where $D_\gamma^u$ is the transversal Jacobian matrix for $\gamma\in \Gamma$ restricted to the unstable sub-bundle $E_u$.  

Since the differential $d\alpha$ of the contact form $\alpha$ restricts to a symplectic form on $\ker \alpha=E_s\oplus E_u$ and is preserved by $Df^t$, we have 
\[
\sqrt{|\det(\mathrm{Id}-D_\gamma^{-m})|}=|\det(D_\gamma^u)|^{m/2}\cdot |\det\left(\mathrm{Id}-(D_\gamma^u)^{-m}\right)|.
\]
Hence, provided that the subbundle $E_u^*$ is orientable, 
we have
\[
\sum_{k=0}^{d}(-1)^{k}\;\flatTr\,\cL_k^t=\sum_{\gamma\in \Gamma}\sum_{m=1}^\infty
\frac{|\gamma|}
{\sqrt{|\det(\mathrm{Id}-D_\gamma^{-m})|}}\cdot \delta(t-m\cdot |\gamma|)
\]
from the algebraic relation
\[
\det(\mathrm{Id}-(D_\gamma^u)^{-m})=
\sum_{k=0}^{d}(-1)^{k} \cdot \trace\, (((D_\gamma^u)^{-m})^{\wedge k}).
\]
\begin{Remark}\label{rem:orientation}
In the case where $E_u^*$  is \emph{not} orientable, the formula above is not valid. In fact, we actually had to replace $F_k^t$ with its extension
\[
\widetilde{F}_k^t=F_k^t\otimes (Df^{-t})^*:(E_u^*)^{\wedge k}\otimes \ell_u\to (E_u^*)^{\wedge k}\otimes \ell_u
\]
where $\ell_u$ is the orientation line bundle of $E^*_u$ and $(Df^{-t})^*:\ell_u\to \ell_u$ denotes the natural pull-back action by $f^{-t}$ on $\ell_u$. (See \cite{Fried86}.) 
Since this modification does no harm in the argument below except making the notation more cumbersome, we proceed with assuming $E_u^*$ to be orientable and keep this modification necessary  for the other cases implicit in the following. 
\end{Remark}

Therefore the semiclassical zeta function $Z_{sc}(s)$ is expressed as
\footnote{The lower bound $+0$ in the integration indicates some small positive number that is smaller than the minimum of the periods of periodic orbits for the flow. } 
\begin{equation}\label{expression_Zsc}
Z_{sc}(s)=\exp\left(-\int_{+0}^\infty \frac{e^{-st}}{t}\sum_{k=0}^{d}(-1)^{k}\flatTr\,\cL_k^t dt\right).
\end{equation}
We define the dynamical Fredholm determinant of $\cL_k^t$ by
\begin{align*}
d_k(s):&=\exp\left(-\int_{+0}^\infty \frac{e^{-st}}{t}\cdot \flatTr\,\cL^t_k dt\right)\\
&=\exp\left(-\sum_{\gamma\in \Gamma}\sum_{m=1}^\infty \frac{e^{-s\cdot m\cdot |\gamma|}\cdot |\det D_\gamma^u|^{m/2}\cdot \trace\, ((D_\gamma^u)^{-m})^{\wedge k}}{m\cdot |\det(\mathrm{Id}-D_\gamma^{-m})|}\right).
\end{align*}
Then the semi-classical zeta function  is expressed as  an alternative product
\begin{equation}\label{eq:Expression_of_Zsc}
Z_{sc}(s)=\prod_{k=0}^{d} d_k(s)^{(-1)^{k}}
\end{equation}
at least for $s$ with sufficiently large real part. The dynamical Fredholm determinant $d_k(s)$ satisfies
\[
(\log d_k(s))'=\frac{d_k(s)'}{d_k(s)}=\int_{+0}^\infty e^{-st} \cdot \flatTr\,\cL_k^t dt.
\]
If $\cL_k^t$ were a finite rank diagonal matrix with diagonal elements $e^{\lambda_\ell t}$ and if the flat trace $\flatTr$ were the usual trace, the right-hand side would be $\sum_\ell (s-\lambda_\ell)^{-1}$ and we would have $d_k(s)={\mathrm{const.}}\prod_\ell(s-\lambda_\ell)$. We therefore expect that the eigenvalues of the generator of $\cL^t_k$ appear as zeros of the dynamical Fredholm determinant $d_k(s)$ and consequently 
zeros (resp.\ poles) of $Z_{sc}(s)$ when $k$ is even (resp.\ odd). 

\section{Grassmann extension}\label{sec:Gext}
A technical difficulty in dealing with the semi-classical zeta functions $Z_{sc}(s)$ is that the coefficient $|\det Df^t|_{E_u}|^{1/2}$ and also the vector bundle $(E_u^*)^{\wedge k}$  in the definition (\ref{eq:ltk}) of  the corresponding transfer operators $\cL^t_k$ is {\em not} smooth but  only H\"older continuous. To avoid this difficulty, we actually consider the corresponding transfer operators on a Grassmann bundle $G$ over the manifold $M$. (In the literature, this kind of idea is found in the papers \cite{CviVat93, GL08}.)

Consider the Grassmann bundle $
\pi_G:G\to M$
that consists of $d$-dimensional subspaces of the tangent bundle $TM$.  By definition,  the fiber $\pi_G^{-1}(x)$ for each point $x\in M$ is the Grassmann space $G_x$ that consists of  $d$-dimensional subspaces of the tangent space $T_xM$, whose dimension is $\dim G_x=d(2d+1-d)=d^2+d$. We suppose that $G$ is equipped with a smooth Riemann metric. 

 The flow $f^t$ naturally induces a flow on $G$:
\[
f^t_G:=(Df^t)_*:G\to G,
\qquad f^t_G(x,\sigma):=(Df^t)_*(x,\sigma)=(f^t(x), Df^{t}(\sigma)).
\]
Let $e_{u}:M\to G$ be the section which assigns the unstable subspace $E_u(x)\in G$ to each point $x\in M$. (But notice that this section $e_u$ is not  smooth in general.) Clearly the following diagrams commute:
\[
\begin{CD}
G@>{f^t_G}>> G\\
@V{\pi_G}VV @V{\pi_G}VV\\
M@>{f^t}>> M
\end{CD}
\qquad\qquad 
\begin{CD}
G@>{f^t_G}>> G\\
@A{e_{u}}AA @A{e_{u}}AA\\
M@>{f^t}>> M
\end{CD}
\]
Since image $\mathrm{Im}(e_u)$ of the section $e_{u}$ is an attracting isolated invariant subset for the extended flow $f^t_G$, we can take a small relatively compact absorbing neighborhood $U_0$ of it so that
\[
f^t_G(U_0)\Subset U_0\quad \mbox{for $t> 0$, \;\;and }\quad \bigcap_{t\ge 0} f^t_G(U_0)=\mathrm{Im}(e_u).
\]
The semi-flow  $f^t_G:U_0\to U_0$ for $t\ge 0$ is hyperbolic in the following sense: There is  a continuous decomposition of the tangent bundle
\begin{equation}\label{eq:hyp_docomp_G}TU_0=\widetilde{E}_u\oplus \widetilde{E}_s\oplus \widetilde{E}_0
\end{equation}
where $\widetilde{E}_s:=D\pi_G^{-1}(E_s)$, $\widetilde{E}_0:=\langle \partial_t f_G^t\rangle$
and $ \widetilde{E}_u$ is a complement of $\widetilde{E}_0\oplus \widetilde{E}_s=\pi_G^{-1}(E_0\oplus E_s)$ such that $
D\pi_{G}(\widetilde{E}_u) =E_u$; 
The semi-flow $f^t_G:U_0\to U_0$ ($t\ge 0$) is  exponentially contracting (resp.\ expanding) on $\widetilde{E}_s$ (resp.\ $\widetilde{E}_u$), that is\footnote{We can and do take the constant $\chi_0$ same as that in (\ref{eq:Anosov}), though this is not necessary.  }, 
\begin{equation}\label{eq:Anosov2}
\|Df_G^t|_{\widetilde{E}_s}\|\le Ce^{-\chi_0 t}\quad \mbox{and} \quad \|Df^{t}_G|_{\widetilde{E}_u}\|_{\min}\ge C^{-1}e^{\chi_0 t}\quad \mbox{for $t\ge 0$}
\end{equation}
where $\|\cdot\|_{\min}$ in the latter inequality denotes the minimum expansion rate 
\[
 \|Df^{t}_G|_{\widetilde{E}_u}\|_{\min}=\min\{ \|Df^{t}_G v\|\mid v\in \widetilde{E}_u, |v|=1\}.
\]
The sub-bundles $\widetilde{E}_0$ and $\widetilde{E}_s$ are (forward) invariant with respect to the semi-flow $f^t_G$, while $\widetilde{E}_u$ will not. 

Let $\pi_{\perp}:V_{\perp}\to G$ be the  $(d^2+d)$-dimensional sub-bundle of $TG$ defined by  
\[
V_\perp:=\{ (z,v)\in TG\mid D\pi_G(v)=0\}\subset TG.
\] 
Let $\pi_G^*(TM)$ be the pull-back  of the tangent  bundle $TM$ by the projection $\pi_G:G\to M$ and let  $\pi:V_u\to G$ be its smooth sub-bundle  defined tautologically   by
\[
V_u=\{(z,v)\in \pi_G^*(TM) \mid v\in [z]\}
\]
where $[z]$ denotes the $d$-dimensional subspace of $T_{\pi_G(z)}M$ that $z\in G$ represents. 
Let $\pi:V_u^*\to G$ be the dual of $V_u$. 
We define 
\[
\pi_{k,\ell}:V_{k,\ell}:=(V_u^*)^{\wedge k}\otimes (V_\perp)^{\wedge \ell}\to G\quad \mbox{ for $0\le k\le d$ and $0\le \ell \le d^2+d$.}
\]
This is a smooth vector bundle. And, instead of the \emph{non-smooth} one-parameter group of vector bundle map $F^t_k$ in (\ref{eq:Fk}), we consider 
the \emph{smooth} one-parameter semi-group $F^t_{k,\ell}:V_{k,\ell}\to V_{k,\ell}$ defined by 
\begin{equation}\label{eq:Fkl}
F^t_{k,\ell}(z, u\otimes v)=\left(f^t_G(z), b^t(z)\cdot ((Df^{-t})^*)^{\wedge k}(u)\otimes  (Df_G^t)^{\wedge\ell}(v)\right)
\end{equation}
where
\begin{equation}\label{eq:bt}
b^t(z)=|\det Df^t_{\pi_G(z)}|_{[z]}|^{1/2}\cdot 
|\det ((Df_G^t)_{z}|_{\ker D\pi_G})|^{-1}.
\end{equation}
The first term on the right-hand side of (\ref{eq:bt}) is the  determinant of the restriction of $Df^t$ at  $\pi_G(z)$ to the subspace $[z]$ of $T_{\pi_G(z)}M$ represented by $z$, while the second term is the  determinant of the restriction of $Df^t_G$ at $z$ to the kernel of $D\pi_G$. Clearly the action of $F^t_{k,\ell}$ is smooth.

Let $\Gamma^\infty(U_0,V_{k,\ell})$ be the set of smooth sections of the vector bundle $V_{k,\ell}$  whose support is contained in the isolating neighborhood  $U_0$ of the attracting subset $\mathrm{Im}(e_u)$. 
The  semi-group of transfer operators associated to  $F^t_{k,\ell}$ is
\begin{equation}\label{eq:Lkl}
\cL_{k,\ell}^t:\Gamma^\infty(U_0,V_{k,\ell})\to \Gamma^\infty(U_0,V_{k,\ell}),\quad
\cL_{k,\ell}^t u(z)= F^t_{k,\ell}(u(f^{-t}_{G}(z))).
\end{equation}
The  flat trace of $\cL_{k,\ell}^t$ is computed as
\[
\flatTr\, \cL_{k,\ell}^t=
\sum_{\gamma\in \Gamma}\sum_{m=1}^\infty \!
\frac{ |\gamma|\cdot |\det (D_\gamma^u)|^{m/2}\cdot   \trace\, ((D_\gamma^u)^{-m})^{\wedge k}\cdot \trace\, ((D_\gamma^\perp)^{m})^{\wedge \ell}}{ |\det (D_\gamma^\perp)|^{m}\cdot |\det(\mathrm{Id}-D_\gamma^{-m})|\cdot |\det(\mathrm{Id}-((D_\gamma^\perp)^{-m})|}
\cdot \delta(t-m |\gamma|)
\]
where $D^\perp_\gamma$ is the restriction of the transversal Jacobian matrix for the prime periodic orbit $\hat{\gamma}(t)=e_u(\gamma(t))$ of the flow $f^t_G$ to  $V^{\perp}=\ker D\pi_G$. 
We define the dynamical Fredholm determinant of $\cL_{k,\ell}^t$  by
\begin{equation}\label{def:ds_ext}
d_{k,\ell}(s)=\exp\left(-\int_{+0}^\infty\frac{e^{-st}}{t}\cdot\flatTr\, \cL^t_{k,\ell} dt\right).
\end{equation}
Computation as in the last section then gives
\begin{equation}\label{eq:expression_zeta}
Z_{sc}(s)=\prod_{k=0}^{d}\prod_{\ell=0}^{d^2+d}d_{k,\ell}(s)^{(-1)^{k+\ell}}
\end{equation}
provided that $E_u^*$ is orientable. (Note that $d^2+d$ is even and recall Remark \ref{rem:orientation}.) 
\begin{Remark}\label{rm:comment_on_Grassman_ext}
This argument using the Grassmann extension resolves the problems related to non-smoothness of the coefficient of the transfer operators $\cL^t_k$ in the formal level.  
However the things are not that simple. The attracting section $e_u$ is not smooth and we will find some technical problems (and solutions to them) in the course of the argument. See Subsection \ref{ss:choice_local_chart} for instance.
\end{Remark}

The next theorem on the spectral property of the generators of one-parameter semi-groups $\cL^t_{k,\ell}$ is the main ingredient of this paper. 
\begin{Theorem} \label{th:spectrum}
Let $0\le k\le d$ and $0\le \ell\le d^2+d$. 
For each $r\ge 0$, there exists a Hilbert space 
\[
\Gamma^\infty(U_0,V_{k,\ell})\subset \widetilde{\cK}^r(U_0,V_{k,\ell}) \subset (\Gamma^\infty(U_0,V_{k,\ell}))'
\]
that consists of distributional sections of the vector bundle $V_{k,\ell}$ and, if $r>0$ is sufficiently large,  the 
 following claims hold true:
\begin{enumerate}
\item  The  one-parameter semi-group of operators $\cL_{k,\ell}^t$ for $t\ge 0$ extends to a strongly continuous semi-group of operators on $\widetilde{\cK}^r(U_0,V_{k,\ell})$ and  the spectral set of the generator
\[
A_{k,\ell}:\cD(A_{k,\ell})\subset  \widetilde{\cK}^r(U_0,V_{k,\ell})\to \widetilde{\cK}^r(U_0,V_{k,\ell})
\]
in the region $\{\Re(z)>-r\chi_0/4\}$ consists of discrete eigenvalues with finite multiplicity. 
These discrete eigenvalues (and their multiplicities)  are independent of the choice of $r$. 
\item For any $\tau>0$, there exist only finitely many eigenvalues of the generator $A_{k,\ell}$ on the region $\Re(s)>-(k+\ell)\chi_0 +\tau$.
\item  For the case $(k,\ell)=(0,0)$, the spectral set of $A_{0,0}$ is contained in the region
$U(\chi_0, \tau)$ defined in (\ref{eq:Uchiepsilon}), up to finitely many exceptions, for arbitrarily small $0<\tau<\chi_0$. Moreover there do exist countably many eigenvalues of $A_{0,0}$ in the strip $U_0(\tau)$ defined in (\ref{eq:Uzero}) and, for any $\delta>0$,  we have
\[
\frac{|\omega|^d}{C}<\frac{\#\{ \text{\rm eigenvalues of $A_{0,0}$ }|\text{ $|\Re(s)|<\tau$,  
$\omega\le \Im(s)\le \omega+  |\omega|^\delta$}\}}{  |\omega|^\delta}<C|\omega|^d
\]
for $\omega$ with sufficiently large absolute value, where $C>1$ is a constant independent of $\omega$. 
\end{enumerate}
\end{Theorem}

The next theorem gives the relation between the eigenvalues of the generator of the semi-group $\cL_{k,\ell}^t$ and zeros of  the dynamical Fredholm determinant $d_{k,\ell}(s)$. 
\begin{Theorem}\label{th:zeta}
The dynamical Fredholm determinant $d_{k,\ell}(s)$ of the one-parameter semi-group of transfer operators $\cL^t_{k,\ell}$ extends to a holomorphic function on the complex plane $\complex$.  
For any $c>0$, there exists $r_*>0$ such that, if $r\ge r_*$, the  zeros of the analytic extension of $d_{k,\ell}(s)$ coincide with the eigenvalues of the generator of 
$\cL^t_{k,\ell}:\widetilde{\cK}^r(U_0,V_{k,\ell})\to \widetilde{\cK}^r(U_0,V_{k,\ell})$  on the region $\Re(s)>-c$, including multiplicity. 
\end{Theorem}

Since the relation (\ref{eq:expression_zeta}) holds at least for $s\in \complex$ with sufficiently large real part, the main theorem (Theorem \ref{mainth:spec}) follows immediately from the two theorems above. 

\begin{Remark}\label{rem:vector_valued_case}
In the proofs of Theorem \ref{th:spectrum} and Theorem \ref{th:zeta}, we will mostly consider  the case $(k,\ell)=(0,0)$ and put a few remarks about the other cases $(k,\ell)\neq (0,0)$ in the course of the argument. Indeed the case $(k,\ell)=(0,0)$ is most important because the zeros of the semi-classical zeta function $Z_{sc}(s)$ along the imaginary axis correspond to the eigenvalues of the generator $A_{0,0}$ of the semi-group $\cL^t_{0,0}$. 
Note that we may and do regard $\cL^t_{0,0}$ as a scalar-valued transfer operator because the vector bundle $V_{0,0}$ is one-dimensional and trivial. 
In the cases where $(k,\ell)\neq (0,0)$, the transfer operators $\cL^t_{k,\ell}$ are vector-valued. But 
we can apply the parallel argument regarding the transfer operators $\cL^t_{k,\ell}$ as matrices of scalar-valued transfer operators as we noted at the end of Remark \ref{rm:dyntrace}. (Actually the argument is much simpler in the cases $(k,\ell)\neq (0,0)$, because they are irrelevant to the claim (3) in Theorem~\ref{th:spectrum}.) 
\end{Remark}

We finish this section by describing\footnote{But note that the rigorous argument in the proofs will somewhat deviate from the explanation here by technical reasons.} the ideas behind the proofs and the plan of the following sections. A basic idea in the proofs is to regard the transfer operators as ``Fourier integral operator", that is to say, to regard functions (or sections of vector bundles actually) as superposition of wave packets ({\it i.e.} functions concentrating both on the real and frequency spaces) and look how the action of the transfer operator transform one wave packet to another (or to a cloud of wave packets more precisely). 
For simplicity, let us consider the case $\cL^t_{0,0}$ of scalar-valued transfer operators. In our argument, the wave packets are parametrized by the points in the cotangent bundle $T^*U_0\subset T^*G$ and the transformation of the wave packets that $\cL^t_{0,0}$ induces is closely related to the map $(Df_G^{-t})^*:T^*U_0\to T^*U_0$, called the canonical map.  Notice that, since the flow $f^t$ preserves the contact form $\alpha$,  the action of the canonical map $(Df_G^{-t})^*$ preserves the submanifold 
\begin{equation}\label{eq:X}
X=\{s \cdot \pi_G^*(\alpha)(w)\in T^*U_0\mid s\in \real, w\in \mathrm{Im}(e_u)\}\subset T^{*}U_0,
\end{equation}
which is called the ``trapped set"\footnote{In terminology of dynamical system theory, this is nothing but the non-wandering set for the dynamics of $(Df_G^{-t})^*$.}, and the action on the outside of a small neighborhood of  $X$ is not recurrent as a consequence of hyperbolicity of the flow $f^t_G$. This fact suggests that, concerning the spectrum and trace, the most essential is the action of the transfer operators on the wave packets corresponding to the points in a small neighborhood of  $X$. 
This idea has been used in the previous papers \cite{Tsujii2009, Tsujii2010} and led to the results which essentially correspond to the claim (2) of Theorem \ref{th:spectrum}.
\begin{Remark}\label{Explanation_for_du}
If the reader is familiar with Dolgopyat argument \cite{Dolgopyat98}, the idea behind the claim (2) of Theorem \ref{th:spectrum} (and the reason for the factors $|\det(Df^t_x|_{E_u})|^{1/2}$ in the definitions of transfer operators)
may be understood roughly as follows. 
For simplicity, let us forget about the Grassmann extension for the moment and consider the simple transfer operator $u\mapsto u\circ f^{-t}$ on~$M$. The trapped set in such setting is the one-dimensional vector subbundle of $T^*M$ spanned by $\alpha$. 
Consider the situation where the transfer operator $\cL^t_{0}$ for $t\gg 1$ acts on wave packets that have high frequency in the direction of $\alpha$ with spatial size  $0<\delta\ll 1$ As usual in Dolgopyat argument, we suppose that this action of the transfer operator $\cL^t_{0}$ is followed by a smoothing (or averaging) operation\footnote{Here we do not explain why we apply this smoothing along stable foliation. Let us just note that the effect of this smoothing will decreases exponentially fast in further evolution.} along the stable foliation in the scale $\delta$. 
The last smoothing enlarge the supports of the images of wave packets in the stable direction by the rate proportional to $|\det(Df^t_x|_{E_s})|^{-1}\sim |\det(Df^t_x|_{E_u})|$. Hence, on the one  hand,  the $L^2$ norms of the images of wave packets decrease by the rate proportional to $|\det(Df^t_x|_{E_u})|^{-1/2}$  and, on the other hand, makes overlaps of the images of wave packets at points that were separated in the stable direction. The analysis of interference (or cancellation by difference of complex phase) between such overlapping images is equivalent to the essential part of Dolgopyat argument. In the papers \cite{Tsujii2009, Tsujii2010}, we showed basically that the overlapping images are \emph{almost} orthogonal to each other in $L^2$, using complete non-integrability of the contact form $\alpha$, and concluded that the essential spectral radius of the (simple) transfer operator is bounded by $\sup_{x}|\det(Df^t_x|_{E_u})|^{-1/2}$. 
With the same idea, we can show the claim corresponding to Theorem \ref{th:spectrum} (2) for $\cL^t_{0}$, because the coefficient of $\cL^t_{0}$ balances the rate $|\det(Df^t_x|_{E_u})|^{-1/2}$.  
\end{Remark}
\begin{figure}\label{fig:explanation_dolgopyat}
\scalebox{0.55}{
{
\begin{pspicture}(0,-3.08)(16.96,3.08)
\definecolor{colour0}{rgb}{0.8,0.8,0.8}
\pscircle[linecolor=black, linewidth=0.04, fillstyle=solid,fillcolor=colour0, dimen=outer](0.5,-2.58){0.5}
\pscircle[linecolor=black, linewidth=0.04, fillstyle=solid,fillcolor=colour0, dimen=outer](0.5,-0.86){0.5}
\pscircle[linecolor=black, linewidth=0.04, fillstyle=solid,fillcolor=colour0, dimen=outer](0.5,0.86){0.5}
\pscircle[linecolor=black, linewidth=0.04, fillstyle=solid,fillcolor=colour0, dimen=outer](0.5,2.58){0.5}
\psellipse[linecolor=black, linewidth=0.04, fillstyle=solid,fillcolor=colour0, dimen=outer](7.1,0.4574603)(2.22,0.082539685)
\psellipse[linecolor=black, linewidth=0.04, fillstyle=solid,fillcolor=colour0, dimen=outer](7.1,0.1658201)(2.22,0.082539685)
\psellipse[linecolor=black, linewidth=0.04, fillstyle=solid,fillcolor=colour0, dimen=outer](7.1,-0.1258201)(2.22,0.082539685)
\psellipse[linecolor=black, linewidth=0.04, fillstyle=solid,fillcolor=colour0, dimen=outer](7.1,-0.41746032)(2.22,0.082539685)
\psellipse[linecolor=black, linewidth=0.04, fillstyle=solid,fillcolor=colour0, dimen=outer](14.74,0.49746034)(2.22,0.6825397)
\psellipse[linecolor=black, linewidth=0.04, fillstyle=solid,fillcolor=colour0, dimen=outer](14.74,0.2758201)(2.22,0.5725397)
\psellipse[linecolor=black, linewidth=0.04, fillstyle=solid,fillcolor=colour0, dimen=outer](14.74,-0.015820106)(2.22,0.4725397)
\psellipse[linecolor=black, linewidth=0.04, fillstyle=solid,fillcolor=colour0, dimen=outer](14.74,-0.35746032)(2.22,0.46253967)
\psline[linecolor=black, linewidth=0.1, arrowsize=0.05291667cm 3.28,arrowlength=1.4,arrowinset=0.0]{->}(1.8,0.06)(3.9,0.08)
\psline[linecolor=black, linewidth=0.1, arrowsize=0.05291667cm 3.28,arrowlength=1.4,arrowinset=0.0]{->}(9.78,0.0)(11.88,0.02)
\rput[bl](2.42,0.94){{\huge$\cL^t$}}
\rput[bl](10,2.02){{\huge smoothing along}}
\rput[bl](10,1.52){{\huge stable foliation}}
\end{pspicture}
}}
\caption{A schematic picture for the explanation in Remark \ref{Explanation_for_du}. This is a picture viewed from the direction of the flow.  }
\label{figR}
\end{figure}
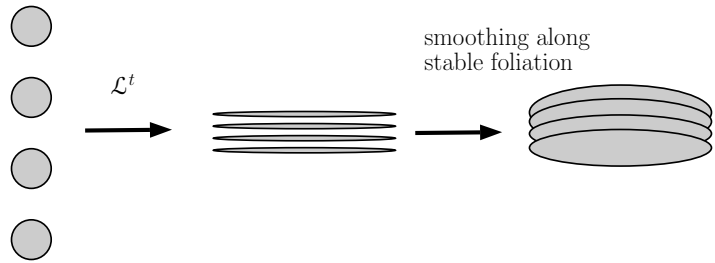

However, in order to get more information on the spectrum as described in the claim (3), we have to analyze more precisely  the action of transfer operators on the wave packets associated to the points in a neighborhood of  the trapped set $X$.  
Such action is modeled by the so-called ``prequantum map" and has already been studied in the paper \cite{Faure07} in the linear setting and then extended to the non-linear setting in the previous paper \cite{FaureTsujii12} of the authors. (The explanation in \cite[Section 2.3]{FaureTsujii12} will be useful to understand the idea that leads to the main theorems.) 
We are going to put the argument developed in those papers into the setting of contact Anosov flows. 
Note that, in \cite{Faure07, FaureTsujii12}, it was crucially important that the trapped set was an invariant  symplectic submanifold of the phase space and normally hyperbolic for the induced dynamical system on the phase space. In our setting of contact Anosov flows, this corresponds to the fact that the projection of the trapped set $X$ above to $T^*M$, 
\begin{equation}\label{eq:trapped_proj}
\check{X}=\{s \cdot \alpha(x)\in T^*M\mid s\in \real, x\in M\}\subset T^{*}M,
\end{equation}
is an invariant symplectic submanifold of $T^*M$ on the outside of the zero section and is normally hyperbolic with respect to the flow $(Df^{-t})^*$. 

\subsection*{Organization of this paper}
The remaining part of this paper is organized as follows. 
In the next section, Section \ref{sec:comments}, we make a few comments related to the main results.  
The proof of the main results starts from Section~\ref{sec:linear}. 
In Section~\ref{sec:linear}, we consider a linear model for contact Anosov flows and prove a proposition  corresponding to Theorem \ref{th:spectrum} in such model. 
The argument in this section is based on that in \cite{Faure07} and \cite{FaureTsujii12} for prequantum Anosov map and will serve as a guideline for the argument developed in the later sections.
In Section \ref{sec:basic},
we introduce some systems of local charts and associated partitions of unity on the Grassmann bundle $G$. Then, using them, we define the (modified) anisotropic Sobolev space $\cK^{r}(K_0)$ in Section \ref{sec:modifiedASS}. The Hilbert space $\widetilde{\cK^r}(K_0)$ appearing in Theorem \ref{th:spectrum} is a slight modification of  ${\cK^r}(K_0)$.
In Section \ref{sec:properties_Lt}, we give several propositions on the properties of the transfer operators $\cL^t=\cL^t_{0,0}$ on the Hilbert space $\cK^{r}(K_0)$. We expect that these propositions are easy to understand and intuitive for the readers. 
Then we prove that Theorem \ref{th:spectrum} follows from them in Section \ref{sec:proof1}. 
In Section \ref{sec:pre}--\ref{sec:proofs_of_props}, we give the proofs of the propositions given in Section \ref{sec:properties_Lt}. The argument in these sections is elementary but necessarily rather involved one because we have to deal with ``non-smooth" objects. The proof of Theorem \ref{th:zeta} (and that of a small lemma given in Section \ref{sec:proofs_of_props}) is put in Section \ref{sec:zeta} in the appendix. (See the remark below.)
\begin{Remark}
It is possible to derive Theorem \ref{th:zeta} from Theorem \ref{th:spectrum} by using a slight generalization of the existing results. In the paper \cite{GLP12}, it is proved that  
the dynamical zeta function $Z(s)$ for a general $C^\infty$ Anosov flow has meromorphic extension to the whole complex plane and that its zeros and poles are related to the discrete spectra of the generators for some associated transfer operators. To deal with the semi-classical zeta function and the Grassmann extension $f^t_G:G\to G$, we actually need a generalization of such results. This is not be difficult to obtain. 
Though a different kind of Banach spaces is used in \cite{GLP12}, it is possible to show that the discrete spectrum does not depend essentially on the function space. (See \cite[Appendix A]{BaladiTsujii08}.) 
We present a proof of  Theorem \ref{th:zeta} for completeness, based on the idea presented in \cite{BaladiTsujii08} in the case of hyperbolic discrete dynamical systems. 
\end{Remark}

\section{Comments}
\label{sec:comments}
\subsection{About this paper and a few related works of the authors}
A few years ago, the authors started the joint project studying transfer operators for geodesic flows on negatively curved manifolds (or more general contact Anosov flows) from the viewpoint of semi-classical analysis.  In the first paper \cite{FaureTsujii12} of the project, we considered  {\em prequantum Anosov maps} and studied the associated transfer operators in detail. A prequantum Anosov map is a $\mathbf{U}(1)$-extension of a symplectic Anosov map, equipped with a specific connection. 
It may be regarded as a model of contact Anosov flow because its  local structure is very similar to that of the time-$t$-maps of a contact Anosov flow and, in technical sense,  it is more tractable because 
the associated transfer operator is decomposed into the Fourier modes with respect to the $\mathbf{U}(1)$ action. (See \cite{FaureTsujii12} for more  explanation.)

In this paper and \cite{FaureTsujii13}, we extend the argument in \cite{FaureTsujii12}  to the contact Anosov flows. This paper concerns the results about the semi-classical transfer operators and also the semi-classical zeta functions. In the other paper \cite{FaureTsujii13}, we consider the band structure of the spectrum of the generators and also on the semi-classical aspect of the argument. (A part of the results in \cite{FaureTsujii13} has been announced in \cite{FaureTsujii13A}.)

\subsection{Recent related works} During the period the authors were writing this paper and the previous paper \cite{FaureTsujii12}, there have been some related developments. We give a few of them that came  into the authors' knowledge. Recently, Giulietti, Liverani and Pollicott published a paper \cite{GLP12} on dynamical zeta functions for Anosov flows. They proved among others that the dynamical zeta functions (including $Z(s)$ defined by Smale) has meromorphic extension to the complex plane $\complex$ if the flow is $C^\infty$ Anosov. 

In the proofs of the main theorems, we will regard the transfer operator as a
``Fourier integral operator" and consider its action in the limit of
high-frequency. (See \cite{FaureTsujii12} for more explanation.) Therefore the
main part of the argument is  naturally in the realm of 
semiclassical analysis. From this view point, the terminology and techniques
developed in semiclassical analysis must  be very  useful. (But this sounds
somewhat strange because the geodesic flow is completely a classical
object!). A first formulation of transfer operators and Ruelle spectrum in
terms of semiclassical analysis was given in the papers of the first author with N.\ Roy and
J.\ Sj\"ostrand \cite{FRS,FaureSjostrand11}. It was shown there that Ruelle resonances
are ``quantum resonances for a scattering dynamics in phase space''. 
Recently a few papers authored by K.\ Datchev, S.\ Dyatlov, S.\ Nonnenmacher and
M.\ Zworski gave precise
results for contact  Anosov flows  using this semiclassical approach: spectral gap estimate and decay of correlations \cite{NZ12}, Weyl law
upper bound \cite{DDZ13} and meromorphic properties of dynamical zeta function
\cite{DZ13}. We would like to mention also a closely related work: in \cite{D13}, for a
problem concerning decay of waves around black holes, S.\ Dyatlov  shows  that the
spectrum of resonances has a band structure similar to what we observe for
contact Anosov flows. In fact these two problems are very similar in the sense
that  in
both cases the trapped set is symplectic  and  normally hyperbolic. This
geometric property is the main reason for the existence of a band structure. However in  \cite{D13}, S.\ Dyatlov requires and uses some regularity of the
hyperbolic foliation that is not present for contact Anosov flows.

\subsection{Why do we consider the semi-classical zeta function?} \quad  
The semi-classical (or Gutzwiller-Voros) zeta function is related to the transfer operator with non-smooth coefficient if we do not consider the Grassmann extension. From this aspect, the semi-classical zeta function is a rather singular and difficult object to study. This may be one reason why the semi-classical zeta function has not been well studied in mathematics, at least compared with in physics. But here we would like to explain that the semi-classical zeta function is a very nice object to study among other kind of dynamical zeta functions. 

\subsubsection{Zeros along the imaginary axis}
In physics, there is a clear reason to study the semi-classical zeta function $Z_{sc}(s)$  rather than the zeta function $Z(s)$. 
The semi-classical zeta function appears in the semi-classical theory of quantum chaos in physics \cite{voros87,voros88,Rugh}. If we consider the semi-classical approximation of the kernel of the semi-group generated by the Schr\"odinger equation (or the wave equation) on a manifold $N$, we get the  Gutzwiller trace formula \cite{GutzwillerBook}. This formula is actually for some fixed range of time and for the limit where the Plank constant $\hbar$ goes to zero (or the energy goes to infinity). But, if we suppose that the formula holds for long time and if the long-time limit $t\to \infty$ and the semi-classical limit $\hbar\to 0$ were exchangeable, we would expect that the zeros of the semi-classical zeta function, which is  defined from the Gutzwiller trace formula, is closely related to the spectrum of the Laplacian on the manifold $N$. Thus the semi-classical zeta function is an object that connects the spectral structure of the quantized system (or the Schr\"odinger equation)  and the structure of the periodic orbits for the chaotic classical dynamical systems. For this reason, the semi-classical zeta function and its zeros have been discussed extensively in the field of ``quantum chaos". Of course, as any mathematician can imagine,  there is much difficulty in making such idea into rigorous argument. Still the semi-classical zeta function and its zeros are interesting objects to study.  To date, mathematically rigorous argument on semi-classical zeta function seems to be limited to the special case of constant curvature, where Selberg trace formula is available. To the authors' knowledge, Theorem \ref{mainth:spec}  is the first rigorous result for the semi-classical zeta function for the geodesic flows on manifolds with negative \emph{variable} curvature. We hope that our results will shed light on the related studies. 

\subsubsection{Cohomological argument}
There is some hope that we can relate the semi-classical zeta function $Z_{sc}(s)$ to transfer operators on some cohomological spaces (rather than those on the spaces of differential forms as in (\ref{eq:expression_zeta})) and get more precise results on its analytic properties. 
Though the idea is not new and may be well known, we would like to present it here and discuss how we may be able to go further than the main results of this paper. 
First of all, let us recall the following argument in the case of Anosov diffeomorphism. 
Let $f:M\to M$ be an Anosov diffeomorphism. 
The Artin-Mazur zeta function of $f$ is defined by
\[
\zeta(z)=\exp\left(-\sum_{n=1}^{\infty} \frac{z^n}{n} \# \mathrm{Fix}(f^n)\right)
\]
where $\# \mathrm{Fix}(f^n)$ denotes the number of fixed points for $f^n$. The flat trace of the transfer operator $\cL_k:\Omega^k(M)\to \Omega^k(M)$ associated to the natural action of $f$ on the space $\Omega^k(M)$ of $k$-forms is 
\[
\flatTr \cL_k=\sum_{p\in \mathrm{Fix}(f)}\frac{\trace\,  (Df_p)^{\wedge k}}{|\det(1-Df_p)| }
\]
and the dynamical Fredholm determinant of $\cL_k$ is defined by 
\[
\cD_k(z)=\exp\left(-\sum_{n=1}^\infty \frac{z^n}{n}\flatTr \cL_k^n\right)=\exp\left(-\sum_{n=1}^{\infty} \frac{z^n}{n} \sum_{p\in \mathrm{Fix}(f^n)}\frac{\trace\,  (Df^n_p)^{\wedge k}}{|\det(1-Df^n_p)| }\right).
\]
Then, similarly to (\ref{eq:expression_zeta}),  the Artin-Mazur zeta function is expressed\footnote{For simplicity, we assume that the stable and unstable subbundle are orientable and $f$ preserves their orientations.} as 
\begin{equation}\label{eq:AMzeta_exp}
\zeta(z)=\prod_{k=1}^{\dim M} \cD_k(z)^{(-1)^{\dim M-k}}.
\end{equation}
We can show  that the dynamical Fredholm determinants $\cD_k(z)$ are  entire functions and its zeros coincide with the reciprocals of  the discrete eigenvalues of the transfer operator $\cL_k$ acting on some Hilbert space. (See \cite{BaladiTsujii08}.) So  the Artin-Mazur zeta function $\zeta(z)$ is a meromorphic function on $\complex$. 
This argument is true for more general (Ruelle) dynamical zeta functions. But, for the Artin-Mazur zeta function $\zeta(z)$, we can simplify the argument as follows. Note that we have  the commutative diagram
\begin{equation}\label{eq:DR}
\begin{CD}
0@>>> \Omega^0 @>{d}>> \Omega^1@>{d}>> \cdots @>{d}>> \Omega^{\dim M}@>>> 0\\
@. @VV{\cL_0}V @VV{\cL_1}V @. @VV{\cL_{\dim M}}V @.\\
0@>>> \Omega^0 @>{d}>> \Omega^1@>{d}>> \cdots @>{d}>> \Omega^{\dim M}@>>> 0.
\end{CD}
\end{equation}
This tells that many of the discrete eigenvalues of $\cL_k$ for adjacent $k$'s coincide and the corresponding zeros of the dynamical Fredholm determinant $\cD_k(z)$ cancel each other in the alternative product  (\ref{eq:AMzeta_exp}). The remaining zeros and poles of the zeta function $\zeta(z)$ should correspond to the eigenvalues of the (push-forward) action of $f$ on the de Rham cohomology $H^*_{DR}(M)$ of $M$. Indeed we can actually count the number of periodic points using the Lefschetz fixed point formula and show that $\zeta(z)$ is a rational function. (See \cite{Smale}.)

For Anosov flows, the corresponding argument become much more subtle and only much less is known. The argument using the flat trace and dynamical Fredholm determinant work as well, as we discussed in the previous sections. But, for the moment, we do not know whether it can be simplified as in the case of Anosov diffeomorphism. This is not a new problem and there are many works on this subject especially in the case of geodesic flows on hyperbolic surfaces. (See the introduction chapter of \cite{Juhl} for instance.) But there seems no much argument in more general cases.  

Note that the Artin-Masur zeta functions was special because the corresponding transfer operators commute with the exterior derivatives $d$. Indeed, if we consider other (Ruelle) zeta functions, the corresponding transfer operator will not have this property and hence the structure of the zeta functions will be much more complicated. The similar will be true in the case of Anosov flows, so that we will have to choose a ``good" dynamical zeta function, although we do not know whether there do exists such choice. For instance, let us consider the Smale's zeta function $Z(s)$. As is presented in \cite{GLP12} (and in many other places), it is expressed as follows. 
Let $f^t:M\to M$ be a contact Anosov flow. 
Let $\Omega^{k}_{\perp}(M)$ be the space of $k$-forms on $M$ which vanish for the generating vector field of the flow. If we write $d^{\perp}_k(s)$ for the dynamical Fredholm determinant of the natural  (push-forward) action  of the flow on $\Omega^{k}_{\perp}(M)$, we have
$Z(s)=\prod_{k=0}^{2d} (d^{\perp}_k(s))^{(-1)^k}$. 
But, unfortunately, the exterior derivative $d$ does not preserves the space $\Omega^{k}_{\perp}\subset \Omega^k(M)$. This is one reason why we can not apply the cohomological argument to the zeta function $Z(s)$. (Of course there is possibility that some better expression of the zeta function $Z(s)$ works.) 

Let us now turn to the case of the semi-classical zeta function $Z_{sc}(s)$. For simplicity, we assume that $E_s$ is orientable and the the stable foliation is smooth. (The latter is a strong assumption.) 
As we discussed in Section \ref{sec:intro},  $Z_{sc}(s)$ is expressed as an alternative product (\ref{expression_Zsc}) of dynamical Fredholm determinants for the transfer operators on differential forms. 
But here we consider  in a slightly different way.  
Let $\pi:L\to M$ be the line bundle $L=(E_u)^{\wedge d}$. Since $E_u$ is assumed to be orientable, $L$ is trivial and therefore we can consider the square root $L^{1/2}$ of  $L$. Note that there is a natural dynamically defined connection \emph{along the unstable manifolds} on the line bundles $L$ and $L^{1/2}$, in which two elements $\sigma,\sigma'\in L$ (or $L^{1/2}$) on an unstable manifold are parallel if and only if the ratio between $(Df^{t})^*(\sigma)$ and $(Df^{t})^*(\sigma')$ (considered in some local chart) converges to $1$ as $t\to +\infty$. By definition, these connections are flat (along unstable manifolds) and preserved by the natural action of the flow $f^t$. 
Let $\Lambda^k$ be the space of smooth sections of the vector bundle $L^{1/2}\otimes (E_u^*)^{\wedge k}$. 
In other words, this is the space of differential $k$-forms along the stable foliation that take values in $L^{1/2}$.  Let $\cL^t_k:\Lambda^k\to \Lambda^k$ be the natural (push-forward) action of the flow $f^t$. 
Then these transfer operators are equivalent to those in Section \ref{sec:intro} denoted by the same symbol and therefore the expression 
(\ref{eq:Expression_of_Zsc}) holds with $d_k(s)$ the dynamical Fredholm determinant for $\cL^t_k$ defined above. One definitely better fact in this expression is that we have the commutative diagram:
\begin{equation}\label{eq:DR2}
\begin{CD}
0@>>> \Lambda^0 @>{D_u}>> \Lambda^1@>{D_u}>> \cdots @>{D_u}>> \Lambda^{d}@>>> 0\\
@. @VV{\cL^t_0}V @VV{\cL^t_1}V @. @VV{\cL^t_{\dim M}}V @.\\
0@>>> \Lambda^0 @>{D_u}>> \Lambda^1@>{D_u}>> \cdots @>{D_u}>> \Lambda^{d}@>>> 0
\end{CD}
\end{equation}
where $D_u$ is the covariant exterior derivative along the stable manifolds. (Again this observation is not new. For instance, we can find it in the paper \cite{Guillemin} by Guillemin.) 
We therefore expect that large part of the zeros of the dynamical Fredholm determinant $d_k(s)$ will cancel each other in the expression (\ref{eq:Expression_of_Zsc}). In fact, under some strong assumptions on smoothness of the unstable foliation, it seems possible to prove this. But, for more general contact Anosov flows, it is not clear whether we can set up appropriate Hilbert spaces as completions of $\Lambda^k$ so that the commutative diagram above is extended to them.  Also it is not clear to what extent the cancellation between zeros will be complete. Still, to be optimistic, we would like to put the following conjecture. 
\begin{Conjecture}
The semi-classical zeta function $Z_{sc}(s)$ for contact Anosov flows will have  {\em holomorphic} extension to the whole complex plane $\complex$. 
Its zeros will be contained in the region 
$
\{ z\in \complex\mid |\Re(s)|\le \tau \mbox{ or } |\Im(s)|\le C\}
$ 
for some $C>0$ and arbitrarily small $\tau>0$ up to finitely many exceptions.
\end{Conjecture}

\section{Linear models}\label{sec:linear}
 In this section, we discuss about a one-parameter family of partially hyperbolic linear transformations. This is a linearized  model of  the Grassmann extension $f_G^t$ of the flow $f^t$ viewed in local coordinate charts. The main statement, Theorem \ref{cor:linear1}, of this section is a prototype of  Theorem \ref{th:spectrum}. The idea presented below is initially given in \cite{Faure07} and the following is basically a restatement of the results there in a modified setting and in a different terminology.
We have given a very similar argument in our previous paper \cite[Chapter 3 and 4]{FaureTsujii12} on prequantum Anosov maps. Since the argument there is self-contained and elementary, we will  refer \cite{FaureTsujii12} for the proofs of some statements and also for more detailed explanations.  
\subsection{A linear model for the flow $f^t_G$}\label{ss:linear_model}
\subsubsection{Euclidean space and coordinates}
Let us consider the Euclidean space  
\[
\real^{2d+d'+1}=\real^{2d}\oplus \real^{d'}\oplus \real
\]
as a local model of the Grassmann bundle $G$, where we suppose that the component $\real^{d'}$ in the middle is the fiber of the Grassmann bundle and the last component $\real$ is the flow direction. 
We equip the space $\real^{2d+d'+1}$ with the coordinates 
\begin{equation}\label{eq:xyz}
(x,y,z)\qquad \mbox{ with $x\in \real^{2d}$, $y\in \real^{d'}$ and $z\in \real$.}
\end{equation}
The first component  $x\in \real^{2d}$ is sometimes written 
\begin{equation}\label{eq:xpq}
x=(q,p)\quad\mbox{ with $q,p\in \real^{d}$.}
\end{equation}
We suppose that the $q$-axis and the $p$-axis are respectively  the expanding and contracting subspaces. 
Also we sometimes write the coordinates (\ref{eq:xyz}) above  as 
\begin{equation}\label{eq:w}
(w,z) \qquad \mbox{with setting }w=(x,y)\in \real^{2d+d'}
\end{equation}
for simplicity. 
In order to indicate which coordinate is used on which component, we sometimes use such notation as
\[
\real^{2d+d'+1}_{(x,y,z)}=\real^{2d}_x\oplus \real^{d'}_y\oplus \real_z=\real^{d}_q\oplus \real^{d}_p\oplus \real^{d'}_y\oplus \real_z=
\real^{2d+d'}_w\oplus \real_z.
\]
The orthogonal projections to some of the components are written as follows:
\begin{equation}\label{eq:pj}
\proj_{(x,z)}\!:\real^{2d+d'+1}_{(x,y,z)}\to \real^{2d+1}_{(x,z)}, \;\; \proj_{(x,y)}\!:\real^{2d+d'+1}_{(x,y,z)}\to \real^{2d+d'}_{(x,y)}, \;\;
 \proj_{x}\!:\real^{2d+d'+1}_{(x,y,z)}\to \real^{2d}_{x}
\end{equation}
We suppose that the space $\real^{2d+1}_{(x,z)}=\real^{2d+1}_{(q,p,z)}$ is equipped with the contact form
\begin{equation}\label{eq:walpha}
\alpha_0=dz-q dp+p dq.
\end{equation}

\subsubsection{Partially hyperbolic linear transformations} 
Let us consider invertible  linear transformations $A:\real^{d}\to \real^{d}$ and $\wA:\real^{d'}\to \real^{d'}$ and suppose that they are expanding and contracting respectively in the sense that  
\begin{equation}\label{eq:lambda}
\|A^{-1}\|<\frac{1}{\lambda}\quad\mbox{and}\quad \|\wA\|<\frac{1}{\lambda}\quad\mbox{for some constant $\lambda\ge 1$. }
\end{equation}
The transpose of the inverse of $A$ will be written as
\begin{equation}\label{eq:Adag}
A^\dag:=({}^tA)^{-1}:\real^d\to \real^d.
\end{equation}
In the following, we investigate the one-parameter family of partially hyperbolic affine transformations  
\begin{equation}\label{eq:Bt}
B^t:\real^{2d+d'+1}_{(x,y,z)}\to \real^{2d+d'+1}_{(x,y,z)},\qquad
B^t(q,p, y,z)=(Aq, A^\dag p, \wA y,z+t),
\end{equation}
as a model of the family of diffeomorphisms $f_G^{t+t_0}$ for $t_0\gg 0$ viewed in flow-box 
coordinate charts\footnote{We will take flow-box coordinate charts $\kappa$ and $\kappa'$ around a point $P$ and its image $f^{t_0}_G(P)$ respectively, so that the $q$-axis and $p$-axis corresponds to the unstable and stable subspace,  and consider the family of maps $(\kappa')^{-1}\circ f^{t+t_0}_G\circ \kappa$. Then its linearization will look like $B^t$.}. Observe that $B^t$ preserves the one form $(\proj_{(x,z)})^*\alpha_0$. Below we consider the one-parameter family of transfer operators 
\begin{equation}\label{eq:Lt}
L^t:C^\infty(\real^{2d+d'+1})\to C^\infty(\real^{2d+d'+1}), \quad L^t u(w)=\frac{|\det A|^{1/2}}{|\det\wA|}\cdot u(B^{-t}(w)).
\end{equation}
\begin{Remark}\label{Rem:Lt}
We ask the readers to check that the coefficient ${|\det A|^{1/2}}/{|\det\wA|}$  is chosen so that $L^t$ is an appropriate model of  the transfer operator $\cL^t_{0,0}$ considered in Section \ref{sec:Gext}. 
See (\ref{eq:Fkl}), (\ref{eq:bt}) and  (\ref{eq:Lkl}) for the definitions. 
\end{Remark}
\subsection{Bargmann transform}
\subsubsection{Definition}
We will employ the  partial Bargmann transform for analysis of the transfer operators. This is a kind of wave-packet transform. To begin with,  
we recall the definition of the (usual) Bargmann transform and its basic properties.  We refer \cite[Chapter 3]{FaureTsujii12} and \cite{FollandBook} for more detailed accounts. 

Let us consider the $D$-dimensional Euclidean space $\real^D_{w}$ and its cotangent bundle 
\[
T^*\real^D_w=\real^{2D}_{(w,\xi_w)}=\real^{D}_w\oplus \real^{D}_{\xi_w},
\]
where we regard $\xi_w\in \real^D$ as the dual variable of $w\in \real^D$. 
Let $\hbar>0$ be a parameter that is related to the sizes of  wave packets. For each point $(w, \xi_w)\in T^*\real^D_w$, we assign a Gaussian wave packet 
\begin{equation}\label{eq:phiD}
\phi_{w,\xi_w}(w')=a_D(\hbar)\cdot \exp(i \xi_w\cdot  (w'-(w/2))/\hbar-|w'-w|^2/(2\hbar))
\end{equation}
where $a_D(\hbar)$ is a normalization constant defined by
\begin{equation}\label{def:ad}
a_D(\hbar)=(\pi\hbar)^{-D/4}.
\end{equation}
The Bargmann transform on $\real^{D}_w$ (for the parameter $\hbar>0$) is defined by 
\begin{equation}\label{eq:Bargmann_transform_h}
\Bargmann_\hbar:L^2(\real_w^{D})\to L^2(\real_{(w,\xi_w)}^{2D}), \quad \Bargmann_\hbar u(w',\xi'_w)=\int \overline{\phi_{w',\xi'_w}(w)}\cdot  u(w) dw.
\end{equation}
Its $L^2$-adjoint $\Bargmann^*_\hbar:L^2(\real^{2D}_{(w,\xi_w)})\to L^2(\real^{D}_w)$ is given as
\[
\Bargmann_\hbar^* v(w')=\int \phi_{w,\xi_w}(w') v(w,\xi_w) \frac{dw d\xi_w}{(2\pi \hbar)^D}.
\]
Here we make a convention that we use the volume form $dw d\xi_w/(2\pi \hbar)^D$ in defining the $L^2$-norm on $L^2(\real^{2D}_{(w,\xi_w)})$. Then we have
\begin{Lemma}[{\cite[Lemma 3.1.2]{FaureTsujii12}}]\label{lm:Bargmann_basic} 
The Bargmann transform $\Bargmann_\hbar$ is an $L^2$-isometric embedding. 
Its adjoint $\Bargmann_\hbar^*$ is a bounded operator with respect to the $L^2$ norm and satisfies $\Bargmann_\hbar^*\circ \Bargmann_\hbar=\mathrm{Id}$. 
\end{Lemma}
The last claim implies that  $u\in  L^2(\real_w^{D})$ is expressed as a superposition (or an integration) of the wave packets  $\phi_{w,\xi_w}(\cdot)$  for $(w,\xi_w)\in \real^{2D}_{(w,\xi_w)}$:
\[
u(w')=\Bargmann_\hbar^*\circ \Bargmann_\hbar u(w')=\int \phi_{w,\xi_w}(w') v(w,\xi_w) \frac{dw d\xi_w}{(2\pi \hbar)^D}\qquad \mbox{with setting  $v=\Bargmann_{\hbar} u$.}
\]
\begin{Lemma}[{\cite[Proposition 3.1.3]{FaureTsujii12}}] 
The operator 
\begin{equation}\label{eq:Bargmann_projector_h}
\BargmannP_\hbar=\Bargmann_\hbar\circ \Bargmann_\hbar^*: L^2(\real^{2D}_{(w,\xi_w)})\to L^2(\real^{2D}_{(w,\xi_w)})
\end{equation}
is an orthogonal projection onto the image of $\Bargmann$ and called the {\em Bargmann projector}.  It is expressed as an integral operator 
\[
\BargmannP_\hbar v(w',\xi'_w)=\int K_{\BargmannP,\hbar}(w',\xi'_w;w,\xi_{w}) v(w,\xi_w)  \frac{dw d\xi_w}{(2\pi \hbar)^D}
\]
with the kernel
\begin{equation}\label{eq:kernel_KP}
K_{\BargmannP,\hbar}(w',\xi'_w;w,\xi_{w})=e^{-i \Omega((w',\xi'_w),(w, \xi_w))/(2\hbar)-|(w',\xi'_w)-(w,\xi_w)|^2/(4\hbar)},
\end{equation}
where  $\Omega((w',\xi'_w), (w, \xi_w))=\xi'_w w-w' \xi_w$ is the standard symplectic form on $\real^{2D}_{(w,\xi_w)}$. 
\end{Lemma}
\subsubsection{Lift of transfer operators with respect to the Bargmann transform}
Let  $Q:\real_w^{D}\to \real_w^{D}$ be an invertible affine transformation. 
Let $Q_0:\real^{D}_w\to \real_w^{D}$ be its linear part and $q_0:=Q(0)\in \real^D$ be the constant part. Let $L_Q:L^2(\real_w^D)\to L^2(\real_w^D)$ be the $L^2$-normalized transfer operator defined by 
\begin{equation}\label{eq:normalizedOp}
L_Q u(w)=|\det Q_0|^{-1/2}\cdot  u(Q^{-1}w).
\end{equation}
We call the operator 
\[
L_{Q}^{\lift}:=\Bargmann_\hbar \circ L_Q\circ \Bargmann_\hbar^*:L^2(\real^{2D}_{(w,\xi_w)})\to L^2(\real^{2D}_{(w,\xi_w)})
\]
the lift of the operator $L_Q$ with respect to the Bargmann transform $\Bargmann_\hbar$, as it makes the following diagram commutes
\[
\begin{CD}
L^2(\real^{2D}_{(w,\xi_w)})@>{L_Q^{\lift}}>> L^2(\real^{2D}_{(w,\xi_w)})\\
@A{\Bargmann_\hbar}AA @A{\Bargmann_\hbar}AA\\
L^2(\real_w^D)@>{L_Q}>> L^2(\real_w^D).
\end{CD}
\]
The next lemma gives a useful expression of the lift $L_{Q}^{\lift}$. 
We consider the natural (push-forward) action of $Q$ on 
the cotangent bundle $T^*\real^{D}_w=\real^{2D}_{(w,\xi_w)}$:
\[
D^\dag Q:\real_{(w,\xi_w)}^{2D}\to \real_{(w,\xi_w)}^{2D},\quad 
D^\dag Q(w,\xi_w)=(Qw, Q_0^{\dag} \xi_w)=(Qw, ({}^tQ_0)^{-1} \xi_w).
\]
Let $L_{D^\dag Q}$ be the associated ($L^2$-normalized) transfer operator, which is defined by (\ref{eq:normalizedOp}) with $A$ replaced by $D^\dag Q$, that is, $L_{D^\dag Q}u:=u\circ (D^\dag Q)^{-1}$.
\begin{Lemma}[{\cite[Lemma 3.2.2 and Lemma 3.2.4]{FaureTsujii12}}] \label{lm:lift}
The lift $L_{Q}^{\lift}$  is expressed as 
\[
L_{Q}^{\lift}v(w,\xi_w)
=d(Q)\cdot e^{-i\xi_w q_0/(2\hbar)}\cdot  \BargmannP_\hbar\circ  L_{D^\dag Q}\circ \BargmannP_\hbar v(w,\xi_w)
\]
where $d(Q)=|\det((Q_0+{}^t Q_0^{-1})/2)|^{1/2}$. If $Q$ is isometric, we have
 $[L_{D^\dag Q}, \BargmannP_\hbar]=0$ and therefore 
 $\BargmannP_\hbar\circ  L_{D^\dag Q}\circ \BargmannP_\hbar=  L_{D^\dag Q}\circ \BargmannP_\hbar=\BargmannP_\hbar\circ  L_{D^\dag Q}$.
\end{Lemma}
\subsection{Partial Bargmann transform}

\subsubsection{Definition}
The partial Bargmann transform, which we will use in later sections, is  roughly the Fourier transform along the $z$-direction in $\real^{2d+d'+1}_{(x,y,z)}$ combined with the Bargmann transform $\Bargmann_\hbar$ in the transverse directions, where the parameter $\hbar$ is related to the frequency $\xi_z$ in the $z$-direction as $
\hbar=\langle \xi_z\rangle^{-1}$. Here and henceforth, we let $\langle s\rangle$ be a smooth function of $s\in \real$ such  that $\langle s\rangle =|s|$ if $|s|\ge 2$ and that  $\langle s\rangle\ge 1$ for all $s\in \real$.

As the (partial) cotangent bundle  of  $\real^{2d+d'+1}_{(x,y,z)}$, we consider the Euclidean space $\real^{4d+2d'+1}_{(x,y,\xi_x, \xi_y,\xi_z)}$ equipped with the coordinates
\[
(x,y,\xi_x, \xi_y,\xi_z)\quad \mbox{ with $x,\xi_x\in \real^{2d}$,\; $y,\xi_y\in \real^{d'}$,\; $\xi_z\in \real$. }
\]
We regard it as the cotangent bundle of the Euclidean space $\real^{2d+d'+1}_{(x,y,z)}$, where 
 $\xi_x$, $\xi_y$, $\xi_z$ are regarded as the dual variable of $x$, $y$, $z$ respectively. 
 (But notice that we omit the variable $z$. This is because we consider the Fourier transform along the $z$-axis.) For simplicity, we  sometimes write the coordinates above as 
\[
(w,\xi_w,\xi_z)\quad \mbox{ with setting
$w=(x,y),\; \xi_w=(\xi_x,\xi_y)$.}
\]
Also, according to (\ref{eq:xpq}), we sometimes write the coordinate $\xi_x\in \real^{2d}$ as
\[
\xi_x=(\xi_q,\xi_p) \quad \mbox{with $\xi_q,\xi_p\in \real^d$.}
\]
Instead of the functions $\phi_{w,\xi_w}(\cdot)$ in (\ref{eq:phiD}), we consider the functions
\begin{equation}\label{eq:pBwp}
\phi_{x,y,\xi_x, \xi_y,\xi_z}:\real^{2d+d'+1}_{(x,y,z)}\to \complex\quad \mbox{for}\quad (x,y,\xi_x, \xi_y,\xi_z)\in \real^{4d+2d'+1}
\end{equation}
defined by 
\begin{align*}
&\phi_{x,y,\xi_x, \xi_y,\xi_z}(x',y',z')\\
&=a_{2d+d'}(\langle\xi_z\rangle^{-1})\cdot \exp
\left(i\xi_z z' +i\langle \xi_z\rangle \,\xi_w  (w'-(w/2))-\langle \xi_z\rangle\cdot\frac{ |w'-w|^2}{2} \right)
\\
&=
a_{2d+d'}(\langle\xi_z\rangle^{-1})\cdot \exp
\left(i\xi_z z' +i\langle\xi_z\rangle\,\xi_x (x'-(x/2))+i\langle\xi_z\rangle\,\xi_y (y'-(y/2))\right)\\
&\qquad \qquad \cdot \exp\left(-\langle \xi_z\rangle\cdot\frac{ |x'-x|^2}{2}
-\langle \xi_z\rangle\cdot \frac{|y'-y|^2}{2} \right)
\end{align*}
where $a_{D}(\cdot)$ is that in (\ref{def:ad}).  For brevity, we sometimes write $\phi_{w,\xi_w, \xi_z}(w',z')$ for $\phi_{x,y,\xi_x, \xi_y,\xi_z}(x',y',z')$. 
\begin{Remark}\label{Rem:rescaledCoordinates}
Note that $\xi_z$ in (\ref{eq:pBwp}) indicates the frequency of  $\phi_{x,y,\xi_x, \xi_y,\xi_z}(\cdot)$ in $z$, and that the frequency in $w=(x,y)$ is actually $\langle \xi_z\rangle \xi_w=\langle \xi_z\rangle (\xi_x,\xi_y)$, that is, it is rescaled by the factor $\langle \xi_z\rangle$.
\end{Remark}
The  partial Bargmann transform 
\[
\pBargmann:L^2(\real_{(x,y,z)}^{2d+d'+1})\to L^2(\real_{(x,y,\xi_x, \xi_y,\xi_z)}^{4d+2d'+1})
\]
is defined by
\begin{equation}\label{eq:volume}
\pBargmann u(x',y',\xi'_x, \xi'_y,\xi'_z)=\int \overline{\phi_{x',y',\xi'_x, \xi'_y,\xi'_z}(x,y,z)}\cdot u(x,y,z) dx dy dz.
\end{equation}
We make a convention that we use the volume form 
\begin{equation}\label{eq:constNormalization}
d\vol= (2\pi)^{-1} \cdot (2\pi \langle \xi_z\rangle^{-1})^{-2d-d'} dx dy d\xi_x d\xi_y d\xi_z
\end{equation}
in defining the $L^2$-norm on $L^2(\real_{(x,y,\xi_x, \xi_y,\xi_z)}^{4d+2d'+1})$.  Then the  $L^2$-adjoint 
\[
\pBargmann^*: L^2(\real_{(x,y,\xi_x, \xi_y,\xi_z)}^{4d+2d'+1}) \to L^2(\real^{2d+d'+1}_{(x,y,z)})
\]
of the partial Bargmann transform $\pBargmann$ is the operator given by
\begin{equation}\label{eq:pBargmannAdj}
\pBargmann^* v(x',y',z')=\int  \phi_{x,y,\xi_x, \xi_y,\xi_z}(x',y',z') v(x,y,\xi_x,\xi_y,\xi_z)  
d\vol.
\end{equation}
\subsubsection{Basic properties of the partial Bargmann transform}
The following is a basic property of the partial Bargmann transform $\pBargmann$, which follows from those of the Bargmann transform and the Fourier transform.
\begin{Lemma} \label{lm:pBargmannL2}
The partial Bargmann transform $\pBargmann$ is an $L^2$-isometric injection and $\pBargmann^*$ is a bounded operator such that $\pBargmann^*\circ \pBargmann =\mathrm{Id}$. 
The composition  
\begin{equation}\label{eq:pBargmannP}
\pBargmannP:=\pBargmann\circ \pBargmann^*:L^2(\real_{(x,y,\xi_x, \xi_y,\xi_z)}^{4d+2d'+1})\to L^2(\real_{(x,y,\xi_x, \xi_y,\xi_z)}^{4d+2d'+1})
\end{equation}
is the $L^2$ orthogonal projection onto the image of $\pBargmann$.
\end{Lemma}

Suppose that  $B:\real^{2d+d'+1}_{(w,z)}\to \real^{2d+d'+1}_{(w,z)}$ is an affine transform of the form
\[
B(w,z)=(B_0(w)+b_0, z+C_0(w)+c_0)
\]
where $B_0:\real^{2d+d'}_{w}\to \real^{2d+d'}_{w}$ and $C_0:\real^{2d+d'}_{w}\to \real$ are linear maps and $b_0$ and $c_0$ are constants. (The linear model $B^t$ in (\ref{eq:Bt}) is a special case of such maps.)  Let $D^\dag B:\real^{4d+2d'+1}_{(w,\xi_w,\xi_z)}\to \real^{4d+2d'+1}_{(w,\xi_w,\xi_z)}$ be the naturally induced (push-forward) action  
\[
D^\dag B(w,\xi_w,\xi_z)=(B_0(w)+b_0, \;B_0^\dag(\xi_w- {}^t C_0 \xi_z),\; \xi_z)
\] 
on the partial cotangent bundle $\real^{4d+2d'+1}_{(w,\xi_w,\xi_z)}$. 
We consider the $L^2$-normalized transfer operators $L_B$ and $L_{D^\dag B}$ defined in (\ref{eq:normalizedOp}) with $A$ replaced by $B$ and $D^{\dag}B$ respectively. The lift $L_B^{\lift}$ of the operator $L_B$ with respect to the partial Bargmann transform $\pBargmann$ is defined by 
\[
L_B^{\lift}:=\pBargmann\circ L_B\circ \pBargmann^*:\real^{4d+2d'+1}_{(w,\xi_w,\xi_z)}\to \real^{4d+2d'+1}_{(w,\xi_w,\xi_z)}
\]
and makes the following diagram commutes:
\[
\begin{CD}
L^2(\real^{4d+2d'+1})@>{L_B^{\lift}}>> L^2(\real^{4d+2d'+1})\\
@A{\pBargmann}AA @A{\pBargmann}AA\\
L^2(\real^{2d+d'+1})@>{L_B}>> L^2(\real^{2d+d'+1}).
\end{CD}
\]
The next lemma is a consequence of  Lemma \ref{lm:lift} and gives an expression of  $L_B^{\lift}$. 
\begin{Lemma}\label{lm:expression_lift} 
The  lift $L_B^{\lift}:=\pBargmann\circ L_B\circ \pBargmann^*$  is expressed as
\[
L_B^{\lift}v(w,\xi_w,\xi_z)=d(B_0)\cdot e^{-(i\langle \xi_z\rangle/2)\xi_w \cdot b_0-i\xi_z\cdot  (C_0 B_0^{-1}(w-b_0)+c_0)}
\cdot \pBargmannP \circ {L}_{D^\dag B}\circ \pBargmannP 
v(w,\xi_w,\xi_z)
\]
If $B$ is isometric, we have $[\pBargmannP, L_{D^\dag B}]=0$ and therefore 
\begin{align*}
L_B^{\lift}v
&= e^{-(i\langle \xi_z\rangle/2)\xi_w \cdot b_0-i\xi_z\cdot (C_0B_0^{-1}(w-b_0)+c_0)}
\cdot  {L}_{D^\dag B}\circ \pBargmannP 
v(w,\xi_w,\xi_z).
\end{align*}
\end{Lemma}
\begin{Remark}\label{Rem:expression_lift}
The relation between (the lift of) the transfer operator $L_B$ and its canonical map 
$D^\dag B$ given in the lemma above realizes the idea explained in the latter part of Section \ref{sec:Gext} in a simple setting. Here note that the kernel of the partial Bargmann projector $\pBargmannP$ is of the form $k(w',\xi'_w;w,\xi_w)\cdot \delta(\xi'_z-\xi_z)$ and we have 
\[
k(w',\xi'_w;w,\xi_w)\le C_\nu \langle \langle \xi_z\rangle^{1/2} |w-w'|\rangle^{-\nu}\cdot \langle \langle \xi_z\rangle^{1/2}|\xi_w-\xi'_w|\rangle^{-\nu}
\]
for arbitrarily large $\nu>0$, where $C_\nu$ is a constant depending on $\nu$.
\end{Remark}

\subsection{A coordinate change on the phase space}\label{ss:coord}
\subsubsection{The lift of the transfer operators $L^t$ and the trapped set}
Let us now consider the family of transfer operators $L^t$ defined in (\ref{eq:Lt}) and its lift with respect to the partial Bargmann transform:
\begin{equation}\label{eq:Ltlift}
(L^t)^{\lift}=\pBargmann\circ L^t\circ \pBargmann^*.
\end{equation}
Below we keep in mind that $L^t$ is a model of the transfer operator $\cL^t_{0,0}$ viewed in the local coordinate charts. 
As we explained at the end of Section \ref{sec:Gext}, we mainly consider the action of 
the transfer operator $\cL^t$ on 
 the wave packets (with high frequency) corresponding to the points near the trapped set $X$ given in (\ref{eq:X}). 
In our linear model, we understand that the hyperplane 
\[
\{(x,y,z)\in \real^{2d+d'+1}\mid y=0\}\subset \real^{2d+d'+1}_{(x,y,z)}
\]
corresponds to the section $e_u$ in the global setting. Then the trapped set $X_0$, which corresponds to $X$ in (\ref{eq:X}), must be
\[
X_0:=\{\mu\cdot \proj_{(x,z)}^* \alpha_0(x,0,z)\mid \mu \in \real,\;(x,z)\in \real^{2d+1}\}.
\]
Since we consider the rescaled coordinates,  as mentioned in Remark \ref{Rem:rescaledCoordinates}, the subset $X_0$ is given by the equations
\begin{equation}\label{eq:X0} 
\langle \xi_z\rangle \xi_p=-\xi_z q, \quad \langle \xi_z\rangle\xi_q=\xi_z p,\quad  y=0,\quad  \xi_y=0,\quad  \xi_z\neq 0
\end{equation}
in the coordinates $(q,p,y,\xi_q,\xi_p,\xi_z)$ on $\real^{4d+2d'+1}$.

By Lemma \ref{lm:expression_lift}, the lift $(L^t)^{\lift}$ of  $L^t$ is expressed as a composition of the transfer operator ${L}_{D^\dag B^t}$ with the Bargmann projectors $\pBargmannP$ and some multiplication operator, where $D^\dag B^t$ is the linear map 
\[
D^\dag B^t:\real^{4d+2d'+1}\to \real^{4d+2d'+1},\quad D^\dag B^t:\real^{4d+2d'+1}(w,\xi_w,\xi_z)=
(B^t w,B_0^\dag\xi_w,\xi_z)
\]
with $B_0:=A\oplus A^\dag\oplus \widehat{A}:\real_w^{2d+d'}\to \real_w^{2d+d'}$ and $B_0^\dag:=({}^tB_0)^{-1}$. 
Note that this linear map $D^\dag B^t$ preserves the trapped set $X_0$. 

Let us consider the level sets of the coordinate $\xi_z$,
\[
Z_c:=\{(x,y,\xi_x,\xi_y,\xi_z)\in \real^{4d+2d'+1}\mid \xi_z=c\},
\]
which is preserved by $D^\dag B^t$. 
This level set $Z_c$ carries the canonical symplectic form $
dw\wedge d\xi_w$, which is also preserved by $D^\dag B^t$. 
Observe that the subspace $X_0\cap Z_c$ is a symplectic subspace of $Z_c$ with respect to this symplectic structure, provided $c\neq 0$. (This is a consequence of the fact that $\alpha_0$ is a contact form.) Hence the action of $D^\dag B^t$ restricted to each $Z_c$ with $c\neq 0$ preserves the decomposition
\begin{equation}\label{eq:symp_decomp}
Z_c=(X_0\cap Z_c)\oplus (X_0\cap Z_c)^{\perp}
\end{equation}
where $(X_0\cap Z_c)^{\perp}$ denotes the symplectic orthogonal of the subspace $(X_0\cap Z_c)$:
\[
(X_0\cap Z_c)^{\perp}:=\{v\in Z_c\mid 
dw\wedge d\xi_w(v,v')=0 \; \forall v'\in X_0\cap Z_c\}.
\]
The restriction of $D^\dag B^t$ to the subspace $(X_0\cap Z_c)$ describes the dynamics inside the trapped set, while that to the symplectic orthogonal $(X_0\cap Z_c)^{\perp}$ describes the dynamics in the transverse (or normal) directions.  
\begin{Remark}\label{Rem:restriction_xiz}
Notice that the symplectic decomposition (\ref{eq:symp_decomp}) does not make good sense when $c=0$. 
Correspondingly we have to treat the hyperplane $Z_0=\{\xi_z=0\}$ as an exceptional set when we consider this decomposition. 
\end{Remark}

\subsubsection{A new coordinate system }
From the observation above (and the argument in \cite[Chapter 2 and 4]{FaureTsujii12}), we introduce the coordinates on $\real^{4d+2d'+1}_{(w,\xi_w,\xi_z)}$,
\[
\zeta=(\zeta_p,\zeta_q)\in \real^{2d},\quad \nu=(\nu_q,\nu_p)\in \real^{2d},\quad \mbox{and}\quad (\tilde{y}, \tilde{\xi}_y)\in \real^{2d'}=\real^{d'}\oplus \real^{d'}
\]
as follows. On the region $\xi_z\ge 0$, we define 
\begin{alignat}{2}
\label{eq:newcoordinates}
\zeta_p&=2^{-1/2}\langle\xi_z\rangle^{-1/2}(\langle\xi_z\rangle\xi_p+\xi_z q),\qquad & 
\zeta_q&=2^{-1/2}\langle\xi_z\rangle^{-1/2}(\langle\xi_z\rangle\xi_q-\xi_z p),\\
\nu_q&=2^{-1/2}\langle\xi_z\rangle^{-1/2}(\xi_z q-\langle\xi_z\rangle\xi_p),&
\nu_p&=2^{-1/2}\langle\xi_z\rangle^{-1/2}(\xi_z p+\langle\xi_z\rangle\xi_q),\notag\\
\tilde{y}&=\langle\xi_z\rangle^{1/2}y,&
\tilde{\xi}_y&=\langle\xi_z\rangle^{1/2}\xi_y\notag
\end{alignat}
while, on the region $\xi_z<0$, we modify the definitions of $\zeta_p$ and $\nu_q$ as  
\begin{alignat}{2}
\label{eq:newcoordinates2}
\zeta_p&=-2^{-1/2}\langle\xi_z\rangle^{-1/2}(\langle\xi_z\rangle\xi_p+\xi_z q),\qquad &
\nu_q&=-2^{-1/2}\langle\xi_z\rangle^{-1/2}(\xi_z q-\langle\xi_z\rangle\xi_p),
\end{alignat}
by changing the signs, but keep the other definitions of $\zeta_q$, $\nu_p$, $\tilde{y}$ and $\tilde{\xi}_y$.
\begin{Remark}\label{Rem:newcoordinate_scope}
Basically we consider these new coordinates on the region $|\xi_z|\ge 2$ where $\langle \xi_z\rangle=|\xi_z|$. But we defined them also on the region $|\xi_z|\le 2$  for convenience in some definitions below. (See Definition \ref{def:Wrh}.)
Note that, if $\xi_z\ge 2$ or $\xi_z\le -2$, we have $\langle \xi_z\rangle=|\xi_z|=\pm \xi_z$ and the relations in the definitions above become simpler. 
\end{Remark}
The coordinates  $(\nu,\zeta,\tilde{y},\tilde{\xi}_y,\xi_z)$ above are defined so that the following hold on the region $|\xi_z|\ge 2$:
\begin{enumerate}
\item The trapped set $X_0$  is characterized by the equation $(\zeta,y,\xi_{y})=(0,0,0)$. 
\item The coordinate change transformation preserves the canonical symplectic form and the Riemann metric  on $Z_c$ up to multiplication by the factor $|\xi_z|$, that is, on each of the level set $Z_c$ with $|c|\ge 2$, we have
\begin{align*}
&|\xi_z|  (dx \wedge d\xi_x+dy\wedge d\xi_y)=d\zeta_p\wedge d\zeta_q+d\nu_q\wedge d\nu_p+
d\tilde{y}\wedge d\tilde{\xi}_y\\
\intertext{and}
&|\xi_z| ( |dx|^2+ |d\xi_x|^2+|dy|^2+ |d\xi_y|^2)=
|d\zeta|^2+|d\nu|^2+|d\tilde{y}|^2+|d\tilde{\xi}_y|^2.
\end{align*}
\item the volume form in (\ref{eq:constNormalization}) is written 
\begin{equation}\label{eq:normalized_volume}
d\vol=(2\pi)^{-(2d+d'+1)} d\nu d\zeta d\tilde{y} d\tilde{\xi}_y d\xi_z,
\end{equation}
\item the $\zeta_p$, $\nu_q$ and $\tilde{\xi}_y$ axes are the expanding directions  whereas the $\zeta_q$, $\nu_p$ and $\tilde{y}$ axes are the contracting directions.
\end{enumerate} 
We write the  corresponding coordinate change transformation as
\[
\Phi:\real^{4d+2d'+1}\to \real^{4d+2d'+1},\quad 
\Phi(x,y,\xi_x,\xi_y,\xi_z)=((\nu_q,\nu_p),((\zeta_p,\tilde{\xi}_y),(\zeta_q,\tilde{y})), \xi_z),
\]
where the order and combination of the variables on the right-hand side is chosen for convenience in the argument below.

\begin{Remark}\label{Rem:dag}
Below we sometimes restrict our attention to the functions supported on the region $|\xi_z|\ge 2$. We use the subscript $\dag$ to remind such restriction. For instance, we write $L^2_{\dag}(\real^{4d+2d'+1}_{(w,\xi_w,\xi_z)})$  for  the subspace of $L^2(\real^{4d+2d'+1}_{(w,\xi_w,\xi_z)})$ that consists of functions supported on the region $|\xi_z|\ge 2$. Correspondingly we let $L^2_\dag(\real^{2d+d'+1}_{(x,y,z)})$ be the preimage of $L^2_{\dag}(\real^{4d+2d'+1}_{(w,\xi_w,\xi_z)})$ with respect to the partial Bargmann transform~$\pBargmann$. 
\end{Remark}
The pull-back operator by $\Phi$ restricted to the region $|\xi_z|\ge 2$,
\[
\Phi^*:L_\dag^2(\real^{4d+2d'+1}_{((\nu_q,\nu_p),((\zeta_p,\tilde{\xi}_y),(\zeta_q,\tilde{y})), \xi_z)})\to L_\dag^2(\real^{4d+2d'+1}_{(x,y,\xi_x,\xi_y,\xi_z)}),\quad \Phi^* u=u\circ \Phi,
\]
is a unitary operator, because the norm on $L^2(\real^{4d+2d'+1}_{((\nu_q,\nu_p),((\zeta_p,\tilde{\xi}_y),(\zeta_q,\tilde{y})), \xi_z)})$ is defined by using the volume form (\ref{eq:normalized_volume}). 

\subsubsection{A tensorial decomposition of the transfer operator $L^t$}
In the next lemma, we express the transfer operator $L^t$ as a tensor product of  three simple operators.  This is a consequence of the fact that $D^\dag B^t$ preserves the symplectic decomposition (\ref{eq:symp_decomp}). 
In the statement below, we write $\Bargmann_1^{(d)}$ for the Bargmann transform defined in (\ref{eq:Bargmann_transform_h}) with $D=d$ and $\hbar=1$ and let $\BargmannP^{(d)}_1$ be the corresponding Bargmann projector defined in (\ref{eq:Bargmann_projector_h}). 
\begin{Lemma}{\cite[Proposition 4.3.1 and Proposition 7.1.2]{FaureTsujii12}}\label{lm:U}
The transformation $\Phi^*$ above satisfies 
\begin{equation}\label{eq:Phi_Bargmann}
\pBargmannP\circ \Phi^*=\Phi^*\circ (\BargmannP^{(2d+d')}_1\otimes \mathrm{Id})=
\Phi^*\circ (\BargmannP^{(d)}_1\otimes \BargmannP^{(d+d')}_1\otimes \mathrm{Id})\
\end{equation}
on $L^2_{\dag}(\real^{4d+2d'+1})$ and is an isomorphism between the images of the operators
\[
\Bargmann_1^{(d)}\otimes \Bargmann_1^{(d+d')}\otimes \mathrm{Id}:L^2_\dag(\real_{(\nu_q,(\zeta_p,\tilde{\xi}_y),\xi_z)}^{2d+d'+1})\to L^2_\dag(\real_{((\nu_q,\nu_p),((\zeta_p,\tilde{\xi}_y),(\zeta_q,\tilde{y})), \xi_z)}^{4d+2d'+1})
\]
and
\[
\pBargmann:L^2_\dag(\real_{(x,y,z)}^{2d+d'+1})\to L^2_\dag(\real_{(x, y,\xi_x,\xi_y,\xi_z)}^{4d+2d'+1}).
\]
The operator  
\[
\mathcal{U}=\Bargmann^* \circ \Phi^*\circ (\Bargmann_1^{(d)}\otimes \Bargmann_1^{(d+d')}\otimes \mathrm{Id}): L_\dag^2(\real^{2d+d'+1}_{(\nu_q,(\zeta_p,\tilde{\xi}_y),\xi_z)})\to L_\dag^2(\real_{(x,y,z)}^{2d+d'+1})
\]
is a unitary operator and makes the following diagram commute:
\begin{equation*}
\begin{CD}
L^2_\dag(\real^{2d+d'+1}_{(x,y,z)})@<{\mathcal{U}}<<  L^2_\dag(\real^{2d+d'+1}_{(\nu_q,(\zeta_p,\tilde{\xi}_y),\xi_z)})\cong L^2(\real^{d}_{\nu_q})\otimes L^2(\real^{d+d'}_{(\zeta_p,\tilde{\xi}_y)})\otimes L^2_\dag(\real_{\xi_z})\\
@VV{L^t}V  @V{\mathfrak{L}^t}VV\\
L^2_\dag(\real^{2d+d'+1}_{(x,y,z)})@<{\mathcal{U}}<<  L^2_\dag(\real^{2d+d'+1}_{(\nu_q,(\zeta_p,\tilde{\xi}_y),\xi_z)})\cong L^2(\real^{d}_{\nu_q})\otimes L^2(\real^{d+d'}_{(\zeta_p,\tilde{\xi}_y)})\otimes L_\dag^2(\real_{\xi_z})
\end{CD}
\end{equation*} 
where the operator $\mathfrak{L}^t$ is  defined by 
\[
\mathfrak{L}^t= \frac{|\det A|^{1/2}}{|\det \wA|^{1/2}}\cdot L_A\otimes (L_{A\oplus \wA^{\dag}})\otimes e^{i\xi_z t}, 
\]
writing $L_A$ and $L_{A\oplus \wA^{\dag}}$ for the $L^2$-normalized transfer operators defined  by
\[
L_A u=\frac{1}{|\det A|^{1/2}} \cdot u\circ A^{-1}, \quad 
L_{A\oplus \wA^{\dag}} u=\frac{|\det\wA|^{1/2}}{|\det A|^{1/2}} \cdot u\circ (A\oplus \wA^{\dag})^{-1}.
\] 
\end{Lemma}
For the proof,  we refer  \cite[Proposition 7.1.2 and Proposition 4.3.1]{FaureTsujii12}.

\subsection{Anisotropic Sobolev space $\cH^r(\real^{2d+d'+1})$}
We introduce the anisotropic Sobolev spaces in order to study spectral properties of transfer operators. This kind of Hilbert spaces have been introduced (in the context of dynamical systems) by Baladi \cite{Baladi05} and the related argument is developed in the papers  \cite{BaladiTsujii07, BaladiTsujii08, Tsujii2009, FRS, FaureSjostrand11}.  This is a kind of (generalized) Sobolev space with the weight function adapted to hyperbolicity of the dynamics (and  is  anisotropic accordingly).    
 Note that the anisotropic Sobolev space is not contained in the space of usual functions but contained in the space of distributions.
 
\subsubsection{The definition of the anisotropic Sobolev space}
For each $r>0$, we will define the anisotropic Sobolev space $\cH^r(\real^{2d+d'+1})$. For the construction below, we do not need any assumption on the range of the parameter $r$. But, for the argument in the later subsections, we assume 
\begin{equation}\label{eq:choice_of_r}
r>2+2(2d+d')
\end{equation}
and also
\begin{equation}\label{eq:choice_of_r3}
\lambda^{r-1}>|\det A|\cdot |\det \wA|^{-1}
\end{equation}
in relation to the affine transformation $B^t$ given in (\ref{eq:Bt}).

For each $\tau>0$, let us  consider the cones  
\begin{align}
{\cone}_{+}^{(d+d',d+d')}(\tau) & =\{(\zeta_p,\tilde{\xi}_y,\zeta_q, \tilde{y})\in\real^{2d+2d'}\mid |(\zeta_q, \tilde{y})| \le \tau\cdot |(\zeta_p,\tilde{\xi}_y)|\}\quad \mbox{and},\label{eq:cone_tau}\\
{\cone}_{-}^{(d+d',d+d')}(\tau)&=\{(\zeta_p,\tilde{\xi}_y,\zeta_q, \tilde{y})\in\real^{2d+2d'}\mid |(\zeta_p,\tilde{\xi}_y)|\le \tau\cdot|(\zeta_q, \tilde{y})|\}
\end{align}
in $\real^{2d+2d'}$ equipped with the coordinates $\zeta_p, \zeta_q\in \real^{d}$ and $\tilde{\xi}_y, \tilde{y}\in \real^{d'}$. 
Next we take and fix a $C^{\infty}$ function on the projective space $\mathbf{P}(\real^{2d+2d'})$,
\[
\mathrm{ord}:\mathbf{P}\left(\real^{2d+2d'}\right)\to[-1,1]
\]
 so that
\begin{equation}
\mathrm{ord}\left([(\zeta_p,\tilde{\xi}_y,\zeta_q, \tilde{y})]\right)=\begin{cases}
-1, & \quad\mbox{if $(\zeta_p,\tilde{\xi}_y,\zeta_q, \tilde{y})\in{\cone}_{+}^{(d+d',d+d')}(1/2)$};\\
+1, & \quad\mbox{if $(\zeta_p,\tilde{\xi}_y,\zeta_q, \tilde{y})\in\cone_{-}^{(d+d',d+d')}(1/2)$}
\end{cases}\label{eq:def_order_function_m}
\end{equation}
and that 
\[
\mathrm{ord}\left([(\zeta'_p,\tilde{\xi}'_y,\zeta'_q, \tilde{y}')]\right)
\le
\mathrm{ord}\left([(\zeta_p,\tilde{\xi}_y,\zeta_q, \tilde{y})]\right)
\quad \mbox{if}\quad \frac{|(\zeta'_q, \tilde{y}')|}{|(\zeta'_p,\tilde{\xi}'_y)|}\le \frac{|(\zeta_q, \tilde{y})|}{|(\zeta_p,\tilde{\xi}_y)|}.
\] 
We then consider  the smooth function 
\[
W^{r}:\real^{2d+2d'}\to\real_{+},\qquad W^{r}(\zeta_p,\tilde{\xi}_y,\zeta_q, \tilde{y})=\langle|(\zeta_p,\tilde{\xi}_y,\zeta_q, \tilde{y})|\rangle^{r\cdot \mathrm{ord}([(\zeta_p,\tilde{\xi}_y,\zeta_q, \tilde{y})])}.
\]
By definition, we have that 
\[
W^{r}(\zeta_p,\tilde{\xi}_y,\zeta_q, \tilde{y}) =
\begin{cases}
\langle|(\zeta_p,\tilde{\xi}_y,\zeta_q, \tilde{y}|)\rangle^{-r} &\quad \mbox{on  ${\cone}_{+}^{(d+d',d+d')}(1/2)$};\\
\langle|(\zeta_p,\tilde{\xi}_y,\zeta_q, \tilde{y})|\rangle^{+r} &\quad \mbox{on ${\cone}_{-}^{(d+d',d+d')}(1/2)$}.
\end{cases}
\]
A simple (but important) property of the function $W^r$ is that we have
\begin{equation}\label{eq:property_w1}
W^{r}(((A\oplus \wA^{\dag})\oplus (A\oplus \wA^{\dag})^\dag)(\zeta_p,\tilde{\xi}_y,\zeta_q, \tilde{y}))\le 
C_0 \lambda^{-r}\cdot W^{r}(\zeta_p,\tilde{\xi}_y,\zeta_q, \tilde{y})
\end{equation}
when $A$ and $\widehat{A}$ satisfy (\ref{eq:lambda}) with some $\lambda\ge 1$ and 
$(\zeta_p,\tilde{\xi}_y,\zeta_q, \tilde{y})\in \real^{2d+2d'}$ is sufficiently far from the origin,
where $C_0>0$ is a constant independent of $A$ and $\widehat{A}$. Another important property is that it is rather smooth in the sense that we have
\begin{equation}\label{eq:property_w2}
W^{r}(\zeta_p,\tilde{\xi}_y,\zeta_q, \tilde{y})\le C_0\cdot  W^{r}(\zeta'_p,\tilde{\xi}'_y,\zeta'_q, \tilde{y}')\cdot \langle |(\zeta_p,\tilde{\xi}_y,\zeta_q, \tilde{y})-(\zeta'_p,\tilde{\xi}'_y,\zeta'_q, \tilde{y}')|\rangle^{2r}
\end{equation}
for some constant $C_0>0$. 
\begin{Definition}[Anisotropic Sobolev space]
\label{def:Wrh} 
Let  $\cW^r:\real^{4d+2d'+1}\to \real_+$ be  the function defined by 
\[
\cW^r(x,y,\xi_x,\xi_y,\xi_z)=(1\otimes W^r\otimes 1)\circ \Phi(x,y,\xi_x,\xi_y,\xi_z)=W^r(\zeta_p,\tilde{\xi}_y,\zeta_q, \tilde{y})
\]
where the variables\footnote{Notice that this function $\cW^r$ is continuous despite of the discontinuity of the coordinates $(\zeta_p,\tilde{\xi}_y,\zeta_q, \tilde{y})$ at points on the hyperplane $\xi_z=0$.} $(\zeta_p,\tilde{\xi}_y,\zeta_q, \tilde{y})$  in the rightmost term are those defined by (\ref{eq:newcoordinates}). 
We define the anisotropic Sobolev norm $\|\cdot \|_{\cH^r}$ on $\cS(\real^{2d+d'+1})$ by 
\[
\|u \|_{\cH^r}:=\|\cW^r\cdot \pBargmann u\|_{L^2}.
\]
The anisotropic Sobolev space $\cH^r(\real^{2d+d'+1})$ is the Hilbert space obtained as the completion of the Schwartz space $\cS(\real^{2d+d'+1})$ with respect to this norm. 
\end{Definition}

By definition, the partial Bargmann transform $\pBargmann$ extends to an isometric embedding
\[
\pBargmann:\cH^{r}(\real^{2d+d'+1}_{(w,z)})\to L^2(\real^{4d+2d'+1}_{(w,\xi_w,\xi_z)},(\cW^r)^2)
\]
where $L^2(\real^{4d+2d'+1}_{(w,\xi_w,\xi_z)},(\cW^{r})^2)$ denotes the weighted $L^2$ space 
\[
L^2(\real^{4d+2d'+1}_{(w,\xi_w,\xi_z)},(\cW^{r})^2)=\{u\in L^2_{\mathrm{loc}}(\real^{4d+2d'+1}_{(w,\xi_w,\xi_z)})\mid 
\| \cW^{r}\cdot u\|_{L^2}<\infty\}.
\]
(Note that the $L^2$ norm on $\real^{4d+2d'+1}_{(w,\xi_w,\xi_z)}$ is defined with respect to the volume form $d\vol$ in (\ref{eq:constNormalization}).)

\subsubsection{Variants of $\cH^r(\real^{2d+d'+1})$} \label{ss:variantsOfH}
The anisotropic Sobolev spaces $\cH^r(\real^{2d+d'+1})$  introduced above  are quite useful when we consider the spectral properties of the transfer operators for hyperbolic maps or flows. But, in using them,  one has to be careful that they have  singular properties related to their  anisotropic nature. For instance, even if a linear map $B:\real^{2d+d'+1}\to \real^{2d+d'+1}$ is close to the identity,  the action of the associated transfer operator $L_B $ on them can be unbounded. This actually leads to various problems. In order to do with such problems, we introduce variants $\cH^{r,\sigma}(\real^{2d+d'+1})$ of $\cH^r(\real^{2d+d'+1})$ below. We consider the index set
\begin{equation}\label{eq:sigma}
\Sigma_0=\{-1,0,+1\}\subset \Sigma=\{-2,-1,0,+1,+2\}. 
\end{equation}
For $\sigma \in \Sigma$, we define $\mathrm{ord}^{\sigma}:\mathbf{P}\left(\real^{2d+2d'}\right)\to[-1,1]$ by 
\begin{equation}
\label{ord+}
\mathrm{ord}^{\sigma}\left(\left[(\zeta_p,\tilde{\xi}_y,\zeta_q, \tilde{y})\right]\right)=
\mathrm{ord}\left(\left[(2^{-\sigma/2}\zeta_p,2^{-\sigma/2}\tilde{\xi}_y,2^{+\sigma/2}\zeta_q, 2^{+\sigma/2}\tilde{y})\right]\right)
\end{equation}
so that we have
\begin{equation}\label{eq:ord_ineq}
\mathrm{ord}^{\sigma'}([(\zeta_p,\tilde{\xi}_y,\zeta_q, \tilde{y})])\le  \mathrm{ord}^{\sigma}([(\zeta_p,\tilde{\xi}_y,\zeta_q, \tilde{y})])\quad \mbox{if $\sigma'\le \sigma$.}
\end{equation}
(Here we use the factor $2^{\pm 1/2}$ for concreteness, but it could be any real number greater than $1$.) 
We define 
\[
W^{r,\sigma}:\real^{2d+2d'}\to\real_{+},\qquad W^{r,\sigma}(\zeta_p,\tilde{\xi}_y,\zeta_q, \tilde{y})=\langle|(\zeta_p,\tilde{\xi}_y,\zeta_q, \tilde{y})|\rangle^{r\cdot \mathrm{ord}^{\sigma}([(\zeta_p,\tilde{\xi}_y,\zeta_q, \tilde{y})])}.
\]
From (\ref{eq:ord_ineq}), we have 
\begin{equation}\label{eq:wrpm}
W^{r,\sigma'}(\zeta_p,\tilde{\xi}_y,\zeta_q, \tilde{y})\le W^{r,\sigma}(\zeta_p,\tilde{\xi}_y,\zeta_q, \tilde{y})\quad \mbox{if $\sigma'\le \sigma$.}
\end{equation}
These functions also satisfy the properties parallel to  (\ref{eq:property_w1}) and (\ref{eq:property_w2}). 

The functions $\cW^{r,\sigma}(\cdot)$, the norms $\|\cdot\|_{\cH^{r,\sigma}}$ and the Hilbert spaces $\cH^{r,\sigma}(\real^{2d+d'+1})$ are defined  in the same manner as $\cW^r$, $\|\cdot\|_{\cH^{r}}$ and $\cH^{r}(\real^{2d+d'+1})$ respectively, with the function $W^r(\cdot)$ replaced by $W^{r,\sigma}(\cdot)$. In particular, we have    
\[
W^{r,0}(\cdot)=W^r(\cdot), \quad \cW^{r,0}(\cdot)= \cW^{r}(\cdot), \quad \cH^{r,0}(\real^{2d+d'+1})=\cH^{r}(\real^{2d+d'+1}).
\]
From (\ref{eq:wrpm}),  we have
\begin{equation}\label{eq:norm_comparison}
\|u\|_{\cH^{r,\sigma'}}\le \|u\|_{\cH^{r,\sigma}}\quad  \mbox{if $\sigma'\le \sigma$,}
\end{equation}
and hence
\[
\cH^{r,\sigma}(\real^{2d+d'+1})\subset  \cH^{r,\sigma'}(\real^{2d+d'+1}) \quad \mbox{if $\sigma'\le \sigma$.}
\]
The partial Bargmann transform $\pBargmann$ extends to  isometric embeddings
\[
\pBargmann:\cH^{r,\sigma}(\real^{2d+d'+1})\to L^2(\real^{4d+2d'+1},(\cW^{r,\sigma})^2)\quad \mbox{for $\sigma\in \Sigma$}. 
\]

\subsection{The spectral structure of the transfer operator $L^t$}
\label{ss:specLt}
We now discuss about the  spectral properties of the transfer operator $L^t$, defined in (\ref{eq:Lt}), on the anisotropic Sobolev spaces $\cH^{r,\sigma}(\real^{2d+d'+1})$. 
First we recall a few results from \cite[Chapter 4 and 7]{FaureTsujii12}. 
Let  $H^{r,\sigma}(\real^{d+d'})$ be the completion of the space $\cS(\real^{d+d'})$ with respect to the norm
\[
\|u\|_{H^{r,\sigma}}=\|W^{r,\sigma} \cdot \Bargmann^{(d+d')}_1 u\|_{L^2}.
\]
From the definition of $\cW^{r, \sigma}$,  the commutative diagram in Lemma \ref{lm:U} extends naturally  to 
\begin{equation}\label{cd:H}
\begin{CD}
\cH^{r,\sigma}_{\dag}(\real_{(x,y,z)}^{2d+d'+1})@<{\mathcal{U}}<<  L^2(\real_{\nu_q}^{d})\otimes H^{r,\sigma}(\real_{(\zeta_p, \tilde{\xi}_y)}^{d+d'})\otimes L^2_{\dag}(\real_{\xi_z})\\
@VV{L^t}V  @V{\mathfrak{L}^t}VV\\
\cH^{r,\sigma'}_{\dag}(\real_{(x,y,z)}^{2d+d'+1})@<{ \mathcal{U}}<<  L^2(\real_{\nu_q}^{d})\otimes H^{r,\sigma'}(\real_{(\zeta_p, \tilde{\xi}_y)}^{d+d'})\otimes L^2_{\dag}(\real_{\xi_z})
\end{CD}
\end{equation}
if $t\ge 0$ and $\sigma'\le \sigma$, where $\mathcal{U}$ is  an isomorphism. (Here we use the subscript $\dag$ in $\cH^{r,\sigma'}_{\dag}(\real_{(x,y,z)}^{2d+d'+1})$ in the same meaning as noted in Remark \ref{Rem:dag}.) 

Therefore the operator $L^t:\cH^{r,\sigma}(\real^{2d+d'+1}_{(x,y,z)})\to \cH^{r,\sigma'}(\real^{2d+d'+1}_{(x,y,z)})$ is identified with the tensor product of the three operators
\begin{align}
&L_A: L^2(\real^{d}_{\nu_q})\to  L^2(\real^{d}_{\nu_q}),\\
&\widetilde{L}:=\frac{|\det A|^{1/2}}{|\det \wA|^{1/2}}\cdot L_{A\oplus \wA^{\dag}}:H^{r,\sigma}(\real^{d+d'}_{(\zeta_p, \tilde{\xi}_y)})\to H^{r,\sigma'}(\real^{d+d'}_{(\zeta_p, \tilde{\xi}_y)}),\label{eq:whLa}
\intertext{
and}
&e^{i\xi_z t}\cdot \mathrm{Id}:L^2_\dag(\real_{\xi_z})\to L^2_\dag(\real_{\xi_z}).
\end{align}
The first and third operators are unitary. In \cite[Chapter 3]{FaureTsujii12}, we studied the second operator $\widetilde{L}$ to some detail, which we recall below. 

Let us consider the projection operator
\begin{equation}\label{def:T0}
T_0:\cS(\real^{d+d'}_{(\zeta_p, \tilde{\xi}_y)})\to \cS(\real^{d+d'}_{(\zeta_p, \tilde{\xi}_y)})',\qquad T_0(u)(x)=u(0) \cdot \mathbf{1}
\end{equation}
where $\mathbf{1}$ denotes the constant function on $\real^{d+d'}_{(\zeta_p, \tilde{\xi}_y)}$ with value $1$.
This is a simple operation that extracts the constant term in the Taylor expansion of a function at the origin. Letting $\Bargmann^{(d+d')}_1:L^2(\real^{d+d'}_{(\zeta_p,\tilde{\xi}_y)})\to L^2(\real^{2d+2d'}_{(\zeta_p,\tilde{\xi}_y,\zeta_q, \tilde{y})})$ be the Bargmann transform with $\hbar=1$, we set 
\begin{equation}\label{eq:TzeroLift}
T_0^{\lift}:=\Bargmann^{(d+d')}_1\circ T_0 \circ (\Bargmann^{(d+d')}_1)^*: 
L^2(\real^{2d+2d'}_{(\zeta_p,\tilde{\xi}_y,\zeta_q, \tilde{y})}) \to L^2(\real^{2d+2d'}_{(\zeta_p,\tilde{\xi}_y,\zeta_q, \tilde{y})}).
\end{equation}
(Here we regard $\zeta_q$ and $\tilde{y}$ as the dual variable of $\zeta_p$ and $\tilde{\xi}_y$ respectively.) Clearly it makes the following diagram commutes:
\begin{equation}\label{cd:T_0Lift}
\begin{CD}
L^2\left(\real^{2d+2d'}_{(\zeta_p,\tilde{\xi}_y,\zeta_q, \tilde{y})}\right)@>{T_0^{\lift}}>> L^2\left(\real^{2d+2d'}_{(\zeta_p,\tilde{\xi}_y,\zeta_q, \tilde{y})}\right)\\
@A{\Bargmann^{(d+d')}_1}AA @A{\Bargmann^{(d+d')}_1}AA\\
L^2(\real^{d+d'}_{(\zeta_p, \tilde{\xi}_y)})@>{T_0}>> L^2(\real^{d+d'}_{(\zeta_p, \tilde{\xi}_y)}).
\end{CD}
\end{equation}
\begin{Lemma}[{\cite[Lemma 3.4.2 and its proof]{FaureTsujii12}}] \label{lm:Tzero}
The operator $T_0^{\lift}$ is written as an integral operator 
\[
T_0^{\lift} u(\zeta'_p,\tilde{\xi}'_y,\zeta'_q, \tilde{y}')=
\int K_+(\zeta'_p,\tilde{\xi}'_y,\zeta'_q, \tilde{y}') K_-(\zeta_p,\tilde{\xi}_y,\zeta_q, \tilde{y}) 
u(\zeta_p,\tilde{\xi}_y,\zeta_q, \tilde{y}) d\zeta_p d\tilde{\xi}_y d\zeta_q d\tilde{y}
\]
where the functions $K_{\pm}(\cdot)$ satisfy, for any $\sigma, \sigma'\in \Sigma$, that 
\begin{align}
&W^{r,\sigma'}(\zeta'_p,\tilde{\xi}'_y,\zeta'_q, \tilde{y}')\cdot K_+(\zeta'_p,\tilde{\xi}'_y,\zeta'_q, \tilde{y}')\le C_0\langle(\zeta'_p,\tilde{\xi}'_y,\zeta'_q, \tilde{y}')\rangle^{-r}\label{eq:kpm1}
\intertext{and}
&W^{r,\sigma}(\zeta_p,\tilde{\xi}_y,\zeta_q, \tilde{y})^{-1}\cdot K_-(\zeta_p,\tilde{\xi}_y,\zeta_q, \tilde{y})\le C_0\langle(\zeta_p,\tilde{\xi}_y,\zeta_q, \tilde{y})\rangle^{-r}\label{eq:kpm2}
\end{align}
for a constant $C_0>0$.
Hence, for any $\sigma, \sigma'\in \Sigma$,  $T_0^{\lift}$ extends to a rank-one  operator 
\[
T_0^{\lift}:L^2(\real^{2d+2d'}, (W^{r,\sigma})^2)\to L^2(\real^{2d+2d'}, (W^{r,\sigma'})^2).
\] 
\end{Lemma}
\begin{Corollary}
The operator $T_0$ extends naturally to a rank-one  operator
\[
T_0:H^{r,\sigma}(\real^{d+d'}_{(\zeta_p, \tilde{\xi}_y)})\to H^{r,\sigma'}(\real^{d+d'}_{(\zeta_p, \tilde{\xi}_y)})\quad \mbox{for any $\sigma, \sigma'\in \Sigma$.}
\]
\end{Corollary}
The next lemma is a rephrase of  the main statement in \cite[Chapter 3]{FaureTsujii12}.  Note that  $A\oplus \wA^{\dag}:\real^{d+d'}\to \real^{d+d'}$ is an expanding map satisfying (\ref{eq:lambda}) for some $\lambda\ge 1$.
\begin{Lemma}[{\cite[Proposition 3.4.6 and Section 4]{FaureTsujii12}}]\label{lm:linear1} 
\noindent
{\rm (1)} The operator $\widetilde{L}$ in {\rm (\ref{eq:whLa})} extends to a bounded operator 
$
\widetilde{L}:H^{r,\sigma}(\real_{(\zeta_p, \tilde{\xi}_y)}^{d+d'})\to H^{r,\sigma'}(\real_{(\zeta_p, \tilde{\xi}_y)}^{d+d'})$
for $\sigma, \sigma'\in \Sigma$ with $\sigma'\le \sigma$ and the operator norm is bounded by a constant independent of $A$. Further, if $\lambda$  is sufficiently large, say $\lambda>10$, then this is true for any $\sigma, \sigma'\in \Sigma$. \\
\noindent
{\rm (2)} The operator $\widetilde{L}:H^{r,\sigma}(\real_{(\zeta_p, \tilde{\xi}_y)}^{d+d'})\to H^{r,\sigma}(\real_{(\zeta_p, \tilde{\xi}_y)}^{d+d'})$ commutes with $T_0$. And it preserves the decomposition 
$
H^{r,\sigma}(\real_{(\zeta_p, \tilde{\xi}_y)}^{d+d'})=\mathfrak{H}_0\oplus \mathfrak{H}_1$ 
where
\[
\mathfrak{H}_0=\mathrm{Im}\, T_0 =\{ c\cdot \mathbf{1} \mid c\in \complex\}
\quad \mbox{and}\quad
\mathfrak{H}_1=\mathrm{Ker}\, T_0 =\{ u \in H^{r,\sigma}(\real_{(\zeta_p, \tilde{\xi}_y)}^{d+d'}) \mid u(0)=0\}.
\]
Further we have that
\begin{itemize}
\item[(i)] the restriction of  $\widetilde{L}$ to $\mathfrak{H}_0$ is the identity, and that
\item[(ii)]  the restriction of  $\widetilde{L}$ to $\mathfrak{H}_1$ is contracting in the sense that 
\[
\|\widetilde{L} u\|_{H^{r,\sigma}}\le C_0\cdot (1/ \lambda)\cdot  \|u\|_{H^{r,\sigma}}\quad \mbox{for all $u\in \mathfrak{H}_1$,} 
\]
where $C_0$ is a constant independent of  $A$ and $\widehat{A}$. 
\end{itemize}
\end{Lemma}

From the last  lemma and  the commutative diagram (\ref{cd:H}), we conclude the next theorem. This is the counterpart of Theorem \ref{th:spectrum} for the operator $L^t$ as a (local) linearized model of the transfer operator $\cL^t$.
We consider the projection operator 
\begin{equation}\label{eq:T}
\cT_0=\mathcal{U}\circ (\mathrm{Id}\otimes T_0\otimes \mathrm{Id})\circ
\mathcal{U}^{-1}: \cS_{\dag}(\real^{2d+d'+1}_{(x,y,z)})\to \cS_\dag(\real^{2d+d'+1}_{(x,y,z)})'
\end{equation}
\begin{Theorem}\label{cor:linear1} 
{\rm (1)}  The operator $\cT_0$ extends naturally to a bounded operator
\[
\cT_0:\cH_\dag^{r,\sigma}(\real^{2d+d'+1}_{(x,y,z)})\to \cH_\dag^{r,\sigma'}(\real_{(x,y,z)}^{2d+d'+1})\quad\mbox{for any $\sigma, \sigma'\in \Sigma$. }
\]
\noindent
{\rm (2)} $L^t$ extends to a bounded operator $L^t:\cH_\dag^{r,\sigma}(\real^{2d+d'+1})\to \cH_\dag^{r,\sigma'}(\real^{2d+d'+1})$ for
 any $\sigma,\sigma'\in \Sigma$ with $\sigma'\le \sigma$ and the operator norms are bounded by a  constant independent of  $A$ and $\widehat{A}$.
Further, if $\lambda$ in the assumption (\ref{eq:lambda}) is sufficiently large, say $\lambda>10$, this is true for any $\sigma, \sigma'\in \Sigma$. \\
\noindent
{\rm (3)} 
 $L^t$ commutes with the projection operator $\cT_0$ and preserves the decomposition 
\[
\cH_\dag^{r, \sigma}(\real^{2d+d'+1})=\cH_0\oplus \cH_1
\quad \mbox{
where}\quad
\cH_0=\mathrm{Im}\, \cT_0
\quad \mbox{and}\quad 
\cH_1=\mathrm{Ker}\, \cT_0.
\]
Further we have that
\begin{itemize}
\item[(i)] the restriction of $L^t$ to $\cH_0$ is a unitary operator, and that
\item[(ii)]  the restriction of $L^t$ to $\cH_1$ is contracting in the sense that 
\[
\|L^t u\|_{\cH^{r,\sigma}}\le C_0\cdot (1/  \lambda) \cdot \|u\|_{\cH^{r,\sigma}} \quad \mbox{for all $u\in \cH_1$,}
\]
where $C_0$ is a constant independent of  $A$ and $\widehat{A}$.
\end{itemize}
\end{Theorem}

The lift of the operator $\cT_0$ with respect to the partial Bargmann transform $\pBargmann$ is
\begin{equation}\label{eq:Tlift}
\cT_0^{\lift}=\pBargmann\circ \cT_0\circ \pBargmann^*:\cS_\dag(\real^{4d+2d'+1}_{(w,\xi_w,\xi_z)})\to \cS_\dag(\real^{4d+2d'+1}_{(w,\xi_w,\xi_z)})'.
\end{equation}
Note that, by the definitions and  the relation (\ref{eq:Phi_Bargmann}), we may write it as 
\begin{equation}\label{eq:Tlift2}
\cT_0^{\lift}=\pBargmann\circ \mathcal{U}\circ (\mathrm{Id}\otimes T_0\otimes \mathrm{Id})\circ
\mathcal{U}^{-1}\circ \pBargmann^*=\Phi^*\circ (\BargmannP^{(d)}_1\otimes T_0^{\lift} \otimes \mathrm{Id}) \circ (\Phi^*)^{-1}.
\end{equation} 
The next  is a simple consequence of this expression and Lemma \ref{lm:Tzero}.
\begin{Corollary}\label{cor:Tzero}
The operator $\cT_0^{\lift}$ is written as an integral operator 
\[
\cT_0^{\lift} u(w',\xi'_w, \xi_z)=
\int K(w',\xi'_w;w,\xi_w;\xi_z)
u(w,\xi_w, \xi_z) dw d\xi_w 
\]
and the kernel satisfies, for any $\sigma, \sigma'\in \Sigma$ and $m>0$, that 
\begin{align*}
\frac{\cW^{r,\sigma'}(w',\xi'_w,\xi_z)}{\cW^{r,\sigma}(w,\xi_w, \xi_z)}  &\cdot K(w',\xi'_w;w,\xi_w;\xi_z)\\
&\le 
C_m \langle(\zeta'_p,\tilde{\xi}'_y,\zeta'_q, \tilde{y}')\rangle^{-r}\langle(\zeta_p,\tilde{\xi}_y,\zeta_q, \tilde{y})\rangle^{-r}
\langle(\nu'_q, \nu'_p)-(\nu_q,\nu_p)\rangle^{-m}\\
&\le C' \langle \langle \xi_z\rangle^{1/2}\cdot |(w',\xi'_w)-(w,\xi_w)|\rangle^{-r}
\end{align*}
where $C_m>0$ and $C'>0$ are constants and where $(\zeta_p,\tilde{\xi}_y,\zeta_q, \tilde{y})$ and $\nu=(\nu_q,\nu_p)$ (resp.\ $(\zeta'_p,\tilde{\xi}'_y,\zeta'_q, \tilde{y}')$ and $\nu'=(\nu'_q,\nu'_p)$) are the coordinates of $(w,\xi_w,\xi_z)$ (resp.\ $(w',\xi'_w,\xi_z)$) defined in (\ref{eq:newcoordinates}). 
In fact, the kernel is written as
\[
K(w',\xi'_w;w,\xi_w;\xi_z)=
 K_+(\zeta'_p,\tilde{\xi}'_y,\zeta'_q, \tilde{y}')\cdot K_-(\zeta_p,\tilde{\xi}_y,\zeta_q, \tilde{y}) \cdot k(\nu',\nu')  
\]
where $K_{\pm}(\cdot)$ are the functions in Lemma \ref{lm:Tzero} and $k(\nu,\nu')$ is the kernel of the Bargmann projector $\BargmannP_1^{(d)}$, which satisfies 
\[
|k(\nu,\nu')|\le C_m \langle \nu-\nu'\rangle^{-m}\quad \mbox{for any $m>0$.}
\]
\end{Corollary}

\subsection{Fibered contact diffeomorphism and affine transformations} \label{ss:affine}
In this subsection and the next, we prepare a few definitions and related facts for the argument in the following sections. 
We first introduce the following definition.
\begin{Definition}
\label{def:fiberedContact}
We call a $C^\infty$ diffeomorphism $f:V\to V'=f(V)$ between open subsets $V,V'\subset \real^{2d+d'+1}_{(x,y,z)}$  a {\em fibered contact diffeomorphism} if it satisfies the following conditions:
\begin{enumerate}
\item  $f$ is written in the form
\begin{equation}\label{eq:fdiff}
f(x,y,z)=(\tilde{f}(x), \hat{f}(x,y), z+\tau(x)),
\end{equation}
\item  the diffeomorphism
\[
\check{f}:\proj_{(x,z)}(V)\to \proj_{(x,z)}(V'), \quad \check{f}(x,z)=(\tilde{f}(x), z+\tau(x))
\]
preserves  the contact form $\alpha_0$ given in (\ref{eq:walpha}).  
\end{enumerate}
\begin{Remark}\label{rem:ext_fibered_contact}
We can always extend a fibered contact diffeomorphism $f:V\to V'$ to $f:\proj_{(x,y)}(V))\times \real_z\to \proj_{(x,y)}(V')\times \real_z$ by the expression (\ref{eq:fdiff}). We will assume this extension in some places.
\end{Remark}
The diffeomorphism $\check{f}$ above is called the {\em base diffeomorphism} of $f$. 
The diffeomorphism
\begin{equation}\label{eq:breve}
\breve{f}:\proj_{(x,y)}(V)\to \proj_{(x,y)}(V'),\quad 
\breve{f}(x,y):=(\tilde{f}(x), \hat{f}(x,y))
\end{equation}
is called the {\em transversal diffeomorphism} of $f$. 
We have the commutative diagrams
\[
\begin{CD}
V@>{f}>>V'\\
@V{\proj_{(x,z)}}VV @V{\proj_{(x,z)}}VV\\
\proj_{(x,z)}(V)@>{\check{f}}>>\proj_{(x,z)}(V)
\end{CD}
\qquad\qquad  
\begin{CD}
V@>{f}>>V'\\
@V{\proj_{(x,y)}}VV @V{\proj_{(x,y)}}VV\\
\proj_{(x,y)}(V)@>{\breve{f}}>>\proj_{(x,y)}(V')
\end{CD}
\]
\end{Definition}
The  function $\tau(x)$ in (\ref{eq:fdiff}) is determined by the transversal diffeomorphism $\breve{f}$ up to an additive constant.  In particular, we have
\begin{Lemma}[{\cite[Lemma 4.1]{Tsujii2010}}] \label{lm:tau} If $f:V\to V'$  is a  fibered contact diffeomorphism as  above and suppose that the transversal diffeomorphism preserves the origin, {\it i.e.}\ $\breve{f}(0)=0$, then the function $\tau(x)$ in the expression (\ref{eq:fdiff}) satisfies
\[
D_x\tau(0)=0, \qquad D_x^2\tau(0)=0.
\]
\end{Lemma}
\begin{proof} The first  equality $D\tau(0)=0$ is obvious. The second is also easy to prove but  we need a little computation. See the proof of \cite[Lemma 5.4.3]{FaureTsujii12}.
\end{proof}

Next we restrict ourselves to the case of affine transformations and introduce the following definitions. 
\begin{Definition}[Groups $\cA_0\supset \cA_1\supset \cA_2$ of affine transforms on $\real^{2d+d'+1}$] \label{def:transG} \leavevmode \newline
\noindent (1) Let $\cA_0$  be the group of affine transformations $a:\real^{2d+d'+1}_{(x,y,z)}\to  \real^{2d+d'+1}_{(x,y,z)}$ that are fibered contact diffeomorphisms (with setting $V=V'=\real^{2d+d'+1}$).\\
\noindent(2)Let $\cA_1 \subset \cA_0$ be the subgroup of all the affine transformations in $\cA_0$ of the form
\begin{equation}\label{eq:a}
a:\real^{2d+d'+1}_{(q,p,y,z)}\to \real^{2d+d'+1}_{(q,p,y,z)}, 
a(q,p,y,z)=(Aq+q_0, A^\dag p+p_0, \wA y, z+b(q,p)+z_0)
\end{equation}
where $A:\real^{d}\to \real^{d}$ and  $\wA:\real^{d'}\to \real^{d'}$ are unitary transformations,   $b:\real^{2d}\to \real$  is a  linear map and $(q_0, p_0, y_0,z_0)\in \real^{2d+d'+1}_{(q,p,y,z)}$  is a constant vector. (Recall (\ref{eq:Adag}) for the definition of $A^\dag$.)\\
(3) Let $\cA _2\subset \cA_1$ be the subgroup of all the affine transforms $a\in \cA_1$ as above  with $A$ and $\widehat{A}$ the identities on $\real^d$ and $\real^{d'}$ respectively. 
\end{Definition}
\begin{Remark}
Suppose that $a\in \cA_1$ is of the form (\ref{eq:a}). Then, from the condition that the base diffeomorphism preserves the contact form $\alpha_0$, we see that the linear map  $b(q,p)$ is determined by $A$, $p_0$ and $q_0$. In fact, by simple calculation, we find   $
b(q,p)=-({}^tA p_0)\cdot  q+(A^{-1} q_0)\cdot p$.
\end{Remark}
The following fact is easy to check and quite useful.  
\begin{Lemma}\label{lm:unitaryaction}
The transfer operator $L_a$ for $a\in \cA _1$ (defined by $L_a u:=u\circ a^{-1}$) extends to a unitary operator on $\cH^r(\real^{2d+d'+1})$ (resp.\ on $\cH^{r,\sigma}(\real^{2d+d'+1})$) and  commutes with the projection operator $\cT_0$, that is, $L_a\circ \cT_0=\cT_0\circ L_a$. 
\end{Lemma}
We use the next lemma in setting up the local charts on $G$ in the next section.
\begin{Lemma}\label{lm:affine_trans}
Let $\ell$ and $\ell'$ be $d$-dimensional subspaces in $T_w \real^{2d+d'+1}_{(x,y,z)}$ at a point $w\in \real^{2d+d'+1}_{(x,y,z)}$.  
Suppose that the projection $\proj_{(x,z)}:\real^{2d+d'+1}_{(x,y,z)}\to \real^{2d+1}_{(x,z)}$ maps $\ell$ and $\ell'$ bijectively onto the images $\proj_{(x,z)}(\ell)$ and $\proj_{(x,z)}(\ell')$ and that we have
\[
\proj_{(x,z)}(\ell) \oplus \proj_{(x,z)}(\ell')=\ker \alpha_0(\check{w}),\qquad d\alpha_0|_{\proj_{(x,z)}(\ell)}=0,\qquad d\alpha_0|_{\proj_{(x,z)}(\ell')}=0
\]
where $\check{w}=\proj_{(x,z)}(w)$. 
Then there exists an affine transform $a\in \cA_0 $ (which is in particular a fibered contact diffeomorphism)  such that 
\[
a(0)=w, \quad (Da)_0(\real^d_{q}\oplus \{0\}\oplus\{0\}\oplus\{0\})=\ell, \quad 
(Da)_0(\{0\}\oplus \real^d_{p} \oplus\{0\}\oplus\{0\})=\ell'.
\]
\end{Lemma}
\begin{proof} 
By changing coordinates by the transformation group $\cA_0$, we may assume that $w=0$ and  that the subspaces $\ell$ and $\ell'$ are subspaces of  $\real^{2d}_{(q,p)}\oplus \{0\} \oplus \{0\}\subset \real^{2d+d'+1}$. Thus we have only to   find a linear  map $\check{a}:\real^{2d+1}\to \real^{2d+1}$ preserving $\alpha_0$ such that
\begin{equation}\label{eq:symp_a}
D\check{a}(\real^d_q\oplus \{0\}\oplus\{0\})=\proj_{(x,z)}(\ell), \quad D\check{a}(\{0\}\oplus \real^d_p\oplus \{0\})=\proj_{(x,z)}(\ell').\end{equation}
Note that a linear map $\check{a}:\real^{2d+1}\to \real^{2d+1}$ preserves the contact form $\alpha_0$ if and only if it is of the form  $\check{a}(x,z)=({a}_0(x),z)$ and ${a}_0:\real^{2d}\to \real^{2d}$ preserves the symplectic form $d\alpha_0$ (identifying $\real^{2d}$ with  $\real^{2d}\oplus \{0\}$). Since the subspaces $\proj_{(x,z)}(\ell)$ and $\proj_{(x,z)}(\ell')$ are Lagrangian subspaces ({\it i.e.} the restriction of $d\alpha_0$ to those subspaces are null) transversal to each other, we can find a linear transform $\check{a}:\real^{2d+1}\to \real^{2d+1}$  preserving $d\alpha_0$ so that (\ref{eq:symp_a}) holds true. 
\end{proof}


\section{Local charts and partitions of unity}\label{sec:basic}
In this section, we discuss about choice of local coordinate charts and partitions of unity that we will use. 
We henceforth fix a small constant  $0<\theta<1$ such that
\begin{equation}\label{eq:choice_of_theta}
0<\theta<\min\{\beta,1-\beta\}/20\le 1/40
\end{equation} 
where $0<\beta<1$ is the H\"older exponent given in (\ref{beta}). Also we let $\chi:\real\to [0,1]$ be a $C^\infty$ function satisfying the condition that\footnote{This definition of $\chi(\cdot)$ may look a bit strange for the argument below. Since we use this function $\chi(\cdot)$ later in a different context, we define it in this way. } 
\begin{equation}\label{eq:def_chi}
\chi(s)=
\begin{cases}
1& \mbox{for $s\le 4/3$;}\\
0& \mbox{for $s\ge 5/3$.}
\end{cases} 
\end{equation}

\subsection{On the choice of local coordinate charts and partitions of unity.}\label{ss:choice_local_chart}

Before giving the choice of local coordinate charts and partitions of unity precisely, we explain the motivation behind the choice. Observe first of all that, in the linear model that we discussed  in the last section, we may decompose the action of the transfer operator with respect to the frequency in $z$-direction. Indeed, in Lemma \ref{lm:U}, we had only point-wise multiplication in the third factor $L^2(\real_{\xi_z})$. 
Though this is not true for the transfer operators $\cL^t_{k,\ell}$ in exact sense, it is important to observe that they ``almost" preserve the frequency in the flow direction. 
And, based on such observation, we will decompose functions on $M$ so that each of the components has frequencies around some $\omega \in \integer$ in the flow direction and then consider the action of the transfer operators on each of them.
To look into such action, we choose a finite system of local charts and an associated partition of unity for each $\omega\in \integer$. 

When we consider the action of the transfer operator on a component with frequency around $\omega$, we will look things in the scale $\langle \omega\rangle^{-1/2}$ in the directions transversal to the flow. 
(Note that this corresponds to the scale in the definition of the partial Bargmann transform.) 
Accordingly we would like to consider a system of local charts and a partition of unity of size $\langle \omega\rangle^{-1/2+\theta}$. Then, in such small scale, the action of the transfer operators will be well-approximated by those for linear transformations considered in the last section. 

However, to proceed with this idea, we face one problem caused by the fact that the section $e_u:M\to G$ is only H\"older continuous. 
If we look its image $\mathrm{Im}\, e_u$ in the local chart of size $\langle \omega\rangle^{-1/2+\theta}$, its variation in the fiber directions of the Grassmann bundle will be proportional to $\langle \omega\rangle^{-\beta(1/2-\theta)}\gg\langle \omega\rangle^{-1/2+\theta}$. 
This is a problem because the image $\mathrm{Im}\, e_u$ must be approximated by the ``horizontal" subspace $\real^{2d}_{x}\oplus \{0\}\oplus \real_{z}\subset \real^{2d+d'+1}_{(x,y,z)}$ to make use of the argument in the last section.  
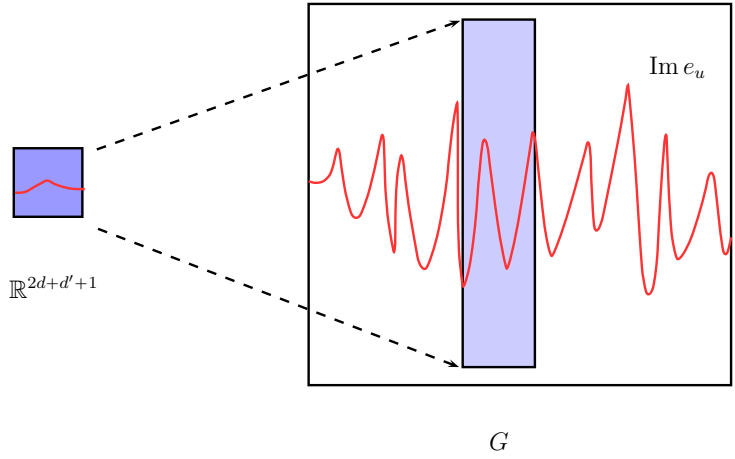
\begin{figure}
\scalebox{0.8}{
\scalebox{1} 
{
\begin{pspicture}(0,-3.7782812)(14.250546,3.7382812)
\definecolor{color75b}{rgb}{0.8,0.8,1.0}
\definecolor{color77}{rgb}{1.0,0.2,0.2}
\definecolor{color78b}{rgb}{0.6,0.6,1.0}
\psframe[linewidth=0.04,dimen=outer,fillstyle=solid,fillcolor=color75b](9.794687,3.4782813)(8.5546875,-2.3417187)
\psframe[linewidth=0.04,dimen=outer](13.0546875,3.7382812)(5.9946876,-2.6417189)
\pscustom[linewidth=0.04,linecolor=color77]
{
\newpath
\moveto(6.0146875,0.7672571)
\lineto(6.0946875,0.74659425)
\curveto(6.1346874,0.7362628)(6.2146873,0.74659425)(6.2546873,0.7672571)
\curveto(6.2946873,0.78791994)(6.3546877,0.8499078)(6.3746877,0.8912329)
\curveto(6.3946877,0.9325583)(6.4246874,1.0255401)(6.4346876,1.077197)
\curveto(6.4446874,1.1288538)(6.4696875,1.2321668)(6.4846873,1.2838233)
\curveto(6.4996877,1.3354797)(6.5246873,1.2941546)(6.5346875,1.2011728)
\curveto(6.5446873,1.1081909)(6.5796876,0.86023927)(6.6046877,0.70526916)
\curveto(6.6296873,0.5502994)(6.6896877,0.31267884)(6.7246876,0.23002838)
\curveto(6.7596874,0.14737792)(6.8346877,0.13704652)(6.8746877,0.20936584)
\curveto(6.9146876,0.28168517)(6.9946876,0.508974)(7.0346875,0.66394407)
\curveto(7.0746875,0.8189139)(7.1346874,1.077197)(7.1546874,1.1805103)
\curveto(7.1746874,1.2838233)(7.2096877,1.4491248)(7.2246876,1.5111127)
\curveto(7.2396874,1.5731006)(7.2596874,1.5214441)(7.2646875,1.4077994)
\curveto(7.2696877,1.2941546)(7.2796874,0.9842151)(7.2846875,0.78791994)
\curveto(7.2896876,0.59162474)(7.3196874,0.23002838)(7.3446875,0.06472717)
\curveto(7.3696876,-0.10057373)(7.4096875,-0.32786316)(7.4246874,-0.38985106)
\curveto(7.4396877,-0.451839)(7.4546876,-0.32786316)(7.4546876,-0.1418994)
\curveto(7.4546876,0.04406433)(7.4596877,0.36433533)(7.4646873,0.49864256)
\curveto(7.4696875,0.6329498)(7.4846873,0.8395764)(7.4946876,0.91189575)
\curveto(7.5046873,0.9842151)(7.5296874,1.1185223)(7.5446873,1.1805103)
\curveto(7.5596876,1.2424982)(7.5896873,1.1288534)(7.6046877,0.95322084)
\curveto(7.6196876,0.7775885)(7.6696873,0.3953296)(7.7046876,0.188703)
\curveto(7.7396874,-0.017923584)(7.7996874,-0.35885742)(7.8246875,-0.49316466)
\curveto(7.8496876,-0.6274719)(7.9196873,-0.7204541)(7.9646873,-0.6791284)
\curveto(8.009687,-0.6378027)(8.114688,-0.31753173)(8.174687,-0.038586427)
\curveto(8.234688,0.2403595)(8.309688,0.7156006)(8.324688,0.91189575)
\curveto(8.339687,1.1081909)(8.384687,1.5111127)(8.414687,1.7177392)
\curveto(8.444688,1.9243659)(8.479688,2.0999982)(8.484688,2.0690043)
\curveto(8.489688,2.0380104)(8.494687,1.7384018)(8.494687,1.4697872)
\curveto(8.494687,1.2011728)(8.494687,0.6329498)(8.494687,0.3333417)
\curveto(8.494687,0.03373291)(8.509687,-0.4105139)(8.524688,-0.5551526)
\curveto(8.539687,-0.69979125)(8.564688,-0.90641725)(8.574688,-0.9684058)
\curveto(8.584687,-1.0303937)(8.639688,-0.8754236)(8.684688,-0.65846556)
\curveto(8.729688,-0.44150758)(8.794687,0.10605225)(8.814688,0.43665466)
\curveto(8.834687,0.7672571)(8.879687,1.2424982)(8.904688,1.3871368)
\curveto(8.9296875,1.5317755)(8.979688,1.4491248)(9.004687,1.2218354)
\curveto(9.029688,0.9945462)(9.104688,0.42632324)(9.154688,0.085389405)
\curveto(9.204687,-0.25554442)(9.279688,-0.6378027)(9.3046875,-0.6791284)
\curveto(9.329687,-0.7204541)(9.439688,-0.25554442)(9.524688,0.2506909)
\curveto(9.609688,0.75692564)(9.709687,1.397468)(9.724688,1.5317752)
\curveto(9.739688,1.6660825)(9.779688,1.5214441)(9.8046875,1.2424982)
\curveto(9.829687,0.96355224)(9.889688,0.3953296)(9.924687,0.10605225)
\curveto(9.959687,-0.1832251)(10.024688,-0.47250184)(10.0546875,-0.47250184)
\curveto(10.084687,-0.47250184)(10.229688,-0.0799115)(10.344687,0.31267884)
\curveto(10.459687,0.70526916)(10.604688,1.2321668)(10.634687,1.366474)
\curveto(10.664687,1.5007813)(10.699688,1.3458115)(10.704687,1.0565344)
\curveto(10.709687,0.7672571)(10.734688,0.27135375)(10.754687,0.06472717)
\curveto(10.774688,-0.1418994)(10.854688,-0.059249267)(10.914687,0.23002808)
\curveto(10.974688,0.5193054)(11.094687,1.1185223)(11.154688,1.4284619)
\curveto(11.214687,1.7384015)(11.284687,2.1413233)(11.294687,2.2343054)
\curveto(11.3046875,2.3272872)(11.319688,2.3892753)(11.324688,2.3582811)
\curveto(11.329687,2.3272872)(11.359688,2.0173473)(11.384687,1.7384018)
\curveto(11.409687,1.4594562)(11.449688,0.9325586)(11.464687,0.6846066)
\curveto(11.479688,0.43665466)(11.504687,0.03373291)(11.514688,-0.12123657)
\curveto(11.524688,-0.27620667)(11.544687,-0.5551526)(11.5546875,-0.6791284)
\curveto(11.564688,-0.8031042)(11.599688,-0.9993994)(11.624687,-1.0717187)
\curveto(11.649688,-1.1440381)(11.704687,-1.1130444)(11.734688,-1.0097308)
\curveto(11.764688,-0.90641725)(11.8046875,-0.58614624)(11.814688,-0.36918822)
\curveto(11.824688,-0.15223022)(11.849688,0.31267884)(11.864688,0.5606305)
\curveto(11.879687,0.8085821)(11.909687,1.2218354)(11.924687,1.3871366)
\curveto(11.939688,1.5524378)(11.959687,1.5937629)(11.964687,1.469787)
\curveto(11.969687,1.3458111)(11.984688,0.9738837)(11.994687,0.725932)
\curveto(12.004687,0.47798035)(12.029688,0.075058594)(12.044687,-0.0799115)
\curveto(12.059688,-0.2348816)(12.104688,-0.451839)(12.134687,-0.5138269)
\curveto(12.164687,-0.57581544)(12.284687,-0.4001831)(12.374687,-0.16256225)
\curveto(12.464687,0.075058594)(12.594687,0.52963686)(12.634687,0.74659455)
\curveto(12.674687,0.96355224)(12.739688,0.96355224)(12.764688,0.74659455)
\curveto(12.789687,0.52963686)(12.829687,0.16804047)(12.844687,0.0234021)
\curveto(12.859688,-0.12123657)(12.889688,-0.32786316)(12.904688,-0.38985106)
\curveto(12.919687,-0.451839)(12.954687,-0.46217042)(12.974688,-0.4105139)
\curveto(12.994687,-0.35885742)(13.024688,-0.2348816)(13.034687,-0.16256225)
}
\psframe[linewidth=0.04,dimen=outer,fillstyle=solid,fillcolor=color78b](2.2746875,1.3382813)(1.0946875,0.15828125)
\psline[linewidth=0.04cm,linestyle=dashed,dash=0.16cm 0.16cm,arrowsize=0.05291667cm 2.0,arrowlength=1.4,arrowinset=0.4]{->}(2.4946876,1.2982812)(8.534687,3.4382813)
\psline[linewidth=0.04cm,linestyle=dashed,dash=0.16cm 0.16cm,arrowsize=0.05291667cm 2.0,arrowlength=1.4,arrowinset=0.4]{->}(2.5146875,-0.02171875)(8.514688,-2.3217187)
\pscustom[linewidth=0.04,linecolor=color77]
{
\newpath
\moveto(1.1346875,0.5782812)
\lineto(1.2146875,0.5782812)
\curveto(1.2546875,0.5782812)(1.3196875,0.59328127)(1.3446875,0.60828125)
\curveto(1.3696876,0.62328124)(1.4196875,0.6532813)(1.4446875,0.66828126)
\curveto(1.4696875,0.68328124)(1.5146875,0.7082813)(1.5346875,0.71828127)
\curveto(1.5546875,0.72828126)(1.5996875,0.75328124)(1.6246876,0.7682812)
\curveto(1.6496875,0.78328127)(1.6946875,0.78328127)(1.7146875,0.7682812)
\curveto(1.7346874,0.75328124)(1.7796875,0.72828126)(1.8046875,0.71828127)
\curveto(1.8296875,0.7082813)(1.8896875,0.68828124)(1.9246875,0.67828125)
\curveto(1.9596875,0.66828126)(2.0246875,0.6532813)(2.0546875,0.6482813)
\curveto(2.0846875,0.6432812)(2.1496875,0.6382812)(2.1846876,0.6382812)
\curveto(2.2196875,0.6382812)(2.2646875,0.6382812)(2.2946875,0.6382812)
}
\usefont{T1}{ptm}{m}{n}
\rput(9.179961,-3.5567188){\Large $G$}
\usefont{T1}{ptm}{m}{n}
\rput(1.779961,-0.9767187){\Large $\real^{2d+d'+1}$}
\usefont{T1}{ptm}{m}{n}
\rput(12.1499605,2.6632812){\Large $\mathrm{Im}\,e_u$}
\end{pspicture} 
}}
\caption{A schematic picture of the choice of local charts.}
\label{fig:choice}
\end{figure}

Our solution for this problem is rather simple-minded: We  choose the system of local  charts so that the section $\mathrm{Im}\, e_u$ ``looks" horizontal. 
 (See Figure \ref{fig:choice}.) More precisely, we choose local coordinate charts for the parameter $\omega\in \integer$ so that the objects look contracted by the rate $\langle \omega \rangle^{-(1-\beta)(1/2-\theta)-4\theta}$ in the fiber directions of the Grassmann bundle.
Then the ``vertical" variation of the section $\mathrm{Im}\, e_u$ in such coordinates will be bounded by
 \[
 \langle \omega\rangle^{-\beta(1/2-\theta)}\cdot \langle \omega \rangle^{-(1-\beta)(1/2-\theta)-4\theta} =\langle \omega\rangle^{-1/2-3\theta}\ll \langle \omega\rangle^{-1/2}
 \]
 so that the section $\mathrm{Im}\, e_u$ will look ``horizontal" in the scale $\langle \omega\rangle^{-1/2}$.  
 Of course, there are some drawbacks of such choice of (asymptotically) singular  local charts. At least, we have to be careful about the  non-linearity of the flow viewed in such local charts. Also we will find a related technical problem, as  discussed in the beginning of the next section.
\begin{Remark}
The idea mentioned above is equivalent to consider a generalization of Bargmann transform using eccentric wave packets. 
\end{Remark}

\subsection{Hyperbolicity of the flow  $f^t_G$}\label{ss:hyperbolicity} 
The flow $f^t_G$ is hyperbolic in the neighborhood $U_0$ of the section $\mathrm{Im}\, e_u$ with hyperbolic decomposition given in (\ref{eq:hyp_docomp_G}), as we noted in Section \ref{sec:Gext}.  
For the argument below, we take the ``maximum" exponent  $\chi_{\max}>\chi_0$ so that  
\begin{equation}\label{eq:chi_max}
|Df^t_G(v)|\le C_0\cdot e^{\chi_{\max}\cdot |t|}\cdot |v|\quad \mbox{for any $t\in \real$ and $v\in TU_0$}
\end{equation}
with some constant $C_0>0$. 
\begin{Remark}\label{rem:choice_chi0}
The hyperbolicity exponent $\chi_0>0$ was taken as the constant satisfying the condition (\ref{eq:Anosov2}). But, in what follows, we additionally suppose that the condition (\ref{eq:Anosov2}) remains true if we replace  $\chi_0$ with $\chi_0+\varepsilon$ for some small $\varepsilon>0$. Since we have not used $\chi_0$ from Section \ref{sec:linear} to this point and since the equalities in the main theorems related to $\chi_0$ are strict ones, this does not cause any loss of generality. 
\end{Remark}

In the next lemma, we introduce a continuous (but not necessarily smooth) Riemann metric $|\cdot|_*$ on $U_0\subset G$ which is adapted to the flow. 
\begin{Lemma}\label{lm:metric}
There exists a $\beta$-H\"older continuous Riemann metric $|\cdot|_{*}$ on  $U_0\subset G$ such that, for the decomposition 
$TU_0=\widetilde{E}_u\oplus \widetilde{E}_s\oplus \widetilde{E}_0$ in {\rm (\ref{eq:hyp_docomp_G})}, we have
 
\begin{enumerate}
\item $|Df^t_G(v)|_*\ge e^{\chi_0 t}\cdot |v|_*$ for $v\in \widetilde{E}_u$ and $t\ge 0$, 
\item $|Df^t_G(v)|_*\le e^{-\chi_0 t}\cdot  |v|_*$ for $v\in \widetilde{E}_s$ and $t\ge 0$,  
\item $|v|_*=|\alpha(v)|$ for $v\in \widetilde{E}_0$, 
\item $\widetilde{E}_s$, $\widetilde{E}_u$ and $\widetilde{E}_0$ are orthogonal to each other with respect to the metric $|\cdot|_*$,  
\item $|v|_*=\sup\{ d\alpha(v,v') \mid  v'\in \widetilde{E}_s, |v'|_*=1\}$ for $v\in \widetilde{E}_u$. 
\end{enumerate}
\end{Lemma}
\begin{proof} The construction is standard. For $v\in\widetilde{E}_s$, we set 
\[
|v|_*=\int_0^T e^{+\chi_0 t} |Df_G^t(v)| dt
\]
where $T>0$ is a constant. Letting  $T>0$ be sufficiently large and using Remark \ref{rem:choice_chi0}, we see that the condition  (2) is fulfilled for $0\le t\le 1$ and hence for all $t\ge 0$. For the vectors in $\widetilde{E}_u$ and $\widetilde{E}_0$, we define the norm $|\cdot|_*$ uniquely so that the conditions (5) and (3) hold. We extend such construction so that the condition (4) holds.
Then it is easy to check the condition (1): For $v\in \widetilde{E}_u$, we have
\begin{align*}
|Df^t_G(v)|_*&=\sup\{ d\alpha(Df^t(v), v')\mid v'\in \widetilde{E}_s, |v'|_*=1\}\\
&=\sup\{ d\alpha(v, v'')\mid v''\in \widetilde{E}_s, |Df^t(v'')|_*=1\}\ge e^{\chi_0 t}\cdot |v|_*. 
\end{align*}
The H\"older continuity of $\|\cdot\|_*$ follows from that of the decomposition  $TU_0=\widetilde{E}_u\oplus \widetilde{E}_s\oplus \widetilde{E}_0$ and the construction above. 
\end{proof}

\subsection{Darboux charts} \label{ss:Darboux} 
The next lemma is a slight extension of the Darboux theorem \cite[pp.168]{Aebischer} for contact structure.  
\begin{Lemma}\label{th:Darboux}
There exists a finite system of local coordinate charts on $M$,
\[
\check{\kappa}_a:\check{V}_a\subset \real^{2d+1}_{(x,z)} \to \check{U}_a:=\check{\kappa}_a(\check{V}_a) \subset M\quad \text{for }a\in A,
\]
and corresponding  local coordinate charts on $G$,
\[
{\kappa}_a:V_a\subset  \real^{2d+d'+1}_{(x,y,z)} \to U_a:=\kappa_a(V_a) \subset G \quad \text{for }a\in A,
\]
such that  
\begin{enumerate}
\item $\check{V}_a=\proj_x(\check{V}_a)\times (-s_0,s_0)$ for some $s_0>0$,
\item $\check{\kappa}_a$ are Darboux charts, that is, $\check{\kappa}^*_a\alpha=\alpha_0$ on $\check{U}_a$, where $\alpha_0$ is the standard contact form in (\ref{eq:walpha}).  
\item $U_a=\pi_G^{-1}(\check{U}_a)\cap U_0$, where $U_0$ is the absorbing neighborhood of the attractor $\mathrm{Im}(e_u)$, and the following diagram commutes:
\[
\begin{CD}
U_a@<{\kappa_a}<< V_a\subset \real^{2d+d'+1}_{(x,y,z)}\\
@V{\pi_G}VV @V{\proj_{(x,z)}}VV\\
\check{U}_a @<{\check{\kappa}_a}<<  \check{V}_a\subset \real^{2d+1}_{(x,z)}
\end{CD}
\]
\item the pull-back of the generating vector field of $f^t_G$ by $\kappa_a$ is the  (constant) vector field $\partial_z$ on $\real^{2d+d'+1}_{(x,y,z)}$.
\end{enumerate}
\end{Lemma}
\begin{proof} The Darboux theorem for contact structure  gives the Darboux charts $\check{\kappa}_a$, $a\in A$, satisfying the condition (2).  The generating vector field viewed in those coordinates are the constant vector field $\partial_z$ on $\real^{2d+1}_{(x,z)}$ because it is characterized as the Reeb vector field of $\alpha$. 
Then  we  can easily define the extended charts $\kappa_a$ for $a\in A$ so that the conditions (1), (3) and (4) hold.\end{proof}

We henceforth fix the local charts $\kappa_a$ in the lemma above. 
The time-$t$-map of the flow $f^t_{G}$ viewed in those local charts are 
\[
f^t_{a\to a'}:=\kappa_{a'}^{-1}\circ f^t_{G}\circ \kappa_a:\kappa^{-1}_a(U_a\cap f^{-t}_G(U_{a'}))\to \kappa_{a'}^{-1}(f^t_G(U_a)\cap U_{a'}).
\] 
The  next lemma is a consequence of  the choice of the local charts $\kappa_a$. 
\begin{Lemma}
The mappings $f^t_{a\to a'}:\kappa^{-1}_a(U_a\cap f^{-t}_G(U_{a'}))\to \kappa_{a'}^{-1}(f^t_G(U_a)\cap U_{a'})$ are fibered contact diffeomorphisms (defined in Definition \ref{def:fiberedContact}). Further we have
\[
f^{t+s}_{a\to a'}(x,y,z)=f^{t}_{a\to a'}(x,y,z+s)
\]
provided that both sides are defined and $s>0$ is sufficiently small. 
\end{Lemma}

Let  $K_0\Subset U_0$ be a compact neighborhood  of the section $\mathrm{Im}\, e_u$ such that 
\begin{equation}\label{eq:choice_K0}
f_G^t(K_0)\Subset K_0\;\;\mbox{ for  all $t> 0$. }
\end{equation}
We take and fix  a family of  $C^\infty$ functions 
\[
\rho_a:V_a\to [0,1]\quad\mbox{for $a\in A$}
\]
such that  $\supp \rho_a\Subset V_a$ and that
\[
\sum_{a\in A}\rho_a\circ \kappa_a^{-1}\equiv 
\begin{cases}
1,&\qquad \mbox{on $K_0$;}\\
0,&\qquad \mbox{on a neighborhood of $G\setminus U_0$.}
\end{cases}
\]
For the argument in the later sections, we take another family of smooth functions 
\[
\tilde{\rho}_a:V_a\to [0,1]
\quad \mbox{
for $a\in A$}
\]
such that $\supp \tilde{\rho}_a\subset V_a$ and $\tilde{\rho}_a\equiv 1$ on $\supp \rho_a$.

\subsection{The local charts adapted to the hyperbolic structure}
\label{ss:local_coord_hyp}
In the next proposition,  we construct local  charts that are more adapted to the hyperbolic structure of the flow $f_G^t$, by pre-composing  affine transformations. These local charts are  centered at the points of the form $e_u(\check{\kappa}_a(x,0))$ with $x\in \proj_{x}({V}_a)\subset \real^{2d}$.
\begin{Proposition}\label{pp:kappaaw}
For $a\in A$ and $x\in \proj_{x}({V}_a)\subset \real^{2d}$, we can choose an affine transformation 
\[
A_{a,x}:\real^{2d+d'+1}_{(x,y,z)}\to \real^{2d+d'+1}_{(x,y,z)}
\]
in the transformation group $\cA_0$ (defined in Definition \ref{def:transG}) so that, if we set  
\[
\kappa_{a,x}:= \kappa_{a}\circ A_{a,x}:A_{a,x}^{-1}(V_a)\to U_a,
\]
it sends the origin $0\in \real^{2d+d'+1}_{(x,y,z)}$ to the point $e_u(\check{\kappa}_a(x,0))\in \mathrm{Im}\, e_u$ and  
 the differential $(D{\kappa}_{a,x})_0$ is isometric with respect to  the Euclidean metric in the source and the Riemann metric $|\cdot|_*$ in the target and, further, $(D{\kappa}_{a,x})_0$ sends the components of the decomposition
\[
T_{0}\real^{2d+d'+1}_{(x,y,z)}=\real^{d}_q\oplus \real^{d}_p\oplus \real^{d'}_y\oplus \real_z 
\]
to those of the decomposition
\[
T_{e_u(\check{\kappa}_a(x,0))} G=\widetilde{E}_u\oplus_* (\widetilde{E}_s\ominus_*  \ker D\pi_G)\oplus \ker D\pi_{G} \oplus \widetilde{E}_0
\]
respectively in this order. ($\widetilde{E}_s\ominus_*  \ker D\pi_G$ denotes the orthogonal complement of $\ker D\pi_G$ in $\widetilde{E}_s$ with respect to the Riemann metric $\|\cdot\|_*$.)

\end{Proposition}
\begin{proof} By applying  Lemma \ref{lm:affine_trans}, we can find the affine map $A_{a,x}$ such that all the conditions in the conclusion hold true, but for the isometric property. By pre-composing a simple linear map of the form 
\[
C\oplus C^{\dag}\oplus \widehat{C}\oplus \mathrm{Id}:\real^{d}_q\oplus \real_p^d\oplus \real_y^{d'}\oplus \real_z\to \real^{d}_q\oplus \real^d_p\oplus \real_y^{d'}\oplus \real_z,
\]
we may modify $A_{a,x}$ so that $(D{\kappa}_{a,x})_0$ restricted to $\{0\}\oplus \real^{d}_p\oplus \{0\}\oplus \{0\}$ and $\{0\}\oplus \{0\} \oplus \real^{d'}_y\oplus \{0\}$  are respectively isometric. Then, from  (3) and (5)   in Lemma \ref{lm:metric}, we  see that  $(D{\kappa}_{a,x})_0$ restricted to 
$\{0\}\oplus \{0\}\oplus \{0\}\oplus \real_{z}$ and $\real^d_q\oplus \{0\}\oplus \{0\}\oplus \{0\}$ are also  isometric and then, from  (4), we obtain the isometric property of $(D{\kappa}_{a,x})_0$. 
\end{proof}

\subsection{The local coordinate charts parametrized by $\omega\in \integer$}\label{ss:chart_omega}
For each integer $\omega\in \integer$, we set up a finite system of local charts and an associated partition of unity. 
Following the idea explained in the beginning of this section, 
we define the local charts  as the composition of  the charts $\kappa_{a,x}$ introduced in the last subsection with the partially expanding linear map  
\begin{equation}\label{eq:Eomega}
E_\omega:\real^{2d+d'+1}_{(x,y,z)}\to \real^{2d+d'+1}_{(x,y,z)},\quad 
E_\omega(x,y,z)=(x,\langle \omega\rangle^{(1-\beta)(1/2-\theta)+4\theta} y,z).
\end{equation} 
\begin{Remark}\label{rem:num}
We will use the numerical relation 
\begin{equation}\label{eq:theta_beta}
(1-\beta)(1/2-\theta)\le (1-\beta)/2\le 1/2-10\theta,
\end{equation}
which follows from the choice of the constant $\theta>0$ in (\ref{eq:choice_of_theta}). 
\end{Remark}

For each $a\in A$ and $\omega\in \integer$, we consider the following  finite subset  of $\real^{2d}$:
\[
\cN (a,\omega)=\{n\in \real^{2d}\;\mid\; 
  n\in (\langle \omega\rangle^{-1/2+\theta}\cdot \integer^{2d}) \cap \proj_x(V_a) \}.
\]
For each element $n\in \cN (a,\omega)$, we define the local  chart $\kappa_{a,n}^{(\omega)}$ by
\begin{equation}\label{def:kappa_anw}
\kappa_{a,n}^{(\omega)}:= \kappa_{a,n}\circ E_{\omega}:
V_{a,n}^{(\omega)}:=E_{\omega}^{-1}\circ A_{a,n}^{-1}(V_a)\to U_a,
\end{equation}
where $\kappa_{a,n}$ is the local chart $\kappa_{a,x}$ defined in Proposition \ref{pp:kappaaw} for $x=n$. 

Next, for  each integer $\omega\in \integer$, we introduce  a partition of unity associated to the system of local coordinate charts $\{\kappa_{a,n}^{(\omega)}\}_{a\in A, n\in \cN (a,\omega)}$. First we take and fix  a smooth  function $\rho_0:\real^{2d}\to [0,1]$ so that the support is contained in the cube $(-1,1)^{2d}$ and that 
\begin{equation}\label{eq:rho_n}
\sum_{n\in \integer^{2d}} \rho_0(x-n)\equiv 1\quad \mbox{for all $x\in \real^{2d}$.}
\end{equation}
({For instance, set
$\rho_0(x)=\prod_{i=1}^{2d} (\chi(x_i+1)-\chi(x_i+2))$ for  $x=(x_i)_{i=1}^{2d}\in \real^{2d}$, using the function $\chi(\cdot)$ in (\ref{eq:def_chi}).}
)
For $a\in A$ and $n\in \cN(a,\omega)$, we define the function $\rho_{a,n}^{(\omega)}:\real^{2d+d'+1}_{(x,y,z)}\to [0,1]$ by
\begin{equation}\label{eq:rhoan}
\rho_{a,n}^{(\omega)}(x,y,z)=\rho_{a}(x',y',z')\cdot \rho_0(\langle \omega\rangle^{1/2-\theta}
(x'-n))
\end{equation}
where $
(x',y',z')=A_{a,n}\circ E_\omega(x,y,z)$. 
From this definition, we have
\begin{equation}\label{eq:relRho}
\rho_{a,n}^{(\omega)}\circ (\kappa_{a,n}^{(\omega)})^{-1}(p)=\rho_a\circ \kappa_a^{-1}(p)\cdot \rho_0(\langle \omega\rangle^{1/2-\theta}(\proj_x\circ \kappa_{a,n}^{-1}(p)-n)).
\end{equation}
Hence, from (\ref{eq:rho_n}) and the choice of $\rho_a$,  we have, for each $\omega\in \integer$, that
\[
\sum_{a\in A}\sum_{n\in \cN(a,\omega)} \rho_{a,n}^{(\omega)}\circ 
(\kappa_{a,n}^{(\omega)})^{-1}
=\sum_{a\in A} \rho_a\circ \kappa_a^{-1}
\equiv 
\begin{cases}
1&\qquad \mbox{on $K_0$;}\\
0&\qquad \mbox{on a neighborhood of $G\setminus U_0$.}
\end{cases}
\]
That is,  the set of functions
\[
\{\rho_{a,n}^{(\omega)}\circ 
(\kappa_{a,n}^{(\omega)})^{-1}:U_0\to [0,1] \mid a\in A, n\in \cN (a,\omega)\}
\]
is  a partition of unity on $K_0$ supported on $U_0$.

For the argument in the later sections, we define an ``enveloping" family of functions, $\tilde{\rho}_{a,n}^{(\omega)}(\cdot)$ for $\omega\in \integer$, $a\in A$ and $n\in \mathcal{N}(a,\omega)$, by
\begin{equation}\label{eq:trhoan}
\tilde{\rho}_{a,n}^{(\omega)}(x,y,z)=\tilde{\rho}_{a}(x',y',z')\cdot \tilde{\rho}_0(\langle \omega\rangle^{1/2-\theta}(x'-n))
\end{equation} 
where $(x',y',z')=A_{a,n}\circ E_\omega(x,y,z)$ and $\tilde{\rho}_0:\real^{2d}\to [0,1]$ is a $C^\infty$ function  supported on the cube $(-1,1)^{2d}$ such that $\tilde{\rho}_0\equiv 1$ on the support of $\rho_0$. 
By definition, we have $\tilde{\rho}_{a,n}^{(\omega)}(x,y,z)\equiv 1$ on the support of $\rho_{a,n}^{(\omega)}$.

\section{Modified  anisotropic Sobolev spaces $\cK^{r,\sigma}(K_0)$} \label{sec:modifiedASS}
We define the function spaces $\cK^{r,\sigma}(K_0)$, called {\em the modified anisotropic Sobolev spaces},  that consist of distributions on the neighborhood $K_0$ of the attracting section $\mathrm{Im}\, e_u$. The function spaces in Theorem \ref{th:spectrum} will be obtained from them by a simple procedure of time averaging of the norm. (See Definition \ref{def:modifiedAHS}.) 
We will use a  simple periodic partition of unity $\{q_\omega:\real\to [0,1]\}_{\omega\in \integer}$ defined by 
\begin{equation}\label{eq:chi_omega}
q_\omega(s)=\chi(s-\omega+1)-\chi(s-\omega+2).
\end{equation}

\subsection{The problems caused by the factor $E_\omega$}\label{ss:problem_K}
A simple idea to define the Hilbert spaces in Theorem \ref{th:spectrum}  is  to patch the anisotropic Sobolev spaces $\cH^{r,\sigma}(\real^{2d+d'+1})$ using  the local charts $\kappa_{a,n}^{(\omega)}$ and the partition of unity $\rho_{a,n}^{(\omega)}$. 
But, proceeding with this idea, we face one problem caused by the singularity of the local charts $\kappa_{a,n}^{(\omega)}$. This forces us to give a more involved definition of the modified anisotropic Sobolev space $\cK^{r,\sigma}(K_0)$ in the following subsections.

The difficulty may be explained as follows. 
(The explanation below may not be very clear and is not indispensable for the argument in the following subsections.)
For facility of explanation, suppose that $M=\real^{2d+1}_{(x,z)}$ and $G=\real^{2d+d'+1}_{(x,y,z)}$ and that the  the local chart to look functions with frequency around $\omega\in \integer$ along the flow direction (or the $z$-axis) is just the partial expanding map $E_\omega$ in (\ref{eq:Eomega}).
Then, if we follow the idea mentioned above, the norm that we consider for a function $u$ on $G=\real^{2d+d'+1}_{(x,y,z)}$ will look like 
\[
\|u\|=\|\mathfrak{W}^r\cdot \pBargmann u\|_{L^2}
\quad \mbox{with}\quad 
\mathfrak{W}^r(w,\xi_w,\xi_z):=\sum_{\omega}q_\omega(\xi_z)\cdot \cW^r\circ D^*E_\omega(w,\xi_w,\xi_z)
\]
where $D^*E_\omega:\real^{4d+2d'+1}_{(w,\xi_w,\xi_z)}\to \real^{4d+2d'+1}_{(w,\xi_w,\xi_z)}$ denotes the pull-back by $E_\omega$. The problem with this norm is that the function $\mathfrak{W}^r(\cdot)$ is rather singular in the limit $|\xi_w|\to \infty$.
To be more precise, let us consider two integers $0\ll \omega\ll \omega'$. 
Since the non-conformal property of $D^*E_\omega$ depends on $\omega$ and since  the weight function  $\cW^r$ is anisotropic, we can find  $\xi_w\in \real^{2d+d'}$ such that 
\begin{equation}\label{prob}
\lim_{s\to +\infty} \frac{\mathfrak{W}^r(0, s\cdot\xi_w,\omega)}{\mathfrak{W}^r(0,s\cdot\xi_w,\omega')}=\infty.
\end{equation}
while the distance between the points $(0, s\cdot\xi_w,\omega)$ and $(0, s\cdot\xi_w,\omega')$ is bounded uniformly in $s$. 
With this singularity\footnote{In terms of the theory of pseudo-differential operator, this implies that the function $\mathfrak{W}^r$ does not belong to an appropriate  class of symbols.}
 of $\mathfrak{W}^r$, even multiplications by moderate smooth functions will be unbounded with respect to the norm above.

We emphasize that the problem mentioned above does not affect the most essential part of our argument because it  happens only for the action of transfer operators on the wave packets that are very far from the trapped set. (Recall the explanation at the end of Section \ref{sec:Gext}.)  
The modification of the definition of the Hilbert spaces will be described in the following subsections. 
The idea is simply to relax the non-conformal property of the factor $E_\omega$ gradually and sufficiently slowly as we go far from the trapped set.

\subsection{Partitions of unity on the phase space}\label{ss:modified_pu}
We introduce a few partitions of unity on the phase space $\real^{4d+2d'+1}_{(w,\xi_w,\xi_z)}$. 
\subsubsection{Interpolating $E_\omega$ and  the identity map}
We first construct a family of linear maps $E_{\omega,m}$ for $m\ge 0$, which interpolates the linear map  $E_\omega$ and  the identity map. To begin with, we introduce two constants
\[
1<\Theta_1<\Theta_2. 
\]
These constants will be used to specify the interval of integers $m$ where we do the relaxation of the singularity (or non-conformal property) of the local coordinate charts. The choice of these constant are rather arbitrary. But, to make sure that the relaxation takes place sufficiently slowly, we  suppose
\begin{equation}\label{eq:choice_of_Theta}
\Theta_2-\Theta_1>10\cdot \frac{\chi_{\max}}{\chi_0}>10,
\end{equation}
so that 
\begin{equation}
\mu:=\frac{(1-\beta)(1/2-\theta)+4\theta}{\Theta_2-\Theta_1}
<\frac{1/2}{\Theta_2-\Theta_1}
<\frac{\chi_0}{20 \cdot \chi_{\max}}<\frac{1}{20}.  
\end{equation}
For each $\omega\in \integer$, we set 
\begin{equation}\label{eq:ni}
 n_0(\omega):=\left[\theta \cdot 
 \log \langle \omega\rangle\right],\quad 
 n_1(\omega):=\left[\Theta_1 \cdot 
 \log \langle \omega\rangle\right],\quad 
  n_2(\omega):=\left[\Theta_2\cdot  \log \langle \omega\rangle\right]
  \end{equation}
so that
\[
e^{n_0(\omega)}\sim \langle \omega\rangle^{\theta},\quad
e^{n_1(\omega)}\sim \langle \omega\rangle^{\Theta_1}\quad  \mbox{and}\quad e^{n_2(\omega)}\sim \langle \omega\rangle^{\Theta_2}.
\]
Then we define a function  $e_\omega :\integer_+ \to \real$
by 
\begin{equation}\label{def:eomega}
e_\omega(m)=
\begin{cases}
1,&\quad \mbox{if }m\le n_1(\omega);\\
e^{\mu(m-n_1(\omega))},&\quad \mbox{if }
n_1(\omega)< m<  n_2(\omega);\\
\langle \omega\rangle^{(1-\beta)(1/2-\theta)+4\theta},&\quad \mbox{if }m\ge n_2(\omega).
\end{cases}
\end{equation}
From the choice of the constants above, this function  varies slowly satisfying
\begin{equation}\label{eq:e}
\frac{e_{\omega}(m)}{e_{\omega}(m')}\le e^{\mu(|m-m'|+2)}< e^{(|m-m'|+2)/20}.
\end{equation}
We define the family of  linear maps $E_{\omega,m}:\real^{2d+d'+1}_{(x,y,z)}\to \real^{2d+d'+1}_{(x,y,z)}$ for $\omega\in \integer$ and $m\in \integer_+$ by 
\[
E_{\omega,m}(x,y,z)
=(x,\;e_{\omega}(m) \cdot y,\; z).
\]
From the definition, this family interpolates $E_\omega$ and the identity map in the sense that  $E_{\omega,m}=\mathrm{Id}$ if $m\le n_1(\omega)$ and $E_{\omega,m}=E_\omega$ if $m\ge n_2(\omega)$.

Let $D^*E_{\omega,m}:\real^{4d+2d'+1}_{(x,y,\xi_x, \xi_y,\xi_z)}\to \real^{4d+2d'+1}_{(x,y,\xi_x, \xi_y,\xi_z)}$ be the natural pull-back action of $E_{\omega,m}$ on the cotangent bundle: 
\begin{equation}\label{eq:Eomeagdag} 
D^*E_{\omega,m}(x,y,\xi_x,\xi_y,\xi_z)=(x,\;e_{\omega}(m)^{-1} y,\;\xi_x,\; e_{\omega}(m)\xi_y,\; \xi_z).
\end{equation}

\subsubsection{A partition of unity on the phase space} 
We next define  partitions of unity on the phase space $\real^{4d+2d'+1}_{(x,y,\xi_x, \xi_y,\xi_z)}$. 
Recall the periodic partition of unity $\{q_\omega\}_{\omega\in \integer}$ on the real line $\real$ and also the function $\chi(\cdot)$, defined in (\ref{eq:chi_omega}) and (\ref{eq:def_chi}) respectively.  
We  define a family of functions 
\begin{equation}\label{eq:Xm}
X_{m}:\real^{4d+2d'+1}_{(x,y,\xi_x, \xi_y,\xi_z)}\to [0,1]\quad \mbox{for $m\in \integer_+$}
\end{equation}
by
\[
X_{m}(x,y,\xi_x,\xi_y,\xi_z)=
\chi(e^{-m}|(\zeta_p,\tilde{\xi}_y,\zeta_q, \tilde{y})|)
\]
where $\zeta_p, \zeta_q, \tilde{ y}, \tilde{\xi}_y$ are coordinates on $\real^{4d+2d'+1}_{(x,y,\xi_x, \xi_y,\xi_z)}$ introduced in  (\ref{eq:newcoordinates}). Then  we define the functions 
\[
\widetilde{X}_{\omega,m}:\real^{4d+2d'+1}_{(x,y,\xi_x, \xi_y,\xi_z)}\to [0,1]\quad \mbox{for $\omega\in \integer$ and $m\in \integer$ with $m\ge   n_0(\omega)$}
\]
by
\[
\widetilde{X}_{\omega,m}=
\begin{cases}
X_{n_0(\omega)}\cdot q_{\omega}=(X_{n_0(\omega)}\circ D^*E_{\omega,m}^{-1})\cdot q_{\omega}&\quad \mbox{if $m=n_0(\omega)$;}\\
(X_{m}\circ D^*E_{\omega,m}^{-1} -X_{m-1}\circ D^*E_{\omega,m-1}^{-1})\cdot q_{\omega}&\quad \mbox{if $m>n_0(\omega)$}
\end{cases}
\]
where (and also in many places in the following) we understand $q_{\omega}(\cdot)$ as a function of the coordinate $\xi_z$ in $(x,y,\xi_x, \xi_y,\xi_z)$.  
By this construction, the family of functions
\[
\{ \widetilde{X}_{\omega,m}:\real^{4d+2d'+1}_{(x,y,\xi_x, \xi_y,\xi_z)}\to [0,1]\mid  \omega\in \integer, m\in \integer_+\mbox{ with $m\ge n_0(\omega)$}\}
\]
is a partition of unity on $\real^{4d+2d'+1}_{(x,y,\xi_x, \xi_y,\xi_z)}$. 
\begin{Remark}\label{Rem:index}
The index $m$ above is related to the distance of the support of $\widetilde{X}_{\omega,m}$ from the trapped set $X_0$, while $\omega$ indicates the values of the coordinate $\xi_z$. 
When $m=n_0(\omega)$, the support is contained in the $e^2 \langle \omega\rangle^{\theta}$-neighborhood of the trapped set $X_0$ in the standard Euclidean norm in coordinates introduced in (\ref{eq:newcoordinates}). 
When $n_0(\omega)\le m\le n_1(\omega)$, it is contained in the region where 
the distance from the trapped set $X_0$ is in between $e^{m-2}$ and 
 $e^{m+2}$. When $m\ge n_1(\omega)$, the situation is a little more involved because the modification by the family of linear maps $D^*E_{\omega,m}^{-1}$  takes effect. 
(If we look things through the linear map $D^*E_{\omega,m}^{-1}$, we have a parallel description.) 
\end{Remark}

Next, for $\sigma\in \Sigma$, we define the functions $
Z^\sigma_+, Z^\sigma_-:\real^{4d+2d'+1}_{(x,y,\xi_x, \xi_y,\xi_z)}\setminus\{0\}\to [0,1]$
by 
\begin{align*}
Z^\sigma_+(x,y,\xi_x, \xi_y,\xi_z)&=\frac12(1-\mathrm{ord}^{\sigma}([(\zeta_p,\tilde{\xi}_y,\zeta_q, \tilde{y})]))
\intertext{and}
Z^\sigma_-(x,y,\xi_x, \xi_y,\xi_z)&=\frac12(\mathrm{ord}^{\sigma}([(\zeta_p,\tilde{\xi}_y,\zeta_q, \tilde{y})])+1).
\end{align*}
Obviously we have $Z^\sigma_+(\cdot)+Z^\sigma_-(\cdot)\equiv 1$. (Recall (\ref{eq:def_order_function_m}) and (\ref{ord+}) for the definition of the function $\mathrm{ord}^{\sigma}(\cdot)$.) 
For each $m\in \integer$ with $m\neq 0$, we set 
\[
Z^{\sigma}_{+,\omega,m}=Z^{\sigma}_{+}\circ D^*E_{\omega,m}^{-1}\quad 
\mbox{and}\quad 
Z^{\sigma}_{-,\omega,m}=Z^{\sigma}_{-}\circ D^*E_{\omega,m}^{-1}.
\]
Again we have $Z^\sigma_{+,\omega,m}(\cdot)+Z^\sigma_{-,\omega,m}(\cdot)\equiv 1$ for each integer $\omega$ and  $m\ge 0$. 
\begin{Remark}
The supports of $Z^\sigma_\pm(\cdot)$ are contained in some conical subsets in the coordinates $(\zeta_p,\tilde{\xi}_y,\zeta_q, \tilde{y})$ in the stable or unstable direction. 
For instance, we have $Z^\sigma_+(x,y,\xi_x, \xi_y,\xi_z)\neq 0$ only if $(\zeta_p,\tilde{\xi}_y,\zeta_q, \tilde{y})$ belongs to the cone $\cone_{+}^{(d+d',d+d')}(2\cdot 2^{-\sigma/2})$.
This is true also for the supports of $Z^\sigma_{\pm,\omega,m}(\cdot)$, but the corresponding cones will be distorted by the factor $D^*E_{\omega,m}$. 
\end{Remark}

Finally  we define the functions 
\[
\Psi_{\omega, m}^\sigma:\real^{4d+2d'+1}_{(x,y,\xi_x, \xi_y,\xi_z)}\to [0,1],\quad\mbox{for $\omega\in \integer$, $m\in \integer$ and $\sigma\in \Sigma$}
\]
by
\begin{equation}\label{def:psi_sigma}
\Psi_{\omega, m}^{\sigma}=
\begin{cases}
\widetilde{X}_{\omega,n_0(\omega)},&\quad \mbox{if $m=0$;}\\
0,&\quad \mbox{if $0<|m|\le n_0(\omega)$;}\\
\widetilde{X}_{\omega,m} \cdot 
Z_{+,\omega,m}^{\sigma},&\quad \mbox{if $m> n_0(\omega)$;}\\
\widetilde{X}_{\omega,|m|} \cdot
Z_{-,\omega,m}^{\sigma},&\quad \mbox{if $m< -n_0(\omega)$.}
\end{cases}
\end{equation}

For each $\sigma\in \Sigma$, the family of functions $\{\Psi_{\omega, m}^{\sigma}\mid \omega\in \integer, m\in \integer\}$ is a partition of unity on $\real^{4d+2d'+1}_{(x,y,\xi_x, \xi_y,\xi_z)}$. For the support of the function $\Psi_{\omega, m}^\sigma$, the index $\omega$ indicates the approximate value of $\xi_z$, the absolute value 
of $m$ indicates the distance from the trapped set, and the  sign of $m$ indicates the (stable or unstable) directions from the trapped set. 

\begin{Remark}
In the definitions above, we suppose that the coordinates $(w,\xi_w,\xi_z)$ corresponds to those given by the distorted local charts $\kappa_{a,n}^{(\omega)}$ in (\ref{def:kappa_anw}). 
Note that, by definition, the supports of the functions $\Psi_{\omega, m}^{\sigma}$ in the partition of unity above will look somewhat regular if $m\le n_1(\omega)$ and they become distorted gradually by the factor $D^*E_{\omega,m}$ as $m$ increase. If $m\ge n_2(\omega)$ and $|\omega|$ is large, the supports of $\Psi_{\omega, m}^{\sigma}$ will be strongly distorted. However, if we look things in the usual local coordinate charts (say, $\kappa_{a,n}^{(\omega)}$ without the factor $E_{\omega}$ in its definition (\ref{def:kappa_anw})), they will look reversely: The supports of $\Psi_{\omega, m}^{\sigma}$ for $m\le n_1(\omega)$ will look strongly distorted while those for $m\ge n_2(\omega)$ will look regular in such coordinates.
\end{Remark}

\subsection{The decomposition of functions}
Suppose $\sigma\in \Sigma$. For each $u\in C^\infty(U_0)$, we assign  a countable family  of functions 
\[
u_{a,\omega, n, m}^\sigma=\Psi_{\omega,m}^\sigma\cdot \pBargmann\left(\rho_{a,n}^{(\omega)}\cdot u\circ \kappa_{a,n}^{(\omega)}\right)
\in C^{\infty}(\supp \Psi_{\omega,m}^\sigma)\subset 
C^{\infty}(\real^{4d+2d'+1}_{(x,y,\xi_x, \xi_y,\xi_z)})
\]
for $a\in A$, $\omega\in \integer$, $n\in\cN (a,\omega)$ and $m\in \integer$, where  $\rho_{a,n}^{(\omega)}\circ \kappa_{a,n}^{(\omega)}$ for $a\in A$, $\omega\in \integer$ and $n\in \cN(a,\omega)$ are those introduced in Subsection \ref{ss:chart_omega} 
and $\Psi_{\omega, m}^\sigma$ for $\omega, m\in \mathbb{Z}$ are those introduced in the last subsection. For simplicity, we set 
\[
\cJ=\{\j=(a,\omega,n,m)\mid \mbox{$a\in A$, $\omega\in \integer$, $n\in \cN (a,\omega)$, $m\in \integer$ s.t.\ $m=0$ or $|m|> n_0(\omega)$}\}
\]
and write 
\[
u_{\j}^{\sigma}=u_{a,\omega, n, m}^\sigma\quad \mbox{for $\j=(a,\omega, n,m)\in \cJ$.}
\]
We will refer the components of $\j=(a,\omega, n,m)\in \cJ$ as
\[
a(\j)=a,\quad \omega(\j)=\omega, \quad n(\j)=n,   \quad m(\j)=m.
\]
Also we set, for $\j\in\cJ$, 
\begin{equation}\label{eq:defsrho}
\rho_{\j}:=\rho_{a(\j),  n(\j)}^{(\omega(\j))},\quad \tilde{\rho}_{\j}:=\tilde{\rho}_{a(\j),  n(\j)}^{(\omega(\j))},\quad 
 \kappa_{\j}:=\kappa_{a(\j), n(\j)}^{(\omega(\j))} \quad\mbox{and}\quad \Psi^\sigma_{\j}:=\Psi^\sigma_{\omega(\j),m(\j)}.
\end{equation}
Then the assignment mentioned above can be regarded as an operator  
\[
\bI^\sigma:C^\infty(U_0)\to \overline{\bigoplus_{\j\in \cJ}L^2(\supp \Psi^\sigma_{\j})},\quad \bI^\sigma u=(u_{\j}^\sigma)_{\j\in \cJ}.
\]
This is of course injective on $C^\infty(K_0)$.
\begin{Remark}\label{rem:bI} Since the intersection multiplicities  of 
\[
\{\supp \Psi^\sigma_{\omega,m}\}_{m\in \integer}\quad\text{and}\quad 
\{\supp \rho_{a,n}^{(\omega)}\circ (\kappa^{(\omega)}_{a,n})^{-1}\mid a\in A, n\in \cN(a,\omega)\}
\]
are bounded by absolute constants and since the functions $\rho_{a,n}^{(\omega)}$ are smooth in the $z$-direction uniformly $a$, $n$ and $\omega$, it is not difficult to see that
\[
\sum_{\j\in \cJ} \|u_{\j}^\sigma\|_{L^2}^2\le 
C_0 \sum_{\omega\in \integer} \bigg\|\sum_{\j\in \cJ; \omega(\j)=\omega} u_{\j}^\sigma\bigg\|_{L^2}^2
\le C_R \sum_{\omega\in \integer} \langle \omega\rangle^{-2R} \|u\|_{H^R}^2\quad\text{for $u\in C^\infty(U_0)$}
\]
for any $R>0$, where $\|\cdot\|_{H^R}$ denotes the norm on the Sobolev space $H^R(K_0)$ of order $R$. (For the definition of the Sobolev spaces using Bargmann transform, we refer \cite[Ch.1]{NicolaBook}.)  
Hence the range of $\bI^\sigma$ above is indeed contained in $\overline{\bigoplus_{\j\in \cJ}L^2(\supp \Psi^\sigma_{\j})}$.
\end{Remark}
A left inverse of $\bI^\sigma$ is defined as  
\[
(\bI^\sigma)^*:{\bigoplus}_{\j\in \cJ}L^2(\supp \Psi^\sigma_{\j})\to L^2(U_0)
, \quad (\bI^\sigma)^* ((u_{\j})_{\j\in \cJ})=\sum_{\j\in \cJ} 
 \left((\tilde{\rho}_{\j}\cdot \pBargmann^*u_{\j})\circ \kappa_{\j}^{-1}\right).  
\]
Note that this is not the $L^2$ adjoint of $\bI^\sigma$. The following is not trivial. 
\begin{Lemma}\label{lm:IastI}   $(\bI^{\sigma})^*\circ (\bI^{\sigma})=\mathrm{Id}$ 
on $C^\infty(K_0)$ for any $\sigma\in \Sigma$.  
\end{Lemma}
\begin{proof} From Remark \ref{rem:bI}, we can see also that the composition $(\bI^{\sigma})^*\circ (\bI^{\sigma})$ is well-defined from $C^\infty(K_0)$ to itself. 
To prove the claim, it is enough to show 
\begin{equation}\label{lm:sum_claim}
\sum_{\j\in \cJ} \left(\tilde{\rho}_{\j}\cdot \pBargmann^* \left(\Psi^\sigma_{\j}\cdot  \pBargmann(\rho_{\j}\cdot (u\circ \kappa_\j))\right)\right)\circ \kappa_\j^{-1}=u 
\end{equation}
for $u\in C^{\infty}(K_0)$.
The proof is simple if we take the sum in an appropriate order. 
On the left hand side of (\ref{lm:sum_claim}), we first take the sum over  $\j$ with $a(\j)=a$, $\omega(\j)=\omega$ and $n(\j)=n\in \cN(a,\omega)$ fixed. Then the sum is  
\[
\sum_{a,\omega, n}\tilde{\rho}_{a,n}^{(\omega)}\circ (\kappa_{a,n}^{(\omega)})^{-1}\cdot\left( 
(\pBargmann^* \circ \mult(q_{\omega})\circ \pBargmann) (\rho_{a,n}^{(\omega)}\cdot u\circ \kappa_{a,n}^{(\omega)})\right)\circ (\kappa_{a,n}^{(\omega)})^{-1}
\]
where $\mult(q_{\omega})$ is the multiplication by the function $q_\omega$. 
Recall that the partial Bargmann transform $\pBargmann$ is a composition of the Fourier transform in the variable $z$ and the Bargmann transforms in the variable $w$ with scaling depending on the frequency $\xi_z$. From this, we see that the  operator $\pBargmann^* \circ \mult(q_{\omega})\circ\pBargmann$ above is a convolution operator which involves only the variable~$z$ and commutes with the action of the fibered contact diffeomorphism
 $(\kappa_{a,n}^{(\omega)})^{-1}\circ \kappa_a$.
In the definition (\ref{eq:trhoan}) of $\tilde{\rho}^{(\omega)}_{a,n}(\cdot)$, the latter factor does not  depend  on the variable $z$ so that the multiplication by that factor commutes with the operator  $\pBargmann^* \circ \mult(q_{\omega})\circ    \pBargmann$. 
Using these facts with (\ref{eq:rho_n}) and (\ref{eq:relRho}), and taking the sum over $n\in \cN_{a,\omega}$, we see that the left hand side of (\ref{lm:sum_claim}) equals 
\[
\sum_{a,\omega}\left(\tilde{\rho}_{a}\cdot (\pBargmann^*\circ \mult(q_{\omega})\circ    \pBargmann)(\rho_{a}\cdot (u\circ \kappa_a))\right)\circ \kappa_a^{-1}.
\]
Taking sum with respect to $\omega\in \integer$ and then to $a\in A$, we see that this equals $u$.\end{proof}

\subsection{The modified anisotropic Sobolev space $\cK^{r,\sigma}(K_0)$}\label{ss:defModifiedAniso}
We define the  Hilbert space $\cK^{r,\sigma}(K_0)$ for $r>0$ and $\sigma\in \Sigma$ as follows.
Though the definition makes sense for any $r>0$, we henceforth assume that $r$ satisfies (\ref{eq:choice_of_r}) and 
\begin{equation}\label{eq:choice_of_r4}
(r-1) \chi_0>4(d+d')\chi_{\max}
\end{equation}
which corresponds to (\ref{eq:choice_of_r3}), in order to make use of the results in Section \ref{sec:linear}. 
\begin{Definition}\label{Def:cK} Let  $\bK^{r,\sigma}$  be the completion of  $\bigoplus_{\j\in \cJ}L^2(\supp \Psi_{\j}^{\sigma})$ with respect to the norm 
\[
\|\bu\|^2_{\bK^{r,\sigma}}=\sum_{\j\in \cJ:m(\j)=0} \| \cW^{r,\sigma}\cdot u_{\j}\|_{L^2}^2+
\sum_{\j\in \cJ:m(\j)\neq 0} 2^{-r\cdot m(\j)}\|u_{\j}\|_{L^2}^2\quad \mbox{for 
$\bu=(u_{\j})_{\j\in \cJ}$.}
\]
Then let  $\cK^{r,\sigma}(K_0)$ be  the completion of the space $C^\infty(K_0)$ with respect to the norm 
\[
\|u\|_{\cK^{r,\sigma}}:=\|\bI^\sigma(u)\|_{\bK^{r,\sigma}}.
\]
For any compact subset $K'\subset K_0$,  $\cK^{r,\sigma}(K')$ denotes the subspace of $\cK^{r,\sigma}(K_0)$ that consists of elements supported on $K'$. 
\end{Definition}
\begin{Remark}\label{rem:HR}
From Remark \ref{rem:bI} and geometric consideration about the position of the supports of functions $\Psi^{\sigma}_\j=\Psi^{\sigma}_{\omega(\j), m(\j)}$, we see that the inequality
\[
C^{-1} \|u\|_{H^{-R}}\le \|u\|_{\cK^{r,\sigma}}\le C \|u\|_{H^{R}}\quad \text{for any $u\in C^\infty(K_0)$}
\]
holds for some $R>r$ and $C>1$.
This implies that  we have 
\[
C^{\infty}(K_0)\subset H^{R}(K_0)\subset \cK^{r,\sigma}(K_0)\subset H^{-R}(K_0) \subset \cD'(K_0)
\]
where $H^R(K_0)$ denotes the (usual) Sobolev space of order $R$. 
\end{Remark}

For convenience in the argument in the later sections, we give a few related definitions. 
For each $\j\in \cJ$, we define  the Hilbert space $\bK_\j^{r,\sigma}$  as the space  $L^2(\supp \Psi^{\sigma}_{\j})$ equipped  with  the norm
\begin{equation}\label{eq:bKnorm}
\|u\|_{\bK_\j^{r,\sigma}}=\begin{cases}
\|\cW^{r,\sigma} u\|_{L^2}  &\quad \mbox{if $m(\j)=0$;}\\
2^{-r\cdot m(\j)/2}\|u\|_{L^2}&\quad \mbox{if $m(\j)\neq 0$.}
\end{cases}
\end{equation}
Then we have 
\[
\bK^{r,\sigma}=\overline{\bigoplus_{\j\in \cJ} \bK^{r,\sigma}_{\j}}\quad \mbox{ for $\sigma\in \Sigma$.}
\]
For each $\j\in \cJ$, we define 
\[
\bI_\j^{\sigma}:\cK^{r,\sigma}(K_0)\to L^2(\supp \Psi^{\sigma}_{\j}), \quad \bI_{\j}^\sigma u= \Psi_{\j}^{\sigma}\cdot \pBargmann(\rho_{\j}\cdot u\circ \kappa_\j)
\]
so that the operator $\bI^{\sigma}$ is the direct product of them:
\begin{equation}\label{eq:bIpm}
\bI^{\sigma}=\overline{\bigoplus_{\j\in \cJ} \bI^\sigma_{\j}}:\cK^{r,\sigma}(K_0)\to \bK^{r,\sigma}=\overline{\bigoplus_{\j\in \cJ} \bK^{r,\sigma}_{\j}}. 
\end{equation}
\begin{Remark}\label{Rem:alternative_def} We could define the modified anisotropic Sobolev space $\cK^{r,\sigma}(K_0)$ in the same spirit as in the definition of $\cH^{r,\sigma}(\real^{2d+d'+1})$. 
Let $\mathfrak{W}^{r,\sigma}_\omega:\real^{2d+d'+1}\to \real$ be the function defined by 
\[
\mathfrak{W}^{r,\sigma}_\omega=\left( (\cW^{r,\sigma}\cdot \Psi_{\omega,0})^2 +
\sum_{|m|>n_0(\omega)} 2^{-rm}\cdot\Psi_{\omega,m}^2\right)^{1/2}.
\]
Then the norm $ \|u\|_{\cK^{r,\sigma}}$ is equivalent to the norm 
\[
 \|u \|'_{\cK^{r,\sigma}}= \left(\sum_{a\in A, \omega\in \integer,n\in \cN(a,\omega)} \| \mathfrak{W}^{r,\sigma}_\omega\cdot \pBargmann (\rho_{a,n}^{(\omega)}\cdot u\circ \kappa_{a, n}^{(\omega)})\|_{L^2}^2\right)^{1/2}.
\]
This definition looks a little simpler than the definition given above, since it avoids the decomposition of functions with respect to the integer $m$.
But the definition that we gave in the text is more useful in our argument. 
\end{Remark}
\begin{Remark}\label{rem:vector_valued_case_1}
In this section, we have defined the modified anisotropic Sobolev spaces $\cK^{r,\sigma}(K_0)$ as a completion of the space of $C^\infty$ functions. This is enough for the purpose of considering the scalar-valued transfer operators $\cL^t=\cL^t_{0,0}$. When we consider the vector-valued transfer operators $\cL^t_{k,\ell}$, we have to define the similar Hilbert spaces as a completion of $\Gamma^\infty(K_0, V_{k,\ell})$. 
The extension of the definition is straightforward once we fix some  local trivializations of $V_{k,\ell}$ subordinate to the local charts $\kappa_a$.  
\end{Remark}

\section{Properties of the transfer operator $\cL^t$}\label{sec:properties_Lt}
In order to clarify the structure of the proof of Theorem \ref{th:spectrum}, we state several propositions below in this section and then  deduce Theorem \ref{th:spectrum} from them in the next section. 
The proofs of the propositions are deferred to the later sections, Sections \ref{sec:pre},\ref{sec:lifted} and \ref{sec:proofs_of_props}. Below we mostly consider the case of scalar-valued transfer operator $\cL^t=\cL^t_{0,0}$. For the other cases of vector-valued transfer operators $\cL^t_{k,\ell}$ with $(k,\ell)\neq (0,0)$, we put a remark,  Remark \ref{rem:vector_valued_case_3}, at the end.

\subsection{Constants and some definitions}\label{ss:constandDef}
In addition to the constants $\chi_0$, $\beta$, $\theta$, $\Theta_1$ and $\Theta_2$ that we have fixed in the previous sections, we introduce two more constants $t_0>0$ and $\epsilon_0>0$. 

We take $t_0>0$ as the time that we need to wait until  the hyperbolicity of the flow takes sufficiently strong effects. Precisely we take and fix $t_0$ such that 
\[
e^{\chi_0 t_0}>10\quad \mbox{and}\quad f^{t_0}_{G}(U_0)\subset K_0. 
\]
Also we take $\epsilon_0>0$ as a small constant and define 
\begin{equation}\label{eq:tomega}
t(\omega):=\max\{\epsilon_0 \log \langle \omega\rangle, t_0\}.
\end{equation}
When we consider functions with  frequency around $\omega$ in the flow direction, we look them in  a small neighborhood of a point with size $\langle \omega\rangle^{-1/2+\theta}$ in the transversal directions to the flow. For each fixed time $t$, if we view the flow $f^t_{G}$ in such neighborhood, the effect of non-linearity will decrease as $|\omega|\to \infty$. By a little more precise consideration, we see that such  estimates on non-linearity remains true for $t$ in the range $0\le t\le t(\omega)$ if $t(\omega)$ grows sufficiently slowly with respect to $|\omega|$, that is,  if the constant $\epsilon_0$ is sufficiently small. Roughly this is  what we want to realize by choosing small $\epsilon_0$. The choice of $\epsilon_0$ will be given in the course of the argument. 

Recall that we took the compact subset $K_0\Subset U_0$ as a neighborhood of the section $ e_u$ satisfying the forward invariance condition (\ref{eq:choice_K0}). We define, in addition,  
\[
 K_0\Supset K_1:=f^{t_0}_{G}(K_0)\Supset K_2:=f^{2t_0}_{G}(K_0)
\]
and take smooth functions 
$\rho_{K_0}, \rho_{K_1}:U_0 \to [0,1]$ such that 
\begin{equation}\label{eq:defrhoK}
\rho_{K_0}(p)=\begin{cases}
1,&\quad\mbox{if $p\in K_1$;}\\
0,&\quad\mbox{if $p\notin K_0$.}
\end{cases},\qquad
\rho_{K_1}(p)=\begin{cases}
1,&\quad\mbox{if $p\in K_2$;}\\
0,&\quad\mbox{if $p\notin K_1$.}
\end{cases}
\end{equation}
We will use these functions to restrict the supports of functions to $K_0$ and $K_1$. 

We introduce the operator $\cQ_\omega$ for $\omega\in \integer$, which extracts the parts of functions whose frequency in the flow direction are around $\omega\in \integer$.  
\begin{Definition}\label{Def:PiOmega}
For each $\omega\in \integer$, let $\bPi_\omega:\bK^{r,\sigma} \to \bK^{r,\sigma}$ be the operator defined by
\[
(\bPi_\omega \mathbf{u})_{\j}=\begin{cases}
u_{\j},\quad \mbox{if $\omega(\j)=\omega$;}\\
0,\quad \mbox{otherwise,}
\end{cases}
\quad\mbox{for $\bu=(u_{\j})_{\j\in \cJ}$.}
\]
Then we define  $\cQ _{\omega}: C^\infty(K_0)\to C^\infty(K_0)$
by 
\[
\cQ _{\omega}= \mult(\rho_{K_0})\circ (\bI^\sigma)^{*}\circ \bPi_{\omega}\circ \bI^\sigma.
\]
\end{Definition}
From Lemma \ref{lm:IastI} and the choice of the function $\rho_{K_0}$,  we have 
\begin{equation}\label{eq:qid}
\sum_{\omega\in \integer} \cQ_\omega u=u\quad \mbox{ for  $u\in C^{\infty}(K_1)$. }
\end{equation}
\begin{Remark}
Since the operator $\cQ _{\omega}$ may enlarge the support of functions, each of  $\cQ_\omega(C^\infty(K_1))$ will not be contained in $C^\infty(K_1)$. 
\end{Remark}
\begin{Remark}\label{rem:qid}
The equality (\ref{eq:qid}) is valid for any distribution $u\in \cD'(K_1)$ if we regard the both sides as distributions.
\end{Remark}

In the next definition, we introduce the operator $\cT_\omega$, which corresponds to $\cT_0$ in (\ref{eq:T}) on local charts (restricted to the frequency around $\omega$).. 
\begin{Definition}\label{Def:bT}
For $\omega\in \mathbb{Z}$ and $\sigma,\sigma'\in \Sigma_0$, let  
$
\mathbf{T}_\omega^{\sigma\to \sigma'}:\bK^{r,\sigma}\to \bK^{r,\sigma'}$ be the operator defined 
by
\begin{equation}\label{eq:mathbfT}
(\mathbf{T}^{\sigma\to \sigma'}_\omega \mathbf{u})_{\j}=\begin{cases}
  X_{n_0(\omega)}\cdot \cT_0^\lift u_{\j},& \mbox{if  $m(\j)=0$ and $\omega(\j)=\omega$ with\footnote{We consider the condition $|\omega|\ge 3$ below from Remark \ref{Rem:dag}.}  $|\omega|\ge 3$;}\\
  0, & \mbox{otherwise,}
  \end{cases}
\end{equation}
for $\bu=(u_{\j})_{\j\in \cJ}$.
See (\ref{eq:Tlift}) and (\ref{eq:Xm}) for the definitions of $\cT_0^\lift$ and $X_{n_0(\omega)}$. 
Then we define 
\begin{equation}\label{def:cTomega}
\cT_\omega=\mult(\rho_{K_1})\circ (\bI^{\sigma'})^*\circ \mathbf{T}_\omega^{\sigma\to \sigma'}\circ \bI^{\sigma}\;:\; C^\infty(K_0)\to  C^\infty(K_1).
\end{equation}
Notice that the right-hand side actually does not depend on $\sigma,\sigma'\in \Sigma$. 
\end{Definition}
\begin{Remark}\label{Rem:suppX} 
The multiplication by $X_{n_0(\omega)}$ in (\ref{eq:mathbfT}) is inserted in order that the operator $\mathbf{T}_\omega^{\sigma\to \sigma'}$ is well-defined as that from  $\bK^{r,\sigma}$  to $\bK^{r,\sigma'}$.
(Note that the operator $\cT_0^\lift$ does not enlarge the support of the function in the $\xi_z$ direction.) Similarly the multiplication operator $\mult(\rho_{K_1})$ in  (\ref{def:cTomega}) is inserted so that the image of $\cT_\omega$ is supported on $K_1$. 
\end{Remark}
 
\begin{Remark}\label{Rem:smoothness}
Since the definition of $\mathbf{T}^{\sigma\to \sigma'}_\omega$ above involves only finitely many components in effect (and erase the other components), 
 it is easy to see that $\cT_\omega$ for each $\omega$ is continuous  from $\cK^{r,\sigma}(K_0)$ to $C^\infty(K_1)$. 
\end{Remark}

\subsection{Properties of the operators $\cL^t$, $\cQ_\omega$ and $\cT_\omega$. }\label{ss:lt}
\subsubsection{Boundedness and continuity of the operators $\cL^t$.}
First of all, we give a basic statement on continuity of the family $\cL^t$. 
From the fact noted in the beginning of Subsection \ref{ss:variantsOfH}, we unfortunately do not know whether 
$\cL^t:\cK^{r,\sigma}(K_0)\to \cK^{r,\sigma}(K_0)$ is bounded when  $t> 0$ is small. Instead, we prove the following proposition. Note that, from (\ref{eq:norm_comparison}), we have 
\begin{equation}\label{eq:cHinclusion}
\cK^{r,+}(K_0)\subset \cK^{r}(K_0)\subset \cK^{r,-}(K_0).
\end{equation}
\begin{Remark}
Here and henceforth, we write  $\cK^{r,-}(K_0)$, $\cK^{r}(K_0)$ and $\cK^{r,+}(K_0)$ for the Hilbert space $\cK^{r,\sigma}(K_0)$ with $\sigma=-1,0,+1$ respectively. Similarly we will write  $\|\cdot\|_{\cK^{r,-}}$, $\|\cdot\|_{\cK^{r}}$ and $\|\cdot\|_{\cK^{r,+}}$ for the norms on them. 
\end{Remark}

\begin{Proposition}\label{pp:continuity}
Suppose that  $\sigma,\sigma'\in \Sigma$. The operator $\cL^t:\cK^{r,\sigma}(K_0)\to \cK^{r,\sigma'}(K_0)$ is bounded if  
\begin{equation}\label{eq:condition_on_t}
\mbox{either\quad {\rm (i)} $t\ge 0$  and $\sigma'<\sigma$\quad or  \quad{\rm (ii)} $t\ge t_0 $ (and any $\sigma,\sigma'$).}
\end{equation}
In the latter case (ii), the image is contained in $\cK^{r,\sigma'}(K_1)$.
Further there exists a constant $C>0$ such that 
\[
\|\cL^t:\cK^{r,\sigma}(K_0)\to \cK^{r,\sigma'}(K_0)\|\le C e^{Ct}
\]
provided that  the condition (\ref{eq:condition_on_t})  holds true. 

\end{Proposition}

\subsubsection{The operator $\cQ_\omega$.}
Since the operator $\cQ_\omega$ extracts the parts of functions 
whose frequencies in the flow direction is around $\omega$, the claims of the next lemma is natural.
\begin{Lemma}\label{lm:Qomega}
Suppose that $\sigma,\sigma'\in \Sigma$ satisfy $\sigma'< \sigma$.  The operator $\cQ_\omega$ extends naturally  to a bounded operator $\cQ_\omega:\cK^{r,\sigma}(K_0)\to \cK^{r,\sigma'}(K_0)$.  
There exists a constant $C_0>0$ such that  
\begin{equation}\label{one:qsum1}
\sum_{\omega\in \integer}\|\cQ_{\omega} u\|_{\cK^{r,\sigma'}}^2\le 
C_0 \| u\|_{\cK^{r,\sigma}}^2\quad\mbox{for 
$u\in \cK^{r,\sigma}(K_0)$.}
\end{equation}
\end{Lemma}
\begin{Remark}
The claim of the lemma above will not hold for the case $\sigma=\sigma'$ because of the anisotropic property of our Hilbert spaces. (Note that the operator $\cQ_\omega$ involves the coordinate change transformations.) 
\end{Remark}

\subsubsection{The operator $\cT_\omega$}
From Lemma \ref{lm:Tzero},  the (Schwartz) kernel of the operator $\cT_0^{\lift}$ concentrate around the trapped set if we view it through the weight function $\cW^{r,\sigma}$. 
Also the operator $\cT_\omega$ concerns the part of functions whose frequency in the flow direction is around $\omega$. Hence  it is not difficult to see that this is a compact operator. More precise consideration leads to the following lemma. 
\begin{Lemma}\label{lm:trace1}
Let $\sigma,\sigma'\in \Sigma_0$. 
The operator $\cT_\omega$  extends to a trace class operator $\cT_\omega:\cK^{r,\sigma}(K_0)\to \cK^{r, \sigma'}(K_1)$. There is a constant $C_0>0$ such that, for any subset $Z\subset \integer$, we have 
\begin{equation}\label{eq:TZ}
 \left\|\sum_{\omega\in Z}\cT_\omega: \cK^{r,\sigma}(K_0)\to \cK^{r,\sigma'}(K_1)\right\|\le C_0.
 \end{equation}
Further there exist constants $C_0>1$ and $\omega_0>0$ such that, for each $\omega\in \integer$ with $|\omega|\ge \omega_0$, we have
\begin{itemize}
\item[{\rm (a)}]   the estimate
\begin{equation}\label{eq-trace-order}
C_0^{-1}\langle \omega\rangle^{d}\le \|\cT_\omega:\cK^{r,\sigma}(K_0)\to \cK^{r, \sigma'}(K_1)\|_{\trace}\le C_0 \langle \omega\rangle^{d}
\end{equation}
where $\|\cdot\|_{\trace}$ denotes the trace norm of an operator, and 
\item[{\rm (b)}] there exists a subspace
$V(\omega)\subset \cK^{r,\sigma}(K_1)$ with $\dim V(\omega)\ge C_0^{-1}\langle \omega\rangle^{d}$ such that 
\begin{equation}\label{eq:Vexp}
\|\cT_\omega u \|_{\cK^{r,\sigma'}}\ge C_0^{-1}\|u\|_{\cK^{r,\sigma}} \quad \mbox{for all $u\in V(\omega)$.}
\end{equation}
Further there exists a constant $C_\nu>0$ for $\nu>0$ such that 
\begin{equation}\label{eq:Vexp2}
|(\cT_\omega u, \cT_{\omega'} v)_{\cK^{r,\sigma}}|\le C_\nu \langle \omega'-\omega\rangle^{-\nu}\cdot \|u\|_{\cK^{r,\sigma}}\|v\|_{\cK^{r,\sigma}}
\end{equation}
for $u\in V(\omega)$ and $v\in V(\omega')$ provided $|\omega|, |\omega'|\ge \omega_0$.
\end{itemize}
\end{Lemma}

\subsubsection{The transfer operators  $\cL^t$}
We give two propositions on the transfer operators $\cL^t$. 
The first one below is  in the same spirit as Theorem  \ref{cor:linear1} in Section \ref{sec:linear}.  
\begin{Proposition}
\label{pp:norm_estimate}
Let $\sigma,\sigma'\in \Sigma_0$. 
 There exist constants $\epsilon>0$ and $C_\nu>1$ for any $\nu>0$ such that, for $\omega,\omega'\in \integer$ and $0\le t\le 2t(\omega)$, we have   
\begin{align}
&\|\cT_{\omega'}\circ {\cL}^t\circ \cT_{\omega}:\cK^{r,\sigma}(K_0)\to \cK^{r, \sigma'}(K_1)\|\le 
 C_\nu \langle \omega'- \omega\rangle^{-\nu}\label{eq:upperboundTLT},\\
&\|(\cQ_{\omega'}-\cT_{\omega'})\circ \cL^t\circ \cT_\omega:\cK^{r,\sigma}(K_0)\to \cK^{r, \sigma'}(K_0)\|\le C_\nu \langle \omega\rangle^{-\epsilon} \langle \omega'- \omega\rangle^{-\nu},\label{eq:upperboundQLT}\\
&\|\cT_{\omega'}\circ \cL^t\circ (\cQ_\omega-\cT_\omega):\cK^{r,\sigma}(K_0)\to \cK^{r, \sigma'}(K_1)\|\le C_\nu  \langle \omega\rangle^{-\epsilon} \langle \omega'- \omega\rangle^{-\nu}\label{eq:upperboundTLQ}
\end{align}
and, under the additional condition
 (\ref{eq:condition_on_t}) on $t$, also
\begin{equation}
\|(\cQ_{\omega'}-\cT_{\omega'})\circ {\cL}^t\circ (\cQ_{\omega}- \cT_\omega):\cK^{r,\sigma}(K_0)\to \cK^{r, \sigma'}(K_0)\|\le 
C_\nu e^{-\chi_0 t}  \langle \omega'- \omega\rangle^{-\nu}.
\label{eq:upperboundQLQ}
\end{equation}
In particular, from the four inequalities above, it follows that 
\begin{equation}\label{eq:bound_QLQ}
\|\cQ_{\omega'}\circ {\cL}^t\circ \cQ_{\omega}:\cK^{r,\sigma}(K_0)\to \cK^{r, \sigma'}(K_0)\|\le 
 C_\nu \langle \omega'- \omega'\rangle^{-\nu}
\end{equation}
for  $\omega,\omega'\in \integer$ and $0\le t\le 2t(\omega)$ satisfying (\ref{eq:condition_on_t}) with respect to $\sigma$ and $\sigma'$.
\end{Proposition}
We need the next proposition when we consider the resolvent of the generator of~$\cL^t$. 
This is essentially an estimate on $\cL^t$ for negative $t<0$. The transfer operators $\cL^t$ with negative $t\ll 0$ will not be a bounded operator on our modified anisotropic Sobolev spaces. 
But, since  $\cT_\omega u$ for  $u\in \cK^{r, \sigma}(K_0)$ is a smooth function as we noted in Remark \ref{Rem:smoothness}, its image  $\cL^t(\cT_\omega u)$ for $t<0$ is well defined and smooth. 
More precise consideration leads to  
\begin{Proposition}\label{pp:backward}
Let $\sigma,\sigma'\in \Sigma_0$.  There exist constants $\epsilon>0$, $C_0>0$ and $C_\nu>0$ for any $\nu>0$ such that,  for  
$u\in \cK^{r, \sigma}(K_0)$, $\omega\in \integer$ and $0\le t\le 2t(\omega)$, there exists $v_\omega\in \cK^{r, \sigma'}(K_1)$ such that 
\begin{align}\label{eq:backward1}
&\| \cL^t v_\omega-\cT_\omega u\|_{\cK^{r,\sigma}}\le C_0 \langle \omega\rangle^{-\theta}\|u\|_{\cK^{r,\sigma}}
\intertext{and, for $\omega'\in \integer$ and $0\le t'\le t$, }
&\| \cQ_{\omega'}\circ \cL^{t'} v_\omega\|_{\cK^{r,\sigma'}}\le C_\nu \langle \omega'-\omega\rangle^{-\nu} \| u\|_{\cK^{r,\sigma}},
\label{eq:cv1}\\
&
\| (\cQ_{\omega'}-\cT_{\omega'})\circ \cL^{t'} v_\omega\|_{\cK^{r,\sigma'}}\le C_\nu \langle \omega\rangle^{-\epsilon}\langle \omega'-\omega\rangle^{-\nu} \|u\|_{\cK^{r, \sigma}}.
\label{eq:cv2}
\end{align}

\end{Proposition}

\subsubsection{Short-time estimates}
The next lemma is a consequence of the fact that the component $\cQ_\omega u$ of $u\in \cK^{r,\sigma}(K_0)$ has frequency close to $\omega$ in the flow direction. 
\begin{Lemma} \label{lm:continuity_omega} Suppose that $\sigma,\sigma'\in \Sigma_0$ satisfy $\sigma'<\sigma$. 
There exists a constant $C_0>0$  such that, for $\omega\in \integer$ and $0\le t\le t_0$,
\begin{equation}\label{eq:continuity_omega}
\| (e^{-i\omega t}\cL^t -1)\circ \cQ_\omega:\cK^{r,\sigma}(K_0)\to \cK^{r,\sigma'}(K_0)\|\le C_0 t.
\end{equation}
For the generator $A:=\lim_{t\to +0}(\cL^t -1)/t$ of $\cL^t$, we have, for $u\in \cK^{r,\sigma}(K_0)$,
\begin{align}\label{eq:continuity_omega2}
&\sum_{\omega\in \integer}\| (i\omega-A)\circ \cQ_\omega u\|^2_{\cK^{r,\sigma'}}\le C_0 \|u\|_{\cK^{r,\sigma}}^2
\intertext{and}
&\sum_{\omega\in \integer}\| \cQ_\omega\circ (i\omega-A) u\|^2_{\cK^{r,\sigma'}}\le C_0 \|u\|_{\cK^{r,\sigma}}^2. \label{eq:continuity_omega3}
\end{align}
Further the claims (\ref{eq:continuity_omega2}) and (\ref{eq:continuity_omega3}) remain valid when $\cQ_\omega$ is replaced by $\cT_{\omega}$. 
\end{Lemma}
\begin{Remark}\label{rem:vector_valued_case_3}
For the vector-valued transfer operators $\cL^t_{k,\ell}$ with $(k,\ell)\neq (0,0)$, we consider their action on the Hilbert spaces given in Remark \ref{rem:vector_valued_case_1} and prove the propositions parallel to those stated in this section, except for Lemma \ref{lm:trace1} and Proposition \ref{pp:backward}. The extensions of the proofs given in later sections to such cases  are straightforward. We will not need Lemma \ref{lm:trace1} and Proposition \ref{pp:backward} for the vector-valued cases.
\end{Remark}


\section{Proof of the main theorems (1): Theorem \ref{th:spectrum} }\label{sec:proof1}
We prove  Theorem \ref{th:spectrum} assuming the propositions given in the last section. Below we consider the case of scalar-valued transfer operator $\cL^t=\cL^t_{0,0}$. For the other cases, we put a remark, Remark \ref{rem:vector_valued_case_4}, at the end of Subsection \ref{ss:bddR}. 

\subsection{Strong continuity and the generator}\label{ss:strongContinuity}
We define the Hilbert space $\widetilde{\cK}^r(K_0)$ in the statement of Theorem \ref{th:spectrum} as follows. 
\begin{Definition}\label{def:modifiedAHS}
We define the  norm $\|\cdot \|_{\wK^r}$ on $C^{\infty}(K_0)$ by 
\[
 \|u \|_{\wK^r}:=\left(\int_0^{t_0} \|\cL^t u\|^2_{\cK^{r,-}}dt\right)^{1/2}.
\]
The Hilbert space $\wK^r(K_0)$ is the completion of the space $C^\infty(K_0)$ with respect to this norm. 
\end{Definition}
From Proposition \ref{pp:continuity},  we have
\begin{equation}\label{eq:wK1}
\cK^{r}(K_0)\subset \wK^{r}(K_0)
\end{equation}
and the transfer operator $\cL^t$ for $t\ge t_0$  extends to
\begin{equation}\label{eq:wK2}
\cL^t:\wK^r(K_0)\to \cK^{r,+2}(f^t(K_0))\subset \cK^r(K_2) . 
\end{equation}
\begin{Remark}
The Hilbert space $\widetilde{\cK}^r(K_0)$ is contained in Sobolev space $H^{-R}(K_0)$ of some negative order $-R<0$, since $\cL^{t_0}(\wK^r(K_0))\subset \cK^{r}(f^{t_0}(K_0))\subset H^{-R}(f^{t_0}(K_0))$ from Remark \ref{rem:HR}  and $\cL^{-t_0}$ is  bounded  from $H^{-R}(f^{t_0}(K_0))$ to  $H^{-R}(K_0)$. 
\end{Remark}
\begin{Proposition}\label{pp:strong_continuity}
The transfer operators $\cL^t:C^\infty(K_0)\to C^\infty(K_0)$ for $t\ge 0$ extend to a  strongly continuous one-parameter semi-group of bounded operators 
\[
\mathbb{L}:=\{\cL^t:\wK^{r}(K_0)\to \wK^{r}(K_0), \;t\ge 0\}.
\]
For some constant $C>0$, we have
\begin{equation}\label{eq:LtExp}
\|\cL^t:\wK^{r}(K_0)\to \wK^{r}(K_0)\| \le C  e^{C t}\quad \mbox{for $t\ge 0$.}
\end{equation}
\end{Proposition}
\begin{proof} The claims follow from the definition of $\|\cdot\|_{\wK^r}$ and Proposition \ref{pp:continuity}. 
Indeed, for $0\le t\le t_0$, we have
\begin{align*}
\|\cL^t u\|_{\wK^r}^2&=\int_0^{t_0} \|\cL^{s+t}u\|^2_{\cK^{r,-}}ds
=\int_t^{t_0} \| \cL^{s}u\|^2_{\cK^{r,-}}ds
+\int_0^{t} \|\cL^{t_0}\circ \cL^{s}u\|^2_{\cK^{r,-}}ds\\
&\le 
\int_t^{t_0} \| \cL^{s}u\|^2_{\cK^{r,-}}ds
+C\int_0^{t} \|\cL^{s}u\|^2_{\cK^{r,-}}ds
\le (1+C) \|u\|_{\wK^r}^2.
\end{align*} 
Then we use this estimate recursively to get (\ref{eq:LtExp}).
The correspondence $t\mapsto \cL^t(u)\in C^\infty(K_0)\subset \wK^{r}(K_0)$ for $u\in C^\infty(K_0)$ is continuous at $t=0$ from Remark \ref{rem:HR}. Since $C^\infty(K_0)$ is dense in $\wK^{r}(K_0)$ by definition, we obtain strong continuity of the semi-group $\mathbb{L}$ by approximation argument. 
\end{proof}

We will denote the generator of the one-parameter semi-group $\mathbb{L}$  by 
\[
A:\cD(A)\subset \wK^{r}(K_0)\to \wK^{r}(K_0).
\]
By general argument (see \cite[\S 1.4 p.51]{engel}), this is a closed operator defined on a dense linear subspace $\cD(A)\subset \wK^{r}(K_0)$ that  contains $C^\infty(K_0)$.

\subsection{Meromorphic property of the resolvent}
In the following, we suppose that  $\tau>0$ is that in the statement of Theorem \ref{th:spectrum}, given as an arbitrarily small positive real number.
The  resolvent of the generator $A$ is written
\[
\cR(s)=(s-A)^{-1}.
\]
As the second step toward the proof of Theorem \ref{th:spectrum}, we prove 
\begin{Proposition}\label{prop:mero_res} The resolvent $\cR(s)$ is meromorphic on the region 
\[
\{s\in \complex\mid \Re(s)>-\chi_0+\tau, |\Im(s)|> s_0\}
\] 
if  $s_0>0$ is sufficiently large. Further there exists a constant $C_0>0$ such that, for  $\omega_*\in \integer$ with sufficiently large absolute value,  
there exist at most $C_0 |\omega_*|^{d}$ poles of the resolvent $\cR(s)$ (counted with multiplicity) in the region   
\begin{equation}\label{region2}
R(\omega_*)=\{s\in \complex\mid -\chi_0+2\tau<\Re(s)<1 ,\; |\Im (s)-\omega_*|\le 1\;\}.
\end{equation}
\end{Proposition}
\begin{Remark}\label{rem:mero12}
We actually can prove the meromorphic property of the resolvent $
\cR(s)=(s-A)^{-1}$ on much larger region. (See Remark \ref{rem:meromorphy} in Appendix \ref{sec:zeta}.)   
\end{Remark}

\begin{proof}
For each integer $\omega_*\in \integer$ with sufficiently large absolute value, 
we  prove that the resolvent $\cR(s)$ is meromorphic on the rectangle 
\[
\widetilde{R}(\omega_*):=\{s\in \complex\mid \Re(s)>-\chi_0+\tau,  |\Im (s)-\omega_*|<2\}
\Supset R(\omega_*).
\]  
Let us consider the operator  
\[
\widetilde{\cT}_{\omega_*}=\sum_{|\omega-\omega_*|\le 2\ell} \cT_\omega:\cK^{r,\sigma}(K_0)\to \cK^{r,\sigma'}(K_1)
\]
where the integer $\ell$ will be specified in the course of the argument below. (But note that we will choose  $\ell$ uniformly for $\omega_*$.) 
Below we write $C_0$ for large constants that are independent of $\omega_*$ and $\ell$.
From Lemma \ref{lm:trace1}, we have 
\begin{equation}\label{eq:TZcT}
\left\|\widetilde{\cT}_{\omega_*}:\cK^{r,\sigma}(K_0)\to \cK^{r,\sigma'}(K_1)\right\|\le C_0
\end{equation} 
and also 
\begin{equation}\label{eq:tracenorm_Tk}
\|\widetilde{\cT}_{\omega_*}:\cK^{r,\sigma}(K_0)\to \cK^{r,\sigma'}(K_1)\|_{\trace}\le C_0\ell\cdot |\omega_*|^{d}
\end{equation}
for any $\sigma, \sigma'\in \Sigma$. 
We regard  the generator $A$ as a compact perturbation of
\[
A':=A-\chi_0\cdot  \widetilde{\cT}_{\omega_*}:\cD(A)\to \wK^{r}(K_0).
\]
As the main step of the proof, we prove 
\begin{Lemma}\label{Lemma3}
The resolvent $\cR'(s):=(s-A')^{-1}$ of $A'$ is  bounded and satisfies $\|\cR'(s)\|_{\wK^r}\le C_0$ for $s\in \widetilde{R}(\omega_*)$ provided that  $|\omega_*|$ is sufficiently large.
\end{Lemma}
We postpone the proof of Lemma~\ref{Lemma3} and finish the proof of  Proposition \ref{prop:mero_res}.  From Proposition \ref{pp:strong_continuity}, the resolvent $\cR(s)=(s-A)^{-1}:\wK^{r}(K_0)\to \wK^{r}(K_0)$ is a uniformly bounded holomorphic family of isomorphisms on the region $\Re(s)\ge C_1$ if we take sufficiently large $C_1$. 
Let us write the operator $s-A$ as
\begin{equation}\label{eq:U}
s-A=(s-A')\circ  \mathcal{U}(s)\quad\mbox{where}\quad \mathcal{U}(s):=1-\chi_0 (s-A')^{-1}\circ \widetilde{\cT}_{\omega_*}.
\end{equation}
From Lemma \ref{Lemma3}, the operator $(s-A')^{-1}\circ\widetilde{\cT}_{\omega_*}$ belongs to the trace class and is holomorphic with respect to $s$ on the region $\widetilde{R}(\omega_*)$. Thus the resolvent $\cR(s)$ extends as a meromorphic family of Fredholm operators of index $0$ to the region $\widetilde{R}(\omega_*)$.  This gives the former claim of the proposition. 
Below we provide a more quantitative argument\footnote{We learned the following argument from the paper \cite{Sjostrand00} of Sj\"ostrand. 
} to get the latter claim. To begin with, note that the resolvent $\cR(s)$ has a pole of order $m$ at $s_0\in \widetilde{R}(\omega_*)$ if and only if
\[
k(s):=\det\mathcal{U}(s) 
\]
has a zero of order $m$ at $s=s_0$. 
\begin{Remark}
The determinant $\det\mathcal{U}(s) $ is well-defined 
since $\mathcal{U}(s)$ is a perturbation of the identity by a trace class operator. We refer \cite{GohbergBook} for the trace and determinant of operators on Hilbert (or Banach) spaces.
\end{Remark}
 From  (\ref{eq:tracenorm_Tk}) and Lemma \ref{Lemma3}, we have
\begin{equation}\label{eq:logks}
\log |k(s)|\le C_0\ell |\omega_*|^{d}\quad \mbox{uniformly for $s\in \widetilde{R}(\omega_*)$.}
\end{equation} 
We may write $\mathcal{U}(s)^{-1}$ for $s\in \widetilde{R}(\omega_*)$ with  $\Re(s)\ge C_1$ as 
\begin{align*}
\mathcal{U}(s)^{-1}
=(s-A)^{-1}(s-A')=\mathrm{Id}+\chi_0 (s-A)^{-1} \circ  \widetilde{\cT}_{\omega_*}.
\end{align*}
Hence, from (\ref{eq:tracenorm_Tk}), we obtain that
\begin{equation}\label{eq:logks2}
\log |k(s)|\ge -C_0 \ell  |\omega_*|^{d}\quad \mbox{uniformly for $s\in \widetilde{R}(\omega_*)$ with $\Re(s)\ge C_1$.}
\end{equation}
By virtue of Jensen's formula \cite[Chapter 5, formula (44) on page 208]{AhlforsBook}\footnote{By the Riemann mapping theorem, we find a biholomorphic mapping which maps the region $\widetilde{R}(\omega_*)$  onto the unit disk $|z|<1$ so that a point $s_*\in \widetilde{R}(\omega_*)$ with $\Re(s_*)>C_1$ and  $\Im(s_*)=\omega_*$ is sent to the origin $0$. Then we apply Jensen's formula to the holomorphic function on the unit disk corresponding to $k(s)$. (See also the proof of Corollary \ref{cor:Rs}.)}, the estimates (\ref{eq:logks}) and (\ref{eq:logks2})  imply that there are at most $C_0|\omega_*|^{d}$ poles in the region $R(\omega_*)\Subset \widetilde{R}(\omega_*)$. This proves the latter claim of the proposition. 
\end{proof}

\begin{proof}[Proof of Lemma \ref{Lemma3}]   
For the proof, we construct approximate right and left inverse of $(s-A)$, that is, 
\[
Q_R=Q_R(s),\;Q_L=Q_L(s)\;: \wK^r(K_0)\to \cD(A)\subset \wK^r(K_0)
\]
for $s\in \widetilde{R}(\omega_*)$ satisfying $\|Q_R\|_{\wK^r}\le C_0$, $\|Q_L\|_{\wK^r}\le C_0$ and
\begin{equation}\label{eq:Qlr}
\|\mathrm{Id}-(s-A') \circ Q_R\|_{\wK^r}<\frac{1}{2},\qquad
\|\mathrm{Id}-Q_L\circ (s-A')\|_{\wK^r}<\frac{1}{2}.
\end{equation}
Once we obtain such operators $Q_L$ and $Q_R$, we can prove Lemma  \ref{Lemma3}, constructing the resolvent $\cR'(s)=(s-A')^{-1}$ by iterative approximation. 
Indeed, if we define   
\[
\widetilde{Q}_R:=Q_R\sum_{k=0}^\infty (\mathrm{Id}-(s-A') Q_R)^k
\quad\text{ 
and}\quad 
\widetilde{Q}_L:=\sum_{k=0}^\infty (\mathrm{Id}-Q_L(s-A'))^k Q_L
\]
we have $\|\widetilde{Q}_R\|_{\wK^r}\le C'_0$, $\|\widetilde{Q}_L\|_{\wK^r}\le C'_0$ and
\[
(s-A')\circ \widetilde{Q}_R=\mathrm{Id},\quad \widetilde{Q}_L\circ (s-A')=\mathrm{Id}, \quad\widetilde{Q}_R=\widetilde{Q}_L \circ (s-A')\circ \widetilde{Q}_R=\widetilde{Q}_L.
\]
Below we construct $Q_R$ satisfying (\ref{eq:Qlr}). The construction of $Q_L$ is parallel and will be mentioned later. First, for $u\in \wK^r(K_0)$, we put
\[
\tilde{u}=e^{-st_0} \cL^{t_0}u -\chi_0 \widetilde{\cT}_{\omega_*} \int_0^{t_0} e^{-st}\cL^{t} u dt.
\]
Then we have $
\|\tilde{u}\|_{\cK^{r,+}}\le C_0 \|u\|_{\wK^r}$ from Proposition \ref{pp:continuity} and Lemma \ref{lm:trace1}.
We define the operator $Q_R:\wK^r(K_0)\to \wK^r(K_0)$ by setting
\[
Q_R u:=Q^{(1)}_R\tilde{u}+Q^{(2)}_R\tilde{u}+Q^{(3)}_R\tilde{u}+\int_0^{t_0} e^{-st}\cL^{t} u dt
\]
where   
\begin{align*}
Q^{(1)}_R \tilde{u}&=\sum_{\omega:|\omega -\omega_*|\le  \ell}
\int_{0}^{t(\omega_*)} e^{-(s+\chi_0)t}\cdot \cL^t\circ \cT_\omega 
\tilde{u} dt,\\
Q^{(2)}_R \tilde{u}&=\sum_{\omega:|\omega -\omega_*|\le  \ell}\int_{0}^{t(\omega_*)} e^{-st} \cL^t\circ (\cQ_\omega-\cT_\omega) \tilde{u} dt,\\
Q^{(3)}_R \tilde{u}&=\sum_{|\omega -\omega_*|> \ell}
(s-i\omega)^{-1}   \cQ_\omega\tilde{u}.
\end{align*} 
We have $\|Q^{(k)}_R \tilde{u}\|_{\wK^r}\le C_0\|u\|_{\wK^r}$ for $k=1,2,3$.
In the cases $k=1,2$, this follows from Proposition \ref{pp:norm_estimate}.   
In the case $k=3$, this follows from Lemma \ref{lm:Qomega} and Schwarz inequality.
 
Since $(d/dt) (e^{-st}\cL^t u)=-(s-A) e^{-st}\cL^t u$, we have
\[
(s-A')\int_0^{t_0} e^{-st}\cL^{t} u dt=u-e^{-st_0}\cL^{t_0} u+\chi_0 \widetilde{\cT}_{\omega_*}\int_0^{t_0} e^{-st}\cL^{t} u dt=u-\tilde{u}.
\]
Therefore, to prove the former claim in (\ref{eq:Qlr}) for $Q_R$, it is enough to show
\begin{enumerate}
\item $\|(s-A') Q^{(1)}_R \tilde{u}-\sum_{\omega:|\omega -\omega_*|\le  \ell}  \cT_\omega \tilde{u}\|_{\wK^r}<(1/6) \|u\|_{\wK^r}$,
\item $\|(s-A') Q^{(2)}_R \tilde{u}-\sum_{\omega:|\omega -\omega_*|\le  \ell}  (\cQ_\omega-\cT_\omega) \tilde{u}\|_{\wK^r}<(1/6) \|u\|_{\wK^r}$,
\item $\|(s-A') Q^{(3)}_R \tilde{u} -\sum_{\omega:|\omega -\omega_*|>  \ell}  
\cQ_\omega \tilde{u}\|_{\wK^r}<(1/6) \|u\|_{\wK^r}$.
\end{enumerate}
To prove the claims  (1) and (2), we  write
\begin{align*}
(s-A') Q^{(1)}_R \tilde{u}&=(s+\chi_0-A)Q^{(1)}_R \tilde{u}-\chi_0(1-\widetilde{\cT}_{\omega_*}) Q^{(1)}_R \tilde{u}\\
&=\sum_{\omega:|\omega -\omega_*|\le  \ell}
\left[\cT_{\omega}\tilde{u}-e^{-(s+\chi_0) t(\omega_*)} \cL^{t(\omega_*)} \circ \cT_{\omega}\tilde{u}\right]\\
&\qquad +
\sum_{\omega:|\omega -\omega_*|\le  \ell} \chi_0(1-\widetilde{\cT}_{\omega_*})\int_0^{t(\omega_*)} e^{-(s+\chi_0)t}\cdot \cL^t\circ \cT_\omega \tilde{u} dt
\end{align*}
and, similarly, 
\begin{align*}
(s-A') Q^{(2)}_R \tilde{u}
&=\sum_{\omega:|\omega -\omega_*|\le  \ell}
\left[(\cQ_\omega-\cT_{\omega})\tilde{u}-e^{-s t(\omega_*)} \cL^{-t(\omega_*)} \circ (\cQ_{\omega}-\cT_{\omega})\tilde{u}\right]\\
&\qquad +
\sum_{\omega:|\omega -\omega_*|\le  \ell} \chi_0
\int_0^{t(\omega_*)} e^{-st}\cdot \widetilde{\cT}_{\omega_*}\circ \cL^t\circ (\cQ_\omega-\cT_\omega) \tilde{u} dt.
\end{align*}
Then using Proposition \ref{pp:norm_estimate} and the relation 
\[
(1-\widetilde{\cT}_{\omega_*})=\sum_{|\omega-\omega_*|\le 2\ell}(\cQ_\omega-\cT_{\omega}) +\sum_{|\omega-\omega_*|>2\ell}\cQ_\omega,
\]
we can deduce the claims (1) and (2), provided that $\ell$ and $|\omega_*|$ are sufficiently large. 
To prove the claim (3), we write
\begin{align*}
(s-A')Q^{(3)}_R \tilde{u}&
=\sum_{|\omega -\omega_*|> \ell} \left(
\cQ_\omega \tilde{u}- 
\frac{A-i\omega}{s-i\omega}\cdot  \cQ_\omega \tilde{u}
+
\frac{\chi_0}{s-i\omega}\cdot \widetilde{\cT}_{\omega_*}\circ \cQ_\omega\tilde{u}\right).
\end{align*}
The sum of the second terms on the right-hand side is bounded in the $\wK^r$-norm by
\begin{align*}
C_0 \left(\sum_{|\omega -\omega_*|> \ell} 
\frac{1}{|s-i\omega|^2}\right)^{1/2}
\left(\sum_{|\omega -\omega_*|> \ell} \|(i\omega-A)\cdot  \cQ_\omega \tilde{u}\|_{\cK^{r}}^2\right)^{1/2}
\le C_0 \frac{\|u\|_{\wK^r}}{\ell^{1/2}} 
\end{align*} 
from Lemma \ref{lm:continuity_omega}.
Hence, if we take sufficiently large $\ell$ and then let $\omega_*$ be sufficiently large, we obtain the claim (3). 

For the construction of $Q_L$, we modify the definition of $\tilde{u}$ as
\[
\tilde{u}=e^{-st_0} \cL^{t_0}u +\chi_0  \int_0^{t_0} e^{-st}\cL^{t} \circ \widetilde{\cT}_{\omega_*} u dt
\]
and replace $\cL^t\circ \cT_\omega$ and $\cL^t\circ (\cQ_\omega-\cT_\omega)$ in the definitions of $Q^{(1)}_R$ and $Q^{(2)}_R$ respectively by $\cT_\omega\circ \cL^t$ and $(\cQ_\omega-\cT_\omega)\circ \cL^t$. Then we can follow the argument above with obvious modifications and obtain the latter claim in (\ref{eq:Qlr}) for $Q_L$.
\end{proof}

From the argument in the proof above, we get the following corollary, which we will use later in the proof of Proposition \ref{pp:lowerBound} in Subsection \ref{ss:lowerbound}. 
\begin{Corollary}\label{cor:Rs}
There exists a constant $C_0>1$ and $\omega_0>0$ such that, if $\omega_*\in \integer$ satisfies $|\omega_*|>\omega_0$, there exists some $\omega\in \real$ with $|\omega-\omega_*|<1$ such that  
\begin{equation}\label{eq:Rs}
\sup_{\mu\in [-\tau,\tau]}  \| \cR(\mu+\omega i):\wK^r(K_0)\to \wK^r(K_0)\|  \le \exp(C_0|\omega_*|^d) .
\end{equation}
\end{Corollary}
\begin{Remark}
The estimate (\ref{eq:Rs}) seems very coarse. But for the moment we do not know whether we can give better estimates. 
\end{Remark}
\begin{proof}
Let us recall  the subharmonic function\footnote{We suppose that subharmonic functions can take value $-\infty$ and that $\log 0=-\infty$. } $\log |k(s)|=\log |\det\mathcal{U}(s)|$ and the constant $C_1$ in the proof of the last proposition. By the Riemann mapping theorem, we take a biholomorphic mapping $\varphi:\mathrm{int}\, \widetilde{R}(\omega_*)\to \disk$ which maps the region $\widetilde{R}(\omega_*)$ onto the unit disk $\disk=\{|z|<1\}$ so that  the point $s_*=C_1+i\omega_*\in \widetilde{R}(\omega_*)$  is mapped to the origin~$0$.   Let $\psi(z)=\log |k(\varphi^{-1}(z))|$ for $z\in \disk$ and observe that 
\begin{enumerate}
\item $
\psi(z)\le C_0|\omega_*|^{d}$ uniformly on  $\disk$ from (\ref{eq:logks}),
\item $\psi(z)$ is a subharmonic function with at most $C_0|\omega_*|^{d}$ points $w_i$ for $1\le i\le I$ (with $I<C_0|\omega_*|^{d}$) such that 
\[
\Delta \psi(z) =\sum_{i=1}^I \delta_{w_i},
\]
\item $\psi(0)=\log|k(s_*)|\ge -C_0|\omega_*|^d$ from (\ref{eq:logks2}) and the choice of $s_*$ above.\end{enumerate}
Let $\psi_i$, $1\le i\le I$, be the Green's function on $\disk$ at $w_i$, which is by definition the subharmonic function satisfying $\Delta \psi_i=\delta_{w_i}$ and $\psi_i\equiv 0$ on $\partial \disk$.
Then we see 
\[
\psi(z)=\psi_0(z)+\sum_{i=1}^I \psi_i(z)
\]
where $\psi_0$ is the harmonic function which takes the same boundary values as $\psi$. From the property (3) and the subharmonic property of $\psi$, 
it follows
\[
 \frac{1}{2\pi}\int_{\partial \disk} \psi_0(z) |dz|=\frac{1}{2\pi}\int_{\partial \disk} \psi(z) |dz|\ge \psi(0)\ge -C_0|\omega_*|^{d}.
\]
From this and the property (1) above,
\[
\frac{1}{2\pi}\int_{\partial \disk} |\psi(z)| |dz| \le C_0|\omega_*|^{d}.
\]
Hence, by  Poisson's formula \cite{AhlforsBook}, we get
\[
\psi(z)\ge -C_0|\omega_*|^{d}+\sum_{i=1}^I \psi_i(z)\quad \mbox{for $z\in \varphi(R(\omega_*))\Subset \disk$.}
\]
Note that the distortion of the Riemann map $\varphi$ on the compact subset $R(\omega_*)\Subset \disk$ is bounded uniformly in $\omega_*$, that is,  
\[
C_0^{-1}|s-s'|\le |\varphi(s)-\varphi(s')|\le C_0|s-s'|\quad \mbox{for $s,s'\in R(\omega_*)$}
\]
because the pairs  $(R(\omega_*), \widetilde{R}(\omega_*))$  for different $\omega_*$ are translations of each other. 
Since $\psi_i(z)\ge \log |z-w'_i|-C_0$, we obtain that 
\[
\log |k(s)|\ge -C_0|\omega_*|^{d} +\sum_{i=1}^I (\log |s-w'_i|-C_0)
\qquad \mbox{for $s\in R(\omega_*)$,}
\]
 where $w'_i=\varphi^{-1}(w_i)$. 
In particular we have
\begin{align*}
\int_{\omega_*-1}^{\omega_*+1} &\left(\inf_{\mu\in [-\tau,\tau]}  
\log |k(\mu+\omega i)| \right) d\omega\\
&\ge -C_0|\omega_*|^{d}+ 
\sum_{i=1}^{I}\left(\int_{\omega_*-1}^{\omega_*+1} \log |\omega -\Im(w'_i)| d\omega-C_0\right)\\
&\ge -C_0|\omega_*|^{d}-C_0I \ge -C_0|\omega_*|^{d}.
\end{align*}
Hence we can find $\omega\in [\omega_*-1,\omega_*+1]$ such that
\begin{equation}\label{eq:choice_omega}
\inf_{\mu\in [-\tau,\tau]}  
|k(\mu+\omega i)| \ge \exp(-C_0|\omega_*|^{d}).
\end{equation} 

For a self-adjoint trace class operator $X$ such that $1+X$ is positive, we have 
\[
\|(1+X)^{-1}\|\le \det(1+X)^{-1} \exp(\|X\|_{\trace}+1)
\]
because, writing $\sigma_i>-1$ for the eigenvalues of $X$, we have\footnote{Note that $e^{x}/(1+x)>1$ for $x>-1$.}
\[
\det(1+X)^{-1} \exp(\|X\|_{\trace})\ge \prod_{i}\frac{e^{\sigma_i}}{1+\sigma_i}\ge \max_{i} e^{\sigma_i}(1+\sigma_i)^{-1}\ge e^{-1}\|(1+X)^{-1}\|.
\]
Applying this to $\mathcal{U}(s)^*\cdot \mathcal{U}(s)=1+X$ with setting
\[
X=-Y-Y^*+Y^*\cdot Y\quad \mbox{and}\quad Y=1-\mathcal{U}(s)=\chi_0(s-A')^{-1}\circ \widetilde{\cT}_{\omega_*},
\] 
we obtain that  
\begin{align*}
\|\mathcal{U}(s)^{-1}\|^{2}&\le C_0\exp(C_0 \|X\|_{\trace})\cdot  |\det\mathcal{U}(s)|^{-2}\quad \mbox{for $s\in R(\omega_*)$.}
\end{align*}
Estimating the trace norm $\|X\|_{\trace}$ by     (\ref{eq:tracenorm_Tk}) and Claim 3 and recalling the relation (\ref{eq:U}), we conclude   
\[
\| \cR(s) \|\le C_0\|\mathcal{U}(s)^{-1}\|\le C_0\exp(C_0|\omega_*|^{d})\cdot |k(s)|^{-1} \quad \mbox{for $s\in R(\omega_*)$.
}
\]
Therefore (\ref{eq:choice_omega}) implies  the required estimate.
\end{proof}

\subsection{Boundedness of  the resolvent} \label{ss:bddR}
The third step toward the proof of Theorem \ref{th:spectrum} is to prove that there are only finitely many eigenvalues of the generator $A$ on the outside of 
$U(\chi_0, \tau)$. (See (\ref{eq:Uchiepsilon}) for the definition of $U(\chi_0, \tau)$.)

\begin{Proposition}\label{pp:resolvent1} 
There exists $s_0>0$ such that the resolvent $\cR(s)$ is  bounded as an operator on $\wK^{r}(K_0)$ uniformly for $s\in \complex\setminus U(\chi_0, \tau)$ satisfying $\Re(s)>-\chi_0+\tau$ and $|\Im(s)|\ge s_0$. 
\end{Proposition}
\begin{proof} 
We consider $s\in \complex\setminus U(\chi_0, \tau)$ satisfying $\Re(s)>-\chi_0+\tau$ and $|\Im(s)|\ge s_0$ and take $\omega_*=\omega_*(s)$ so that $s\in R(\omega_*)$. 
We show that, for any $u\in \wK^r(K_0)$, there exists $w\in \wK^r(K_0)$  such that  
$\|w\|_{\wK^r}\le  C_0\|u\|_{\wK^r}$
and that 
\begin{equation}\label{eq:s-Aclaim}
\|(s-A) w-u\|_{\wK^r}\le  \frac{1}{2}\|u\|_{\wK^r}.
\end{equation}
By iterative approximation as in the proof of Lemma \ref{Lemma3}, this implies that $(s-A)$ has a right inverse whose operator norm is bounded by $C_0$. Then, by Lemma \ref{Lemma3} and the relation (\ref{eq:U}), we see that the right inverse thus obtained is actually the resolvent $\cR(s)$ and therefore obtain the conclusion of the proposition.

The following argument is mostly parallel to that in the proof of Lemma \ref{Lemma3}. In the case $\Re(s)>\tau$, we set 
\[
w=w_1+w_2+w_3+\int_0^{t_0}e^{-st}\cL^t u dt
\]
where, letting $\tilde{u}=e^{-st_0}\cL^{t_0} u$, we set 
\begin{align}
w_1&=\sum_{\omega:|\omega -\omega_*|\le  \ell}\int_0^{t(\omega_*)} e^{-(s+\chi_0)t}\cdot \cL^t\circ \cT_\omega \tilde{u} dt,\\
w_2&=\sum_{\omega:|\omega -\omega_*|\le  \ell}\int_0^{t(\omega_*)} e^{-st}\cdot \cL^t\circ (\cQ_\omega-\cT_\omega) \tilde{u} dt,\\
w_3&=\sum_{|\omega -\omega_*|> \ell}
(s-i\omega)^{-1}    \cQ_\omega \tilde{u}
\end{align}
where $\ell>0$ is an integer that we will specify later.
In the case $\Re(s)<-\tau$, we replace the definition of $w_1$ above by 
\[
w_1=-\sum_{\omega:|\omega -\omega_*|\le \ell}\int_0^{t(\omega_*)} e^{s (t(\omega_*)-t)}  \cL^t  v_\omega dt\]
where $v_\omega\in \cK^{r,+}(K_1)$ is that in Proposition \ref{pp:backward} with setting  
$\sigma=0$, $\sigma'=+1$ and 
letting $u$ in its statement be $\tilde{u}\in \cK^{r,+2}(K_2)$. 

We can check that $\|w\|_{\wK^r}\le C_0 \|u\|_{\wK^r}$ as in the proof of Lemma \ref{Lemma3}. 
Since 
\[
(s-A)\int_0^{t_0} e^{-st} \cL^t u dt=u-e^{-st_0} \cL^{t_0} u=u-\tilde{u},
\]
the inequality (\ref{eq:s-Aclaim}) follows if we prove the claims   
\begin{enumerate}
\item $\|(s-A) w_1-\sum_{\omega:|\omega -\omega_*|\le  \ell}  \cT_\omega \tilde{u}\|_{\wK^r}<(1/6) \|u\|_{\wK^r}$,
\item $\|(s-A) w_2-\sum_{\omega:|\omega -\omega_*|\le  \ell}  (\cQ_\omega-\cT_\omega) \tilde{u}\|_{\wK^r}<(1/6) \|u\|_{\wK^r}$,
\item $\|(s-A) w_3-\sum_{\omega:|\omega -\omega_*|>  \ell}  \cQ_\omega \tilde{u}\|_{\wK^r}<(1/6) \|u\|_{\wK^r}$.
\end{enumerate}
The proofs of the claims  (1) and (2) are obtained from that of the corresponding claims in the proof of  Lemma \ref{Lemma3}, letting $\chi_0$ be $0$ in some places. 
But, in the case $\Re(s)<-\tau$, we need to modify the proof of (2)  slightly as follows: 
We write
\begin{align*}
(s-A) w_1&=-\sum_{\omega:|\omega -\omega_*|\le \ell}\int_0^{t(\omega_*)} e^{s (t(\omega_*)-t)}  \cL^t  v_\omega dt\\
&=\sum_{\omega:|\omega -\omega_*|\le \ell}[e^{st(\omega_*)}\cdot  v_\omega-\cL^{t(\omega_*)}v_\omega]
\end{align*}
and check that the claim follows from the choice of $v_\omega$. 
The proof of the claim (3) is again parallel to that in Lemma \ref{Lemma3}.
\end{proof}

\begin{Remark}\label{rem:vector_valued_case_4}
In the cases of vector-valued transfer operators $\cL^t_{k,\ell}$ with $(k,\ell)\neq (0,0)$, we can get the proof of the corresponding claims of Theorem \ref{th:spectrum} by the argument parallel to that in this section  for the case $\Re(s)>\tau$.  (See Remark \ref{rem:vector_valued_case_1} and Remark \ref{rem:vector_valued_case_3} also.) Note that we do not need statements corresponding to Lemma \ref{lm:trace1} and Proposition \ref{pp:backward} for these cases.
\end{Remark}

\subsection{Lower bound for the density of eigenvalues}
\label{ss:lowerbound}
To complete\footnote{As we noted in Remark \ref{rem:mero12}, we actually have proved discreteness of the spectral set of the generator $A$ only on the region $\Re(s)>-\chi_0+\tau$ and $|\Im(s)|\ge s_0$ for some large $s_0$. We will see that this is true for the region $\Re(s)>-r\chi_0/4$ in Appendix \ref{sec:zeta}.}
 the proof of Theorem \ref{th:spectrum}, it is enough to prove the following lower bound on the density of eigenvalues of the generator $A$. 
 
\begin{Proposition}\label{pp:lowerBound} For any  $\delta>0$, there exist  constants $C_0>1$ and $\omega_0>0$ such that, for $\omega_*\in \integer$ with $|\omega_*|\ge \omega_0$, it holds
\[
\frac{\#\{\;\mbox{poles of $\cR(s)$ such that $|\Re(s)|<\tau$ and $|\Im(s)-\omega_*|\le |\omega_*|^{\delta}$\;}\}}{  |\omega_*|^{\delta}}\ge \frac{|\omega_*|^{d}}{C_0}.
\] 
\end{Proposition}
\begin{proof} 
For  $\omega_*\in \integer$, we consider the rectangle
\begin{equation}\label{region3}
\mathrm{Rect}(\omega_*)=\{s\in \complex\mid |\Re(s)|<\tau,\; - |\omega_*|^{\delta}+\Delta<\Im(s)-\omega_*<|\omega_*|^{\delta}-\Delta'\}
\end{equation}
where we choose  $\Delta,\Delta'\in [0,2]$ so that  the estimate (\ref{eq:Rs}) in Corollary \ref{cor:Rs} holds true on the horizontal sides of $\mathrm{Rect}(\omega_*)$.
Then we consider the spectral projector 
\[
\Pi_{\omega_*}=\frac{1}{2\pi i}\int_{\partial \mathrm{Rect}(\omega_*)}\cR(s) ds
\]
for the spectral set of $A$ in this rectangle. 
  For the proof of the proposition, it is enough to show that 
\[
\mathrm{rank}\,\Pi_{\omega_*} \ge  \frac{1}{C_0}{|\omega_*|^{d+\delta}} \quad \mbox{ when $|\omega_*|$ is sufficiently large.}
\] 
We prove this claim  by contradiction.  Let $\ell>0$ be a constant which we will   specify in the course of the argument (independently of $\omega_*$). We consider an integer $\omega_*$ with large absolute value and  take a sequence 
\[
\omega_*-|\omega_*|^{\delta}/2<\omega(1)<\omega(2)<\cdots<\omega(k)<\omega_*+|\omega_*|^{\delta}/2
\]
so that 
\[
|\omega(j+1)-\omega(j)|\ge 2\ell\qquad \mbox{and}\qquad  
k\ge \frac{|\omega_*|^{\delta}}{4\ell}.
\]
We take the subspace $V(\omega(j))$ in Lemma \ref{lm:trace1}(b) for each $1\le j\le k$ and 
set 
\[
\widetilde{V}(\omega(j))=\cT_{\omega(j)}({V}(\omega(j)))\subset \cK^{r,+}(K_1)\subset \cK^{r}(K_1).
\]
From the choice of $V(\omega(j))$  in Lemma \ref{lm:trace1}(b), we have
\[
\dim \widetilde{V}(\omega(j))=\dim {V}(\omega(j))\ge C_0^{-1} | \omega_*|^{d}.
\]
Let us set 
\[
W(\omega_*)=\sum_{j=1}^k \widetilde{V}(\omega(j))\subset \cK^{r,+}(K_1).
\]
From (\ref{eq:Vexp}) and (\ref{eq:Vexp2}) in Lemma \ref{lm:trace1}, the subspaces $\widetilde{V}(\omega(j))$ for $1\le j\le k$ are almost orthogonal to each other (and linearly independent) provided that $\ell$ is sufficiently large. More precisely, if we set $\tilde{v}_j=\cT_{\omega(j)} v_j\in \widetilde{V}(\omega(j))$ for any given $v_j\in V(\omega(j))$ for $1\le j\le k$, we have
\begin{align}\label{eq:almostorthogonal}
\sum_{i,j:i\neq j}|(\tilde{v}_i,\tilde{v}_j)_{\cK^{r,\sigma}}|&\le C_0\sum_{i,j:i\neq j} \langle \omega(i)-\omega(j)\rangle^{-2} \|v_i\|_{\cK^{r,\sigma}}\|v_j\|_{\cK^{r,\sigma}}\\
&\le C_0 \ell^{-1}\cdot \sum_j \|v_j\|_{\cK^{r,\sigma}}^2\le C_0 \ell^{-1}\cdot \sum_j \|\tilde{v}_j\|_{\cK^{r,\sigma}}^2.\notag
\end{align}
Hence, provided that $\ell$ is so large that $C_0\ell^{-1}<1$, we see 
\[
\dim  W(\omega_*)=\sum_{j=1}^k \dim \widetilde{V}(\omega(j))=\sum_{j=1}^k \dim V(\omega(j))\ge  \frac{ |\omega_*|^{d+\delta}} {4\ell  C_0}.
\]

From our assumption (for the proof by contradiction), we can take arbitrarily large $\omega_*\in \integer$ and an element  $\tilde{w}\in W(\omega_*) $ with $\|\tilde{w}\|_{\cK^r}=1$ that belongs to $\ker \Pi_{\omega_*}$. 
We express $\tilde{w}$ as 
\[
\tilde{w}=\sum_{j=1}^{k} \cT_{\omega(j)} w(j) \quad \mbox{ with }w(j)\in V(\omega(j))
\]
and,  for simplicity,  set 
\[
\tilde{w}(j):=\cT_{\omega(j)}w(j)\in \widetilde{V}(\omega(j))\subset \cK^{r,+}(K_1).
\]
Note that we have
\begin{equation}\label{eq:wi_norm}
C_0^{-1}\|w(j)\|_{\cK^r}\le \|\tilde{w}(j)\|_{\cK^r}\le  C_0\|w(j)\|_{\cK^r}
\end{equation}
from the choice of $V(\omega(j))$ and the uniform boundedness of the operators $\cT_{\omega}$ in $\omega$.
We choose an integer $1\le k_*\le k$ so that $\|w(k_*)\|_{\cK^r}$ is the largest among $\|w(j)\|_{\cK^r}$ for $1\le j\le k$.

For further argument, we introduce  an entire holomorphic function
\begin{equation}\label{eq:OmegaBd}
\Omega:\complex\to \complex\setminus \{0\}, \quad \Omega(s):=\exp(1-\cos(s)).
\end{equation}
This function converges to zero rapidly when $|s|\to \infty$ in the strip $|\Re(s)|<\pi/3$. More precisely, we have
\[
|\Omega(s)|=\exp(1-\Re(\cos(s))) \le \exp(1-\exp(|\Im(s)|)/4) 
\]
for $s\in \complex$ with  $|\Re(s)|\le \pi/3$, because, if $s=x+iy$ with $x\in [-\pi/3,\pi/3]$,  
\[
\Re(\cos(s))=\cos(x) \cdot \cosh(y)\ge \exp(|y|)/4.
\]
Also we can check  that $y\mapsto \Omega(x+iy)$ for $x\in [-\pi/3,\pi/3]$ is a function in the Schwartz class $\cS(\real)$ and uniformly bounded. 

Let $b>0$ be a small constant, which we will specify in the last part of the proof, and define the $\cK^r(K_0)$-valued function 
\[
\mathcal{Y}: \mathrm{Rect}(\omega_*)\to \cK^r(K_0),\quad  \mathcal{Y}(s)=\Omega(b (s-i\omega(k_*)))\cdot \cR(s) \tilde{w}.
\]  
Since $\tilde{w}$  belongs to the kernel of the spectral projector $\Pi_{\omega_*}$, this is  holomorphic  on a  neighborhood of the rectangle $\mathrm{Rect}(\omega_*)$ and hence we have
\begin{equation}\label{eq:Cauchy}
\int_{\partial \mathrm{Rect}(\omega_*)}\mathcal{Y}(s)ds=0. 
\end{equation}
Below we show that this can not be true. In fact, we  claim that, if $|\omega_*|$ is sufficiently large, we have 
\begin{equation}\label{eq:Cauchy-contradiction}
\left\|\cQ_* \left(\int_{\partial \mathrm{Rect}(\omega_*)} \mathcal{Y}(s)ds\right)- \tilde{w}(k_*) \right\|_{\cK^r}\le \frac{1}{2}\|\tilde{w}(k_*) \|_{\cK^r}
\end{equation}
where  
\[
\cQ_*=\sum_{\omega:|\omega -\omega(k_*)|\le \ell} \cQ_\omega.
\]
Since $\tilde{w}(k_*)\neq 0$ from the choice of $k_*$ and (\ref{eq:wi_norm}), this proves the proposition by  contradiction.

To begin with,  note that we have
\begin{equation}\label{eq:cQ-1}
\left\|\cQ_* \tilde{w}(k_*)- \tilde{w}(k_*) \right\|_{\cK^r}\le \frac{1}{4}\|\tilde{w}(k_*) \|_{\cK^r}
\end{equation}
provided that $\ell$ is sufficiently large. 
In fact, since $\tilde{w}(k_*)$ is supported on $K_1$ and satisfies (\ref{eq:qid}),  we may write the left hand side as 
\begin{align*}
\left\|\sum_{\omega':|\omega'-\omega(k_*)|>\ell}\cQ_{\omega'} \circ \cT_{\omega(k_*)}{w}(k_*) \right\|_{\cK^r}.
\end{align*}
Hence, using (\ref{eq:upperboundTLT}) and (\ref{eq:upperboundQLT}) for $t=0$  and also (\ref{eq:wi_norm}), we can check (\ref{eq:cQ-1}).

Let us write  $\partial^{h} \mathrm{Rect}(\omega_*)$ and $\partial^{v} \mathrm{Rect}(\omega_*)$ for the horizontal and vertical sides of the rectangle $\mathrm{Rect}(\omega_*)$.  For the integral (\ref{eq:Cauchy}) restricted to the horizontal sides, we have, from the choice of $\Delta, \Delta'>0$ in the definition of $\mathrm{Rect}(\omega_*)$ and (\ref{eq:OmegaBd}), that 
\begin{equation}\label{eq:cY}
\left\|\int_{\partial^h \mathrm{Rect}(\omega_*)} \!\!\!\!\mathcal{Y}(s)ds\right\|_{\cK^r}\le
  \exp\bigg( -\exp (|\omega_*|^{\delta}-1)/4\bigg)\cdot
\exp\big( C_0 |\omega_*|^d\big) \cdot\|\tilde{w}\|_{\cK^r}.
\end{equation}
Since $\|\tilde{w}\|_{\cK^r}\le C_0 |\omega_*|^{\delta} \|w(k_*)\|_{\cK^r}$ from the choice of $k_*$, this part of integral is much smaller than  $\|w(k_*)\|_{\cK^r}$
provided $|\omega_*|$ is large.

To evaluate the integral on the vertical sides, we prepare the next lemma.  Recall that $\epsilon_0$ is the constant that appear in the definition (\ref{eq:tomega}) of $t(\omega)$. 
\begin{Lemma} \label{lem:integral_on_vertical_side}
Suppose that $\rho\in \real$ satisfies $|\rho-\omega_*|\le |\omega_*|^{\delta}+1$.
There exists a constant $C_0>0$, independent of $\omega_*$ and $\rho$, such that, for  $1\le j\le k$, we have
\[
\left\|\cR(\tau+i\rho) \tilde{w}(j)-\int_{0}^{t(\omega_*)} 
 e^{-(\tau +i\rho) t}   \cL^t \tilde{w}(j) dt \right\|_{\cK^r}
\le C_0 | \omega_*|^{-\tau \epsilon_0} \|\tilde{w}(j)\|_{\cK^{r}}
\]
and further that, for the function  $v_{\omega(j)}$ given in  Proposition \ref{pp:backward} for the setting $\sigma=\sigma'=0$, $u=w(j)$, $\omega=\omega(j)$ and $t=t(\omega_*)\le 2t(\omega(j))$, we have 
\[
\left\| \cR(-\tau+i\rho) \tilde{w}(j)+\int_{0}^{t(\omega_*)} 
 e^{(-\tau +i\rho) t}  \cL^{t(\omega_*)-t}  v_{\omega(j)} dt \right\|_{\cK^r}
 \le C_0 |\omega_*|^{-\tau \epsilon_0}  \|\tilde{w}(j)\|_{\cK^{r}}.
\]
\end{Lemma}
\begin{proof} Since
\[
((\tau+i\rho)-A) \left(\int_{0}^{t(\omega_*)} 
 e^{-(\tau +i\rho) t} \cL^t  dt\right)=1-e^{-(\tau+i\rho)t(\omega_*)} \cL^{t(\omega_*)},
\]
we obtain, by applying $\cR(\tau+i\rho)$ to the both sides, that
\[
 \cR(\tau+i\rho)=\left(\int_{0}^{t(\omega_*)} 
 e^{-(\tau +i\rho) t} \cL^t  dt\right)+e^{-(\tau+i\rho)t(\omega_*)}\cR(\tau+i\rho)\circ \cL^{t(\omega_*)}.
\]
We apply this operator to $\tilde{w}(j)$. Since   $\cL^{t(\omega_*)}\tilde{w}(j)$ is supported on $K_1$ and satisfies (\ref{eq:qid}), we may use  (\ref{eq:upperboundTLT}) and (\ref{eq:upperboundQLT}) in Proposition \ref{pp:norm_estimate} and also Proposition \ref{pp:resolvent1} to get the estimate
\begin{align*}
\left\|\cR(\tau+i\rho) \tilde{w}(j)-\int_{0}^{t(\omega_*)} 
 e^{-(\tau +i\rho) t}     \cL^t \tilde{w}(j) dt \right\|_{\cK^r}
 &\le C_0 e^{-\tau t(\omega_*)} \|{w}(j)\|_{\cK^{r}}\\
& \le C_0 |\omega_*|^{-\tau \epsilon_0} \|\tilde{w}(j)\|_{\cK^{r}}.
\end{align*}
This is  the first claim. 
We can get the second inequality by a similar manner.  
Since
\[
((-\tau+i\rho)-A) \left(\int_{0}^{t(\omega_*)} 
 e^{(-\tau +i\rho)t} \cL^{t(\omega_*)-t} dt \right)=e^{(-\tau+i\rho)t(\omega_*)}-\cL^{t(\omega_*)},
\]
we have
\[
 \cR(-\tau+i\rho)\circ \cL^{t(\omega_*)}=-\left(\int_{0}^{t(\omega_*)} 
 e^{(-\tau +i\rho) t} \cL^{t(\omega_*)-t}  dt\right)+e^{(-\tau+i\rho)t(\omega_*)} \cR(-\tau+i\rho).
\]
We obtain the second inequality by applying this operator to $v_{\omega(j)}$ and using the estimate
\[
\left\| \cR(-\tau +i\rho)\circ \cL^{t(\omega_*)} v_{\omega(j)}-\cR(-\tau +i\rho)\tilde{w}(j)\right\|_{\cK^r}
\le C_0|\omega_*|^{-\theta} \|w(j)\|_{\cK^r}
\]
that follows from the condition (\ref{eq:backward1}) in the choice of $v_{\omega(j)}$. 
(We choose $\epsilon_0$ so small that $\tau \epsilon_0<\theta$.)
\end{proof}

From  Lemma \ref{lem:integral_on_vertical_side} above and the choice of $k_*$, we have that 
\begin{align*}
&\cQ_* \int_{\partial \mathrm{Rect}(\omega_*)}\mathcal{Y}(s) ds\\
&=\sum_{j=1}^k \int_{-\infty}^{+\infty} d\rho \int_0^{t(\omega_*)}
\Omega_{b,+}(\rho-\omega(k_*))  e^{ -i(\rho-\omega(k_*)) t} e^{(-\tau-i\omega(k_*)) t} \cQ_* \circ \cL^t \widetilde{w}(j)  \,dt \\
&\quad+\sum_{j=1}^k \int_{-\infty}^{+\infty}\!\! d\rho \int_0^{t(\omega_*)}
 \!\! \Omega_{b,-}(\rho-\omega(k_*)) e^{i(\rho-\omega(k_*)) t} e^{(-\tau+i\omega(k_*)) t}   \cQ_*\circ \cL^{t(\omega_*)-t} {v}_{\omega(j)}  \, dt\\
& \quad +\mathcal{O}_{\cK^{r}}\left(|\omega_*|^{-\tau \epsilon_0+\delta}  \|\tilde{w}(k_*)\|_{\cK^{r}}\right)
\end{align*}
where we set
\[
\Omega_{b,\pm}(\rho)=\Omega(b(\pm \tau+\rho i)).
\]
\begin{Remark}
The last term $\mathcal{O}_{\cK^{r}}\left(|\omega_*|^{-\tau \epsilon_0+\delta}  \|\tilde{w}(k_*)\|_{\cK^{r}}\right)$ denotes an error term whose  $\cK^r$-norm is bounded  by $C|\omega_*|^{-\tau \epsilon_0+\delta}   \|\tilde{w}(k_*)\|_{\cK^{r}}$.  We  use this notation below. 
\end{Remark}

Note that the integration along the horizontal sides of $\mathrm{Rect}(\omega_*)$ and also the integration with respect to  $\rho$ on the outside of the interval
\[
[\omega_*-|\omega_*|^{\delta}+\Delta, \;\omega_*+|\omega_*|^{\delta}-\Delta']
\]
 is included  in the last error term. (The former is small as we have seen in (\ref{eq:cY}). The latter is also very small because of the property of the function $\Omega(\cdot)$ in the integrand.)

Performing integration with respect to $\rho$, we get
\begin{align*}
\cQ_* \int_{\partial^v \mathrm{Rect}(\omega_*)}\mathcal{Y}(s) ds
&= \sum_{j=1}^k \int_0^{t(\omega_*)}
\widehat{\Omega}_{b,+}(t) \cdot e^{(-\tau-i\omega({k_*})) t}\cdot \cQ_* \circ \cL^t \widetilde{w}(j) \,dt \\
&\quad +\sum_{j=1}^{k} \int_0^{t(\omega_*)}
\widehat{\Omega}_{b,-}(-t) \cdot e^{(-\tau+i\omega({k_*})) t} \cdot   \cQ_*\circ\cL^{t(\omega_*)-t}  {v}_{\omega(j)}  \, dt\\
&\qquad +\mathcal{O}_{\cK^{r}}\left( |\omega_*|^{-\tau \epsilon_0+\delta} \|\tilde{w}(k_*)\|_{\cK^{r}}\right)
\end{align*}
where $\widehat{\Omega}_{b,\pm}(\cdot)$ is the Fourier transform of ${\Omega}_{b,\pm}(\cdot)$, 
\[
\widehat{\Omega}_{b,\pm}(t)=\frac{1}{2\pi}\int_{-\infty}^{\infty} e^{-i\rho t}\Omega_{b,\pm}(\rho)  d\rho=\frac{1}{2\pi}\int_{-\infty}^{\infty} e^{-i\rho t}\Omega(b(\pm \tau+i\rho))  d\rho.
\]
In both of the sums over $1\le j\le k$  on the right-hand side above, the contribution from the terms other than $j=k_*$ is relatively small provided that $\ell>0$ is large enough (independently of $\omega_*$). 
In fact, from Proposition \ref{pp:norm_estimate} and the choice of $k_*$,  
\begin{align*}
\left\| \cQ_* \circ \cL^t \widetilde{w}(j) \right\|_{\cK^r}
&\le \frac{C_\nu \ell\cdot  \|\tilde{w}(j)\|_{\cK^r}}{\langle|\omega(j)-\omega(k_*)|-\ell \rangle^\nu}  \le  \frac{C_\nu \ell \cdot \|\tilde{w}(k_*)\|_{\cK^r}}{\langle |\omega(j)-\omega(k_*)|-\ell\rangle^{\nu}}  
\end{align*}
for $j\neq k_*$ and $0\le t\le t(\omega_*)$,  with arbitrarily large $\nu>0$. Similarly, from (\ref{eq:cv1}) in the choice of $v_\omega$, we have
\begin{align*}
\left\| \cQ_* \circ \cL^t {v}_{\omega(j)} \right\|_{\cK^r}\le  \frac{C_\nu \ell}{\langle |\omega(j)-\omega(k_*)|-\ell\rangle^{\nu}}\cdot \|\tilde{w}(k_*)\|_{\cK^r}.
\end{align*}
Therefore the sum of contributions from the integrals for $j\neq k_*$ is bounded by $C \ell^{-\nu+2} \|\tilde{w}(k_*)\|_{\cK^r}$ in the $\cK^r$-norm. That is, we have 
\begin{align*}
\cQ_* \int_{\partial^v \mathrm{Rect}(\omega_*)}\mathcal{Y}(s) ds
&= \cQ_* \left( \int_0^{t(\omega_*)}
\widehat{\Omega}_{b,+}(t) \cdot e^{(-\tau-i\omega(k_*)) t}\cdot \cL^t \widetilde{w}(k_*)\,dt\right) \\
&\quad +\cQ_* \left( \int_0^{t(\omega_*)}
\widehat{\Omega}_{b,-}(-t) \cdot e^{(-\tau+i\omega(k_*)) t} \cdot   \cL^{t(\omega_*)-t}  {v}_{\omega(k_*)}  \, dt\right)\\
&\quad  +\mathcal{O}_{\cK^r}\left((\ell^{-\nu+1}+|\omega_*|^{- \tau \epsilon_0+\delta})\cdot  \|\tilde{w}(k_*)\|_{\cK^{r}}\right).
\end{align*}
We (finally) fix the constant $\ell >0$ so that the last error term is bounded by $\|\tilde{w}(k_*)\|_{\cK^{r}}/10$ when $\omega_*$ is sufficiently large. (Since we have only to prove Proposition \ref{pp:lowerBound} for sufficiently small $\delta$, we assume $\delta<\tau \epsilon_0$.)

Now we  let the constant $b$ be small. Then  the functions $
\widehat{\Omega}_{b,\pm}(t)$ concentrate around $0$ (in the $L^1$ sense) and, further,  $\int_0^{+\infty} \widehat{\Omega}_{b,\pm}(t) dt$ and $\int_{-\infty}^{0} \widehat{\Omega}_{b,\pm}(t) dt$ become close to $1/2$ by symmetry. 
By (\ref{eq:continuity_omega}) in Lemma \ref{lm:continuity_omega}, we have 
\[
\|e^{(-\tau -i \omega(k_*)) t} \cL^t   \tilde{w}(k_*)-\tilde{w}(k_*)\|_{\cK^{r}}\le C_0 |t| \|\tilde{w}(k_*)\|_{\cK^{r}}. 
\]
Similarly, by (\ref{eq:continuity_omega}) in Lemma \ref{lm:continuity_omega} and (\ref{eq:backward1}) in the choice of $v_{\omega(k_*)}$, we have also
\[
\|e^{(-\tau+i\omega(k_*)) t}   \cL^{t(\omega_*)-t}  {v}_{\omega(k_*)}-\tilde{w}(k_*)\|_{\cK^{r}}\le C_0 
(|t|+|\omega_*|^{-\theta}) \|\tilde{w}(k_*)\|_{\cK^{r}}.
\]
Therefore each of the integrations on the right-hand side above become close to $\cQ_*\tilde{w}(k_*)/2$ in $\cK^r(K_0)$ as $b\to +0$ uniformly in $\omega_*$. 
Recalling (\ref{eq:cQ-1}), we conclude  (\ref{eq:Cauchy-contradiction}) when $|\omega_*|$ is sufficiently large, provided the constant $b>0$ is sufficiently small. 
We finished the proof of Proposition \ref{pp:lowerBound}. 
\end{proof}


\section{Some preparatory lemmas}\label{sec:pre} 
In the last section, we have deduced Theorem \ref{th:spectrum} from the propositions given in Subsection \ref{ss:lt}. 
The remaining task (in proving Theorem \ref{th:spectrum}) is to prove those propositions. 
This is done in this section and the following two sections. This section is devoted to some basic estimates. 
\subsection{Multiplication by functions} \label{ss:mult}
We begin with considering the multiplication operator by a function $\psi\in C^\infty_0(\real^{2d+d'+1})$, 
\[
\mult(\psi):C^{\infty}(\real^{2d+d'+1})\to C^{\infty}(\real^{2d+d'+1}), \qquad \mult(\psi) u=\psi\cdot u
\]
and its lift with respect to the partial Bargmann transform,
\begin{equation}\label{eq:mult-lifted}
\mult(\psi)^\lift:=\pBargmann \circ \mult(\psi)\circ \pBargmann^*.
\end{equation}
Below we assume the following setting, which abstracts the situations that we will meet later.
\begin{framed}
\noindent {\bf Setting I:}\; For each $\omega\in \integer$, there is a given set $\cX _\omega$  of $C^\infty$ functions on $\real^{2d+d'+1}$ such that the following conditions hold for all $\omega\in\integer$ and $\psi\in \cX_{\omega}$ with uniform positive constants $C$ and $C_{\alpha,k}$ (independent of $\omega$ and $\psi$):
\begin{itemize}
\item[(C1)] the support of $\psi\in \cX_\omega$ is  contained in 
\[
\disk^{(2d)}(C\langle \omega\rangle^{-1/2+\theta}) \oplus \disk^{(d')}(C)\oplus \disk^{(1)}(C)\subset \real^{2d+d'+1}
\] 
where $\disk^{(D)}(\delta)\subset \real^{D}$ is the disk of radius $\delta$ with center at the origin.
\medskip
\item[(C2)] $\psi \in \cX_\omega$ satisfies the uniform estimate
\[
|\partial^\alpha_{w} \partial^k_z \psi(w,z)|<C_{\alpha,k}
\langle \omega\rangle^{(1-\theta)|\alpha|/2}, \quad \mbox{$\forall w\in \real^{2d+d'}$, $\forall z\in \real$,}
\]
 for any multi-indices $\alpha\in \integer_+^{2d+d'}$ and $k\in \integer_+$.
\end{itemize}
\end{framed}
\begin{Remark}
The conditions (C2) above means that the normalized family 
\[
\widetilde{\cX}_{\omega}=\{ \tilde{\psi}(w,z)=\psi(\langle \omega\rangle^{-(1-\theta)/2} w,  z)\mid \psi \in {\cX}_{\omega}\}
\]
is uniformly bounded in $C^k$-norm for any $k$, that is, they look very smooth (or almost constant) in the variable $w$ if we view them in the scale $\langle \omega\rangle^{-1/2}$. This observation is basic in the following argument.
\end{Remark}
\begin{Remark}\label{ex:setting1}
If we set $\mathcal{X}_{\omega}=\{\rho_{\j}\mid \j\in \cJ\mbox{ with } \omega(\j)=\omega\}$ where $\rho_{\j}$ are those defined in (\ref{eq:rhoan}) and (\ref{eq:defsrho}), then the conditions (C1) and (C2) above hold. 
But notice that a little stronger condition than (C2) holds: we may replace  $\langle \omega\rangle^{(1-\theta)|\alpha|/2}$ by $\langle \omega\rangle^{(1-2\theta)|\alpha|/2}\ll \langle \omega\rangle^{(1-\theta)|\alpha|/2}$ in (C2). 
\end{Remark}

In the next lemma, we consider the lifted multiplication operator $\mult(\psi)^\lift$ for $\psi\in {\cX}_{\omega}$ precomposed by $\mult(q_{\omega})$ and approximate it by a simpler operator constructed as follows. 
For $\psi\in \cX_\omega$, let $\widehat{\psi}$ be the Fourier transform of $\psi$ along the $z$-axis:
\begin{equation}\label{eq:hatpsi}
\widehat{\psi}(w,\xi_z)= \frac{1}{2\pi} \int e^{-i\xi_z z}\psi(w,z) dz,\qquad w\in \real^{2d+d'}, \xi_z\in \real. 
\end{equation}
Note that, from the conditions (C1) and (C2) in Setting I, 
there exists a constant $C_{\alpha,\nu}>0$ for any  $\alpha\in \integer_+^{2d+d'}$ and $\nu>0$ such that
\begin{equation}\label{eq:wpsi}
 |\partial^\alpha_{w} \widehat{\psi}(w,\xi_z)|<C_{\alpha,\nu}
\langle \omega\rangle^{(1-\theta)|\alpha|/2} \langle \xi_z\rangle^{-\nu}\qquad \mbox{for $\psi\in \mathcal{X}_\omega$}.
\end{equation}
We then define the operator 
$\mathcal{C}({\psi}):\cS(\real^{4d+2d'+1})\to \cS(\real^{4d+2d'+1})$ 
by 
\[
\mathcal{C}({\psi}) u(w',\xi'_w,\xi'_z)=\int \widehat{\psi}(w',\xi'_z-\xi_z)\, u(w,(\langle\xi'_z\rangle/\langle \xi_z\rangle)\xi'_w,\xi_z)\, d\xi_z.
\]
This is essentially the convolution operator in the variable $\xi_z$ but looks a little more complicated because of the rescaling mentioned in Remark \ref{Rem:rescaledCoordinates}. 
Recall the Bargmann projection operator $\pBargmannP$ from (\ref{eq:pBargmannP}). For brevity of notation, we will write $L^2(\cW^{r,\sigma})$ for $L^2(\real^{4d+2d'+1}; (\cW^{r,\sigma})^2)$ below.
\begin{Lemma}\label{lm:multiplication} 
Let $\sigma\in \Sigma$. There exists a constant $C_\nu>0$ for each $\nu>0$ such that, for any $\omega, \omega'\in \integer$ and any $\psi\in \cX_\omega$, we have
\[
\|\mult(q_{\omega'})\circ (\mult(\psi)^\lift -\pBargmannP\circ \mathcal{C}({\psi}))\circ \mult(q_{\omega})\|_{L^2(\cW^{r,\sigma})}\le  C_\nu \langle \omega\rangle^{-\theta/2}\langle \omega'-\omega\rangle^{-\nu}
\]
and
\[
\|\mult(q_{\omega'})\circ (\mult(\psi)^\lift -\mathcal{C}({\psi})\circ \pBargmannP)\circ \mult(q_{\omega})\|_{L^2(\cW^{r,\sigma})}\le  C_\nu \langle \omega\rangle^{-\theta/2}\langle \omega'-\omega\rangle^{-\nu}.
\]
\end{Lemma}
\begin{proof}
Below we prove the first inequality. The second inequality is proved in a parallel manner. We write the kernel 
$K(w',\xi'_w,\xi'_z;w,\xi_w,\xi_z)$ of the operator $\mult(\psi)^\lift -\pBargmannP\circ \mathcal{C}({\psi})$ explicitly and find 
\begin{align*}
&  |K(w',\xi'_w,\xi'_z;w,\xi_w,\xi_z)|\\
&\;=a_{2d+d'}(\langle \xi'_z\rangle^{-1})^2 \\
&\quad \times 
 \left|\int 
e^{-i(\langle \xi'_z\rangle \xi'_w -\langle \xi_z\rangle\xi_w) w''}\cdot \left[e^{-\langle \xi'_z\rangle |w''-w'|^2/2-\langle \xi_z\rangle |w''-w|^2/2}
\cdot  \Delta(w,w'',\xi'_z,\xi_z)\right] dw''\right|
\end{align*}
where
\begin{align*}
&\Delta(w,w'',\xi'_z,\xi_z)\\
&\qquad =
\left(
\frac{\langle \xi_z\rangle^{(2d+d')/4}}
{\langle\xi'_z\rangle^{(2d+d')/4}} \cdot \widehat{\psi}(w'',\xi'_z-\xi_z)-e^{(\langle \xi_z\rangle-\langle \xi'_z\rangle) |w''-w|^2/2}\cdot \widehat{\psi}(w,\xi'_z-\xi_z)\right).
\end{align*}
The computation to get the expression above is straightforward. 
We take integration with respect to  $z$ and perform the change of  variables $(\langle\xi'_z\rangle/\langle \xi_z\rangle)\xi_w\mapsto \xi_w$. 
Note that we consider the volume form $d\vol$ in (\ref{eq:constNormalization}) and that we ignored the term $e^{i\langle \xi'_z\rangle \xi'_w\cdot w'/2-i\langle \xi_z\rangle\xi_w\cdot w/2}$ as we take  absolute value of the both sides.

We claim that, for any $\nu>0$,  
\begin{equation}\label{eq:KernelEstimate1}
|K(w',\xi'_w,\xi'_z;w,\xi_w, \xi_z)|\le C_\nu 
\frac{\langle\omega\rangle^{-\theta/2}\cdot \langle \langle \omega\rangle^{1/2}  |w'-w| \rangle^{-\nu}}
{  
 \langle \langle \omega\rangle^{-1/2} (\langle \xi'_z\rangle\xi'_w-\langle \xi_z\rangle\xi_w)\rangle^{\nu}\cdot \langle \xi'_z-\xi_z\rangle^{\nu}}
\end{equation}
provided $\xi'_z\in \supp q_{\omega'}$ and $\xi_z\in \supp q_{\omega}$, where the constant $C_\nu>0$ is uniform for $\psi\in \cX_\omega$ and $\omega\in \integer$.  
Once we get this estimate, we obtain the  conclusion by Schur test. (See the remark below and also recall (\ref{eq:property_w2}).)
Notice that $\nu>0$ in (\ref{eq:KernelEstimate1}) is arbitrary large and may be different from that in the statement of the lemma. 
\begin{Remark} 
Schur test mentioned above reads as follows: For an integral operator $T:L^2(\real^D)\to L^2(\real^D)$ of the form $Tf(x)=\int K(x,y) dy$, we have
\[
\|T:L^2(\real^D)\to L^2(\real^D)\|\le \left(\sup_{x} \int |K(x,y)| dy\right)^{1/2}
\left(\sup_y \int |K(x,y)| dx\right)^{1/2}.
\]
For the proof, see \cite[p.50]{Martinez} for instance.
\end{Remark}

In order to prove the estimate (\ref{eq:KernelEstimate1}), we apply integration by parts several times,  regarding the term $w''\mapsto e^{-i(\langle \xi'_z\rangle \xi'_w -\langle \xi_z\rangle\xi_w) \cdot w''}$ as the oscillating part and using the differential operator
\begin{equation}\label{eq:D1}
\cD_1=\frac{1+i\langle \omega\rangle^{-1}
(\langle \xi'_z\rangle\xi'_w -\langle \xi_z\rangle \xi_w)\cdot \partial_{w''}
}{1+\langle \omega\rangle^{-1} |\langle \xi'_z\rangle \xi'_w -\langle \xi_z\rangle  \xi_w|^2}.
\end{equation}
\begin{Remark}
Here and henceforth, we mean, by ``integration by parts  regarding $e^{i\varphi(x)}$ as the oscillatory part",  application of the formula
\[
\int e^{i\varphi(x)} \Phi(x) dx =\int (\cD^m e^{i\varphi})(x) \Phi(x) dx =\int  e^{i\varphi(x)} ({}^t\cD)^m \Phi(x) dx 
\]
which holds when a differential operator $\cD$ satisfies $(\cD e^{i\varphi})(x)=e^{i\varphi(x)}$. 
\end{Remark}
To get (\ref{eq:KernelEstimate1}), it is enough to show the estimate
\begin{equation}\label{eq:DiffEst}
|\partial^\alpha_{w''} \Delta(w,w'',\xi'_z,\xi_z)|
\le  
C_{\alpha,\nu} \langle \omega\rangle^{-\theta/2 +|\alpha|/2} \cdot 
 \langle \langle \omega\rangle^{1/2}  |w''-w| \rangle^{|\alpha|}\cdot \langle \xi'_z-\xi_z\rangle^{-\nu}
\end{equation}
when $\xi'_z\in \supp q_{\omega'}$ and $\xi_z\in \supp q_{\omega}$. For convenience, we separate the cases where $\omega$ and $\omega'$ are relatively close and apart, that is, the cases
\begin{equation}\label{cond:near-apart}
\mathrm{(i)} \; |\langle \omega\rangle-\langle\omega'\rangle |\le \langle \omega\rangle^{1/2}\quad \text{and} \quad \mathrm{(ii)}\;
|\langle \omega\rangle-\langle\omega'\rangle |> \langle \omega\rangle^{1/2}.
\end{equation} 
In the case (ii), the proof is easy: We apply (\ref{eq:wpsi}) to differentials of the two terms in $\Delta(w,w'',\xi'_z,\xi_z)$ separately and get (\ref{eq:DiffEst}) using  
\begin{equation}\label{cos:apart}
\langle \xi'_z-\xi_z\rangle^{-1}\le C_0\langle \omega'-\omega\rangle^{-1} < C_0 \langle \omega\rangle^{-1/2}.
\end{equation}
In the case (i), the proof is a little more complicated. Note that we have
\begin{equation}\label{eq:cos:near}
\left|\langle \xi'_z\rangle/\langle \xi_z\rangle -1\right|\le C_0 \langle \omega\rangle^{-1/2}
\end{equation}
in this case. If we replace the coefficient 
 of $\widehat{\psi}(w'',\xi'_z-\xi_z)$ in $\Delta(w,w'',\xi'_z,\xi_z)$ with $1$, the difference made in $\Delta(w,w'',\xi'_z,\xi_z)$ satisfies (\ref{eq:DiffEst}) even with the factor $\langle \omega\rangle^{-\theta/2}$ on the right hand side replaced with $\langle \omega\rangle^{-1/2}$ and is therefore negligible.
Also, noting the factor $e^{-\langle \xi_z\rangle |w''-w|^2/2}$ in $K(\cdot)$,  we can replace the coefficient of  $\widehat{\psi}(w,\xi'_z-\xi_z)$ with $1$, producing a negligible term. Therefore it is enough to prove (\ref{eq:DiffEst}) supposing 
\[
\Delta(w,w'',\xi'_z,\xi_z) =
\widehat{\psi}(w'',\xi'_z-\xi_z)- \widehat{\psi}(w,\xi'_z-\xi_z).
\]
But this is now an easy consequence of (\ref{eq:wpsi}). 
\end{proof}

\begin{Corollary} \label{cor:multi}
Let $\sigma\in \Sigma$. There is a constant $C_\nu>0$ for each $\nu>0$ such that 
\[
\|\mult(q_{\omega'})\circ \mult(\psi)^\lift \circ \mult(q_{\omega})\|_{L^2(\cW^{r,\sigma})}\le
 C_\nu \langle \omega'-\omega\rangle^{-\nu}
 \]
 for all $\omega, \omega'\in \integer$ and $\psi\in \mathcal{X}_{\omega}$. 
\end{Corollary}
\begin{Remark} \label{Rem:multiplication_L2}
Lemma \ref{lm:multiplication} and Corollary \ref{cor:multi} remain true when we consider the operators on the space  $L^2(\real^{4d+2d'+1})$ instead of $L^2(\cW^{r,\sigma})$, because we have proved the estimate (\ref{eq:KernelEstimate1}) on the kernel. 
\end{Remark}

\subsection{Transfer operator for nonlinear diffeomorphisms}\label{ss:small_nonlinear}
We now assume the following setting in addition to Setting~I. 
\begin{framed}
\noindent {\bf Setting II:}\; For each $\omega\in \integer$, there is a given set $\cG _\omega$  of 
 fibered contact diffeomorphisms $g:U_g\to \real^{2d+d'+1}$ whose domain $U_g\subset \real^{2d+d'+1}$ contains 
\[
\disk^{(2d)}(C\langle \omega\rangle^{-1/2+\theta})\oplus \disk^{(d')}(C)\oplus \disk^{(1)}(C)
\]
where $C$  is the constant in Setting I and, further,  
the following conditions hold for  $\omega\in \integer$ and $g\in \cG_{\omega}$ with positive constants $C'$ and $C_{\alpha}$ uniform for $\omega$ and~$g$:
\begin{itemize}
\setlength{\itemsep}{1mm}
\item[(G0)] $g(0,0,0)=(0,y_*,0)$ where $y_*\in \real^{d'}$ satisfies $|y_*|\le C'\langle \omega\rangle^{-1/2-3\theta}$.
\item[(G1)]  For the first derivative of $g$ at the origin $0\in \real^{2d+d'+1}$, we have  
\[
\|Dg(0)-\mathrm{Id}\|<C'\max\{\langle \omega\rangle^{-\beta(1/2-\theta)}, \langle \omega\rangle^{-(1-\beta)(1/2-\theta)-2\theta}\}.
\]
\item[(G2)] We have 
\[
\|\partial_w^\alpha   g(w,z)\|<C_{\alpha}  \langle \omega\rangle^{((1-\beta)(1/2-\theta)+4\theta)(|\alpha|-1)+|\alpha|\theta/2}\quad \mbox{ on $U_g$}
\]
for $\alpha\in \integer^{2d+d'}_+$  with $|\alpha|\ge 1$. Further, for the base diffeomorphism (defined in Definition \ref{def:fiberedContact}) $\check{g}:\proj_{(x,z)}(U_g)\to \real^{2d+1}$ of $g$ , we have 
\[
\|\partial_{x,z}^\alpha   \check{g}(x,z)\|<C_{\alpha}\langle \omega\rangle^{|\alpha|\theta/2}\quad \mbox{ on $\proj_{(x,z)}(U_g)$}
\]
for $\alpha\in \integer^{2d+d'}_+$.
\end{itemize}
\vspace{-1mm}
\end{framed}

\begin{Remark} From the numerical relation (\ref{eq:theta_beta}) given in Remark \ref{rem:num}, 
the conditions above implies that the diffeomorphisms in $\cG_{\omega}$ is close to identity including their derivatives when we look them in the scale $\langle \omega\rangle^{-1/2}$ (or even in a little more larger scale) in the source and target. The exponents in the conditions above are slightly different from those in the corresponding argument in the previous paper \cite[Ch. 7]{FaureTsujii12}. This is because of the involved definition of the partition of unity $\widetilde{X}_{\omega,m}$ introduced in Subsection \ref{ss:modified_pu}. But the difference is not essential.
\end{Remark}
\begin{Remark}
We will see in Corollary \ref{lm:distortion} that we can set up the sets $\cG_\omega$ of diffeomorphisms satisfying the conditions as above so that each of the diffeomorphisms $\kappa_{\j'}^{-1}\circ f^t_G\circ \kappa_{\j}$ with $\omega(\j)\sim \omega(\j')\sim \omega$ and $0\le t\le 2t(\omega)$ is expressed as a composition of a diffeomorphism in  $\cG_\omega$ with some affine maps. 
\end{Remark}

For a pair of a function $\psi\in \cX _\omega$ and a diffeomorphism $g\in \cG_\omega$, 
we consider the transfer operator 
\begin{equation}\label{eq:defLfphi}
L(g,\psi)u=\psi\cdot (u\circ g^{-1}) 
\end{equation}
and also its lift with respect to the partial Bargmann transform
\[
L(g,\psi)^{\lift}=\pBargmann\circ L(g,\psi)\circ \pBargmann^*.
\]
\begin{Remark}
The assumption on the support $U_g$ in Setting II is actually not indispensable for the argument below. In fact, since we will consider the transfer operator as above, it is enough to assume that $U_g$ contains the support of $\psi$.  
\end{Remark}

In the next lemma, we would like to show that the diffeomorphism $g$ in the operator $L(g,\psi)^{\lift}$ will not take much effect in its action and, consequently, $L(g,\psi)^{\lift}$ is well approximated by $L(\mathrm{Id},\psi)^{\lift}=\mult(\psi)^{\lift}$. But, actually, this conclusion is not true if we consider the action of the operator $L(g,\psi)^{\lift}$ on a region far from the trapped set $X_0$. 
In order to restrict the action of the lifted transfer operator $L(g,\psi)^{\lift}$ to a region near the trapped set~$X_0$,  we introduce the $C^\infty$ function 
\begin{align}\label{def:Y}
Y:\real^{4d+2d'+1}_{(w,\xi_w,\xi_z)}\to [0,1],\quad Y(w,\xi_w,\xi_z)&=\chi(\langle \xi_z\rangle^{-2\theta}|(\zeta_p, \tilde{\xi}_y, \zeta_q,\tilde{y}) |)
\end{align}
where $\zeta_p,\zeta_q,\tilde{y},\tilde{\xi}_y$ are the coordinates defined in (\ref{eq:newcoordinates}). 
 
\begin{Lemma}\label{lm:IdMinusG}
Let $\sigma\in \Sigma$. There exist constant $C_\nu>0$ for each $\nu>0$  such that, for any $\omega,\omega'\in \integer$, $\psi\in \cX _\omega$ and $g\in \cG_\omega$, we have
\begin{align*}
&\|\mult(q _{\omega'})\circ 
(L(g,\psi)^{\lift}-\mult(\psi)^{\lift}) \circ  \mult(Y)\circ \mult(q_\omega)
\|_{L^2(\cW^{r,\sigma})}\le C_\nu  \langle \omega\rangle^{-\theta/2} \langle \omega'-\omega\rangle^{-\nu}
\intertext{and}
&\|\mult(q _{\omega'})\circ \mult(Y)\circ
(L(g,\psi)^{\lift}-\mult(\psi)^{\lift}) \circ   \mult(q_\omega)
\|_{L^2(\cW^{r,\sigma})}\le C_\nu  \langle \omega\rangle^{-\theta/2} \langle \omega'-\omega\rangle^{-\nu}.
\end{align*}

\end{Lemma}
\begin{proof} 
Below we prove the former inequality. The latter is proved in a parallel manner. 
For the proof, it is enough to show that the kernel $K(w',\xi'_w,\xi'_z, w,\xi_w,\xi_z)$ of the operator  
\[
\mult(q_{\omega'})\circ(L(g,\psi)-\mult(\psi))^{\lift}\circ \mult(Y)\circ\mult(q_{\omega})
\]
satisfies 
\begin{equation}\label{eq:kernel_Lg}
|K(w',\xi'_w,\xi'_z, w,\xi_w,\xi_z)|\le C_\nu  
\frac{\langle \omega\rangle^{-\theta/2}\cdot  \langle \langle \omega\rangle^{1/2} |w'-w|\rangle^{-\nu} }
{ 
\langle \langle\omega\rangle^{-1/2}|\langle \xi'_z\rangle\xi'_w -\langle \xi_z\rangle\xi_w|\rangle^{\nu} 
\cdot \langle \xi_z-\xi'_z\rangle^{\nu}}
\end{equation}
for arbitrarily large $\nu>0$. Indeed we can deduce the former inequality in the lemma as a consequence of (\ref{eq:kernel_Lg}) by Schur test.
From the definitions, we have
\begin{align}\label{eq:estimate_kernel_MLM}
&|K(w',\xi'_w,\xi'_z;w,\xi_w,\xi_z)|\\
&\quad \le  \left|\int e^{i \varphi(w'',z'';\xi'_w,\xi'_z;\xi_w,\xi_z) }\cdot \Phi(w'',z'';w',\xi'_w,\xi'_z;w,\xi_w,\xi_z) dw'' dz''\right|\notag
\end{align}
where
\[
\varphi(w'',z'';\xi'_w,\xi'_z;\xi_w,\xi_z)=
 (\langle \xi_z\rangle \xi_w- \langle \xi'_z\rangle \xi'_w)\cdot  w''  
+( \xi_z-\xi'_z) z'' 
\]
and
\begin{align*}
&\Phi(w'',z'';w',\xi'_w,\xi'_z;w,\xi_w,\xi_z) \\
&\quad= a_{2d+d'}(\langle\xi'_z\rangle^{-1})\cdot a_{2d+d'}(\langle\xi_z\rangle^{-1})\cdot\psi(w'',z'')\cdot 
e^{-\langle \xi'_z\rangle|w''-w'|^2/2-\langle \xi_z\rangle|w-w''|^2/2}
\\
& \qquad\quad  \times 
\left[-1 +
e^{i \xi_z  \tau(w'')-i\langle \xi_z\rangle\xi_w\cdot  (w''-\breve{g}^{-1}(w''))+
\langle \xi_z\rangle|w-w''|^2/2-\langle \xi_z\rangle|w-\breve{g}^{-1}(w'')|^2/2)}
 \right].
\end{align*}
In the last line, the function $\tau(w)$ and the diffeomorphism $\breve{g}^{-1}$ are those in the expression of the fibered contact diffeomorphism $g^{-1}$,
\[
g^{-1}(w'',z'')=(\breve{g}^{-1}(w''), z''+\tau(x''))\quad \text{for $(w'',z'')=(x'',y'',z'')$}
\]
in Definition \ref{def:fiberedContact}. 
Note that we neglected the multiplication operators $\mult(q_{\omega'})$, $ \mult(Y)$ and $\mult(q_{\omega})$ on the right-hand side of  (\ref{eq:estimate_kernel_MLM}), though we will remember and use the fact that the kernel $K(\cdot)$ vanishes unless $\xi_z\in \supp q_{\omega'}$, $\xi'_z\in \supp q_{\omega}$ and $(w,\xi_w, \xi_z)\in \supp Y$. 

For the proof of the estimate (\ref{eq:kernel_Lg}), we 
 apply integration by parts several times regarding the term $(w'',z'')\mapsto e^{i \varphi(w'',z'';\xi_w,\xi_z;\xi'_w,\xi'_z) }$ as the oscillatory part and using  the differential operators
\begin{equation}\label{eq:D0}
\cD_0=\frac{1-i(\xi'_z-\xi_z)\partial_{z''}}{1+|\xi'_z-\xi_z|^2}
\end{equation}
and $\cD_1$ in (\ref{eq:D1}). In order to get the required estimate (\ref{eq:kernel_Lg}), it is enough to show 
\begin{equation}\label{eq:estimatePhi}
|\partial^\alpha_{w''} \partial_{z''}^k\Phi(w'';w',\xi'_w,\xi'_z;w,\xi_w,\xi_z)|<
\frac{C_{\alpha,k,\nu} \langle \omega\rangle^{-\theta/2+|\alpha|/2}}
{ \langle \langle \omega'\rangle |w'-w''|\rangle^{\nu} \cdot 
\langle \langle \omega\rangle |w-w''|\rangle^{\nu}  }
\end{equation}
for $(w',\xi'_w,\xi'_z)\in \supp q_{\omega'}$, $(w,\xi_w,\xi_z)\in \supp (Y\cdot q_{\omega})$ and $w''\in \proj_{(x,y)}(\supp \psi)$. 

To proceed, it is convenient to consider the two cases in (\ref{cond:near-apart})  separately, as in the proof of Lemma \ref{lm:multiplication}.
In the case (ii), we have (\ref{cos:apart}) and hence each application of integration by parts using $\cD_0$ yields a small factor $\langle \xi'_z-\xi_z\rangle^{-1}< C_0 \langle \omega\rangle^{-1/2}$. 
Hence it is actually enough to prove (\ref{eq:estimatePhi}) with an extra factor $\langle \omega\rangle^m$ for some $m>0$ on the right-hand side. But such estimate can be obtained by plane estimates using the conditions in Setting I and II and also noting (\ref{eq:theta_beta}) for the exponents.

We next consider the case (i) in (\ref{cond:near-apart}), where we need  more precise estimate.  Let us write  $w,w'w''\in \real^{2d+d'}=\real^{2d}\oplus \real^{d'}$ as $w=(x,y)$, $w'=(x',y')$, $w''=(x'',y'')$. 
The condition $w''\in \proj_{(x,y)}(\supp \psi)$ implies that  $|x''|\le C\langle \omega\rangle^{-1/2+\theta}$. Also the condition $(w,\xi_w, \xi_z)\in \supp Y$ implies that $|y|\le C \langle \omega\rangle^{-1/2+2\theta}$. 
Below we assume that $|w|$, $|w'|$, $|w''|$ are bounded by $C\langle \omega\rangle^{-1/2+2\theta}$ for some large $C>0$ because, otherwise, the factor $e^{-\langle \xi'_z\rangle|w''-w'|^2/2-\langle \xi_z\rangle|w-w''|^2/2}$ is bounded by $C_\nu \langle \omega\rangle^{-\nu}$ for arbitrarily large $\nu$ and we can get  (\ref{eq:estimatePhi}) by easy estimate as in the case (ii). 
Then, under such an assumption, the condition $(w,\xi_w,\xi_z)\in \supp (Y\cdot q_{\omega})$ implies  $|\xi_w|<C\langle \omega\rangle^{-1/2+2\theta}$. 

Let us write $\Phi_0(w'';w',\xi'_w,\xi'_z;w,\xi_w,\xi_z)$ for the term in the square bracket $[\cdot]$ in the expression of $\Phi(\cdot)$ above. The required estimate (\ref{eq:estimatePhi}) follows immediately if we show
\begin{equation}\label{eq:estimatePhi0}
|\partial^\alpha_{w''} \Phi_0(w'';w',\xi'_w,\xi'_z;w,\xi_w,\xi_z)|<C_{\alpha} \langle \omega\rangle^{-\theta/2+|\alpha|/2}.
\end{equation} 
From the condition (G2) in Setting II, we have $
|\partial_{w''}^\alpha\tau(x'')|\le C_{\alpha} \langle \omega\rangle^{|\alpha|\theta/2}$ for any multi-index $\alpha$. 
Hence, from Lemma \ref{lm:tau} and Taylor theorem, we get  
\[
|\partial_{w''}^\alpha\tau(x'')|\le C \langle \omega\rangle^{-(3-|\alpha|)(1/2-\theta)+(3/2)\theta} \quad \mbox{if $0\le |\alpha|\le 2$ }
\]
provided $|x''|\le C\langle \omega\rangle^{1/2-\theta}$. 
Also, by Taylor theorem for $\breve{g}^{-1}$ at the origin and using the conditions in Setting II, we have
\begin{align*}
|\breve{g}^{-1}(w'')-w''|&\le C 
\langle \omega\rangle^{-1/2-3\theta}
+C\langle \omega\rangle^{\max\{-\beta(1/2-\theta), -(1-\beta)(1/2-\theta)-2\theta\}}\langle \omega\rangle^{-1/2+2\theta}\\
&\quad +
C\langle \omega\rangle^{((1-\beta)(1/2-\theta)+4\theta)+\theta}
\langle \omega\rangle^{2(-1/2+2\theta)}\\
&\le C'\langle \omega\rangle^{-1/2-3\theta}\qquad\text{by (\ref{eq:choice_of_theta})}.
\end{align*}
The last estimate gives for instance  
\[
||w-w''|^2/2-|w-\breve{g}^{-1}(w'')|^2/2|\le C \langle \omega\rangle^{-1/2+2\theta}\langle \omega\rangle^{-1/2-3\theta}=C\langle \omega\rangle^{-1-\theta}
\]
and further, together with the conditions (G1) and (G2) in Setting II,  
\[
\left|\partial^{\alpha}_{w''}\left(|w-w''|^2/2-|w-\breve{g}^{-1}(w'')|^2/2\right)\right|\le C_\alpha \langle \omega\rangle^{-1+|\alpha|/2-\theta}.
\]
It is now straightforward to show the claim (\ref{eq:estimatePhi0}) by using the estimates above, (\ref{eq:cos:near}) and the conditions in Setting II.
\end{proof}

\subsection{The projection operator $\cT_0$ and its lift}
We next consider the lift $\cT^\lift_0$ of the projection operator $\cT_0$ defined in (\ref{eq:TzeroLift}). Recall from Corollary \ref{cor:Tzero} that the kernel of the operator
$\cT_0^{\lift}$ concentrates around the trapped set $X_0$ if we view it through the weight $\cW^{r,\sigma}$. 
The following two lemmas are direct consequences of this fact and Lemma \ref{lm:multiplication}.
We omit the proofs since they are straightforward. (Similar statements and their proofs can be found in \cite[Lemma 5.1.6, Lemma 5.3.1]{FaureTsujii12}.)
\begin{Lemma}\label{lm:mphi} Let $\sigma,\sigma'\in \Sigma$.
There is a constant $C_\nu>0$ for each $\nu>0$ such that 
\[
\|\mult(q_{\omega'})\circ [\mult(\psi)^{\lift}, \cT^\lift_0]\circ  \mult(q_{\omega})
\|_{L^2(\cW^{r,\sigma})\to L^2(\cW^{r,\sigma'})}\le C_\nu \langle \omega\rangle^{-\theta/2} \langle \omega'-\omega\rangle^{-\nu}
\]
for $\omega, \omega'\in \integer$ and $\psi\in \cX_\omega$. Here $[A,B]$ denotes the commutator of two operators $A$ and $B$, $[A,B]=A\circ B-B\circ A$.
\end{Lemma}

\begin{Lemma}\label{lm:TM} Let $\sigma,\sigma'\in \Sigma$. There is a constant $C_\nu>0$ for each $\nu>0$ such that 
\[
\|\mult(q_{\omega'})\circ \mult(1-Y)\circ  {\cT}^\lift_0 \circ \mult(q_{\omega})\|_{L^2(\cW^{r,\sigma})\to L^2(\cW^{r,\sigma'})}\le C\langle \omega\rangle^{-\theta} \langle \omega'-\omega\rangle^{-\nu}
\]
and
\[
\|\mult(q_{\omega'})\circ {\cT}^\lift_0\circ \mult(1-Y)\circ \mult(q_{\omega})\|_{L^2(\cW^{r,\sigma})\to L^2(\cW^{r,\sigma'})}\le C\langle \omega\rangle^{-\theta} \langle \omega'-\omega\rangle^{-\nu}
\]
for $\omega, \omega'\in \integer$ and $\psi\in \cX_\omega$.
\end{Lemma}
We prepare the next elementary lemma for the proof of Lemma \ref{lm:trace1}(b).
\begin{Lemma}\label{lm:choiceV} 
Let $\sigma,\sigma'\in \Sigma$.
For any $\epsilon>0$, there exist a constant $c>0$ and $\omega_0>0$ such that, for  $\omega\in \integer$ with $|\omega|\ge \omega_0$, we can find  a finite dimensional subspace
\begin{equation}\label{eq:VclaimSupport}
W(\omega)\subset C^{\infty}(\{(w,z)\in \real^{2d+d'+1}\mid |w|\le \epsilon\langle \omega\rangle^{-1/2+\theta}, 
|z|\le \epsilon\})
\end{equation}
with  
\begin{equation}\label{eq:Vclaimdim}
\dim W(\omega)\ge c \langle \omega\rangle^{2 d\theta}
\end{equation}
such that 
\begin{align}\label{eq:Vclaim}
&\|\mult(X_{n_0(\omega)})\circ \cT^\lift_0\circ \mult(\Psi^{\sigma}_{\omega,0})\circ \pBargmann v\|_{L^2(\cW^{r,\sigma'})}\ge c  
\| \mult(\Psi^{\sigma}_{\omega',0})\circ \pBargmann v\|_{L^2(\cW^{r,\sigma})} 
\end{align}
for all $v\in W(\omega)$.
\end{Lemma}
\begin{proof} 
Let $N>0$ be a large integer and take the lattice points
\[
S=\{ \mathbf{n}\in \real^{2d}\mid \mathbf{n}\in N \langle \omega\rangle^{-1/2} \integer^{2d}, \; |\mathbf{n}|<(\epsilon/2)\langle \omega\rangle^{-1/2+\theta} \}.
\]
Clearly we have $
C^{-1}\langle \omega\rangle^{ 2d\theta}
\le \# S \le C \langle \omega\rangle^{2d\theta}
$
for a constant $C>0$ which  depends on the choice of $N$ but not on~$\omega$. 
For each $\mathbf{n}\in S$, 
we let $\xi(\mathbf{n})=(\xi_w,\xi_z)\in \real^{2d+d'+1}$ be the unique point such that $\xi_z=\omega$ and that $((\mathbf{n},0),\xi(\mathbf{n}))$ belongs to the trapped set $X_0$, where $(\mathbf{n},0)\in \real^{2d+d'}_{(x,y)}$.
Then we define  $v_{\mathbf{n}}\in C^{\infty}_0(\real^{2d+d'+1})$ by 
\[
v_{\mathbf{n}}(w,z)=\chi(2 \epsilon^{-1} |z|) \cdot 
\chi\left(2\epsilon^{-1}\langle \omega\rangle^{1/2-\theta}|w|\right) \cdot \phi_{(\mathbf{n},0),\xi(\mathbf{n})}(w,z)
\]
where $\chi(\cdot)$ is the function defined in (\ref{eq:def_chi}). 
By definition, the partial Bargmann transform $\pBargmann v_{\mathbf{n}}$ of  $v_{\mathbf{n}}$ concentrates around the point $((\mathbf{n},0),\xi(\mathbf{n}))\in X_0$ in the scale $\langle \omega\rangle^{-1/2}$ in the variable $(w,\xi_w)$ and the unit scale in the variable $\xi_z$. Hence, in the coordinates (\ref{eq:newcoordinates})  introduced in Subsection \ref{ss:coord}, it concentrates around the point
\begin{equation*}
\Xi_{\mathbf{n}}:=(\zeta,\nu,\tilde{y},\tilde{\xi}_y)=(0,2^{1/2}\langle\omega\rangle^{1/2}\mathbf{n},0,0)
\end{equation*}
in the unit scale. Note that we have
\begin{equation}\label{eq:distXi}
d(\Xi_{\mathbf{n}}, \Xi_{\mathbf{n}'})=2^{1/2}\langle\omega\rangle^{1/2}\cdot d(\mathbf{n},\mathbf{n}')\ge 2^{1/2}N
\end{equation}
for different points $\mathbf{n}\neq\mathbf{n}'\in S$. 

Let $W(\omega)$ be the linear space spanned by the functions $v_{\mathbf{n}}$ for ${\mathbf{n}}\in S$. Then the claim (\ref{eq:VclaimSupport}) holds by definition. It is easy to check that the functions $v_{\mathbf{n}}$ for $\mathbf{n}\in S$ are almost orthogonal in the sense that
\[
|(v_{\mathbf{n}}, v_{\mathbf{n}'})_{\cH^{r,\sigma}}|\le 
C_{\nu} \cdot d(\Xi_{\mathbf{n}}, \Xi_{\mathbf{n}'})^{-\nu}\quad \text{for $\mathbf{n}\neq\mathbf{n}'\in S$}
\]
with arbitrarily large $\nu$. Hence we obtain (\ref{eq:Vclaimdim}):
\[
\dim W(\omega)\ge \# S \ge c \langle \omega\rangle^{2 d\theta}
\]
provided that we take sufficiently large constant $N$. 
From the description (\ref{eq:Tlift2}) of the operator $\cT^\lift_0$, we see that the claim (\ref{eq:Vclaim}) holds for each $v=v_{\mathbf{n}}$ with some small constant $c>0$ uniform for ${\mathbf{n}}\in S$ and $\omega$ with $|\omega|\ge \omega_0$. 
But, since the kernel of the operator $\cT^\lift_0$ is localized as described in Corollary \ref{eq:Tlift2} and since we have (\ref{eq:distXi}), we can extend the estimate to the linear combinations of them, again provided that we take sufficiently large $N>0$.  
\end{proof}

\section{Components of lifted operators}
\label{sec:lifted}
\subsection{The lifted transfer operators and their components}
\label{ss:lifted}
The semi-group of (scalar-valued) transfer operators
\[
\cL^t=\cL^t_{0,0}:C^\infty(K_0)\to C^\infty(K_0)\quad \mbox{for }t\ge 0
\]
induces the family of lifted operators:
\begin{equation}\label{eq:def_bLt}
\bL^{t,\sigma\to \sigma'}:=\bI^{\sigma'}\circ \cL^t\circ (\bI^\sigma)^*:\bigoplus_{\j\in \cJ}\bK_\j^{r,\sigma}\to \prod_{\j\in \cJ}\bK_\j^{r,\sigma'},\qquad \sigma,\sigma'\in \Sigma.
\end{equation}
As we will see, these  operators extend to bounded operators
$
\bL^{t,\sigma\to \sigma'}:\bK^{r,\sigma}\to \bK^{r,\sigma'}$ 
provided $t\ge 0$ satisfies (\ref{eq:condition_on_t}) with respect to $\sigma$ and $\sigma'$. Hence the operators $\cL^t:\cK^{r,\sigma}(K_0)\to \cK^{r,\sigma'}(K_0)$ are also bounded and  the diagram  
\begin{equation}\label{cd:bL}
\begin{CD}
\bK^{r,\sigma}@>{\bL^{t,\sigma\to \sigma'}}>>\bK^{r,\sigma'}\\
@A{\bI^\sigma}AA@A{\bI^{\sigma'}}AA\\
\cK^{r,\sigma}(K_0)@>{\cL^t}>>\cK^{r,\sigma'}(K_0)
\end{CD}
\end{equation}
commutes where  $\bI^\sigma:\cK^{r,\sigma}(K_0)\to \bK^{r,\sigma}$ is the natural   isometric embedding. 

The lifted transfer operator $\bL^{t,\sigma\to \sigma'}$ in (\ref{eq:def_bLt}) is expressed  as an infinite matrix of operators:
\begin{equation}
\bL^{t,\sigma\to \sigma'}\bu =
\left(
\sum_{\j\in \cJ} \bL_{\j\to \j'}^{t,\sigma\to \sigma'} u_{\j}
\right)_{\j'\in \cJ} \quad \mbox{for $\bu=(u_{\j})_{\j\in \cJ}$.}
\end{equation}
The components $\bL_{\j\to \j'}^{t,\sigma\to\sigma'}:L^2(\supp \Psi^{\sigma}_{\j})\to L^2(\supp \Psi^{\sigma'}_{\j'})
$ are the operators written in the form 
\begin{equation}\label{eq:ljj}
\bL_{\j\to \j'}^{t,\sigma\to \sigma'}= \mult( \Psi_{\j'}^{\sigma'})\circ \pBargmann \circ L(\, f^t_{\j\to \j'},  \,\tilde{b}^t_{\j'}\cdot \rho^t_{\j\to \j'}\,)\circ \pBargmann^*
\end{equation}
where $L(g,\psi)$ denotes the transfer operator defined by (\ref{eq:defLfphi}).
The diffeomorphism $f^t_{\j\to \j'}$ and the function $\rho^t_{\j\to \j'}$ in (\ref{eq:ljj}) are defined by 
\begin{align}\label{def:ftjj}
f^t_{\j\to \j'}&:=\kappa_{\j'}^{-1}\circ f^t_{G}\circ \kappa_{\j}
\intertext{and}
 \rho^t_{\j\to \j'}&:=
\rho_{\j'}\cdot (\tilde{\rho}_{\j}\circ (f^t_{\j\to \j'})^{-1}).
\label{eq:rhot}
\end{align}
(See (\ref{eq:defsrho}) and also (\ref{def:kappa_anw}), (\ref{eq:rhoan}) for the definitions of $\kappa_{\j}$ and $\rho_{\j}$, $\tilde{\rho}_{\j}$.)
The function $\tilde{b}^t_{\j'}(\cdot)$ in (\ref{eq:ljj}) is defined from $b^t(\cdot)$ in (\ref{eq:bt}) by
\begin{align}\label{eq:tbt}
\tilde{b}^t_{\j'}(w,z)&=b^t(f^{-t}_{G}\circ \kappa_{\j'}(w,z))\quad \mbox{for $(w,z)\in \supp \rho^t_{\j\to \j'}$.}
\end{align} 
\begin{Remark}\label{rem:difference}
If we replace  $\cL^t$ by $\cL^t\circ \mult(\rho_{K_i})$, $i=0,1$, in the definitions above, we obtain slightly different operators. But the difference is only that  
 $\tilde{\rho}_{\j}$ is replaced by $\tilde{\rho}_{\j}\cdot (\rho_{K_i}\circ \kappa_{\j})$, $i=0,1$, and hardly affects the validity of the argument below. In this section and Section \ref{sec:proofs_of_props}, we suppose that 
we are discussing about these variants in parallel.  
\end{Remark}

From the definition of the partial Bargmann transform and its adjoint,  given in (\ref{eq:volume}) and (\ref{eq:pBargmannAdj}), the operator $\bL^{t,\sigma\to \sigma'}_{\j\to \j'}$ is written as an integral operator
\begin{equation*}
\bL_{\j\to \j'}^{t,\sigma\to \sigma'}u(w',\xi'_w,\xi'_z)=\int K^{t,\sigma\to \sigma'}_{\j\to \j'}(w',\xi'_w,\xi'_z;w,\xi_w,\xi_z) u(w,\xi_w,\xi_z)
d\vol(w,\xi_w,\xi_z)
\end{equation*}
where $d\vol$ is the volume form given in (\ref{eq:volume}). 
The kernel is expressed as the integral  
\begin{align}
\label{eq:bigK}
&K^{t,\sigma\to \sigma'}_{\j\to \j'}(w',\xi'_w,\xi'_z;w,\xi_w,\xi_z)\\
&\qquad =\Psi^{\sigma'}_{\j'}(w',\xi'_w,\xi'_z)\cdot\int k^{t}_{\j\to \j'}(w'', z'';w',\xi'_w,\xi'_z;w,\xi_w,\xi_z) dw'' dz''\notag
\end{align}
where
\begin{align*}
&k^{t}_{\j\to \j'}(w'', z'';w',\xi'_w,\xi'_z;w,\xi_w,\xi_z)\\
&\qquad = (\tilde{b}^t_{\j'}\cdot   \rho^t_{\j\to \j'})(w'',z'') \cdot \overline{\phi_{w',\xi'_w,\xi'_z}(w'',z'')}\cdot  \phi_{w,\xi_w,\xi_z}((f^{t}_{\j\to \j'})^{-1}(w'',z'')).
\end{align*}
For further argument, it is convenient to write the last function $k^{t}_{\j\to \j'}(\cdot)$ in the form 
\begin{align}\label{eq:Kt_expression}
k^t_{\j\to \j'}(&w'', z'';w',\xi'_w,\xi'_z;w,\xi_w,\xi_z)\\
&\;\;= e^{i \varphi(w'',z'';w',\xi'_w,\xi'_z;w,\xi_w,\xi_z) }\cdot \Phi(w'',z'';w',\xi'_w,\xi'_z;w,\xi_w,\xi_z)\notag
\end{align}
with setting 
\[
\varphi(w'',z'';w',\xi'_w,\xi'_z;w,\xi_w,\xi_z)= 
-\langle \xi'_z\rangle  \xi'_w \cdot w''  + \langle \xi_z\rangle \xi_w\cdot  (\breve{f}^{t}_{\j\to \j'})^{-1}(w'')  
-(\xi'_z-\xi_z) z''
\]
and
\begin{align}
&\Phi(w'',z'';w',\xi'_w,\xi'_z;w,\xi_w,\xi_z)\label{eq:Ksp}\\
&\;\quad=a_{2d+d'}(\langle\xi_z\rangle^{-1})\, a_{2d+d'}(\langle\xi'_z\rangle^{-1})\cdot 
\exp(i\langle \xi'_z\rangle \xi'_w \cdot w'/2-i\langle \xi_z\rangle \xi_w\cdot w/2)
\notag \\
&\qquad\quad  \times (\tilde{b}^t_{\j'}\cdot   \rho^t_{\j\to \j'})(w'',z'')\cdot \exp(i \xi_z  \tau_{\j\to \j'}(x''))\notag\\
&\qquad \quad \times 
\exp(-\langle \xi'_z\rangle|w''-w'|^2/2-\langle \xi_z\rangle|(\breve{f}^{t}_{\j\to \j'})^{-1}(w'')-w|^2/2
), \notag
\end{align}
where $x''\in \real^{2d}$ is the first component of $w''=(x'',y'')\in \real^{2d+d'}$ and $\tau_{\j\to \j'}$ and $(\breve{f}^{t}_{\j\to \j'})^{-1}$ are those in the expression  
\[
({f}^{t}_{\j\to \j'})^{-1}(w'',z'')=((\breve{f}^{t}_{\j\to \j'})^{-1}(w''), z''+\tau_{\j\to\j'}(x''))
\]
of the fibered contact diffeomorphism $(f^t_{\j\to \j'})^{-1}$ corresponding to  (\ref{eq:fdiff}). 

We may then regard the integral (\ref{eq:bigK}) as an oscillatory integral and expect that it becomes small if the  term $e^{i \varphi(\cdot) }$ oscillates fast with respect to $w''$ or $z''$. Though we will give precise statements later, it is reasonable to make the following observations for the kernel $K^{t,\sigma\to \sigma'}_{\j\to \j'}(w',\xi'_w,\xi'_z;w,\xi_w,\xi_z)$ at this moment:
\begin{itemize} 
\item[(Ob1)] It decays rapidly as the distance $|\xi'_z-\xi_z|$ gets large. This is because the term $e^{i\varphi(\cdot)}$ oscillates fast with respect to $z''$ (while the other terms do not). 
\item[(Ob2)] It decays rapidly as the distance $|\breve{f}^t_{\j\to \j'}(w)-w'|$ gets large in the scale $\langle\omega(\j)\rangle^{-1/2}\sim \langle \xi_z\rangle^{-1/2}$ because so is the last term of (\ref{eq:Ksp}). 
\item[(Ob3)] It  decays rapidly as the distance between $\langle \xi'_z\rangle \xi'_w$ and 
$\langle \xi_z\rangle ((D\breve{f}^t_{\j\to \j'})^*_{w''})^{-1}\xi_w$ gets large (uniformly for  $w''\in \supp \rho^{t}_{\j\to \j'}$) in the scale $\langle\omega(\j)\rangle^{1/2}\sim \langle \xi_z\rangle^{1/2}$. This is because the term  $e^{i\varphi(\cdot)}$ oscillates fast with respect to  $w''$. 
\end{itemize}
\begin{Remark}
Intuitively, the observations above implies that the operator $\bL^{t,\sigma\to \sigma'}_{\j\to \j'}$ is localized in ``energy (or the frequency in $z$)", ``position" and ``momentum" and 
 that the transport of wave packets induced by $\bL^{t,\sigma\to \sigma'}_{\j\to \j}$ is described by the canonical map $((Df^t_{\j\to \j'})^*)^{-1}$.  (Recall the discussion at the end of Section \ref{sec:Gext}.)
\end{Remark}

\subsection{Distortion estimates on $f^t_{\j\to \j'}$} 
We give some estimates on the differentials of diffeomorphisms $f^t_{\j\to \j'}$. The estimates are quite elementary, but may not be completely obvious. Note that the main point in the estimates below is  the effect of the factor $E_{\omega(\j)}$ in the definition of the local charts $\kappa_{\j}$.  
We henceforth consider only those diffeomorphisms $f^t_{\j\to \j'}$ for which the function $\rho^{t}_{\j\to \j'}$ does not vanish. This is of course enough for our consideration on the operators $\bL^{t,\sigma\to \sigma'}_{\j\to \j'}$.

From the definitions in Subsection \ref{ss:chart_omega}, the diffeomorphism $f^t_{\j\to \j'}$ is a fibered contact diffeomorphism and  is written in the form 
\begin{equation}\label{eq:EhE}
f^t_{\j\to \j'}=E_{\omega(\j')}^{-1}\circ h^t_{\j\to \j'}\circ E_{\omega(\j)}
\end{equation}
where 
\begin{equation}\label{eq:htj}
h^t_{\j\to \j'}:=\kappa_{a(\j'),n(\j')}^{-1}\circ f^t_G\circ \kappa_{a(\j),n(\j)}.
\end{equation}
The local charts $\kappa_{a,n}$ are composition of the coordinate charts $\kappa_a$ and bijective affine maps whose derivatives and their inverses are bounded uniformly. Hence we have, for any $\j,\j'\in \cJ$ and any multi-index $\alpha$, that 
\[
|\partial_{w}^\alpha h^t_{\j\to \j'}(w,z)|\le C_{\alpha} e^{|\alpha| \chi_{\max} |t|}
\]
for $t\in \real$ and $(w,z)\in \real^{2d+d'+1}_{(w,z)}$ at which $h^t_{\j\to \j'}(w,z)$ is defined.
In particular, we have the estimate
\begin{equation}\label{eq:dalphah}
|\partial_{w}^\alpha h^t_{\j\to \j'}(w,z)|\le C_{\alpha} \langle \omega\rangle^{\theta |\alpha|/4} \quad \mbox{for $0\le t\le 2t(\omega)$}
\end{equation}
if we let the constant $\epsilon_0$ in the definition of $t(\omega)$ be small. 
 
Recall that there was some arbitrariness in the choice of the local coordinate charts $\kappa_a:U_a\to V_a$ and the associated family of functions $\rho_a$ in Subsection \ref{ss:Darboux}.   
By modifying them if necessary (so that the supports of the functions $\rho_a\circ \kappa_a$ become smaller), 
we may assume the following estimate on the diffeomorphism $h^t_{\j\to \j'}$  for \emph{uniformly bounded time}, that is, for $0\le t\le 2t_0$. 
\begin{Lemma}\label{lm:crude_estimates}  For  $0\le t\le 2t_0$, $\j,\j'\in \cJ$ and $(w,z)\in E_{\omega(\j)}(\supp \rho_{\j})$, we have  
\begin{align*}
\|(Dh^t_{\j\to \j'})_{(w,z)}\circ B^{-1}-\mathrm{Id}\|<10^{-2} 
\end{align*}
with some linear map of the form
\begin{equation}\label{eq:B}
B:\real^{2d+d'+1}_{(x,y,z)}\to \real^{2d+d'+1}_{(x,y,z)},\qquad
B(q,p, y,z)=(Aq, A^\dag p, \wA y,z).
\end{equation}
depending on $t$, $\j$, $\j'$ and $(w,z)$, where  
\begin{equation}\label{eq:normA}
e^{\chi_0 t}\le \|A^{-1}\|^{-1}\le \|A\|\le e^{\chi_{\max} t},
\qquad
 e^{-\chi_{\max} t}\le \|\widehat{A}^{-1}\|^{-1}\le \|\widehat{A}\|\le e^{-\chi_{0} t}
\end{equation}
and $A^\dag={}^t A^{-1}$ as we defined in (\ref{eq:Adag}). (Recall (\ref{eq:chi_max}) for the choice of $\chi_{\max}$.)
\end{Lemma}
\begin{proof}
Let us recall the construction of the local chart $\kappa_a:V_a\to U_a$ in Lemma \ref{th:Darboux} and that of $\kappa_{a,x}:A^{-1}_{a,x}\to U_a$ in Proposition \ref{pp:kappaaw}. 
Since we are assuming $0\le t\le 2t_0$, we may let (the diameter of) $V_a$ be small so that the conclusion of the lemma holds true with $1/100$ replaced by $1/200$ and the matrix $B$ by      
\[
B'(q,p, y,z)=(Aq, A^\dag p, \wA y+\widetilde{B}(q,p),z)
\]
where $\widetilde B:\real^{2d}_{(q,p)}\to \real^{d'}_y$ is bounded uniformly for $0\le t\le t_0$, $\j,\j'\in \cJ$ and $(w,z)\in E_{\omega(\j)}(\supp \rho_{\j})$. 
In order to suppress the extra term $\widetilde B$, we further modify the local chart $\kappa_a$ and $\kappa_{a,x}$ by pre-composing the linear map $(x,y,z)\mapsto (x,\eta y,z)$. Then the extra term $\widetilde B$ becomes $\eta^{-1}\widetilde{B}$ while the other conditions remain valid. Therefore, letting $\eta>1$ large and incorporating $\eta^{-1}\widetilde{B}$ in the error term, we obtain the conclusion of the lemma. 
\end{proof}

Let us recall the affine transformation groups $\cA_2\subset \cA_1\subset \cA_0$ in Definition \ref{def:transG}. 
For each $f^t_{\j\to \j'}$, we take and fix an element  $a^t_{\j\to \j'} \in \cA_2$ whose inverse carries the point $f^t_{\j\to \j'}(0,0,0)=(x_*,y_*,z_*)\in \real^{2d+d'+1}_{(x,y,z)}$ to $(0,y_*,0)$, so that 
\begin{equation}\label{eq:tildef}
\tilde{f}^t_{\j\to \j'}:=(a^t_{\j\to \j'})^{-1}\circ f^{t}_{\j\to \j'}\quad \mbox{satisfies }\tilde{f}^t_{\j\to \j'}(0,0,0)=(0,y_*,0).
\end{equation}
Recall, from Lemma \ref{lm:unitaryaction}, that the transfer operator associated to $a^t_{\j\to \j'}$ is a unitary operator on $\cH^{r,\sigma}(\real^{2d+d'+1})$ and therefore basically negligible.  
\begin{Lemma}\label{lm:crude_estimates2} 
There exist constants $C>0$ and $C_{\alpha}>0$ for each multi-index $\alpha\in \integer_+^{2d+d'}$ such that, for $\j,\j'\in \cJ$ and $0\le t\le 2t(\omega(\j))$, the following hold true:
\begin{enumerate}
\setlength{\leftskip}{-6mm}
\item For $(x_*,y_*,z_*)={f}^t_{\j\to \j'}(0,0,0)$, we have
\[
|x_*|<C\langle \omega(\j')\rangle^{-(1/2-\theta)},\quad 
|y_*|<C\langle \omega(\j')\rangle^{-(1/2+3\theta)},\quad |z_*|<C.
\]
\item The first derivative of $\tilde{f}^t_{\j\to \j'}:=(a^t_{\j\to \j'})^{-1}\circ f^{t}_{\j\to \j'}$ defined in (\ref{eq:tildef}) above at the origin $0\in \real^{2d+d'+1}$ is written in the form
\[
(D\tilde{f}^t_{\j\to \j'})_0=\begin{pmatrix}a_{1,1}&0&0\\
a_{2,1}&a_{2,2}&0\\
0&0&1
\end{pmatrix}\;:\;\real^{2d}_x\oplus\real^{d'}_y\oplus \real_z \to \real^{2d}_x\oplus\real^{d'}_y\oplus \real_z
\]
and the entries satisfy 
\begin{align*}
&\|a_{1,1}\|\le \langle \omega(\j)\rangle^{\theta/4},\\
&\|a_{2,1}\|\le  \langle \omega(\j)\rangle^{\theta/4} \langle \omega(\j')\rangle^{-(1-\beta)(1/2-\theta)-4\theta}, \mbox{ and }\\  
&\|a_{2,2}\|\le  e^{-\chi_0 t}\cdot  (\langle \omega(\j)\rangle/\langle \omega(\j')\rangle)^{(1-\beta)(1/2-\theta)+4\theta}.
\end{align*}
\item For any multi-index $\alpha\in \integer_{+}^{2d+d'+1}$ with $|\alpha|\ge 2$, we have 
\[
\|\partial_w^{\alpha} \tilde{f}^t_{\j\to \j'}\|_\infty <C_{\alpha} 
\langle \omega(\j)\rangle^{\theta|\alpha|/4}\cdot 
\left\langle \langle \omega(\j)\rangle^{|\alpha|}/\langle \omega(\j')\rangle\right\rangle^{(1-\beta)(1/2-\theta)+4\theta}.
\]
Also its base diffeomorphism $(\check{a}^t_{\j\to \j'})^{-1}\circ \check{f}^t_{\j\to \j'}$ satisfies
\[
\|\partial_{x,z}^{\beta}((\check{a}^t_{\j\to \j'})^{-1}\circ \check{f}^t_{\j\to \j'})\|<C_\beta \langle \omega(\j)\rangle^{\theta|\alpha|/4}
\]
for any multi-index $\beta\in\integer_{+}^{2d+1}$ with $|\beta|\ge 2$.
\end{enumerate}
\end{Lemma}
\begin{Remark}\label{Rem:derivz}
For the derivatives of $f^t_{\j\to \j'}$ and $\tilde{f}^t_{\j\to \j'}$ with respect to $z$, we always have 
$
\partial_z f^t_{\j\to \j'}(x,y,z)=\partial_z \tilde{f}^t_{\j\to \j'}(x,y,z)\equiv (0,0,1)$ 
because $f^t_{\j\to \j'}$ and $\tilde{f}^t_{\j\to \j'}$ are fibered contact diffeomorphisms. 
\end{Remark}
\begin{proof}
Since $h^t_{\j\to \j'}$ defined in (\ref{eq:htj}) is a fibered contact diffeomorphism, its  Taylor expansion up to the first order at the origin is of the form
\[
h^t_{\j\to \j'}\begin{pmatrix}x\\y\\z\end{pmatrix}=\begin{pmatrix}x_0\\y_0\\z_0\end{pmatrix}+
\begin{pmatrix}A_{1,1}&0&0\\
A_{2,1}&A_{2,2}&0\\
A_{3,1}&0&1
\end{pmatrix}
\begin{pmatrix}x\\y\\z\end{pmatrix}
+
\begin{pmatrix}h_1(x)\\h_2(x,y)\\h_3(x)\end{pmatrix}.
\]
We have the following estimates when $0\le t\le 2t(\omega(\j))$:
\begin{enumerate}
\setlength{\leftskip}{-6mm}
\setlength{\itemsep}{1mm}
\item $|x_0|\le C\langle \omega(\j')\rangle^{-(1/2-\theta)}$, 
$ |y_0|\le C\langle \omega(\j')\rangle^{-\beta(1/2-\theta)}$ and $|z_0|\le C$.
\item $\|A_{1,1}\|\le \langle \omega(\j)\rangle^{\theta/4}$, $\|A_{2,1}\|\le \langle \omega(\j)\rangle^{\theta/4}$, 
 $\|A_{2,2}\|\le e^{-\chi_0 t}$, and\footnote{We do not need the estimate on $\|A_{3,1}\|$ as it is determined by $x_0$ and $A_{1,1}$. } 
\item $\|\partial_w^{\alpha} h^t_{\j\to \j'}\|_\infty\le C_{\alpha}\langle \omega(\j)\rangle^{\theta |\alpha|/4} $ for $\alpha$ with $|\alpha|\ge 2$, from (\ref{eq:dalphah}), 
\end{enumerate}
provided that we take sufficiently small constant $\epsilon_0>0$ in the definition of $t(\omega)$. 
\begin{Remark} For the second estimate on $|y_0|$ in (1) above, recall that the origin of the local chart $\kappa_{\j}$ corresponds to a point on the section $e^u$ which is $\beta$-H\"older continuous. 
\end{Remark}
From the relation (\ref{eq:EhE}) and the choice of $a^t_{\j\to \j'}$, we see that 
the diffeomorphism $\tilde{f}^t_{\j\to \j'}$ is expressed as 
\begin{align*}
\tilde{f}^t_{\j\to \j'}\begin{pmatrix}x\\y\\z\end{pmatrix}=&\begin{pmatrix}0\\ 
\langle \omega(\j')\rangle^{-\Theta} \cdot  y_0
\\0\end{pmatrix}\\
&\quad +
\begin{pmatrix}A_{1,1}&0&0\\
\langle \omega(\j')\rangle^{-\Theta}  A_{2,1}&\langle \omega(\j)\rangle^{\Theta}\cdot \langle \omega(\j')\rangle^{-\Theta} A_{2,2}&0\\
0&0&1
\end{pmatrix}
\begin{pmatrix}x\\y\\z\end{pmatrix}\\
&\quad 
+
\begin{pmatrix}g_1(x)\\ \langle \omega(\j')\rangle^{-\Theta} \cdot  g_2(x, \langle \omega(\j)\rangle^{\Theta}\cdot y)\\g_3(x)\end{pmatrix}
\end{align*}
where we put $\Theta=(1-\beta)(1/2-\theta)+4\theta$ for brevity  and  the functions $g_i(\cdot)$ for $i=1,2,3$ are those obtained from $h_i(\cdot)$ by changing  the variable $x$ by a translation. 
The required estimates follow immediately from this expression.
\end{proof}

In the case where  $\omega(\j)$ and $\omega(\j')$ are relatively close to each other, we have
\begin{Corollary}\label{lm:distortion}
Suppose that $\j, \j'\in \cJ$ satisfies the condition 
\begin{equation}
\label{eq:close}
|\langle \omega\rangle-\langle\omega'\rangle |\le \langle \omega\rangle^{1/2}
\end{equation} with $\omega=\omega(\j)$ and $\omega'=\omega(\j')$. Then the diffeomorphism $f^t_{\j\to \j'}$ for $0\le t\le 2t(\omega(\j))$ is expressed as the composition 
\begin{equation}\label{eq:agb}
f^t_{\j\to \j'}=a^t_{\j\to \j'}\circ g^t_{\j\to \j'}\circ B^t_{\j\to \j'}
\end{equation}
where 
\begin{enumerate}
\setlength{\leftskip}{-6mm}
\setlength{\itemsep}{1mm}
\item $a^t_{\j\to \j'}\in\cA_2$ is the affine transform that we chose just before Lemma \ref{lm:crude_estimates2}, 
\item $B^t_{\j\to \j'}$ is a linear map of the form (\ref{eq:B}) (or (\ref{eq:Bt}) with $t=0$) with the linear maps $A:\real^{d}\to \real^d$ and $\widehat{A}:\real^{d'}\to \real^{d'}$ satisfying (\ref{eq:normA}),
\item $g^t_{\j\to \j'}$ is a fibered contact diffeomorphism such that the family \[
\cG_\omega=\{g^t_{\j\to \j'}\mid\;\omega(\j)=  \omega, \mbox{ $\omega'=\omega(\j')$ satisfies (\ref{eq:close}), and  } 0\le  t\le 2t(\omega)\}
\] 
fulfills the conditions (G0), (G1) and (G2) in Setting II.
\end{enumerate}
\end{Corollary}
\begin{proof}
From the choice of the local coordinate charts $\kappa_{\j}=\kappa^{(\omega(\j))}_{a(\j), n(\j)}$, we see that the linear map $a_{1,1}:\real^{2d}\to \real^{2d}$ in Lemma \ref{lm:crude_estimates2}, is  written in the form
\[
a_{1,1}\begin{pmatrix}q\\p\end{pmatrix}=\left(\mathrm{Id}+\mathcal{O}
(\langle \omega(\j)\rangle^{-\beta(1/2-\theta)})\right)\begin{pmatrix}A&0\\0&A^{\dag}\end{pmatrix}\begin{pmatrix}q\\p\end{pmatrix}
\]
where $A:\real^d\to \real^d$ is a linear map satisfying the condition in  (\ref{eq:normA}) and the term $\mathcal{O}
(\langle \omega(\j)\rangle^{-\beta(1/2-\theta)})$ denotes a linear map whose operator norm is bounded by 
$C\langle \omega(\j)\rangle^{-\beta(1/2-\theta)}$ with $C$ independent of $\j,\j'$ and $t\in [0,2t(\omega)]$. 

Let  ${B}^t_{\j\to \j'}:\real^{2d+d'+1}\to \real^{2d+d'+1}$ be the linear map  
\[
{B}^t_{\j\to \j'}\begin{pmatrix}q\\p\\y\\z\end{pmatrix}=\begin{pmatrix}A&0&0&0\\
0&A^{\dag}&0&0\\
0&0&a_{2,2}&0\\
0&0&0&1
\end{pmatrix}
\begin{pmatrix}q\\p\\y\\z\end{pmatrix}
\]
where $a_{2,2}$ is that in the claim of Lemma \ref{lm:crude_estimates2}. 
This corresponds to the differential $D\tilde{f}^t_{\j\to \j'}$ in Lemma \ref{lm:crude_estimates2} (2). 
But notice that we omitted the term $a_{2,1}$, which enjoys the estimate
\[
\|a_{2,1}\|\le C\langle \omega(\j)\rangle^{-(1-\beta)(1/2-3\theta)-\theta}
\le C\langle \omega(\j)\rangle^{-3\theta}
\]
and is incorporated in the nonlinear term
\[
g^t_{\j\to \j'}=(a^t_{\j\to \j'})^{-1}\circ f^t_{\j\to \j'}\circ ({B}^t_{\j\to \j'})^{-1}.
\]
The claims other than (3) are obvious. 
We can check that Claim (3) follows from  Lemma \ref{lm:crude_estimates2} by elementary estimates using the fact that the expansion rate of $f^t_{\j\to \j'}$ and $B^t_{\j,\j'}$ (and their inverses) for $0\le t\le 2t(\omega(\j))$ are bounded by $Ce^{2\chi_{\max} t(\omega)}\le C\langle \omega(\j)\rangle^{\theta/4}$, provided that we choose sufficiently small $\epsilon_0$ in the definition of $t(\omega)$. \end{proof}

We next consider the functions  $\tilde{b}^t_{\j'}\cdot \rho^t_{\j\to \j'}$ that appears as the coefficient in the definition (\ref{eq:ljj}) of $\bL^{t,\sigma\to \sigma'}_{\j\to \j'}$. 
As a bound for this function, we define 
\begin{equation}\label{eq:barbt}
\bar{b}^t_{\j'}:=\max \{|\tilde{b}^t_{\j'}(x,y,z)|\mid (x,y,z)\in \supp \rho_{\j'}  \}.
\end{equation}
Letting $\epsilon_0$ in the definition of $t(\omega)$ be small if necessary, we may and do assume 
\begin{align*}
C_0^{-1}\bar{b}^t_{\j'}\le \tilde{b}^t_{\j'}(x,y,z) \le C_0 \bar{b}^t_{\j'}\quad\text{for $(x,y,z)\in \supp \rho_{\j'}$ and $0\le t\le t(\omega(\j'))$.}
\end{align*}
\begin{Corollary}\label{cor:rho_distortion}
There exists a constant $C_{\alpha,k}>0$ for each multi-indices $\alpha$ and integer $k\ge 0$ such that 
\begin{equation}\label{eq:rhoz}
\|\partial_w^{\alpha}\partial_z^k (\tilde{b}^t_{\j'}\cdot \rho^t_{\j\to \j'})\|_\infty <C_{\alpha, k}  \bar{b}^t_{\j'}\cdot 
\max\{\langle \omega(\j)\rangle, \langle \omega(\j')\rangle\}^{(1-\theta)|\alpha|/2}
\end{equation}
for $0\le t\le 2t(\omega(\j))$. In particular, the family  
\[
\mathcal{X}_{\omega}=\{ (\bar{b}^t_{\j'})^{-1}\cdot \tilde{b}^t_{\j'}\cdot \rho^t_{\j\to \j'}\mid \mbox{$\omega(\j)=\omega$, $\omega'=\omega(\j')$ satisfies  (\ref{eq:close}),  $0\le t\le 2t(\omega)$} \}
\]
for $\omega\in \integer$ satisfies the conditions (C1) and (C2) in Setting I. 
\end{Corollary}
\begin{proof}
From Remark \ref{ex:setting1}, the functions $\rho_{\j'}$ and $\tilde{\rho}_{\j}$ satisfy the conditions on the derivatives stronger than that given in the condition (C2). This is also true for the function $\tilde{b}^t_{\j'}\circ f^t_{\j\to \j'}=b^t\circ \kappa_{\j}$. Hence we can get the corollary by estimating the distortion of $f^t_{\j\to \j'}$ using Lemma \ref{lm:crude_estimates2}, recalling Remark \ref{Rem:derivz} for the derivatives with respect to $z$ and letting  $\epsilon_0$ in the definition of $t(\omega)$ smaller if necessary. 
\end{proof}

\subsection{Supplementary estimates on the operator $\bL^{t,\sigma\to \sigma'}_{\j\to \j'}$.}
We finish this section by providing two supplementary lemmas on the components $\bL^{t,\sigma\to \sigma'}_{\j\to \j'}$. The first one below gives a solid bound on their $L^2$-operator norm. 
\begin{Lemma} \label{cor:multi2}There exists a constant $C_\nu>0$ for any $\nu>0$  such that
\[
\|\bL^{t,\sigma\to \sigma'}_{\j\to \j'}:L^2(\supp \Psi^{\sigma}_{\j})\to L^2(\supp \Psi^{\sigma'}_{\j'})\|\le C_\nu \bar{b}^t_{\j'}\cdot \langle \omega(\j')-\omega(\j)\rangle^{-\nu}
\]
for any $\j, \j'\in \cJ$  and $t\ge 0$. 
(See (\ref{eq:barbt}) for the definition of\/ $\bar{b}^t_{\j'}$.)
\end{Lemma}
\begin{Remark}Notice that this estimate holds for arbitrarily large $t\ge 0$ with uniform constant $C_\nu$ and also that we consider  the $L^2$ norm {\em without weight}. 
\end{Remark}
\begin{proof} Below we suppose that $f^t_{\j\to \j'}$ is extended naturally in the variable $z$ as we noted in Remark \ref{rem:ext_fibered_contact}. 
Take a $C^\infty$ function 
$\hat{\rho}_{\j'}:\real^{2d+d'+1}_{(x,y,z)}\to [0,1]$ which does not depend on the variable $z$,  $\hat{\rho}_{\j'}(x,y,z)=1$ if $(x,y,z')\in \supp \rho_{\j'}$ for some $z'\in \real$ and $\hat{\rho}_{\j'}(x,y,z)= 0$ if $\tilde{\rho}_{\j'}(x,y,z')=0$ for all $z'\in \real$. Then we may write the operator $\bL^{t,\sigma\to \sigma'}_{\j\to \j'}$ as the composition 
\[
\bL^{t,\sigma\to \sigma'}_{\j\to \j'}=\mult( \Psi_{\j'}^{\sigma'})\circ\mult(\tilde{b}^t_{\j'}\cdot \rho^t_{\j\to \j'})^{\lift} \circ L(f^t_{\j\to \j'}, \hat{\rho}_{\j'})^{\lift}.
\]
From Lemma \ref{lm:pBargmannL2}, the operator norm of  $L(f^t_{\j\to \j'}, \hat{\rho}_{\j'})^{\lift}$ with respect to the $L^2$ norm is bounded by that of $L(f^t_{\j\to \j'}, \hat{\rho}_{\j'})$ and hence by some uniform constant  $C$. Note that, since the function $\hat{\rho}_{\j'}(x,y,z)$ does not depend on $z$ and since $f^t_{\j\to \j'}$ is a fibered contact diffeomorphism,  this operator will not enlarge the support of the function in the $\xi_z$ direction. 
Hence, for the proof of the lemma, it is enough to show  
\[
\|\mult(\tilde{b}^t_{\j'}\cdot \rho^t_{\j\to \j'})^{\lift}
:L^2(\supp q_{\omega(\j)})\to L^2(\supp q_{\omega(\j')}) \|\le C_\nu \bar{b}^t_{\j'}\cdot\langle \omega(\j')-\omega(\j)\rangle^{-\nu}
\]
for any $t\ge 0$ with a uniform constant $C_\nu>0$. 
Since $f^t_{\j\to \j'}$ is a fibered contact diffeomorphism and is just a translation in the lines parallel to the $z$-axis, there exists a constant $C_k>0$ for each $k>0$ such that 
\[
\|\partial_z^k (\tilde{b}^t_{\j'}\cdot \rho^t_{\j\to \j'})\|_{\infty}\le C_k \bar{b}^t_{\j'}\quad \mbox{for any $t\ge 0$.}
\]
If we set $\psi=\tilde{b}^t_{\j'}\cdot \rho^t_{\j\to \j'}$ and let $\widehat{\psi}(w,\xi_z)$ be that defined in (\ref{eq:hatpsi}), we have
\[
\sup_{w} |\widehat{\psi}(w,\xi_z)|\le C_k \bar{b}^t_{\j'}\langle \xi_z\rangle^{-k}
\]
with possibly different constant $C_k>0$.
Since the partial Bargmann transform $\pBargmann$ is a combination of the Fourier transform in $z$ and the (scaled) Bargmann transform in $w$, we see that the $L^2$-operator norm of $\mult(\tilde{b}^t_{\j'}\cdot \rho^t_{\j\to \j'})^{\lift}$ is bounded by
\[
C\int_{|\omega(\j')-\omega(\j)|-2<|\xi_z|<|\omega(\j')-\omega(\j)|+2} \left(\sup_{w} |\widehat{\psi}(w,\xi_z)|\right) d\xi_z\le C_k \bar{b}^t_{\j'}\langle \omega(\j')-\omega(\j)\rangle^{-k}.
\]
This is the required estimate. 
\end{proof}

In the next lemma, we give a coarse estimate on the component $\bL^{t,\sigma\to \sigma'}_{\j\to \j'}$. This is a crude estimate and will be used to get  rough bounds for components that are far from the trapped set $X_0$.
(So we do not pursue optimality in any sense for this lemma and its consequences.) 
We set
\[
\delta(\omega,\xi_w)=\langle \omega\rangle^{(1+\theta)/2}\cdot  \left\langle \langle \omega\rangle^{1/2-4\theta}|\xi_w|\right\rangle^{1/2}\qquad\text{for $\omega\in \integer$ and $\xi_w\in \real^{2d+d'}$}
\] 
so that 
\[
\delta(\omega,\xi_w) \sim \begin{cases}
\langle \omega \rangle^{(1+\theta)/2},&\quad\text{if $|\xi_w|\le \langle \omega\rangle^{-1/2+4\theta}$ };\\
\langle\omega \rangle^{1/4-(3/2)\theta}  \cdot (\langle \omega\rangle |\xi_w|)^{1/2},&\quad\text{if $|\xi_w|> \langle \omega\rangle^{-1/2+4\theta}$. }
\end{cases}
\]
We actually state the lemma for the operator $\pBargmann \circ L(\, f^t_{\j\to \j'},  \,\tilde{b}^t_{\j'}\cdot \rho^t_{\j\to \j'}\,)\circ \pBargmann^*$.
Recall from (\ref{eq:ljj}) that $\bL^{t,\sigma\to \sigma'}_{\j\to \j'}$ is just this operator followed by the multiplication by $\Psi_{\j'}^{\sigma'}$.
\begin{Lemma}\label{lm:coarse}  
The operator $\pBargmann \circ L(\, f^t_{\j\to \j'},  \,\tilde{b}^t_{\j'}\cdot \rho^t_{\j\to \j'}\,)\circ \pBargmann^*$ for any $\j,\j'\in \cJ$ and $0\le t\le 2t(\omega(\j))$ is written as an integral operator of the form
\begin{align*}
&(\pBargmann \circ L(\, f^t_{\j\to \j'},  \,\tilde{b}^t_{\j'}\cdot \rho^t_{\j\to \j'}\,)\circ \pBargmann^*) u(w',\xi'_w,\xi'_z)\\
& =\int \left(\int_{\supp \rho^{t}_{\j\to \j'}} \!\!\!\!\!\!\!\!\!\!\!\!\!\!\!\! K(w'',z'';w',\xi'_w,\xi'_z;w,\xi_w,\xi_z) \frac{dw''dz''}{\langle \omega\rangle^{-(2d+d')/2}}\right) u(w,\xi_w,\xi_z) 
d\vol(w,\xi_w,\xi_z)
\end{align*}
and the function $K(\cdot)$ in the integrand satisfies the following estimate: There exists a constant $C_\nu>0$ for each $\nu>0$ (independent of $\j,\j'$ and $t$), such that 
\begin{align}\label{eq:Kof}
|K(w''&,z'';w',\xi'_w,\xi'_z;w,\xi_w,\xi_z)| \\
&  
 \le\frac{ C_\nu \langle \omega\rangle^{\theta}\cdot \left\langle \delta(\omega,\xi_w)^{-1}\cdot |\langle \xi_z\rangle((D\breve{f}^t_{\j\to \j'})^*_{w''})^{-1} \xi_w -\langle \xi'_z\rangle \xi'_w| \right\rangle^{-\nu}} 
{\langle \langle \omega \rangle^{1/2} |(\breve{f}^t_{\j\to \j'})^{-1}(w'')-w|\rangle^{\nu}\cdot 
\langle \langle \omega' \rangle^{1/2} |w''-w'|\rangle^{\nu}\cdot  \langle \omega'-\omega\rangle^{\nu}}\notag
\end{align}
with setting $\omega=\omega(\j)$ and $\omega'=\omega(\j')$, provided
\begin{equation}\label{ass:w}
w''\in \proj_{(x,y)}(\supp \rho^{t}_{\j\to \j'})\quad\text{and}\quad (w,\xi_w,\xi_z)\in \supp \Psi^{\sigma}_{\j}.
\end{equation}
\end{Lemma}
\begin{proof}
In  (\ref{eq:bigK}), 
we have a similar expression for the kernel of $\bL^{t,\sigma\to \sigma'}_{\j\to \j'}$ as in the statement above. 
But the function $k^t_{\j\to \j'}(\cdot)$ in (\ref{eq:bigK}) does not satisfy the required estimates. Below we apply integration by parts to the integral in (\ref{eq:bigK}) to realize the required estimate. 
For simplicity, we consider $\tilde{f}^t_{\j\to \j'}$ given in (\ref{eq:tildef}) in the place of $f^t_{\j\to \j'}$. This does not violate the validity of the proof because, from Lemma \ref{lm:expression_lift}, the action of the lift of the transfer operator for  $a_{\j\to \j'}^{t}$ is that for $D^\dag a_{\j\to \j'}^{t}$ composed with the partial Bargmann projection $\pBargmannP$.
We write $k^t_{\j\to \j'}(\cdot)$ as  (\ref{eq:Kt_expression}) and apply integration by parts to it several times, viewing $e^{i \varphi(\cdot) }$ as oscillatory part and using the differential operator $\cD_0$ in (\ref{eq:D0}) and 
\[
\cD_2=\frac{1-i\cdot \delta(\omega,\xi_w)^{-2}\cdot 
(\langle \xi_z\rangle ((D\breve{f}^t_{\j\to \j'})^*_{w''})^{-1}\xi_w -\langle \xi'_z\rangle \xi'_w)\cdot \partial_{w''}
}{1+ \delta(\omega,\xi_w)^{-2}\cdot |\langle \xi_z\rangle ((D\breve{f}^t_{\j\to \j'})^*_{w''})^{-1}\xi_w -\langle \xi'_z\rangle  \xi'_w|^2}.
\]
To get the required estimate, it is sufficient to prove
\begin{equation}\label{eq:req_est_1}
|\partial_{w''}^\alpha\partial_{z''}^k\Phi(w'',z'';w',\xi'_w,\xi'_z;w,\xi_w,\xi_z)|\le C_{\alpha,k} 
\langle \omega\rangle^{\theta/2+(2d+d')/2}\cdot \delta(\omega,\xi_w)^{|\alpha|}
\end{equation}
for the function $\Phi(\cdot)$ in (\ref{eq:Ksp}) and that
\begin{equation}\label{eq:req_est_2}
|\partial_{w''}^\alpha\partial_{z''}^k(\langle \xi_z\rangle ((D\breve{f}^t_{\j\to \j'})^*)_{w''}^{-1}\xi_w -\langle \xi'_z\rangle \xi'_w)|\le C_{\alpha,k} \cdot \delta(\omega,\xi_w)^{|\alpha|+1}\quad\text{if $\alpha\neq \emptyset$.}
\end{equation}
To proceed, we again consider the two cases in (\ref{cond:near-apart}) separately. 
In the case (ii), we may suppose that we have an extra factor $\langle \omega\rangle^m$ with arbitrary $m>0$ on the right-hand side of (\ref{eq:req_est_1}) and (\ref{eq:req_est_2}), as we noted in the proof of Lemma \ref{lm:IdMinusG}. 
Then such an estimate can be obtained by crude estimates using Lemma \ref{lm:crude_estimates2}. 

In the case (i), we need more precise estimates, but the estimates are still easy. The condition $w''=(x'',y'')\in \proj_{(x,y)}(\supp \rho^{t}_{\j\to \j'})$ implies that $|x''|\le C \langle \omega\rangle^{-1/2+\theta}$. 
Then, from Lemma \ref{lm:tau} and Lemma \ref{lm:crude_estimates2}(3), it is easy to see 
\[
|\xi_z \partial_{x''}^\alpha \tau_{\j\to \j'}(x'')|\le 
C_\alpha  \delta(\omega,\xi_w)^{|\alpha|}\quad\text{for any multi-index $\alpha\neq \emptyset$}.
\]
We may and do assume that $|w''-w|$ and 
$|(\breve{f}^t_{\j\to \j'})^{-1}(w'')-w|$ is bounded by $C\langle \omega\rangle^{1/2+\theta}$ for some constant $C>0$, because, otherwise, the last factor on the right-hand side of 
(\ref{eq:Ksp}) is bounded by $C_\nu\langle \omega\rangle^{-m}$ for arbitrarily large $m$ and, hence, we can prove the claims easily as in the case (ii). Then, under such an assumption, it is straightforward to check the claims (\ref{eq:req_est_1}) and (\ref{eq:req_est_2}) using  (\ref{eq:cos:near}), Lemma \ref{lm:crude_estimates2} and Corollary \ref{cor:rho_distortion}. 
\end{proof}

For ease of use in the next section, we derive a corollary from the last lemma.
We write  $\Delta_{\j\to \j'}^{t,\sigma\to \sigma'}$ for the supremum of 
 the quantity appearing in (\ref{eq:Kof}), 
\begin{equation}\label{eq:Deltajj}
\frac{\left\langle \delta(\omega(\j),\xi_w)^{-1} \cdot |\langle \xi_z\rangle((D\breve{f}^t_{\j\to \j'})^*_{w''})^{-1} \xi_w -\langle \xi'_z\rangle \xi'_w| \right\rangle^{-1} }
{\langle \langle \omega(\j) \rangle^{1/2} |(\breve{f}^t_{\j\to \j'})^{-1}(w'')-w|\rangle\cdot 
\langle \langle \omega(\j') \rangle^{1/2} |w''-w'|\rangle\cdot  \langle \omega(\j')-\omega(\j)\rangle}
\end{equation}
under the conditions 
\[
w''\in \proj_{(x,y)}(\supp \rho^{t}_{\j\to \j'})\quad\text{and}\quad (w,\xi_w,\xi_z)\in \supp \Psi^{\sigma}_{\j}, \quad (w',\xi'_w,\xi'_z)\in \supp \Psi^{\sigma'}_{\j'}.
\]
\begin{Corollary} \label{cor:coarse} Let $\sigma,\sigma'\in \Sigma$. 
There exist constants  $C_\nu>0$ for each $\nu>0$ and $m_0>0$ such that the trace norm of \/
$\bL^{t,\sigma\to \sigma'}_{\j\to \j'}:\bK^{r,\sigma}_{\j}\to \bK^{r,\sigma'}_{\j'}$ for $\j,\j'\in \cJ$ and $0\le t\le 2t(\omega)$ is bounded by
\begin{equation}\label{eq:CDelta}
C_\nu (\Delta_{\j\to \j'}^{t,\sigma\to \sigma'})^{\nu} \cdot 
\max\{ \langle \omega(\j)\rangle, \langle \omega(\j')\rangle,  e^{m(\j)}, e^{m(\j')}\}^{m_0}.
\end{equation}
\end{Corollary}
\begin{proof} 
Note first of all that the claim is a crude estimate as we admit the factor $\max\{ \langle \omega(\j)\rangle, \langle \omega(\j')\rangle,  e^{m(\j)}, e^{m(\j')}\}^{m_0}$. 
Consider $\j,\j'\in \cJ$ and $0\le t\le 2t(\omega)$ and let  $H:\real^{4d+2d'+1}_{(w,\xi_z,\xi_z)}\to \real$ be the function 
\begin{align*}
H(&w,\xi_w,\xi_z)\\
&=\inf\{\langle \xi_z-\xi''_z\rangle \cdot \langle \langle \xi_z\rangle^{1/2}|\xi_w-\xi''_w|\rangle\cdot \langle \langle \xi_z\rangle^{1/2}|w-w''|\rangle \mid (w'',\xi''_w,\xi''_z)\in \supp \Psi^\sigma_{\j'}\},
\end{align*}
which measures the distance of a point $(w,\xi_w,\xi_z)\in \real^{4d+2d'+1}_{(w,\xi_z,\xi_z)}$ from the support of the function $\Psi^\sigma_{\j'}$. 
We regard  the operator $\bL^{t,\sigma\to \sigma'}_{\j\to \j'}$ as the composition  of\begin{equation}\label{eq:op1}
\mult(H^{-\ell})\circ \pBargmann \circ L(f^t_{\j\to \j'}, \tilde{b}^t_{\j'}\cdot \rho^t_{\j\to \j'})\circ \pBargmann^*:
\bK^{r,\sigma}_{\j}\to L^2(\real^{4d+2d'+1})
\end{equation}
and
\begin{equation}\label{eq:op2}
\mult(\Psi^\sigma_{\j'})\circ \pBargmann\circ \mult(\tilde{\rho}_{\j'})\circ \pBargmann^*\circ \mult(H^{\ell}):
L^2(\real^{4d+2d'+1})\to \bK^{r,\sigma'}_{\j'}
\end{equation}
for some large $\ell>0$. We already have estimates on the kernels of 
\[
\pBargmann \circ L(f^t_{\j\to \j'}, \tilde{b}^t_{\j'}\cdot \rho^t_{\j\to \j'})\circ \pBargmann^*
\quad \text{and}\quad \pBargmann\circ \mult(\tilde{\rho}_{\j'})\circ \pBargmann^*
\]
in Lemma \ref{lm:coarse} and (in the proof of) Lemma \ref{lm:multiplication} respectively. 
With using  those estimates, we see that, 
if $\ell$ is sufficiently large, the kernels of these operators are square-integrable as functions on $(\real^{4d+2d'+1})^2$ and therefore the operators (\ref{eq:op1}) and (\ref{eq:op2}) are Hilbert-Schmidt operators\footnote{We refer \cite[Chapter IV, Section 7]{GohbergBook} for Hilbert-Schmidt operators.}.
Their Hilbert-Schmidt norms are just the $L^2$-norm of their kernels and easy to estimate. Indeed, for a very crude bound for the volume of the supports of $\Psi^{\sigma}_{\j}$ and $\Psi^{\sigma'}_{\j'}$, we may take  some power of $\max\{ \langle \omega(\j)\rangle, \langle \omega(\j')\rangle,  e^{m(\j)}, e^{m(\j')}\}$. Then, from the estimates on the kernels mentioned above and the definition of $\Delta_{\j\to \j'}^{t,\sigma\to \sigma'}$, the Hilbert-Schmidt norms of the operators above are bounded respectively by   
\begin{align*}
&C_\nu (\Delta_{\j\to \j'}^{t,\sigma\to \sigma'})^{\nu} \cdot \max\{ \langle \omega(\j)\rangle, \langle \omega(\j')\rangle,  e^{m(\j)}, e^{m(\j')}\}^{m_0/2}
\intertext{and} 
& C \max\{ \langle \omega(\j)\rangle, \langle \omega(\j')\rangle,  e^{m(\j)}, e^{m(\j')}\}^{m_0/2}
\end{align*}
for sufficiently large $m_0$. Since the trace norm of $\bL^{t,\sigma\to \sigma'}_{\j\to \j'}$ is bounded by the product of the Hilbert-Schmidt norms of  (\ref{eq:op1}) and (\ref{eq:op2}),
 we obtain the claim. 
\end{proof}

\section{Properties of the lifted operators}
\label{sec:proofs_of_props}
In this section, we prove several propositions for the lifted operators $\bL^{t, \sigma\to \sigma'}$ using the estimates on their components $\bL^{t,\sigma\to \sigma'}_{\j\to \j'}$ obtained in the previous sections and then deduce the propositions in  Subsection \ref{ss:lt} from them. This finishes the proof of Theorem \ref{th:spectrum}. 
 
\subsection{Decomposition of the lifted transfer operator}\label{sec:lifted_operator2}
We classify the components $\bL^{t,\sigma\to \sigma'}_{\j\to \j'}$ of the lifted operator $\bL^{t,\sigma\to \sigma'}$ into three classes, namely ``low frequency", ``hyperbolic (or peripheral)" and ``central" components. The classification  depends on a constant  $k_0>0$ that we will specify in the course of the argument. 

\begin{Definition}\label{Def:partsOfOperators} (1) 
A component $\bL^{t,\sigma\to \sigma'}_{\j\to \j'}$ is  a {\em low frequency component} if 
\begin{itemize}
\item[{\rm(LF)}] either $\max\{ |\omega(\j)|, |m(\j)|\}\le k_0$ or $\max\{|\omega(\j')|, |m(\j')|\}\le k_0$.
\end{itemize}
(2)
A component $\bL^{t,\sigma\to \sigma'}_{\j\to \j'}$ is  a \emph{central  component} if it is not a low frequency component but  
\begin{itemize}
\item[{\rm(CT)}] $m(\j)=m(\j')=0$.
\end{itemize}
(3) The other components are called {\em hyperbolic (or peripheral) components}. That is, a component $\bL^{t,\sigma\to \sigma'}_{\j\to \j'}$ is  a hyperbolic component if 
\begin{itemize}
\item[{\rm(HYP)}] $\max\{ \omega(\j), m(\j)\}> k_0$, $\max\{\omega(\j'),  m(\j')\}> k_0$ and (either $m(\j)\neq 0$ or $m(\j')\neq  0$).
\end{itemize}
\end{Definition}
The low frequency components  are responsible for the action of transfer operators on low frequency part of  functions (in all the directions) and will be treated as a negligible part in our argument.  The central components are of primary importance in our argument. 
In the global picture discussed at the end of Section \ref{sec:Gext}, the central part is responsible for the action of transfer operators $\cL^{t}$ on the wave packets corresponding to points near the trapped set $X$ in (\ref{eq:X}). 
We are going to see that the central components are well approximated by the linear models considered in Section \ref{sec:linear}. 
The hyperbolic components are those components which are strongly affected by the hyperbolicity and non-linearity of the flow. 
For these components, we will see that  the weight $2^{-rm(\j)}$ in the definition (\ref{eq:bKnorm}) of the norm on $\bK^{r,\sigma}$ takes effect and makes the operator norms  small (at least if the constants $k_0>0$ and $r>0$ are sufficiently large). 
 
Correspondingly to the classification of the components above, we decompose the transfer operator $\bL^{t,\sigma\to \sigma'}$ into three parts: 
\begin{equation}\label{eq:decomp_Lt}
\bL^{t,\sigma\to \sigma'}=\bL^{t,\sigma\to \sigma'}_{\mathrm{low}}+\bL^{t,\sigma\to \sigma'}_{\mathrm{ctr}}+\bL^{t,\sigma\to \sigma'}_{\mathrm{hyp}}
\end{equation}
where the low frequency part $\bL^{t,\sigma\to \sigma'}_{\mathrm{low}}$ (resp.\ the central part $\bL^{t,\sigma\to \sigma'}_{\mathrm{ctr}}$,  the hyperbolic part $\bL^{t,\sigma\to \sigma'}_{\mathrm{hyp}}$) is defined as the operator that consists of only the low frequency (resp.\  central, hyperbolic) components of 
$\bL^{t,\sigma\to \sigma'}$. For instance, the low frequency part $\bL^{t,\sigma\to \sigma'}_{\mathrm{low}}$ is defined by 
\[
\bL^{t,\sigma\to \sigma'}_{\mathrm{low}}\bu =
\left(\sum_{\rm low} L_{\j\to \j'}^{t,\sigma\to \sigma'}u_{\j}\right)_{\j'\in \cJ} \quad \mbox{for $\bu=(u_{\j})_{\j\in \cJ}$,}
\]
where $\sum_{\rm low}$ is the sum over $\j\in \cJ$ such that $\bL^{t,\sigma\to \sigma'}_{\j\to \j'}$ is a low frequency component. 

\begin{Remark}
In some places below, we will let the constant $k_0$ be larger to get preferred estimates. If we let $k_0$ be larger, the components classified as central and hyperbolic components will become fewer and this will enable us to get better uniform estimates for the corresponding parts.
Of course then the components classified as low frequency components will increase. But, since we need only a few simple estimates for this part, this will not cause any problem.   
\end{Remark}

\subsection{The central part}
For the central components, we prove two propositions. The first one below is a counterpart of Proposition \ref{pp:norm_estimate} for the lifted operators and corresponding to Theorem  \ref{cor:linear1} in the linear setting. 
Let us recall the operator $\mathbf{T}^{\sigma\to \sigma'}_\omega$ from Definition \ref{Def:bT}.
For simplicity, we write $\mathbf{T}^{\sigma}_\omega$ for $\mathbf{T}^{\sigma\to \sigma'}_\omega$ when $\sigma'=\sigma$.
\begin{Proposition}\label{lm:TL}
Let $\sigma, \sigma'\in \Sigma$.  There exist constants $\epsilon>0$ and $C_\nu>1$ for each $\nu>0$ such that 
\begin{align*}
&\|\mathbf{T}_{\omega'}^{\sigma'}\circ \bL^{t, \sigma\to \sigma'}_{\mathrm{ctr}}\circ \mathbf{T}^{\sigma}_{\omega}:\bK^{r,\sigma}\to \bK^{r,\sigma'}\|\le 
C_\nu  \langle \omega'-\omega\rangle^{-\nu},\\
&\|\mathbf{T}^{\sigma'}_{\omega'}\circ \bL^{t, \sigma\to \sigma'}_{\mathrm{ctr}}\circ (\bPi_\omega-\mathbf{T}^{\sigma}_{\omega}):\bK^{r,\sigma}\to \bK^{r,\sigma'}\|\le C_\nu \langle \omega\rangle^{-\epsilon}  \langle \omega'-\omega\rangle^{-\nu},\\
&\|(\bPi_{\omega'}-\mathbf{T}^{\sigma'}_{\omega'})\circ \bL^{t, \sigma\to \sigma'}_{\mathrm{ctr}}\circ \mathbf{T}^{\sigma}_{\omega}:\bK^{r,\sigma}\to \bK^{r,\sigma'}\|\le C_\nu \langle \omega\rangle^{-\epsilon}  \langle \omega'-\omega\rangle^{-\nu}
\intertext{for any $\omega, \omega'\in \integer$ and $0\le t\le 2t(\omega)$, and further, if the condition (\ref{eq:condition_on_t}) with respect to $\sigma$ and $\sigma'$ holds, then we have }
&\|(\bPi_{\omega'}-\mathbf{T}^{\sigma'}_{\omega'})\circ  \bL^{t, \sigma\to \sigma'}_{\mathrm{ctr}}\circ (\bPi_\omega-\mathbf{T}^{\sigma}_{\omega}):\bK^{r,\sigma}\to \bK^{r,\sigma'}\|\le 
C_\nu e^{-\chi_0 t}  \langle \omega'-\omega\rangle^{-\nu}
\end{align*}
for any $\omega, \omega'\in \integer$ and $0\le t\le 2t(\omega)$.
\end{Proposition}
\begin{proof}
We prove the claims assuming that the condition (\ref{eq:close}) for $\omega$ and $\omega'$ holds because, otherwise, the claims follow immediately from Lemma~\ref{cor:multi2}. (See Remark \ref{Rem:far2}.)
The following lemma is the {\em component-wise version} of  Proposition  \ref{lm:TL}, in which we write $\mathbf{T}_{\omega}^{\sigma}$ for the restrictions of $\mathbf{T}_{\omega}^{\sigma}=\mathbf{T}_{\omega}^{\sigma\to\sigma}$ to each $\bK^{r,\sigma}_\j$ with $m(\j)=0$.
\begin{Lemma}\label{lm:main} There exist constants $\epsilon>0$ and $C_\nu>0$ for each $\nu>0$, independent of $\omega, \omega'\in \integer$, such that, for  any central component \/ $\bL_{\j\to \j'}^{t,\sigma\to \sigma'}$ with  $\omega(\j)=\omega$, $\omega(\j')=\omega'$ and $m(\j)=m(\j')=0$, we have  
\begin{align*}
&\|\mathbf{T}_{\omega'}^{\sigma'}\circ \bL_{\j\to \j'}^{t,\sigma\to \sigma'}\circ \mathbf{T}_{\omega}^{\sigma}:
\bK^{r,\sigma}_{\j}\to \bK^{r,\sigma'}_{\j'}\|\le  C_\nu \langle  \omega'-\omega\rangle^{-\nu},\\
&\|\mathbf{T}_{\omega'}^{\sigma'}\circ \bL_{\j\to \j'}^{t,\sigma\to \sigma'}\circ (1-\mathbf{T}_{\omega}^{\sigma}):
\bK^{r,\sigma}_{\j}\to \bK^{r,\sigma'}_{\j'}\|\le  C_\nu \langle \omega\rangle^{-\epsilon} \langle  \omega'-\omega\rangle^{-\nu},\\
&\|(1-\mathbf{T}_{\omega'}^{\sigma'})\circ \bL_{\j\to \j'}^{t,\sigma\to \sigma'}\circ \mathbf{T}_{\omega}^{\sigma}:
\bK^{r,\sigma}_{\j}\to \bK^{r,\sigma'}_{\j'}\|\le  C_\nu \langle \omega\rangle^{-\epsilon} \langle  \omega'-\omega\rangle^{-\nu}
\intertext{for $0\le t\le 2t(\omega(\j))$, and further, if the condition (\ref{eq:condition_on_t}) with respect to $\sigma$ and $\sigma'$ holds for $t$ in addition, then}
&\|(1-\mathbf{T}_{\omega'}^{\sigma'})\circ \bL_{\j\to \j'}^{t,\sigma\to \sigma'} \circ (1-\mathbf{T}_{\omega}^{\sigma}) : \bK^{r,\sigma}_{\j}\to \bK^{r,\sigma'}_{\j'}\| \le C_\nu  e^{-\chi_0 t} \langle  \omega'-\omega\rangle^{-\nu}.
\end{align*} 
\end{Lemma}
\begin{proof} 
We first prove the claim that, if $0\le t\le 2t(\omega)$ satisfies the condition (\ref{eq:condition_on_t}) with respect to $\sigma$ and $\sigma'$, then
\[
\|\bL_{\j\to \j'}^{t,\sigma\to \sigma'}:
\bK^{r,\sigma}_{\j}\to \bK^{r,\sigma'}_{\j'}\|\le  C_\nu \langle  \omega'-\omega\rangle^{-\nu}.
\]
We express the diffeomorphism $f^t_{\j\to \j'}$  as in Corollary \ref{lm:distortion} and correspondingly write the operator $\bL_{\j\to \j'}^{t,\sigma\to \sigma'}$ as 
\begin{equation*}
\mult(X_{n_0(\omega(\j'))}\cdot {q}_{\omega(\j')}) \circ L^{\lift}_a\circ \mult^{\lift}(\rho^t_{\j\to \j'}\circ a^t_{\j\to \j'}) \circ L^{\lift}_g \circ L^{\lift}_B:\bK^{r,\sigma}_{\j}\to \bK^{r,\sigma'}_{\j'}
\end{equation*}
where $L^{\lift}_a$, $L^{\lift}_g$ and $L^{\lift}_B$ are the lifts of transfer operators for $a^t_{\j\to \j'}$, $g^t_{\j\to \j'}$ and $B^t_{\j\to \j'}$ respectively. We regard that the rightmost factor $L^{\lift}_B$ above as an operator from $\bK^{r,\sigma}_{\j}$ to $L^2(\mathcal{W}^{r,\sigma}):=L^2(\real^{4d+2d'+1},(\mathcal{W}^{r,\sigma})^2)$ and that the rest is that from $L^2(\mathcal{W}^{r,\sigma})$ to $\bK^{r,\sigma'}_{\j'}$. For the latter, we also note that $\cT^{\lift}_0:L^2(\mathcal{W}^{r,\sigma})\to L^2(\mathcal{W}^{r,\sigma'})$ is bounded for any combination of $\sigma, \sigma'\in \Sigma$ from Theorem \ref{cor:linear1} (1).

Recall the function $Y$ defined in (\ref{def:Y}) and also Lemma \ref{lm:expression_lift} and Remark \ref{Rem:expression_lift} for the operator $L^{\lift}_a$. 
Then we see 
\[
\|\mult(X_{n_0(\omega(\j'))}\cdot {q}_{\omega(\j')}) \circ L^{\lift}_a\circ \mult(1-Y):L^2(\mathcal{W}^{r,\sigma})\to \bK^{r,\sigma'}_{\j'}\|\le C_\nu \langle \omega'\rangle^{-\nu}.
\]
From this and Theorem \ref{cor:linear1} (2) for $L_B^{\lift}$, we see that  the difference between the operator $\bL_{\j\to \j'}^{t,\sigma\to \sigma'}$  and 
\[
\mult(X_{n_0(\omega(\j'))}\cdot {q}_{\omega(\j')}) \circ L^{\lift}_a\circ \big[\mult(Y)\circ  \mult^{\lift}(\rho^t_{\j\to \j'}\circ a^t_{\j\to \j'}) \circ L^{\lift}_g\big] \circ L^{\lift}_B:\bK^{r,\sigma}_{\j}\to \bK^{r,\sigma'}_{\j'}
\] 
is bounded by $C_\nu \langle \omega'\rangle^{-\nu}$ and hence negligible.
By virtue of the factor $\mult(Y)$, we can apply Lemma \ref{lm:IdMinusG} to the operator in the square bracket $[\cdot]$ above. 
Since $L^{\lift}_a$ and $L^{\lift}_B$ preserve $\supp q_{\omega(\j')}$ and $\supp q_{\omega(\j)}$ respectively, 
Lemma \ref{lm:IdMinusG} together with Corollary \ref{cor:multi} and Corollary \ref{cor:rho_distortion} yield the estimate that the operator norm of the operator above other than the rightmost factor $L_B^{\lift}$ is bounded by $C_\nu \bar{b}^t_{\j'}\langle \omega'-\omega\rangle^{\nu}$. On the other hand, from Theorem \ref{cor:linear1}, the operator norm of  $\bar{b}^t_{\j'}\cdot L^{\lift}_B$ is bounded by a constant $C_0$. (For this, recall the definition (\ref{eq:Lt}) of the operator $L^t$ appearing in Theorem \ref{cor:linear1}.) Therefore we obtain the claim. 

Now we prove the claims of the lemma. 
The proofs of the four claims are all similar to that in the preceding   paragraphs. Below we prove the second claim and mention for the other cases at the end.  We write 
$\mathbf{T}_{\omega'}^{\sigma'}\circ \bL_{\j\to \j'}^{t,\sigma\to \sigma'}\circ (1-\mathbf{T}_{\omega}^{\sigma})$ as 
\begin{align}\label{eq:bare}
\mult&(X_{n_0(\omega(\j'))})\circ \cT^{\lift}_0\circ \mult(X_{n_0(\omega(\j'))}\cdot {q}_{\omega(\j')}) \circ L^{\lift}_a\\
&\circ \mult^{\lift}(\rho^t_{\j\to \j'}\circ a^t_{\j\to \j'}) \circ L^{\lift}_g \circ L^{\lift}(B)\circ (1-\mult(X_{n_0(\omega(\j))})\circ \cT^{\lift}_0).\notag
\end{align}
Then, producing error terms bounded by $C_\nu \langle \omega\rangle^{-\epsilon} \langle  \omega'-\omega\rangle^{-\nu}$, we may 
\begin{itemize}
\item introduce the factor $\mult(Y)$ before $\mult^{\lift}(\rho^t_{\j\to \j'}\circ a^t_{\j\to \j'})$, 
\item replace $L^{\lift}_g$ by the identity, using Lemma \ref{lm:IdMinusG},
\item replace $\mult(X_{n_0(\omega(\j'))})$ and $\mult(Y)$ by the identity, using the localized property of the kernel of $\cT^{\lift}_0$ given in Lemma \ref{cor:Tzero}, and
\item change the order of $\cT^{\lift}_0$ and $\mult^{\lift}(\rho^t_{\j\to \j'}\circ a^t_{\j\to \j'})$,  using Lemma \ref{lm:mphi}.
\end{itemize}
With these deformations, we reach the operator 
\begin{equation}\label{eq:resultingop}
\mult({q}_{\omega(\j')}) \circ L^{\lift}_a\circ \mult^{\lift}(\rho^t_{\j\to \j'}\circ a^t_{\j\to \j'}) \circ  \big[\cT^{\lift}_0 \circ L^{\lift}(B)\circ (1- \cT^{\lift}_0)\big],
\end{equation} 
noting that $\cT^{\lift}_0$ commutes with $\mult({q}_{\omega(\j')})$ and $L^{\lift}_a$. But this operator is null from Theorem \ref{cor:linear1}. We therefore obtained the second claim. 

The proofs of the other claims are parallel: We deform the operators to the form  corresponding to (\ref{eq:resultingop}) in the same manner as above,  producing error terms bounded by $C_\nu \langle \omega\rangle^{-\epsilon} \langle  \omega'-\omega\rangle^{-\nu}$,
and then use Corollary \ref{cor:multi} (with Corollary \ref{cor:rho_distortion}) and Theorem \ref{cor:linear1} (2). 
(For the last claim, we assume $e^{-\chi_0 t(\omega)}\gg \langle \omega\rangle^{-\epsilon}$ by letting the constant $\epsilon_0>0$ in the definition of $t(\omega)$ smaller.)
\end{proof}

Proposition  \ref{lm:TL} is obtained by summing the estimates in Lemma \ref{lm:main}.  
But there is a small problem. Notice that, for  $\j'\in \cJ$ with $\omega(\j')=\omega'$ and $m(\j')=0$, the number of $\j\in \cJ$ satisfying $\omega(\j)=\omega$, $m(\j)=0$ and $\bL^{t,\sigma\to \sigma'}_{\j\to \j'}\neq 0$ will grow exponentially  with respect to $t$. Hence if we just add  the inequalities in Lemma \ref{lm:main}, we will not get the claims of Proposition  \ref{lm:TL}. 
This problem is resolved by using the localized properties of the kernels of the operators 
$\cT_0^{\lift}$ and $\bL^{t,\sigma\to \sigma'}_{\j\to \j'}$ and also the fact that the intersection multiplicities of the supports of  $\{\rho_{\j}\circ \kappa_{\j}^{-1}\mid \j\in \cJ, \omega(\j)=\omega\}$ is bounded uniformly for $\omega$. The argument is easy but may not be completely obvious.
Since we will use similar argument later, we present it below in some detail. 

Recall the definition (\ref{eq:rhot}) of the functions $\rho^t_{\j\to \j'}$ 
(and also that of $\rho_{a,n}^{(\omega)}$ in Subsection \ref{ss:chart_omega}).
Let $\widetilde{D}^t_{\j\to \j'}$ and $D^t_{\j\to \j'}$ be the $\langle \omega\rangle^{-(1-\theta)/2}$-neighborhoods of the subsets 
\[
\proj_{(x,y)}(\supp \rho^t_{\j\to \j'}) \quad\mbox{and}\quad  
\proj_{(x,y)}(\supp (\rho^t_{\j\to \j'}\circ f^t_{\j\to \j'}))\quad\mbox{in $\real^{2d+d'}_{(x,y)}$}
\]
respectively.  For each $\j'\in \cJ$ with $\omega(\j')=\omega'$ and $m(\j')=0$ (resp.\ $\j\in \cJ$ with $\omega(\j)=\omega$ and $m(\j)=0$) and for $0\le t\le 2t(\omega)$,  the intersection multiplicity of the subsets in
\begin{align}\label{eq:Djj}
&\{ \widetilde{D}^t_{\j\to \j'} \mid \j\in \cJ,\; \omega(\j)=\omega,\; m(\j)=0,\; \widetilde{D}^t_{\j\to \j'}\neq \emptyset\}\\
 (\mbox{resp.\  }&\{ D^t_{\j\to \j'} \mid \j'\in \cJ,\; \omega(\j')=\omega',\; m(\j')=0,\;  D^t_{\j\to \j'}\neq \emptyset\})\notag
\end{align}
is bounded by a constant $C$ independent of $\omega$, $\omega'$ and $t$, provided that the constant $\epsilon_0$ in the definition of $t(\omega)$ is small enough.
\begin{Remark}\label{rem:cardD}
The cardinality of the sets in (\ref{eq:Djj}) will not be bounded uniformly with respect to $t$ and $\omega$, as we have noted.
But note that, letting the constant $\epsilon_0$ be small, we may assume that this  is bounded by $C\langle \omega\rangle^{\theta}$. 
\end{Remark}

In order to discuss about the four claims in Lemma \ref{lm:main} in parallel, we write $\mathbf{M}^{t,\sigma\to \sigma'}_{\j\to \j'}$ for either of the operators
\begin{align*}
&\mathbf{T}_{\omega(\j')}^{\sigma'}\circ \bL_{\j\to \j'}^{t,\sigma\to \sigma'}\circ \mathbf{T}_{\omega(\j)}^{\sigma}, \quad \mathbf{T}_{\omega(\j')}^{\sigma'}\circ \bL_{\j\to \j'}^{t,\sigma\to \sigma'}\circ (1-\mathbf{T}_{\omega(\j)}^{\sigma}),\quad \mbox{or}\\
& (1-\mathbf{T}_{\omega(\j')}^{\sigma'})\circ \bL_{\j\to \j'}^{t,\sigma\to \sigma'}\circ \mathbf{T}_{\omega(\j)}^{\sigma}, \quad (1-\mathbf{T}_{\omega(\j')}^{\sigma'})\circ \bL_{\j\to \j'}^{t,\sigma\to \sigma'} \circ (1-\mathbf{T}_{\omega(\j)}^{\sigma}).
\end{align*}
Let us denote by $\mathbf{1}'_{\j\to \j'}$ (resp.\  $\mathbf{1}_{\j\to \j'}$) the indicator function of the subset
\[
\{ (x,y,\xi_x,\xi_y, \xi_z)\in \real^{4d+2d'+1}\mid (x,y)\in \widetilde{D}^t_{\j\to \j'} \mbox{(resp.\ } (x,y)\in D^t_{\j\to \j'} )\}.
\]
We first approximate the operator $\mathbf{M}^{t,\sigma\to \sigma'}_{\j\to \j'}$ by 
\begin{equation}\label{eq:wtM}
\widetilde{\mathbf{M}}^{t,\sigma\to \sigma'}_{\j\to \j'}=\mult(\mathbf{1}'_{\j\to \j'})\circ \mathbf{M}^{t,\sigma\to \sigma'}_{\j\to \j'}\circ \mult(\mathbf{1}_{\j\to \j'}),
\end{equation}
by cutting off the tail part.  
By crude estimates using the localized properties of the kernels of the operators 
$\mathbf{T}_{\omega(\j)}^{\sigma}$ and $\bL^{t,\sigma\to \sigma'}_{\j\to \j'}$, given in Corollary \ref{cor:Tzero} and Lemma \ref{lm:coarse}, together with the definitions of $\widetilde{D}^t_{\j\to \j'}$ and $D^t_{\j\to \j'}$, we see that 
\begin{equation}\label{eq:wtM2}
\|\mathbf{M}^{t,\sigma\to \sigma'}_{\j\to \j'}-\widetilde{\mathbf{M}}^{t,\sigma\to \sigma'}_{\j\to \j'}:\bK^{r,\sigma}_{\j}\to \bK^{r,\sigma'}_{\j'} \|\le C_{\nu}\langle \omega\rangle^{-\theta}\cdot \langle \omega'-\omega\rangle^{-\nu}
\end{equation}
for $\j,\j'\in \cJ$ with $\omega(\j)=\omega$, $\omega(\j')=\omega'$ and $m(\j)=m(\j')=0$. 
(The factor $\langle \omega\rangle^{-\theta}$ could be much better but this is enough.) 
Since the cardinality of the set (\ref{eq:Djj}) is bounded by $C\langle \omega\rangle^{\theta}$ as we noted in Remark \ref{rem:cardD} above, we obtain
\begin{align*}\label{eq:sep1}
\sum_{\j':\omega(\j')=\omega', m(\j')=0}&\left\| \sum_{\j:\omega(\j)=\omega, m(\j)=0}
({\mathbf{M}}^{t,\sigma\to \sigma'}_{\j\to \j'} - \widetilde{\mathbf{M}}^{t,\sigma\to \sigma'}_{\j\to \j'})u_{\j}\right\|_{\bK^{r,\sigma'}_{\j'}}^2\\
 &\qquad\qquad \le  C_{\nu}\langle \omega\rangle^{\theta}\cdot \langle \omega\rangle^{-2\theta}\cdot \langle \omega'-\omega\rangle^{-2\nu} \sum_{\j:\omega(\j)=\omega, m(\j)=0}\|u_{\j}\|_{\bK^{r,\sigma}_{\j}}^2
\end{align*}
for $(u_{\j})_{\j\in \cJ} \in \bK^{r,\sigma}$.
Therefore the claims of Proposition  \ref{lm:TL} follow if we prove the required estimates with  ${\mathbf{M}}^{t,\sigma\to \sigma'}_{\j\to \j'}$ replaced by $\widetilde{\mathbf{M}}^{t,\sigma\to \sigma'}_{\j\to \j'}$. 

 By boundedness of the intersection multiplicities of (\ref{eq:Djj}), we have 
\begin{align*}
&\sum_{\j':\omega(\j')=\omega', m(\j')=0}\left\| \sum_{\j:\omega(\j)=\omega, m(\j)=0}\widetilde{\mathbf{M}}^{t,\sigma\to \sigma'}_{\j\to \j'} u_{\j} \right\|_{\bK^{r,\sigma'}_{\j'}}^2\\
 &\le 
C \sum_{\j':\omega(\j')=\omega', m(\j')=0}\sum_{\j:\omega(\j)=\omega, m(\j)=0}\left\|{\mathbf{M}}^{t,\sigma\to \sigma'}_{\j\to \j'}:\bK^{r,\sigma}_{\j}\to \bK^{r,\sigma'}_{\j'}\right\|^2
\|\mult(\mathbf{1}_{\j\to \j'}) u_{\j}\|_{\bK^{r,\sigma}_{\j}}^2
\end{align*}
and also
\[
\sum_{\j':\omega(\j')=\omega', m(\j')=0}\|\mult(\mathbf{1}_{\j\to \j'}) u_{\j}\|_{\bK^{r,\sigma}_{\j}}^2\le C\| u_{\j}\|^2_{\bK^{r,\sigma}_{\j}}.
\]
Hence, applying the claims of Lemma \ref{lm:main} to the term $\|{\mathbf{M}}^{t,\sigma\to \sigma'}_{\j\to \j'}:\bK^{r,\sigma}_{\j}\to \bK^{r,\sigma'}_{\j'}\|$, we obtain the required estimates. 
\end{proof}

\begin{Remark}\label{Rem:far2}
In the case where (\ref{eq:close}) does not hold, we have 
\[
\langle \omega'-\omega\rangle \ge \max\{\langle \omega\rangle^{1/2}, \langle \omega'\rangle^{1/2}\}/2,
\]
and we can prove Lemma \ref{lm:main} by crude estimate using Lemma \ref{lm:coarse}. 
We do not need the argument on the intersection multiplicity as above because we have the factor $\langle \omega'-\omega\rangle^{-\nu}\lesssim \max\{\langle \omega\rangle, \langle \omega'\rangle\}^{-\nu/2}$. (See Remark \ref{rem:cardD}.)
\end{Remark}

The second proposition below is essentially same as Proposition \ref{pp:backward} but stated in terms of lifted operators.  
\begin{Proposition} \label{pp:backward_lifted}
Let $\sigma, \sigma'\in \Sigma$. 
There exist constants $\epsilon>0$, $C_0>0$ and $C_\nu>0$ for each $\nu>0$ such that, for  
 $u\in \cK^{r,\sigma}(K_0)$, $\omega\in \integer$ and $0\le t\le 2t(\omega)$, there exists $v_\omega\in \cK^{r,\sigma'}(K_1)$  such
that 
\begin{equation}\label{eq:backward1_lifted}
\|\bL^{t, \sigma'\to \sigma} \circ \bI^{\sigma'}v_\omega- 
 \bI^{\sigma}\circ \mult(\rho_{K_1})\circ (\bI^\sigma)^*\circ\mathbf{T}^{\sigma}_{\omega} \circ \bI^{\sigma}u\|_{\bK^{r,\sigma}}\le C_0 \langle \omega\rangle^{-\theta}\|u\|_{\cK^{r,\sigma}}
\end{equation}
and that, for $\omega'\in \integer$ and $0\le t'\le t$, we have
\begin{align}
&\| \bPi_{\omega'}\circ \bI^{\sigma'}\circ \cL^{t'}  v_{\omega}\|_{\bK^{r,\sigma'}}
\le 
C_\nu \langle \omega'-\omega\rangle^{-\nu}  \|u\|_{\cK^{r,\sigma}},\quad \mbox{and}\label{eq:vomegatrans}
\\
&\|({\bPi}_{\omega'}-\mathbf{T}^{\sigma'}_{\omega'})\circ  \bI^{\sigma'}\circ \cL^{t'}  v_{\omega}\|_{\bK^{r,\sigma'}}
\le 
C_\nu \langle \omega\rangle^{-\epsilon} \langle \omega'-\omega\rangle^{-\nu}  \|u\|_{\cK^{r,\sigma}}.\label{eq:vomegatrans2}
\end{align}
\end{Proposition} 
\begin{proof} 
For construction of $v_{\omega}$, we first set  
\[
u_{\omega}=\cT_\omega u=\mult(\rho_{K_1})\circ (\bI^\sigma)^*\circ \mathbf{T}_\omega^{\sigma}\circ \bI^\sigma u\in \cK^{r,\sigma}(K_0).
\] 
(Recall (\ref{eq:mathbfT}) and (\ref{def:cTomega})  for the definition of $\cT_{\omega}$.) 
As we noted in Remark \ref{Rem:smoothness}, this is a smooth function and hence so is the function   
\[
v_\omega:=\rho_{K_0}\cdot \cL^{-t}u_{\omega}=
\cL^{-t}((\rho_{K_0}\circ f^{-t}_{G})\cdot u_{\omega})
\]
where $\rho_{K_0}:G\to [0,1]$ is the smooth function defined  in (\ref{eq:defrhoK}) and is thrown in because the support of $\cL^{-t}u_{\omega}$ for $0\le t\le 2t(\omega)$ may not be contained in $K_0$. 
\begin{Remark}
Beware that we are considering the transfer operator $\cL^{-t}$ for a negative time $-t<0$, which will not be bounded on $\cK^{r,\sigma}$. 
\end{Remark} 

For the proof of  (\ref{eq:backward1_lifted}), we write the left hand side of (\ref{eq:backward1_lifted}) as 
\[
\|\bI^{\sigma}\circ \mult(\rho_{K_1}(\rho_{K_0}\circ f^{-t}_{G}-1))\circ (\bI^{\sigma})^* \circ  \mathbf{T}_{\omega}^{\sigma} \circ \bI^{\sigma}u\|_{\bK^{r,\sigma}}.
\]
Since we are assuming  $0\le t\le 2t(\omega)$, the function $\rho_{K_1}(\rho_{K_0}\circ f^{-t}_{G}-1)$ is supported on the outside of the $C^{-1} \langle \omega\rangle^{-\theta}$ neighborhood of the section $\mathrm{Im}\,e_u$. (To ensure this, let the constant $\epsilon_0$ in the definition of $t(\omega)$ be smaller if necessary.) Hence, by the localized property of the kernel of  
$\mathbf{T}_{\omega}^{\sigma}$ given in Corollary \ref{cor:Tzero}, we obtain  (\ref{eq:backward1_lifted}).

For the proof of (\ref{eq:vomegatrans}) and (\ref{eq:vomegatrans2}), we basically follow the argument in the proof of Proposition \ref{lm:TL}. Below we assume the condition (\ref{eq:close}) because, otherwise, the proof is obtained easily by using crude estimates. (See Remark \ref{Rem:far2}.)  We take $0\le t'\le t\le 2t(\omega)$ arbitrarily and write
\[
\bI^{\sigma}u=(u_{\j})_{\j\in \cJ}\quad\text{and}\quad \bI^{\sigma'}(\cL^{t'} v_{\omega})=(v_{\j})_{\j\in \cJ}\quad\text{respectively.}
\]
Then $v_{\j'}$ for $\j'\in \cJ$ with $\omega(\j')=\omega'$ and $m(\j')=m'$ is written as the sum 
\[
v_{\j'}=\sum_{\j:\omega(\j)=\omega, m(\j)=0} \bL^{-(t-t'),\sigma\to \sigma'}_{\j\to \j'}\circ \mathbf{T}_{\omega}^{\sigma} u_{\j}
\]
where $\bL^{-(t-t'),\sigma\to \sigma'}_{\j\to \j'}$ is defined in (\ref{eq:ljj}), but we read (\ref{eq:rhot}) as 
\[
{\rho}^{-(t-t')}_{\j\to \j'}= 
\tilde{b}^{-(t-t')}_{\j'}\cdot
( (\rho_{K_1}\circ f^{t-t'}_G\cdot \rho_{K_0}\circ f^{-t'}_G)\circ  \kappa_{\j'})\cdot  {\rho}_{\j'}\cdot \tilde{\rho}_{\j}\circ f^{t-t'}_{\j'\to \j}.
\]
Here we have the additional factor $(\rho_{K_1}\circ f^{t-t'}_G\cdot \rho_{K_0}\circ f^{-t'}_G)\circ  \kappa_{\j'}$, but this hardly affects the argument below as we noted in Remark \ref{rem:difference}. 
The following lemma is the component-wise version of the claims  (\ref{eq:vomegatrans}) and (\ref{eq:vomegatrans2}).
\begin{Lemma}\label{lm:backward local}
There exist constants $\epsilon>0$ and  $C_\nu>0$ for any $\nu>0$, independent of  $\omega, \omega'\in \integer$ and $0\le t'\le 2t(\omega)$,   such that  we have 
\begin{align}\label{eq:vjj}
&\|\mathbf{T}_{\omega'}^{\sigma'}\circ \bL^{-(t-t'),\sigma\to \sigma'}_{\j\to \j'}\circ \mathbf{T}_{\omega}^{\sigma}:\bK^{r,\sigma}_{\j}
\to \bK^{r,\sigma'}_{\j'}\|\le
C_\nu    \langle \omega'-\omega\rangle^{-\nu} 
\intertext{and}
 &\|(1-\mathbf{T}_{\omega'}^{\sigma'})\circ \bL^{-(t-t'),\sigma\to \sigma'}_{\j\to \j'}\circ \mathbf{T}_{\omega}^{\sigma}:\bK^{r,\sigma}_{\j}
\to \bK^{r,\sigma'}_{\j'}\|\le
C_\nu  \langle \omega\rangle^{-\epsilon} \langle \omega'-\omega\rangle^{-\nu} 
\label{eq:vjj2}
\end{align}
for $\j, \j'\in \cJ$ with $\omega(\j)=\omega$, $\omega(\j')=\omega'$, $m(\j)=m(\j')=0$ and further
\begin{equation}\label{eq:vjj3}
\|\bL^{-(t-t'),\sigma\to \sigma'}_{\j\to \j'}\circ \mathbf{T}_{\omega}^{\sigma}:\bK^{r,\sigma}_{\j}
\to \bK^{r,\sigma'}_{\j'}\|\le
C_\nu   e^{- |m(\j')|}\cdot \langle \omega'-\omega\rangle^{-\nu} 
\end{equation}
for $\j, \j'\in \cJ$ with $\omega(\j)=\omega$, $\omega(\j')=\omega'$, $m(\j)=0$ and $m(\j')\neq 0$.  
\end{Lemma}

\begin{proof}[Proof of Lemma \ref{lm:backward local}] 
We apply the argument in the proofs of Lemma \ref{lm:crude_estimates2} and Corollary \ref{lm:distortion} to the time reversed system. Then, for $0\le t\le 2t(\omega)$, we get the decomposition $f^{-t}_{\j\to \j'}=a^{-t}_{\j\to \j'}\circ g^{-t}_{\j\to \j'}\circ B^{-t}_{\j\to \j'}$ corresponding to (\ref{eq:agb}), such that
\begin{enumerate}
\item $a^{-t}_{\j\to \j'}$ is an affine transform in the group $\cA_2$, 
\item the \emph{inverse} of $B^{-t}_{\j\to \j'}$ is a hyperbolic linear map of the form (\ref{eq:Bt}) with the linear maps $A:\real^{d}\to \real^d$ and $\widehat{A}:\real^{d'}\to \real^{d'}$ satisfying (\ref{eq:normA}), and
\item $g^{-t}_{\j\to \j'}$ is a fibered contact diffeomorphism and the family
\[
\cG_\omega=\{g^{-t}_{\j\to \j'}\mid\;\omega(\j)=\omega, \omega'=\omega(\j') \mbox{ satisfy (\ref{eq:close}), and } 0\le  t\le 2t(\omega)\},
\] 
fulfills the conditions (G0), (G1) and (G2) in  Setting II  in Section \ref{sec:pre}. 
\end{enumerate}
Also, in parallel to Corollary \ref{cor:rho_distortion}, we can show that the family 
\[
\mathcal{X}_{\omega}=\{(\bar{b}^{-t}_{\j'})^{-1}\cdot \tilde{b}^{-t}_{\j'}\cdot \rho^{-t}_{\j\to \j'}\mid \; \mbox{$\omega(\j)=\omega$, $\omega'=\omega(\j')$ satisfy  (\ref{eq:close}),  $0\le t\le 2t(\omega)$}\}
\]
satisfies the conditions (C1) and (C2) in Setting I in Section \ref{sec:pre}.  

\begin{Remark}
The main point of the argument below is that, though the lift of the transfer operator associated to $B^{-t}_{\j\to \j'}$ will not be bounded as an operator on $\bK^{r,\sigma}$, we have precise description about its inverse in Theorem \ref{cor:linear1}.\end{Remark}

To prove (\ref{eq:vjj}), we suppose  $m(\j)=m(\j')=0$ and write $\mathbf{T}_{\omega}^{\sigma'}\circ \bL^{-(t-t'),\sigma\to \sigma'}_{\j\to \j'}\circ \mathbf{T}_{\omega}^{\sigma}$ as the composition of 
\begin{align}\label{eq:op1a}
&\mult(X_{n_0(\omega')}) \circ \cT_0^{\lift}\circ\mult(\Psi^{\sigma'}_{\j})\circ \pBargmann\circ 
 L\left(a^{-t}_{\j\to \j'}\circ g^{-t}_{\j\to \j'}, {\rho}^{-(t-t')}_{\j\to \j'}\right)\circ \pBargmann^*
\qquad\text{and}\\
 &\pBargmann\circ L(B^{-(t-t')}_{\j\to \j'}, 1)\circ \pBargmann^*\circ \mult(X_{n_0(\omega)}) \circ 
 \cT_0^{\lift}. \label{eq:op2a}
\end{align}
Below we regard (\ref{eq:op2a}) as an operator from $\bK^{r,\sigma}_{\j}$ to  $L^2(\supp q_\omega,(\cW^{r,\sigma'+1})^2)$ and (\ref{eq:op1a}) as that from $L^2(\supp q_\omega,(\cW^{r,\sigma'+1})^2)$ to $\bK^{r,\sigma'}_{\j'}$.

From Theorem \ref{cor:linear1} for $B=(B^{-(t-t')})^{-1}$, we see that the operator norm of
\[
\pBargmann\circ L(B^{-(t-t')}_{\j\to \j'}, 1)\circ \pBargmann^*\circ 
 \cT_0^{\lift} :L^2(\supp \Psi_{\j}^{\sigma},(\cW^{r,\sigma})^2)\to L^2(\supp q_{\omega}, (\cW^{r,\sigma})^2) 
\]
is bounded by $C_0|\det \widehat{A}|^{-1} |\det {A}|^{1/2}$. The difference of this operator and (\ref{eq:op2a}) is $\mult(X_{n_0(\omega)})$ in the middle of (\ref{eq:op2a}). But, by the localized property of the kernel of $\cT_0^{\lift}$ in Corollary~\ref{cor:Tzero} and also by Lemma \ref{lm:expression_lift},  we see that the insertion of $\mult(X_{n_0(\omega)})$ makes only a negligible difference bounded by $C\langle \omega\rangle^{-\theta}$. Hence the operator norm of (\ref{eq:op2a}) is bounded by $C_0|\det \widehat{A}|^{-1} |\det {A}|^{1/2}$.
The operator norm of (\ref{eq:op1a}) is bounded by $C_\nu\bar{b}^{-t}_{\j'}\langle \omega(\j')-\omega(\j)\rangle^{-\nu}$ for any $\nu>0$, by Lemma~\ref{lm:IdMinusG} and Corollary~\ref{cor:multi}. 
Since $\bar{b}^{-t}_{\j'}\le C |\det \widehat{A}|^{-1} |\det {A}|^{1/2}$, the claim 
(\ref{eq:vjj}) follows from these estimates. 

To prove the claim (\ref{eq:vjj2}), we express the operator on its left-hand side  
as the composition of (\ref{eq:op1a}) and (\ref{eq:op2a}), with  $\cT_0^{\lift}$ in (\ref{eq:op1a}) replaced by $(1-\cT_0^{\lift})$. 
The operator $\cT_0^{\lift}$ commutes with 
$
\pBargmann\circ L(B^{-(t-t')}_{\j\to \j'}, 1)\circ \pBargmann^*$ and approximately with  
$
\mult(\Psi^{\sigma'}_{\j})$ and $\mult(X_{n_0(\omega)})$, {\it i.e.} producing negligible terms bounded by $C_\nu  \langle \omega\rangle^{-\epsilon} \langle \omega'-\omega\rangle^{-\nu} $,  by the localized property of the kernel of $\cT_0^{\lift}$ in Corollary~\ref{cor:Tzero}.
Also we see that $\cT_0^{\lift}$ commutes approximately with  $\pBargmann\circ 
 L(a^{-t}_{\j\to \j'}\circ g^{-t}_{\j\to \j'}, {\rho}^{-(t-t')}_{\j\to \j'})\circ \pBargmann^*$ by using  Lemma~\ref{lm:IdMinusG}, Lemma \ref{lm:mphi} and Lemma \ref{lm:TM}.
Therefore we get (\ref{eq:vjj2}) from the fact that $\cT_0^{\lift}$ is a projection operator.

To prove the last claim (\ref{eq:vjj3}), we express the operator on its left-hand side as the composition of (\ref{eq:op1a}) and (\ref{eq:op2a}), but now without the term $\mult(X_{n_0(\omega')}) \circ \cT_0^{\lift}$ in (\ref{eq:op1a}). 
On the one hand,  from the properties of $\cT_0^{\lift}$ we have mentioned above, the image of (\ref{eq:op2a}) concentrates around the trapped set $X_0$ if we view it in the scale $\langle \omega'\rangle^{-1/2}$ and through the weight function $\cW^{r,\sigma'}$.
On the other hand, since we have $|m(\j')|\ge n_0(\omega')=[\theta\cdot \log \langle \omega'\rangle]$ from the assumption, the distance of the support of $\Psi_{\j'}^{\sigma'}$ from the trapped set $X_0$ is not less than 
 $e^{|m(\j')|}\cdot \langle \omega'\rangle^{-1/2}\gtrsim \langle \omega'\rangle^{-1/2+\theta}$ . 
Therefore we conclude the claim (\ref{eq:vjj3}), applying Lemma \ref{lm:IdMinusG} (and Lemma \ref{lm:multiplication}) for  the operator  $\pBargmann\circ 
 L(a^{-t}_{\j\to \j'}\circ g^{-t}_{\j\to \j'}, {\rho}^{-(t-t')}_{\j\to \j'})\circ \pBargmann^*$.
\end{proof}

The claims of Proposition \ref{pp:backward_lifted} are obtained by summing the estimates for the components in Lemma \ref{lm:backward local}. 
We actually have to deal with the same problem as that in deducing Proposition \ref{lm:TL} from Lemma \ref{lm:main}. 
But we omit it because the argument is exactly same as that in the proof of Proposition \ref{lm:TL}.  
\end{proof}

\subsection{The hyperbolic part}
We next consider the hyperbolic part $\bL^{t,\sigma\to \sigma'}_{\mathrm{hyp}}$. 
We decompose it further into two parts.
For this, we first introduce the new index  
\begin{equation}\label{def:tm}
\tilde{m}(\j)=\begin{cases}
m(\j)+(1/2)\log \langle \omega(\j)\rangle, &\mbox{if $m(\j)>0$};\\
m(\j)-(1/2)\log \langle \omega(\j) \rangle, &\mbox{if $m(\j)<0$};\\
0, &\mbox{if $m(\j)=0$}. 
\end{cases}
\end{equation}
Recall that the frequency vector  of the wave packet $\phi_{w,\xi_w,\xi_z}(\cdot)$ is 
$(\langle \xi_z\rangle \xi_w, \xi_z)$ and its distance from the trapped set $X_0$, disregarding the normalization mentioned in Remark \ref{Rem:rescaledCoordinates},  is $\langle \xi_z\rangle^{1/2} |(\zeta_p,\tilde{\xi}_y,\zeta_q,\tilde{y})|$. 
The absolute value of $\tilde{m}(\j)$ is directly related to this distance (without normalization).
Indeed, from the definition, we have
\[
e^{|\tilde{m}(\j)|-1}\le \langle \xi_z\rangle^{1/2} |(\zeta_p,\tilde{\xi}_y,\zeta_q,\tilde{y})| \le e^{|\tilde{m}(\j)|+1}\quad \mbox{for $(w,\xi_w,\xi_z)\in \supp \Psi_{\j}^{\sigma}$,}
\]
provided that $0<|m(\j)|\le n_1(\omega)$. (In the case $|m(\j)|> n_1(\omega)$, we can get a similar estimate but need modification related to the factor $E_{\omega,n}$ in the definition of $\Psi_{\j}^{\sigma}$.) 
Motivated by the observations (Ob3) in Subsection \ref{ss:lifted}, we introduce
\begin{Definition}[The relation $\hookrightarrow^t$]\label{def:arrow} We write $\j \hookrightarrow^t \j'$ for $\j,\j'\in \cJ$ and $t\ge 0$ if either of the following conditions holds true:
\begin{enumerate}
\item $\tilde{m}(\j)\le 0$ and $\tilde{m}(\j')\ge  0$, or
\item $\tilde{m}(\j)\cdot \tilde{m}(\j')>0$ and $\tilde{m}(\j')\ge \tilde{m}(\j) +[t\chi_0]-10$.
\end{enumerate}
Otherwise we write $\j \not\hookrightarrow^t\j'$. 
\end{Definition}
\begin{Remark}\label{rem:notarrow}
For the argument below, we ask the readers  to observe that the relation $\j \not\hookrightarrow^t\j'$ implies that $((Df^t_{\j\to \j'})^{-1})^*(\supp \Psi_{\j}^{\sigma})$ is separated from  $\supp \Psi_{\j'}^{\sigma'}$ for $t$ satisfying the condition (\ref{eq:condition_on_t}), at least if $\omega=\omega(\j)$ and $\omega'=\omega(\j')$ satisfy (\ref{eq:close}) and that 
$|m(\j)|,|m(\j')|\le n_1(\omega)$. (For this, Remark \ref{Rem:index} will be useful.) 
We will give a related quantitative estimate in Lemma \ref{lm:basic2}.
\end{Remark}
Correspondingly to the definition above,  we decompose the hyperbolic part $\bL^{t,\sigma\to \sigma'}_{\mathrm{hyp}}$ into two parts as follows.
 For $t\ge 0$ and $\bu=(u_{\j})_{\j\in \cJ}$, we  set
\[
\bL^{t,\sigma\to \sigma'}_{\mathrm{hyp},\hookrightarrow} \bu=
\left(\sum_{\j:\j\hookrightarrow^t \j'} \bL^{t,\sigma\to \sigma'}_{\j\to \j'} u_{\j}\right)_{\j'\in \cJ}
\]
and
\[
\bL^{t,\sigma\to \sigma'}_{\mathrm{hyp},\not\hookrightarrow} \bu=
\left(\sum_{\j:\j\not\hookrightarrow^t \j'} \bL^{t,\sigma\to \sigma'}_{\j\to \j'} u_{\j}\right)_{\j'\in \cJ}
\]
where the sum $\sum_{\j:\j\hookrightarrow^t \j'} $ (resp.\ $\sum_{\j:\j\not\hookrightarrow^t \j'}$) denotes that over $\j$ such that $ \bL^{t,\sigma\to \sigma'}_{\j\to \j'}$ is a hyperbolic component and that $\j\hookrightarrow^t \j'$ (resp.\  $\j\not\hookrightarrow^t \j'$) holds. Obviously we have
\begin{equation}\label{eq:decomphcL}
{\bL}^{t,\sigma\to \sigma'}_{\mathrm{hyp}}=\bL^{t,\sigma\to \sigma'}_{\mathrm{hyp},\hookrightarrow}+\bL^{t,\sigma\to \sigma'}_{\mathrm{hyp},\not\hookrightarrow}.
\end{equation}

By geometric consideration based on the observations (Ob1) and (Ob3) discussed in Subsection \ref{ss:lifted} and crude estimates using Corollary \ref{cor:coarse}, we will see that a hyperbolic component $\bL^{t,\sigma\to \sigma'}_{\j\to \j'}$ satisfying $\j \not\hookrightarrow^t \j'$  is (extremely) small in the trace norm as well as in the operator norm and, consequently, so is
the latter part $\bL^{t,\sigma\to \sigma'}_{\mathrm{hyp},\not\hookrightarrow}$.  The former part $\bL^{t,\sigma\to \sigma'}_{\mathrm{hyp},\hookrightarrow}$ will not be small if we view it in the $L^2$ norm. But, recall from (\ref{eq:bKnorm})  that  the norm on $\bK^{r,\sigma}$ counts the component in  $\bK^{r,\sigma}_\j$ with  the weight $2^{-r m(\j)}$ (provided $m(\j)\neq 0$).  This weight and the definition of the relation $\hookrightarrow^t$ allow us to show that the latter part 
 $\bL^{t,\sigma\to \sigma'}_{\mathrm{hyp},\hookrightarrow}$  has a small operator norm. 
\begin{Proposition}\label{prop:hyp} Let $\sigma, \sigma'\in \Sigma$. 
The  hyperbolic part\/ $\bL^{t,\sigma\to \sigma'}_{\mathrm{hyp}}:\bK^{r,\sigma}\to \bK^{r,\sigma'}$ is bounded for $0\le t\le 2t_0$  provided that  (\ref{eq:condition_on_t}) holds with respect to $\sigma$ and $\sigma'$. There exist constants $C_\nu>0$ and $C'_\nu>0$ for each $\nu>0$ such that 
\begin{equation}\label{claim:hyp1}
\|\bPi_{\omega'}\circ \bL^{t,\sigma\to \sigma'}_{\mathrm{hyp},\hookrightarrow}\circ \bPi_{\omega}\|\le 
C_\nu   e^{-(r/2) \chi_0 t}\cdot \langle \omega'-\omega\rangle^{-\nu}
\end{equation}
and
\begin{equation}\label{claim:hyp2}
\|\bPi_{\omega'}\circ \bL^{t,\sigma\to \sigma'}_{\mathrm{hyp},\not\hookrightarrow}\circ \bPi_{\omega}\|  \le \|\bPi_{\omega'}\circ \bL^{t,\sigma\to \sigma'}_{\mathrm{hyp},\not\hookrightarrow}\circ \bPi_{\omega}\|_{\trace}\le C'_\nu \langle \omega  \rangle^{-\nu}\langle \omega'-\omega\rangle^{-\nu}
\end{equation}
for any $\omega,\omega'\in \integer$ and $0\le t\le 2t(\omega)$ provided that the condition   (\ref{eq:condition_on_t}) holds.
\end{Proposition}
\begin{proof}
The first claim follows from the claims (\ref{claim:hyp1}) and (\ref{claim:hyp2}). 
Below we first prove  (\ref{claim:hyp1}) on the part $\bL^{t,\sigma\to \sigma'}_{\mathrm{hyp},\hookrightarrow}$.
We restrict ourselves to the case where (\ref{eq:close}) holds for $\omega, \omega'\in \integer$ by the same reason as in the proofs of a few previous propositions. 
First we prove the following lemma for the components.
\begin{Lemma}\label{Claim3a} 
Suppose that  $\omega, \omega'\in \integer$ satisfy (\ref{eq:close}). There exists a constant $C_\nu>0$ for each $\nu>0$, independent of $\omega$ and $\omega'$, such that, 
if\/  $\bL^{t,\sigma\to \sigma'}_{\j\to \j'}$ for $\j,\j'\in \cJ$ with $\omega(\j)=\omega$ and $\omega(\j')=\omega'$ is a component of $\bL^{t,\sigma\to \sigma'}_{\mathrm{hyp},\hookrightarrow}$ and if $0\le t\le 2t(\omega)$ satisfies (\ref{eq:condition_on_t}), then we have 
\begin{equation}\label{eq:lhypnorm}
 \left\| \bL^t_{\j\to \j'}:\bK^{r,\sigma}_{\j}\to \bK^{r,\sigma'}_{\j'}\right\|\le C_\nu 
 \bar{b}_{\j'}^t\cdot  \Lambda_\nu(m(\j),\omega(\j);m(\j'),\omega(\j');t)
 \end{equation}
 where $\bar{b}_{\j'}^t$ is that defined in (\ref{eq:barbt}) and we set 
 \[
\Lambda_\nu(m, \omega; m',\omega';t)=\begin{cases}
e^{-rm'+rm}\langle \omega'-\omega\rangle^{-\nu},& \mbox{ if $m\neq 0$, $m'\neq 0$;}\\
\min\{e^{-r\chi_0 t}, e^{-r (m'-n_{0}(\omega))}\}\langle \omega'-\omega\rangle^{-\nu},& \mbox{ if $m= 0$, $m'> 0$;}\\
\min\{e^{-r\chi_0 t}, e^{r(m+n_{0}(\omega))}\}\langle \omega'-\omega\rangle^{-\nu},& \mbox{ if $m< 0$,  $m'= 0$.}
\end{cases}
\]
\end{Lemma}
\begin{proof} In the case where $m(\j)\neq 0$ and $m(\j')\neq 0$, the conclusion is a direct consequence of Lemma \ref{cor:multi2} and the definition (\ref{eq:bKnorm}) of the weight on $\bK^{r,\sigma}_\j$. Below we prove the lemma for the case  $m(\j)= 0$, but the case  $m(\j')= 0$ is  proved in a parallel manner.

Recall that the image and target spaces of $\bL^t_{\j\to \j'}$ are
\[
\bK^{r,\sigma'}_{\j'}=L^2(\supp \Psi_{\j'}^{\sigma'},2^{-2rm(\j')}), \quad
\bK^{r,\sigma}_\j=L^2(\supp \Psi_{\j}^{\sigma},(\cW^{r,\sigma})^2).
\] 
If $e^{-r\chi_0 t}\ge e^{-r (m(\j')-n_{0}(\omega(\j)))}$, we get the required estimate easily by Lemma \ref{cor:multi2} because $C^{-1} e^{-r n_0(\omega)}\le \cW^{r,\sigma}(\cdot)\le C e^{+r n_0(\omega)}$ on $\supp \Psi^{\sigma}_{\j}=\supp \Psi^{\sigma}_{\omega, 0}$.
Below we assume $e^{-r\chi_0 t}<e^{-r (m(\j')-n_{0}(\omega(\j)))}$. 
But this implies   
\begin{equation}\label{eq:smallm}
e^{m(\j')-n_{0}(\omega(\j))}< e^{\chi_0 t} < \langle \omega(\j')\rangle^{\theta}\quad \mbox{and so}\quad 
e^{m(\j')}< \langle \omega(\j')\rangle^{2\theta}
\end{equation}
and hence both of the supports of  $\Psi^{\sigma'}_{\j'}$ and $\Psi^{\sigma}_{\j}$ are contained in that of the function  $Y$ in (\ref{def:Y}). This enables us to use the argument we have used in the proof of Lemma \ref{lm:main}.
We express $f^t_{\j\to \j'}$ as in Corollary \ref{lm:distortion} and note that, if the non-linear diffeomorphism $g^t_{\j\to \j'}$ were identity, we could conclude (\ref{eq:lhypnorm}) with $\Lambda_\nu(m,\omega;m',\omega';t)=e^{-r\chi_0 t}\langle \omega'-\omega\rangle^{-\nu}$ immediately from the precise estimate on the kernel of $\bL^t_{\j\to \j'}$ in Lemma \ref{lm:IdMinusG} and  Lemma \ref{lm:expression_lift}. 
But, by virtue of Lemma \ref{lm:multiplication} and the estimate  (\ref{eq:kernel_Lg}) in its proof, we may indeed replace  $g^t_{\j\to \j'}$ by the identity producing a negligible error term bounded by $C_\nu \langle \omega\rangle^{-\theta}\langle \omega'-\omega\rangle^{-\nu}$. 
\end{proof}

We deduce the claim  (\ref{claim:hyp1}) from the estimates (\ref{eq:lhypnorm}) by the argument parallel to that in the latter part of the proof of Proposition  \ref{lm:TL}. To begin with we note the following estimate;   
Since  $\bar{b}_{\j'}^t<e^{(1/4)r\chi_0 t}$ from the choice of $r$ in (\ref{eq:choice_of_r4}), we have
\begin{align}\label{eq:sumLambda}
&\sup_{m} \sum_{m'} \bar{b}_{\j'}^t\cdot \Lambda_\nu(m,\omega;m',\omega';t)\le C_\nu e^{-(1/2)r\chi_0 t}\langle \omega'-\omega\rangle^{-\nu}
\intertext{and} 
&\sup_{m'} \sum_{m} \bar{b}_{\j'}^t\cdot \Lambda_\nu(m,\omega;m',\omega';t)\le C_\nu e^{-(1/2)r\chi_0 t}\langle \omega'-\omega\rangle^{-\nu}.
\label{eq:sumLambda2}
\end{align}
We write $\bL_{\j\to \j'}$ for the $\j\to \j'$ component of $\bL^{t,\sigma\to \sigma'}_{\mathrm{hyp},\hookrightarrow}$ for brevity. As in the proof of Proposition \ref{lm:TL}, we approximate it by  $\widetilde{\bL}_{\j\to \j'}=\mult(\mathbf{1}'_{\j\to \j'})\circ \bL_{\j\to \j'} \circ \mult(\mathbf{1}_{\j\to \j'})$ where  $\mathbf{1}'_{\j\to \j'}$ and $\mathbf{1}_{\j\to \j'}$ are those defined in the paragraph preceding  (\ref{eq:wtM}). It is not difficult to see
\[
\|\bL_{\j\to \j'}-\widetilde{\bL}_{\j\to \j'}:\bK^{r,\sigma}_{\j}\to \bK^{r,\sigma'}_{\j'} \|\le C_{\nu}\langle \omega\rangle^{-2\theta}\cdot \Lambda_\nu(m(\j),\omega(\j);m(\j'),\omega(\j');t)
\]
from the proof of Lemma \ref{Claim3a}. Hence, from the uniform boundedness of the intersection multiplicity of the sets (\ref{eq:Djj}) and Remark \ref{rem:cardD}, we obtain  
\begin{align*}
\|\bPi_{\omega'}\circ \bL^{t,\sigma\to \sigma'}_{\mathrm{hyp},\hookrightarrow}\circ \bPi_{\omega} \bu\|^2_{\bK^{r,\sigma'}}
&\le \sum_{m'}\;\sum_{\j':\omega',m'}
\left\|\sum_{m}\;\sum_{\j:\omega, m} \widetilde{\bL}_{\j\to\j'} u_{\j}\right\|_{\bK^{r,\sigma'}_{\j'}}^2\\
&\qquad +C_\nu \langle \omega\rangle^{-\theta} \langle \omega'-\omega\rangle^{-\nu} \|\bu\|_{\bK^{r,\sigma}}^2
\end{align*}
for $\bu=(u_{\j})_{\j\in \cJ}\in \bK^{r,\sigma}$, 
where $\sum_{\j:\omega,m}$ denotes the sum over $\j\in \cJ$ such that $\omega(\j)=\omega$ and $m(\j)=m$ and so on. For the sum of the right-hand side, we have 
\begin{align*}
&\sum_{m'}\;\sum_{\j':\omega',m'}
\left\|\sum_{m}\;\sum_{\j:\omega, m} \widetilde{\bL}_{\j\to\j'} u_{\j}\right\|_{\bK^{r,\sigma'}_{\j'}}^2\\
&\le C_\nu  e^{-r\chi_0 t/2}\langle \omega'-\omega\rangle^{-\nu} \sum_{m,m'}\sum_{\j':\omega', m'}(\bar{b}_{\j'}^t\cdot \Lambda_\nu(\omega,m;\omega',m';t))^{-1} \left\|\;\sum_{\j:\omega, m} \bL_{\j\to\j'} u_{\j}\right\|_{\bK^{r,\sigma'}_{\j'}}^2\\
&\le C_\nu   e^{-r\chi_0 t/2} \langle \omega'-\omega\rangle^{-\nu}\sum_{m,m'}\sum_{\j:\omega, m}
\bar{b}_{\j'}^t\cdot\Lambda_\nu(\omega,m;\omega',m';t)\left\|u_{\j}\right\|_{\bK^{r,\sigma}_{\j}}^2\\
&\le C_\nu  e^{-r\chi_0 t} \langle \omega'-\omega\rangle^{-2\nu}\sum_{m}\sum_{\j:\omega, m}
\left\| u_{\j}\right\|_{\bK^{r,\sigma}_{\j}}^2=  C_\nu  e^{-r\chi_0} \langle \omega'-\omega\rangle^{-2\nu}
\left\|\bu\right\|_{\bK^{r,\sigma}}^2
\end{align*} 
where, besides boundedness of the intersection multiplicities of (\ref{eq:Djj}),  we used Schwartz inequality and (\ref{eq:sumLambda2}) in the first inequality, 
Lemma \ref{Claim3a} in the second, and (\ref{eq:sumLambda}) in the third. 
We therefore obtain  the required estimate (\ref{claim:hyp1}).

We next prove the claim (\ref{claim:hyp2}). We deduce it from the following lemma. Recall the definition of the quantity $\Delta_{\j\to \j'}^{t, \sigma\to\sigma'}$ preceding Corollary \ref{cor:coarse}. 
\begin{Lemma}\label{lm:basic2}
Suppose that $\sigma, \sigma'\in \Sigma$ and consider a (non-zero) component  $\bL^{t,\sigma\to \sigma'}_{\j\to \j'}$ of $\bL^{t,\sigma\to \sigma'}_{\mathrm{hyp},\not\hookrightarrow}$ and  $0<t\le 2t(\omega(\j))$ satisfying (\ref{eq:condition_on_t}). Then there  exist constants $\gamma_0>0$ and $C_0>0$ (independent of $\j, \j'$ and $t$) such that  
\begin{equation}\label{Delta}
\Delta_{\j\to \j'}^{t, \sigma\to\sigma'}<C_0 \max\{\omega(\j), \omega(\j'), e^{|m(\j)|}, e^{|m(\j')|}\}^{-\gamma_0}.
\end{equation}
\end{Lemma}
The conclusion of Lemma \ref{lm:basic2} is just an estimate between the supports of 
$(Df^t_{\j\to j'})^*(\Psi^{\sigma'}_{\j'})$ and $\Psi^\sigma_{\j}$ and hence should be obtained by elementary geometric consideration. 
Also, from Remark \ref{rem:notarrow} and the construction of $E_{\omega,m}$ in Subsection \ref{ss:modified_pu}, the claim may be intuitively rather obvious. This is indeed the case, however, because of the involved definition of $\Psi^{\sigma}_{\omega,m}$, we need to separate several cases and go through cumbersome estimates that are not very essential.   
Thereby we defer the proof of Lemma \ref{lm:basic2} to Section \ref{ssA} in the appendix. 

Once we obtain Lemma \ref{lm:basic2}, we combine it with  Corollary \ref{cor:coarse} to bound the trace norms of the components of $\bL^{t,\sigma\to \sigma'}_{\mathrm{hyp},\not\hookrightarrow}$. 
Observe that the term $(\Delta^{t, \sigma\to \sigma'}_{\j\to \j'})^{\nu}$ in the bound thus obtained dominates the latter factor in  (\ref{eq:CDelta}) provided $\nu$ is sufficiently large. Hence we  conclude the claim (\ref{claim:hyp2}) simply by summing such bounds on the trace norms. 
\end{proof}

\begin{Remark}\label{Rem:const}
For the proof of Theorem \ref{th:zeta} in Section \ref{sec:zeta} in the appendix, we will actually need a little more information above the constants $C_\nu$ and $C'_\nu$ in the claims of Proposition \ref{prop:hyp}. Note that the constant $C_\nu$ in (\ref{claim:hyp1}) comes from that in Corollary \ref{cor:multi2} and can be chosen independently of the choice of $t_0>0$ if we take larger $k_0$ according to $t_0$. Hence, by letting $t_0$ be larger if necessary, we may suppose that 
$\|\bL^{t,\sigma\to \sigma'}_{\mathrm{hyp},\hookrightarrow}:\bK^{r,\sigma}\to \bK^{r,\sigma'}\|\le e^{-(r/2) \chi_0 t}$ for $t_0\le t\le 2t_0$. 
Also it is easy to see that, by letting the constant $k_0$ in Definition \ref{Def:partsOfOperators} be larger, we may restrict the sum over $(\j,\j')\in \cJ^2$ and  let $C'_\nu$ in (\ref{claim:hyp2}) be arbitrarily small. 
\end{Remark}

From the definition of the hyperbolic part $\bL^{t, \sigma\to \sigma'}_{\mathrm{hyp}}$, its components $\bL^{t,\sigma\to \sigma'}_{\j\to \j'}$ satisfy either $m(\j)\neq 0$ or $m(\j')\neq 0$, that is, either $|m(\j)|> n_0(\omega(\j))$ or $|m(\j')|> n_0(\omega(\j'))$. 
Hence either of the supports of $\Psi^{\sigma}_{\j}$ or $\Psi^{\sigma'}_{\j'}$ are separated from the trapped set $X_0$ by the distance proportional to $\langle \omega(\j)\rangle^{-1/2+\theta}$ or $\langle \omega(\j')\rangle^{-1/2+\theta}$. 
On the other hand, from Corollary \ref{cor:Tzero}, the kernel of $\cT_0^{\lift}$ is localized around the trapped set $X_0$ in the scale $\langle \omega\rangle^{-1/2}$ if we view it through the weight $\cW^{r,\sigma}$. 
These observations lead to
\begin{Lemma} \label{lm:center_peri}
Let  $\sigma, \sigma'\in \Sigma$. There exists a constant   $C_\nu>0$ for $\nu>0$ such that 
\begin{align}
&\|\mathbf{T}_{\omega'}^{\sigma'} \circ \bL^{t, \sigma\to \sigma'}_{\mathrm{hyp}}\circ \bPi_{\omega}\|\le C_\nu \langle \omega\rangle^{-\theta} \langle \omega'-\omega\rangle^{-\nu}\quad \mbox{and}\\
&\|\bPi_{\omega'}\circ  \bL^{t, \sigma\to \sigma'}_{\mathrm{hyp}}\circ \mathbf{T}^{\sigma}_{\omega}\|\le C_\nu \langle \omega\rangle^{-\theta}  \langle \omega'-\omega\rangle^{-\nu}\label{eq:bLt_remainder2}
\end{align}
for any $\omega, \omega'\in \integer$ and $0\le t\le 2t(\omega)$.  
\end{Lemma}
\begin{proof}
From the observations made above, it is not difficult to see that 
\[
\|\mathbf{T}_{\omega'}^{\sigma'}\circ \bL_{\j\to \j'}^{t,\sigma\to \sigma'}:
\bK^{r,\sigma}_{\j}\to \bK^{r,\sigma'}_{\j'}\|\le  C_\nu e^{-|m(\j)|} \cdot \langle \omega\rangle^{-\theta} \langle  \omega'-\omega\rangle^{-\nu}
\]
for $\j,\j'\in \cJ$ with $\omega(\j)=\omega$, $m(\j)\neq 0$ and $\omega(\j')=\omega'$, $m(\j')= 0$, and also that
\[
\| \bL_{\j\to \j'}^{t,\sigma\to \sigma'}\circ \mathbf{T}_{\omega}^\sigma:
\bK^{r,\sigma}_{\j}\to \bK^{r,\sigma'}_{\j'}\|\le  C_\nu e^{-|m(\j')|} \cdot\langle \omega\rangle^{-\theta}  \langle  \omega'-\omega\rangle^{-\nu} 
\]
for $\j,\j'\in \cJ$ with $\omega(\j)=\omega$, $m(\j)=0$ and $\omega(\j')=\omega'$, $m(\j')\neq 0$.
Then we obtain the claims of Lemma \ref{lm:center_peri}, adding these estimates for the components by using the  argument parallel to that in the latter part of the proof of Proposition \ref{lm:TL}. 
\end{proof}

\subsection{The low frequency part}
For the low frequency part, we prove
\begin{Lemma}\label{lm:low} Let $\sigma, \sigma'\in \Sigma$. 
The low frequency part\/ $\bL^{t,\sigma\to \sigma'}_{\mathrm{low}}:\bK^{r,\sigma}\to \bK^{r,\sigma'}$ for $0\le t\le 2t_0$  are trace class operators. Further, there exists a constant $C_\nu>0$ for each $\nu>0$ such that 
\[
\|\bPi_{\omega'}\circ \bL^{t,\sigma\to \sigma'}_{\mathrm{low}}\circ \bPi_{\omega}\|
\le 
\|\bPi_{\omega'}\circ\bL^{t,\sigma\to \sigma'}_{\mathrm{low}}\circ \bPi_{\omega}\|_{\trace}
\le C_\nu\langle \omega'\rangle^{-\nu}\langle \omega\rangle^{-\nu}
\]
for any $\omega,\omega'\in \integer$ and  $0\le t\le 2t(\omega)$.  
\end{Lemma}
The lemma above is obtained immediately by adding the estimates on the trace norms of the components of $\bL^{t,\sigma\to \sigma'}_{\mathrm{low}}$ given by the next lemma and Corollary \ref{cor:coarse}. 
 
\begin{Lemma}\label{lm:basic3}
Let $\sigma, \sigma'\in \Sigma$.  
Suppose that   $\bL^{t,\sigma\to \sigma'}_{\j\to \j'}$ for $\j, \j'\in \cJ$ is a low frequency component of 
$\bL^{t,\sigma\to \sigma'}$ and  that $0<t\le 2t(\omega(\j))$ satisfies (\ref{eq:condition_on_t}) w.r.t.~$\sigma$ and $\sigma'$. Then there  exist constants $\gamma_0>0$ and $C_0>0$ (independent of $\j, \j'$ and $t$ but dependent on the constant $k_0$) such that  
\begin{equation}\label{Delta_bound}
\Delta_{\j\to \j'}^{t, \sigma\to\sigma'}< C_0 \max\{\langle\omega(\j)\rangle, \langle \omega(\j')\rangle, e^{|m(\j)|}, e^{|m(\j')|}\}^{-\gamma_0}.
\end{equation}
\end{Lemma}
Since we can take large constant $C_0$ in the statement above depending on  $k_0$, we can prove the claims above by crude geometric estimates 
on the supports of $\Psi^{\sigma}_{\omega,m}$. We omit the proof because it is straightforward and is obtained as an easier case of the proof of Lemma \ref{lm:basic2} in the appendix.

\subsection{The operator $\mathbf{T}^{\sigma\to \sigma'}_{\omega}$}
Let us recall the operator $\mathbf{T}^{\sigma\to \sigma'}_\omega$ from Definition \ref{Def:bT}.
The following  is the counterpart of Lemma \ref{lm:trace1}.
\begin{Lemma}\label{lm:trace_lifted} Let $\sigma, \sigma'\in \Sigma$.  The operator $\mathbf{T}^{\sigma\to \sigma'}_\omega:\bK^{r,\sigma}\to \bK^{r,\sigma'}$ is bounded  and its operator norm is bounded uniformly in $\omega$.
There is a constant $C_0>0$ such that 
\begin{equation}
\label{eq:Tlifted2}
\|\bI^{\sigma'}\circ \mult(1-\rho_{K_1})\circ (\bI^\sigma)^*\circ\mathbf{T}^{\sigma}_\omega:\bK^{r,\sigma}\to \bK^{r,\sigma'}\|\le C_0 \langle \omega\rangle^{-\theta}
\end{equation}
for any $\omega\in \mathbb{Z}$. For the trace norm, we have that
\begin{equation}\label{eq:traceTlifted}
\|\bI^{\sigma'}\circ \mult(\rho_{K_1})\circ (\bI^\sigma)^*\circ\mathbf{T}^{\sigma}_\omega:\bK^{r,\sigma}\to \bK^{r,\sigma'}\|_{\trace}\le C_0 \langle \omega\rangle^d
\end{equation}
for any $\omega\in \mathbb{Z}$.
Further, for each $\omega\in \integer$ with $|\omega|> k_0$,  there exists  a finite dimensional vector subspace $V(\omega)\subset \cK^{r,\sigma}(K_1)$ with  
$\dim V(\omega)\ge \langle \omega\rangle^{d}/C_0$ such  that 
\begin{equation}\label{eq:trace_V}
\|\bI^{\sigma'}\circ \mult(\rho_{K_1})\circ (\bI^{\sigma})^*\circ \mathbf{T}^{\sigma}_\omega \circ \bI^\sigma u\|_{\bK^{r,\sigma'}}\ge C_0^{-1} \|u\|_{\cK^{r,\sigma}}\quad \mbox{for all $u\in V(\omega)$. }
\end{equation}
\end{Lemma}
\begin{proof}
Recall the definition (\ref{eq:mathbfT}) of $\mathbf{T}^{\sigma\to \sigma'}_\omega$ and note that we have a precise description of the kernel of $\cT_0^{\lift}$ from  Corollary \ref{cor:Tzero}.
Then we obtain the uniform boundedness of the operator norm of  $\mathbf{T}^{\sigma\to \sigma'}_\omega$ and the claim (\ref{eq:Tlifted2}) as immediate consequences.
 To prove  (\ref{eq:traceTlifted}), 
it is enough to show, for some constant $C>0$ independent of $\omega$, that
\[
\|\bI^{\sigma'}\circ \mult(\rho_{K_1})\circ (\bI^\sigma_\j)^*
\circ \mult(X_{n_0(\omega)}) \circ \cT_0^{\lift}:\bK^{r,\sigma}_{\j}\to \bK^{r,\sigma'}\|_{\trace}\le C\langle \omega\rangle^{2d\cdot\theta}
\]
for any $\j\in \cJ$ with $\omega(\j)=\omega$ and $m(\j)=0$, because the cardinality of the element $\j\in \cJ$ satisfying $\omega(\j)=\omega$ and $m(\j)=0$ is bounded by $C\langle \omega\rangle^{(1/2-\theta)\cdot 2d}$. 
To prove this claim, we express the operator on the left-hand side above as integration of rank one operators by applying  Corollary \ref{cor:Tzero} to $\cT_0^{\lift}$. Applying Lemma \ref{lm:IdMinusG} to the $\j\to \j'$ component of $\bI^{\sigma'}\circ \mult(\rho_{K_1})\circ (\bI^{\sigma})^*$ with $m(\j)=0$, we see that those rank one operators are uniformly bounded. Therefore we get the required inequality by using triangle inequality on the trace norm.

We prove the last claim as a consequence of 
Lemma \ref{lm:choiceV}. 
Suppose that $c>0$ is sufficiently small. Then, for each $\omega\in \integer$, we can choose a finite  subset $\cJ_\omega$ in $\{\j\in \cJ\mid \omega(\j)=\omega, m(\j)=0\}$ so that $\# \cJ_{\omega}>c\langle \omega\rangle^{(1/2-\theta)\cdot 2d}$ and that the supports of  $\tilde{\rho}_{\j}\circ \kappa_{\j}^{-1}$ for $\j\in \cJ_\omega$ are separated by distance not less than $c\langle \omega\rangle^{-1/2+\theta}$. 
By choosing the points in $\cJ_\omega$ appropriately (and also $\rho_a$ if necessary), we may and do assume that, for some $\epsilon>0$ independent of $\omega$, the function $\rho_{\j}$ for each $\j\in \cJ_\omega$ takes constant value $1$ 
on the subset 
\[
 \{(w,z)\in \real^{2d+d'+1}\mid |w|\le \epsilon\langle \omega\rangle^{-1/2+\theta},\;
|z|\le \epsilon\}.
\]
Let $W(\omega)$ be the finite dimensional subspace given in Lemma \ref{lm:choiceV}  and put
\[
V(\omega):=\bigoplus_{\j\in \cJ_{\omega}}(\kappa_{\j}^{-1})^{*}(W(\omega))\subset C^\infty(K_0).
\]
This is a direct sum because the subspaces on the right-hand side are almost orthogonal\footnote{Here and a few lines below, we mean by ``almost orthogonal" that we have the estimate corresponding to (\ref{eq:almostorthogonal}). } from the assumption made above. Hence we have
\[
\dim V(\omega)=\dim W(\omega)\cdot \# \cJ \ge C_0^{-1}\langle \omega\rangle^d
\]
for some constant $C_0>0$ independent of $\omega$. 
From (\ref{eq:Vclaim}),  
we have (\ref{eq:trace_V}) for $v\in (\kappa_{\j}^{-1})^{*}W(\omega)$ and   $\j\in \cJ_\omega$.
Then it extends to all $v\in V(\omega)$ again by almost orthogonality between the images of $(\kappa_{\j}^{-1})^{*}(W(\omega))$ by $\bI^{\sigma'}\circ \mult(\rho_{K_1})\circ (\bI^\sigma)^*\circ\mathbf{T}^{\sigma}_\omega$.   
\end{proof}

\subsection{Short time estimates}
Lastly we give an estimate on the lifted operator $\bL^{t,\sigma\to \sigma'}$ for small $t>0$. This is the counterpart of Lemma \ref{lm:continuity_omega}. 
\begin{Lemma}\label{lm:cintinuity_lifted} 
Suppose that $\sigma, \sigma'\in \Sigma$ satisfies $\sigma'<\sigma$.
Then there exist constants $t_*>0$ and $C_\nu>0$ for each $\nu>0$  such that
\begin{equation}\label{eq:Aoo2}
\|\bPi_{\omega'}\circ ( e^{-i\omega t} \bL^{t,\sigma\to \sigma'}-\bL^{0,\sigma\to \sigma'})\circ \bPi_\omega:\bK^{r,\sigma}\to \bK^{r,\sigma'}\|\le C_\nu t\cdot 
\langle \omega'-\omega\rangle^{-\nu}
\end{equation}
for $0\le t<t_*$ and $\omega, \omega'\in \integer$. The limit
\begin{equation}\label{eq:Alimit}
\mathbf{A}_{\omega,\omega'}:=\lim_{t\to +0} \frac{1}{t} \cdot \bPi_{\omega'}\circ( e^{-i\omega t} \bL^{t,\sigma\to \sigma'}-\bL^{0,\sigma\to \sigma'})\circ \bPi_\omega
\end{equation}
exists and is a bounded operator from $\bK^{r,\sigma}$ to $\bK^{r,\sigma'}$ satisfying  
\begin{equation}\label{eq:Aoo}
\|\mathbf{A}_{\omega,\omega'}:\bK^{r,\sigma}\to \bK^{r,\sigma'}\|\le C_\nu  \langle \omega'-\omega\rangle^{-\nu}.
\end{equation}
\end{Lemma}
\begin{proof}
From the expression (\ref{eq:ljj}), we can write 
the $\j\to \j'$ component of  
\begin{equation}\label{eq:bLt0}
\frac{1}{t} \cdot \bPi_{\omega'}\circ (e^{-i\omega t}\cdot \bL^t-\bL^0)\circ \bPi_{\omega}
\end{equation}
with $\omega(\j)=\omega$ and $\omega(\j')=\omega'$ as 
\begin{align*}
\mult(\Psi_{\omega(\j'), m(\j')})\circ 
\pBargmann\circ \left(\frac{1}{t}\left(e^{-i\omega t}\cdot L(f^{t}_{\j\to \j'}, \tilde{b}^t_{\j'}\cdot \rho^{t}_{\j\to \j'})-L(f^{0}_{\j\to \j'}, \tilde{b}^0_{\j'}\cdot\rho^{0}_{\j\to \j'})\right)\right)\circ \pBargmann^*. 
\end{align*}
Note that we have $
f^t_{\j\to \j'}(x,y,z)=f^0_{\j\to \j'}(x,y,z+t)$ for sufficiently small $t$, provided that the both sides are defined. Therefore, setting $T_t(x,y,z)=(x,y,z+t)$, we
 rewrite the operator above in the middle  as  
\begin{align}\label{eq:smallT}
\frac{1}{t}&\left(e^{-i\omega t} L(f^{t}_{\j\to \j'}, \tilde{b}^t_{\j'}\cdot\rho^{t}_{\j\to \j'})-L(f^{0}_{\j\to \j'}, \tilde{b}^0_{\j'}\cdot\rho^{0}_{\j\to \j'})\right)\\
&\qquad=\frac{1}{t}(e^{-i\omega t}-e^{-i\omega' t})\cdot L(f^{t}_{\j\to \j'}, \tilde{b}^t_{\j'}\cdot\rho^{t}_{\j\to \j'})\notag\\
&\qquad \qquad
 +e^{-i\omega' t}\cdot L\left(f^{t}_{\j\to \j'}, \frac{1}{t}\left(\tilde{b}^t_{\j'}\cdot\rho^{t}_{\j\to \j'}-(\tilde{b}^0_{\j'}\cdot\rho^{0}_{\j\to \j'})\circ T_{-t}\right)\right)\notag
\\
&\qquad \qquad+ \left(
\frac{e^{-i\omega' t}}{t}\cdot  L(T_{t},1)-\mathrm{Id}\right)\circ L(f^{0}_{\j\to \j'},\tilde{b}^0_{\j'}\cdot\rho^{0}_{\j\to \j'}).\notag
\end{align}
Correspondingly we decompose (\ref{eq:bLt0}) into three parts $\bL_0$, $\bL_1$ and $\bL_2$, which have the first, the second and the third operators on the right-hand side above respectively in their $(\j\to \j')$-components. 

To prove the former claim in the lemma, it is enough to prove 
\begin{equation}\label{ClaimSmallt}
\|\bL_k:\bK^{r,\sigma}\to \bK^{r,\sigma'}\|\le C_\nu \langle \omega'-\omega\rangle^{-\nu}\quad \mbox{for $k=0,1,2$}
\end{equation}
uniformly for small $t>0$.
The case $k=0$ is obvious, because we have the same estimate for $\bPi_{\omega'}\circ \bL^{t,\sigma\to \sigma'}\circ \bPi_{\omega}$ by Proposition \ref{lm:TL}, Proposition \ref{prop:hyp} and Lemma \ref{lm:low} and because $t^{-1}(e^{-i\omega t}-e^{-i\omega' t})$ is bounded by $2\langle \omega'-\omega\rangle $.
For the case $k=1$, we note that the only difference between the operators $\bPi_{\omega'}\circ \bL^{t,\sigma\to \sigma'}\circ \bPi_{\omega}$ and $\bL_1$ is that  the 
multiplication by $\tilde{b}^t_{\j'}\cdot\rho^{t}_{\j\to \j'}$ in each component is replaced with that by 
$(\tilde{b}^t_{\j'}\cdot\rho^{t}_{\j\to \j'}-(\tilde{b}^0_{\j'}\cdot\rho^{0}_{\j\to \j'})\circ T_{-t})/t$.
Since the last function satisfies the similar estimates for the derivatives and the support  as $\tilde{b}^t_{\j'}\cdot\rho^{t}_{\j\to \j'}$, we can follow the argument in the previous subsections to get 
(\ref{ClaimSmallt}) for $k=1$. 
For the case $k=2$, we regard the components of the operator $\bL_2$ as the operators $\bL^{0,\sigma\to \sigma'}_{\j\to \j'}$ post-composed by the lift of $t^{-1} (e^{-i\omega t}L(T_{t},1)-\mathrm{Id})$. Since the supports of the functions in $\bK^{r,\sigma}_{\j'}$ is contained in $[\omega'-1,\omega'+1]$ in the coordinate $\xi_z$, the operator norm of 
\[
\left(\frac{1}{t} (e^{-i\omega' t}L(T_{t},1)-\mathrm{Id})\right)^{\lift}=\pBargmannP\circ \mult(t^{-1}(e^{i(\xi_z-\omega') t}-1)):\bK_{\j'}^{r,\sigma'}\to \bK^{r,\sigma'}_{\j'}
\]
is bounded by $2$.  Therefore we obtain (\ref{ClaimSmallt}) in the case $k=2$. 

To prove the latter claim on $\mathbf{A}_{\omega,\omega'}$, we consider the limit $t\to +0$ in the argument above.  Convergence of  (\ref{eq:Alimit}) is clear from the expression (\ref{eq:smallT}). Then the estimate (\ref{eq:Aoo})  follows from (\ref{eq:Aoo2}). 
\end{proof}

\subsection{Proof of propositions in Subsection \ref{ss:lt}}
Finally we deduce the propositions given in Subsection \ref{ss:lt}. 
Below we keep in mind that the diagram (\ref{cd:bL}) commutes and that $\bI^\sigma:\cK^{r,\sigma}\to \bK^{r,\sigma}$ is an isometric embedding. 

\begin{proof}[Proof of Proposition \ref{pp:continuity}]
Applying Proposition \ref{lm:TL}, Proposition \ref{prop:hyp} and  Lemma \ref{lm:low} to the central, hyperbolic and low frequency parts of $\bL^{t,\sigma\to \sigma'}$ respectively, we see that there exists a constant $C_\nu>0$ for any $\nu>0$ such that 
\begin{equation}\label{eq:bL_omega_omega'}
\|\bPi_{\omega'}\circ \bL^t\circ \bPi_{\omega}:\bK^{r,\sigma}\to \bK^{r,\sigma'}\|\le C_\nu\langle \omega'-\omega\rangle^{-\nu}
\end{equation}
for $0\le t\le 2t_0$ satisfying   (\ref{eq:condition_on_t}). (Note that $t(\omega)\ge t_0$ by definition.) 
By the definition of the norm on $\bK^{r,\sigma}$,  this implies 
\begin{equation}\label{eq:bL_omega_omega'2}
\| \bL^t:\bK^{r,\sigma}\to \bK^{r,\sigma'}\|\le C
\quad \mbox{for $0\le t\le 2t_0$ satisfying   (\ref{eq:condition_on_t}).}
\end{equation}
Then, in view of the commutative diagram (\ref{cd:bL}), we get
\[
\|\cL^t:\cK^{r,\sigma}(K_0)\to \cK^{r,\sigma'}(K_0)\|\le C \quad \mbox{for $0\le t\le 2t_0$ satisfying   (\ref{eq:condition_on_t}).}
\]
We obtain the last claim of Proposition \ref{pp:continuity} by iterative use of this estimate. 
\end{proof}

\begin{proof}[Proof of Lemma \ref{lm:Qomega}]
We consider the inequality  (\ref{eq:bL_omega_omega'}) for the case $t=0$, but with $\cL^t$ replaced by $\cL^t\circ \mult(\rho_{K_i})$, $i=0,1$. (Recall Remark \ref{rem:difference}.) 
Then we have
\begin{equation}\label{eq:I_omega_omega'}
\|\bPi_{\omega'}\circ \bI^{\sigma'} \circ \mult(\rho_{K_i})\circ  (\bI^\sigma)^{*}\circ \bPi_{\omega}:\bK^{r,\sigma}\to \bK^{r,\sigma'}\|\le
C_\nu \langle \omega'-\omega\rangle^{-\nu}
\end{equation}
and hence
\begin{equation}\label{eq:I_omega_omega'2}
\|\bI^{\sigma'} \circ \mult(\rho_{K_i})\circ  (\bI^\sigma)^{*}:\bK^{r,\sigma}\to \bK^{r,\sigma'}\|\le
C
\end{equation}
provided $\sigma'<\sigma$. 
From the definition of the norm $\|\cdot\|_{\cK^{r,\sigma}}$ in Definition \ref{Def:cK}, we have $\|u\|_{\cK^{r,\sigma}}^2=\sum_{\omega} \|\bPi_{\omega}\circ \bI^{\sigma} u\|_{\bK^{r,\sigma}}^2$ and 
\[
\|\cQ_{\omega} u\|^2_{\cK^{r,\sigma'}}=
\sum_{\omega'}\|(\bPi_{\omega'}\circ \bI^{\sigma'} \circ \mult(\rho_{K_0})\circ (\bI^\sigma)^{*}\circ \bPi_{\omega})\circ \bPi_{\omega}\circ \bI^{\sigma} u\|_{\bK^{r,\sigma'}}^2. 
\]
Therefore the inequality (\ref{one:qsum1}) follows from (\ref{eq:I_omega_omega'}). 
\end{proof}
\begin{proof}[Proof of Lemma \ref{lm:trace1}]
We deduce the claims from Lemma \ref{lm:trace_lifted}.
From uniform boundedness of $\mathbf{T}_\omega^{\sigma\to \sigma'}$ given in Lemma \ref{lm:trace_lifted} and the fact that $\mathbf{T}_\omega^{\sigma\to \sigma'}$ for different $\omega$ acts on different components in effect,  there exists a constant $C>0$ such that the operator norm of 
\[
\sum_{\omega\in Z} \mathbf{T}_\omega^{\sigma\to \sigma'+1}:\bK^{r,\sigma}\to \bK^{r,\sigma'+1}
\]
for any subset $Z \subset \integer$
is bounded uniformly by $C$. 
The claim  (\ref{eq:TZ}) follows from this  and (\ref{eq:I_omega_omega'2})
because
the operator norm in (\ref{eq:TZ}) is bounded by that of 
\[
\bI^{\sigma'}\circ \mult(\rho_{K_1})\circ (\bI^{\sigma'+1})^*\circ \left(\sum_{\omega\in Z} \mathbf{T}_\omega^{\sigma\to \sigma'+1}\right):\bK^{r,\sigma}\to \bK^{r,\sigma'}.
\]
We next prove the latter claims (a) and (b) in Lemma \ref{lm:trace1}. Since  
\[
\|\cT_\omega:\cK^{r,\sigma}(K_0)\to \cK^{r, \sigma'}(K_0)\|_{\trace}\le
\|\bI^{\sigma'}\circ \mult(\rho_{K_1})\circ (\bI^\sigma)^*\circ\mathbf{T}^{\sigma}_\omega:\bK^{r,\sigma}\to \bK^{r,\sigma'}\|_{\trace},
\]
the upper bound in (\ref{eq-trace-order}) follows from (\ref{eq:traceTlifted}). 
The claim (b) is basically a literal translation of the latter claim in Lemma \ref{lm:trace_lifted}. 
The estimates (\ref{eq:Vexp}) and (\ref{eq:Vexp2}) are immediate consequences of (\ref{eq:trace_V}) and (\ref{eq:I_omega_omega'}).
Finally we see that the lower bound in  (\ref{eq-trace-order}) is a consequence of the claim (b).
\end{proof}

\begin{proof}[Proof of Proposition \ref{pp:norm_estimate}]
We first prove the claim (\ref{eq:upperboundTLT}). 
From  (\ref{eq:I_omega_omega'2}) and the commutative diagram (\ref{cd:bL}), we have that 
\begin{align*}
\|\cT_{\omega'}\circ &\cL^t\circ \cT_{\omega}:\cK^{r, \sigma}(K_0)\to \cK^{r, \sigma'}(K_0)\|\\
&\le  C\| \bI^{\sigma'}\circ \mult(\rho_{K_1})\circ  (\bI^{\sigma'+1})^*\circ  \mathbf{T}_{\omega'}^{\sigma-1\to \sigma'+1}\circ \bL^{t, \sigma\to \sigma-1}\circ \mathbf{T}_{\omega}^{\sigma}:\bK^{r, \sigma}\to \bK^{r,\sigma'}\|\\
&\le C \|  \mathbf{T}_{\omega'}^{\sigma-1\to \sigma'+1}\circ \bL^{t, \sigma\to \sigma-1}\circ \mathbf{T}_{\omega}^{\sigma}:\bK^{r, \sigma}\to \bK^{r,\sigma'+1}\|\\
&= C \|  \mathbf{T}_{\omega'}^{\sigma'+1}\circ \bL^{t, \sigma\to \sigma'+1}\circ \mathbf{T}_{\omega}^{\sigma}:\bK^{r, \sigma}\to \bK^{r,\sigma'+1}\|
\end{align*}
where $\bL^{t, \sigma\to \sigma-1}$ is to be read as the lift of $\cL^t\circ \mult(\rho_{K_1})$. (Recall Remark \ref{rem:difference}.)
We decompose $\bL^{t, \sigma\to \sigma'+1}$ as in (\ref{eq:decomp_Lt}) and apply Proposition \ref{lm:TL}, Lemma \ref{lm:center_peri} and Lemma \ref{lm:low} to each part, noting uniform  boundedness of $\mathbf{T}_{\omega}^{\sigma}$ and the relation 
\begin{equation}\label{eq:comm_TQ}
\bPi_\omega\circ \mathbf{T}_{\omega}^{\sigma\to \sigma'}=\mathbf{T}_{\omega}^{\sigma\to \sigma'}\circ \bPi_\omega=\mathbf{T}_{\omega}^{\sigma\to \sigma'}.
\end{equation} 
Then we obtain the first claim:
\begin{align*}
\|\cT_{\omega'}\circ &\cL^t\circ \cT_{\omega}:\cK^{r, \sigma}(K_0)\to \cK^{r, \sigma'}(K_0)\|\\
&\le 
C_\nu \langle \omega'-\omega\rangle^{-\nu}
+C_\nu \langle\omega\rangle^{-\theta} \langle \omega'-\omega\rangle^{-\nu}
+C_\nu \langle \omega'\rangle^{-\nu} \langle\omega\rangle^{-\nu}
\le C_\nu \langle \omega'-\omega\rangle^{-\nu}.
\end{align*}
We can prove (\ref{eq:upperboundQLT}), (\ref{eq:upperboundTLQ}), (\ref{eq:upperboundQLQ}) in a parallel manner. Just note that, in proving the last inequality, we apply Proposition \ref{prop:hyp} instead of Lemma \ref{lm:center_peri} for the hyperbolic part. 
\end{proof}

\subsubsection{Proof of Proposition \ref{pp:backward}}
The statements are direct consequences of those in Proposition \ref{pp:backward_lifted}. 
In proving (\ref{eq:cv1}) and (\ref{eq:cv2}), we use (\ref{eq:Tlifted2}) and (\ref{eq:I_omega_omega'2}) to deal with the multiplication operators $\mult(\rho_{K_0})$ and $\mult(\rho_{K_1})$. 

\subsubsection{Proof of Lemma  \ref{lm:continuity_omega}} 
The statements are direct consequences of those in Lemma \ref{lm:cintinuity_lifted}, but with $\cL^t$ replaced by $\cL^t\circ \mult(\rho_{K_0})$. (Recall Remark \ref{rem:difference} again.) For the last statement on replacement of $\cQ_\omega$ with $\cT_{\omega}$, we use the relation (\ref{eq:comm_TQ}). 

\appendix
\section{Proof of Lemma \ref{lm:basic2}}
\label{ssA}
As we noted in the text after the statement of Lemma \ref{lm:basic2}, the proof is obtained by elementary geometric estimates about diffeomorphisms with some hyperbolicity. We begin with preliminary argument. 
For definiteness, we assume that
\begin{itemize}
\item[($\bigstar 1$)] the condition (ii) in (\ref{eq:condition_on_t}) holds, that is, $t\ge t_0$, and
\item[($\bigstar 2$)] the condition (\ref{eq:close})  holds.
\end{itemize}
The cases where these conditions do not hold are treated in a parallel and simpler manner.
(See Remark \ref{rem:hype1}.) 
Further we may and do assume
\begin{itemize}
\item[($\bigstar 3$)] $\min\{\max\{e^{|m(\j)|}, \langle \omega(\j)\rangle\}, \max\{e^{|m(\j')|}, \langle \omega(\j')\rangle\}\}$ is large, and  
\item[($\bigstar 4$)] $e^{\chi_{\max}\cdot 2t(\omega)}\le \langle \omega\rangle^{\theta/10}$
\end{itemize}
by choosing large constant $k_0$ in the definition of the low-frequency part and  small constant $\epsilon_0>0$ in the definition of $t(\omega)$ in (\ref{eq:tomega}) respectively. 

For simplicity, we set  $
\omega=\omega(\j)$, $\omega'=\omega(\j')$, $m=m(\j)$, $m'=m(\j')$. Below we  consider  points
\begin{align}
&w''\in \supp \rho^t_{\j\to \j'},\\\label{psi1}
&\mathbf{p}=(w,\xi_w,\xi_z)=(q,p,y,\xi_q,\xi_p,\xi_y,\xi_z)\in \supp \Psi^{\sigma}_{\j}\quad 
\mbox{and}\\
&\mathbf{p}'=(w',\xi'_w,\xi'_z)=(q',p',y',\xi'_q,\xi'_p,\xi'_y,\xi'_z)\in \supp \Psi^{\sigma'}_{\j'}\label{psi2}
\end{align}
and estimate the quantity (\ref{eq:Deltajj}). 
By changing the coordinates by a transformation in $\cA_2$, we may and do assume
\begin{itemize}
\item[($\bigstar 5$)] $\proj_{(x,z)}(w'')=0$ and $\proj_{(x,z)}((\breve{f}^t_{\j\to \j'})^{-1}(w''))=0$
\end{itemize}
without loss of generality. 
From the assumption that $\bL^{t,\sigma\to \sigma'}_{\j\to \j'}$ is a component of $\bL^{t,\sigma\to \sigma'}_{\mathrm{hyp},\not\hookrightarrow}$, we have $m\cdot m'\neq 0$ and hence
\begin{equation}\label{eq:mgeomega}
e^{\max\{|m|,|m'|\}}\ge e^{n_0(\omega)}\ge C^{-1}\langle \omega\rangle^{\theta}.
\end{equation}
We can get the conclusion of the lemma for small $\gamma_0$ easily by using (\ref{eq:mgeomega}) if either     
\[
\langle \omega'\rangle^{1/2}|w'-w''|\ge e^{\max\{|m|,|m'|\}/3}\quad \mbox{or}\quad \langle \omega\rangle^{1/2}|w-(\breve{f}^t_{\j\to \j'})^{-1}(w'')|\ge e^{\max\{|m|,|m'|\}/3}.
\]
Therefore we will assume 
\begin{itemize}
\item[($\bigstar 6$)]
$\max\{\;\langle\omega'\rangle^{1/2}|w'-w''|, \; \langle \omega\rangle^{1/2}|w-(\breve{f}^t_{\j\to \j'})^{-1}(w'')|\;\}< e^{\max\{|m|,|m'|\}/3}$.
\end{itemize}

We prove the following claim under the additional assumptions above. 
\begin{Sublemma}\label{subl}There exists a constant $C_0>0$ such that
\begin{align}\label{claimA}
&\left|(D^*E_{\omega,m})^{-1}\left((\breve{f}^t_{\j\to \j'}(w), 
((D\breve{f}^t_{\j\to \j'})^*_{w''})^{-1} \xi_w,\xi_z) -(w',(\langle \xi'_z\rangle/\langle \xi_z\rangle)\xi'_w, \xi_z)\right)\right| \\
&\qquad \qquad\qquad \qquad \qquad \qquad> C_0^{-1} e^{(2/3)\max\{|m|,|m'|\}} \langle\omega\rangle^{-1/2},\notag
\end{align}
where $D^*E_{\omega,m}$ is the linear map defined in (\ref{eq:Eomeagdag}).
\end{Sublemma}
We defer the proof of this sublemma for a while and finish the proof of Lemma \ref{lm:basic2}. 
We show that $w$-component of the quantity on the left hand side of (\ref{claimA}), {\it i.e.}  $E_{\omega,m}(\breve{f}^t_{\j\to \j'}(w)-w')$, is much smaller than the right-hand side of (\ref{claimA}). 
In the case $|m|\le n_1(\omega)$, we have $D^*E_{\omega,m}=\mathrm{Id}$ and $|\breve{f}^t_{\j\to \j'}(w)-w'|$ is much smaller than the right-hand side of (\ref{claimA})
from  ($\bigstar 4$), ($\bigstar 5$), ($\bigstar 6$) and (\ref{eq:mgeomega}).
In the case $|m|> n_1(\omega)$, we have $e^{|m|}\ge C_0^{-1}\langle \omega\rangle^{\Theta_1}$ and 
hence
\[
|E_{\omega,m}(\breve{f}^t_{\j\to \j'}(w)-w')|\le e_{\omega}(m)\cdot |\breve{f}^t_{\j\to \j'}(w)-w'|
\le 
 \langle \omega\rangle^{(1-\beta)(1/2-\theta)+4\theta}\cdot |\breve{f}^t_{\j\to \j'}(w)-w'|
\]
is again much smaller than the right-hand side of (\ref{claimA}).

Comparing Sublemma \ref{subl} with what we proved in the last paragraph, we see     
\begin{equation}\label{eq:db}
|((D\breve{f}^t_{\j\to \j'})^*_{w''})^{-1} (\langle \xi_z\rangle \xi_w) -\langle \xi'_z\rangle\xi'_w|
> (2C'_0)^{-1} e^{(2/3)\max\{|m|,|m'|\}} \langle\omega\rangle^{1/2}.
\end{equation}
To finish, we prove
\begin{align}\label{eq:num}
\langle \omega \rangle^{-(1+\theta)/2} & \cdot \langle\langle \omega\rangle^{1/2-4\theta} |\xi_w|\rangle^{-1/2}\cdot |((D\breve{f}^t_{\j\to \j'})^*_{w''})^{-1} (\langle \xi_z\rangle\xi_w) -\langle \xi'_z\rangle \xi'_w|\\
&\qquad  \ge (C''_0)^{-1} \max\{\omega, \omega', e^{|m|}, e^{|m'|}\}^{\gamma_0}\notag
 \end{align}
for a small constant $\gamma_0>0$. Clearly the required estimate in Lemma \ref{lm:basic2} follows from this claim.
In the case $\langle \omega\rangle^{1/2-4\theta} |\xi_w|\le 2$, we may neglect the second factor on the left hand side and hence obtain (\ref{eq:num}) immediately using (\ref{eq:mgeomega}) and (\ref{eq:db}). 
Thus we consider the case $\langle \omega\rangle^{1/2-4\theta} |\xi_w|> 2$ below. In this case, we have  
\[
\langle \omega\rangle^{1/2}|\xi_w|\le e^{\max\{|m|,|m'|\}+2}\cdot e_{\omega}(|m|)
\]
from  (\ref{psi1}), ($\bigstar 5$) and ($\bigstar 6$). Hence it holds
\[
\langle\langle \omega\rangle^{1/2-4\theta} |\xi_w|\rangle^{-1/2}
=\langle \omega\rangle^{-1/4+2\theta} |\xi_w|^{-1/2}
\ge e^{-\max\{|m|,|m'|\}/2-1}\langle \omega\rangle^{2\theta}\cdot e_{\omega}(|m|)^{-1/2}
\] 
and the left-hand side of (\ref{eq:num}) is bounded from below by 
\[
(e C'_0)^{-1} \langle \omega\rangle^{(3/2)\theta}\cdot e^{\max\{|m|,|m'|\}/6}\cdot e_{\omega}(|m|)^{-1/2}.
\]
Since $e_\omega(|m|)\le e^{\mu \max\{|m|,|m'|\}}\le e^{\max\{|m|,|m'|\}/20}$, we obtain (\ref{eq:num}) again.

\begin{proof}[Proof of Sublemma \ref{subl}]
For the points $\mathbf{p}$ and $\mathbf{p}'$ in (\ref{psi1}) and (\ref{psi2}), we set
\begin{align}\label{eq:pthat}
&(\hat{x},\hat{y},\hat{\xi}_x,\hat{\xi}_y, \xi_z):=(D^*E_{\omega,m})^{-1}(\mathbf{p})=(D^*E_{\omega,m})^{-1}(w,\xi_w, \xi_z),\\\label{eq:pthatdash}
&(\hat{x}',\hat{y}',\hat{\xi}'_x,\hat{\xi}'_y, \xi_z):=(D^*E_{\omega,m'})^{-1}(\mathbf{p}')=(D^*E_{\omega,m'})^{-1}(w',\xi'_w, \xi_z),
\intertext{and also set}
&(\tilde{x},\tilde{y},\tilde{\xi}_x,\tilde{\xi}_y, \xi_z):=(D^*E_{\omega,m})^{-1}(\breve{f}^t_{\j\to \j'}(w), 
((D\breve{f}^t_{\j\to \j'})^*_{w''})^{-1} \xi_w,\xi_z)\label{eq:tildept}\\
& =(E_{\omega,m}\circ \breve{f}^t_{\j\to \j'}\circ E_{\omega,m}^{-1}(\hat{w}),
 (E_{\omega,m}^*)^{-1}\circ((D\breve{f}^t_{\j\to \j'})^*_{w''})^{-1} \circ E_{\omega,m}^*(\hat{\xi}_w),\xi_z).\notag
\end{align} 
The claim of Sublemma \ref{subl} follows if we prove
\begin{equation}\label{eq:claimSublemma}
|(\tilde{y},\tilde{\xi}_x,\tilde{\xi}_y)-(\hat{y}',(\langle \xi'_z\rangle/\langle \xi_z\rangle)\hat{\xi}'_x,\hat{\xi}'_y)|\ge C_0^{-1} e^{(2/3)\max\{|m|,|m'|\}} \langle\omega\rangle^{-1/2}.
\end{equation}
Here we note that $\langle \xi'_z\rangle/\langle \xi_z\rangle$ is bounded because of the assumption ($\bigstar 2$).

For the proof of (\ref{eq:claimSublemma}), we investigate the conditions on the points (\ref{eq:pthat}) and 
(\ref{eq:pthatdash}) that come from the choice (\ref{psi1}) and (\ref{psi2}) of $\mathbf{p}$ and $\mathbf{p}'$. 
Then we look into the correspondence from the point (\ref{eq:pthat}) to (\ref{eq:tildept}). 

From the assumptions ($\bigstar 5$) and ($\bigstar 6$),  the points $\mathbf{p}$ and $\mathbf{p}'$ satisfy respectively
\begin{equation}\label{eq:pqcoord}
|(q,p)|<e^{\max\{|m|,|m'|\}/3}\langle\omega\rangle^{-1/2}, \quad |(q',p')|<e^{\max\{|m|,|m'|\}/3}\langle\omega'\rangle^{-1/2}.
\end{equation}
Recall that the function $\Psi^{\sigma}_{\j}$ is defined in (\ref{def:psi_sigma}) using the coordinates (\ref{eq:newcoordinates}). We let $(\zeta_q,\zeta_p,\tilde{y}, \tilde{\xi}_y)$ and $(\zeta'_q,\zeta'_p,\tilde{y}', \tilde{\xi}'_y)$ be the coordinates  (\ref{eq:newcoordinates}) for the points $\mathbf{p}$ and $\mathbf{p}'$ in (\ref{psi1}) and (\ref{psi2}) respectively. 
Then, from (\ref{eq:pqcoord}), we have 
\begin{align}\label{eq:zetap1}
\left||(\xi_q,\xi_p)|-2^{1/2}\langle \xi_z\rangle^{-1/2}|(\zeta_q,\zeta_p)|\right|<Ce^{\max\{|m|,|m'|\}/3}\langle\omega\rangle^{-1/2}
\intertext{and}\label{eq:zetap2}
\left||(\xi'_q,\xi'_p)-2^{1/2}\langle \xi'_z\rangle^{-1/2}|(\zeta'_q,\zeta'_p)|\right|<Ce^{\max\{|m|,|m'|\}/3}\langle\omega'\rangle^{-1/2}.
\end{align}
From the former estimate and the condition (\ref{psi1}), we see that the point (\ref{eq:pthat}) satisfies the following conditions  
 up to errors bounded by $Ce^{\max\{|m|,|m'|\}/3}\langle\omega\rangle^{-1/2}$: For the distance from the origin,
\begin{align}\label{est1}
e^{|m|-1}\langle \omega\rangle^{-1/2}\le |(\hat{\xi}_q,\hat{y},\hat{\xi}_p,\hat{\xi}_y)|\le e^{|m|+1}\langle \omega\rangle^{-1/2}\quad \mbox{if $m\neq 0$,}\\
\label{est2}
|(\hat{\xi}_q,\hat{y},\hat{\xi}_p, \hat{\xi}_y)|\le e\cdot \langle \omega\rangle^{-1/2+\theta}\quad \mbox{if $m= 0$}
\end{align}
where we write $\hat{\xi}_x=(\hat{\xi}_p,\hat{\xi}_q)$;  For the direction from the origin, 
\begin{align}\label{est3}
&|(\hat{\xi}_q,\hat{y})|<2 \cdot 2^{-\sigma/2}|(\hat{\xi}_p, \hat{\xi}_y)| \quad \mbox{if }m>0,\\
&2\cdot 2^{\sigma/2}|(\hat{\xi}_q,\hat{y})|>|(\hat{\xi}_p, \hat{\xi}_y)| \quad \mbox{if }m<0.\label{est4}
\end{align}
We have the parallel estimates on the point (\ref{eq:pthatdash}) as a consequence of  (\ref{psi2}).

Next we consider the correspondence from the point (\ref{eq:pthat}) to the point (\ref{eq:tildept}).
By contracting property of $f^t_{\j\to \j'}$ along the $y$-axis and Lemma \ref{lm:crude_estimates2}(1), we have
\begin{equation}\label{eq:ty}
|\tilde{y}|\le  e^{-\chi_0 t}|\hat{y}|+Ce_{\omega}(m)\cdot \langle\omega\rangle^{-(1/2+3\theta)}.
\end{equation}
Recall the diffeomorphism $h^t_{\j\to \j'}$ in (\ref{eq:htj}), which is roughly the flow $f^t_G$ viewed in the local charts without the factor $E_{\omega,m}$.  
By the relation (\ref{eq:EhE}),  we have
\begin{equation}\label{eq:corresp}
(\tilde{\xi}_x, \tilde{\xi}_y, \xi_z)=(E_{\omega,m}^*)^{-1}\circ E_{\omega'}^*\circ  
((Dh^t_{\j\to \j'})_{\check{w}}^{*})\circ (E_{\omega}^*)^{-1} \circ  E_{\omega,m}^*(\hat{\xi}_x,\hat{\xi}_y, \xi_z)
\end{equation}
where $\check{w}$ is chosen so that $(\check{w},z)=E_{\omega}(w'',z)$. 

To proceed, let us first consider the case where $|m|\le n_1(\omega)$ and $|m'|\le n_1(\omega')$. 
In this case, we have $e_{\omega}(m)=e_{\omega'}(m')=1$, $E_{\omega,m}=E_{\omega',m'}=\mathrm{Id}$ and therefore the correspondence (\ref{eq:corresp}) is given by nothing but the map $f^t_{\j\to \j'}$. 
Then, by Lemma~\ref{lm:crude_estimates2}, we get (\ref{eq:claimSublemma}) by simple geometric estimates. Indeed, if we ignore  
\begin{itemize}
\item the difference between $f^t_{\j\to \j'}$ and its linearization at the origin,
\item the errors mentioned about the estimates (\ref{est1}) -- (\ref{est4}) and the corresponding estimates on $(\hat{y}',\hat{\xi}'_x,\hat{\xi}'_y)$, and 
\item the difference between $\langle \xi'_z\rangle/\langle \xi_z\rangle$ and $1$,
\end{itemize} 
then the conclusion (\ref{eq:claimSublemma}) is an easy consequence of hyperbolicity of $f^t_{\j\to \j'}$. But we can check easily that the differences above are negligible in fact.

Next we consider the case $|m|\ge n_2(\omega)$ and $|m'|\ge n_2(\omega')$, the other extreme. 
In this case, we have $E_{\omega,m}=E_{\omega}$ and $E_{\omega',m'}=E_{\omega'}$ and therefore the correspondence (\ref{eq:corresp}) is given by the map $h^t_{\j\to \j'}$. Note that $h^t_{\j\to \j'}$ is given as iteration of the maps satisfying the conditions in Lemma \ref{lm:crude_estimates} and therefore have nice hyperbolic property. (Note that Lemma \ref{lm:crude_estimates} is valid only for $0\le t\le 2t_0$.)
Then we can get the conclusion (\ref{eq:claimSublemma}) by essentially same manner as in the previous case. 

The situations in the middle, {\it i.e.} the case where either $n_1(\omega)\le |m|\le n_2(\omega)$ or $n_1(\omega')\le |m'|\le n_2(\omega')$ is slightly more complicated.
If $||m|-|m'||\ge  2\chi_{\max} t$, it is easy to get the conclusion (\ref{eq:claimSublemma})
because the ratio between $e^{|m|}$ and $e^{|m'|}$ is much larger (or smaller) than the expansion (or contraction) given by the correspondence (\ref{eq:corresp}). 
So we may assume $||m|-|m'||\le  2\chi_{\max} t$. 
If we have $e_\omega(m)=e_{\omega'}(m')$ and $E_{\omega,m}=E_{\omega',m'}$, we can see that the correspondence (\ref{eq:corresp}) has good hyperbolic property in the same manner as in the proof of Lemma \ref{lm:crude_estimates2} and hence we can get the conclusion (\ref{eq:corresp}) again.  
But, from the slowly varying property (\ref{eq:e}) of $e_{\omega}(m)$, it is clear that the conclusion remains true without this assumption.   
\end{proof}
\begin{Remark}\label{rem:hype1}
In the case where the condition (i) in (\ref{eq:condition_on_t}) holds and $0\le t\le t_0$, we can follow the argument given in the proof above.  The only difference is that, 
instead of the condition $t\ge t_0$ that ensure enough hyperbolicity of $f^t_{G}$, we use the fact that $\sigma'< \sigma$ in the estimates.  
In the case where the condition (\ref{eq:close}) does not hold, the proof is much simpler. 
We can obtain the conclusion immediately unless both of $|m(\j)|$ and $|m(\j')|$ satisfy $|m(\j)|>n_2(\omega)$ and $|m(\j')|> n_2(\omega')$. But, in such case, the conclusion is again easy to obtain as we mentioned in the proof above.  
\end{Remark}

\section{Proof of the main theorem (2): Theorem \ref{th:zeta}}\label{sec:zeta}
In this section, we prove Theorem \ref{th:zeta}, using the propositions given in Section~\ref{sec:lifted}. 
Below we continue to consider the case $\cL^t=\cL^t_{0,0}$ as in the previous sections. 
But we can proceed in parallel in the case of vector-valued transfer operators $\cL^t_{k,\ell}$ with $(k,\ell)\neq (0,0)$ by regarding them as matrices of scalar valued transfer operators, as we have noted in Remark \ref{rem:vector_valued_case}.  (See also Remark \ref{rem:vector_valued_case_1}, Remark \ref{rem:vector_valued_case_3} and Remark \ref{rem:vector_valued_case_4}.)
The argument in this section is essentially parallel to that in \cite{BaladiTsujii08}, where dynamical zeta functions for hyperbolic diffeomorphisms are considered. The main idea is to decompose the operators in consideration into two parts: a trace class part and a ``upper-triangular'' part whose flat trace is zero. To realize this idea, we will consider the lifted operator $\bL^t$ rather than $\cL^t$.
Below we suppose that the operators are acting on $\cK^r(K_0)$ or $\bK^r$ if we do not specify otherwise.
\subsection{Analytic extension of the dynamical Fredholm determinant}
The dynamical Fredholm determinant $d(s)$ of the one-parameter group  of transfer operators $\mathbb{L}=\{\cL^t=\cL^t_{0,0}\}$ is well-defined if the real part of $s$ is sufficiently large. 
In fact, the sum in the definition (\ref{def:ds_ext}) of $d(s)$ converges absolutely if $\Re(s)$ is larger than the topological pressure 
$P_{top}:=P_{top}(f^t, -(1/2)\log |\det Df^t|_{E_u}|)$. (See \cite[Theorem 4.1]{Lu} for instance.) Hence $d(s)$ is a holomorphic function without zeros on $\Re(s)>P_{top}$.
We take a constant $P\ge P_{top}$ such that 
\begin{equation}\label{eq:bLnorm}
 \|\cL^t\|\le C\|{\bL}^{t}\|\le C e^{P t }\quad \mbox{for $t\ge t_0$}
\end{equation}
and consider the function $\log d(s)$  in the disk
\begin{equation}\label{eq:D(s0)}
\disk(s_0,r_0)=\{z\in \complex\mid |z-s_0|<r_0 \}
\end{equation}
for $s_0\in \complex$ with $\Re(s_0)>P$ and $r_0:=\Re(s_0)+r\chi_0/4$.
The $n$-th coefficient of the Taylor expansion of $\log d(s)$ at the center $s_0$ is
\[
a_n:=\frac{1}{n!} \left(\frac{d^n}{ds^n}\log d\right)(s_0)=(-1)^{n-1}\frac{1}{n!}\int_{+0}^\infty t^{n-1} e^{-s_0t} \cdot \flatTr\,\cL^t dt.
\]
Since we have  
\[
\cR(s_0)^n=\frac{1}{(n-1)!}\int_{+0}^\infty t^{n-1}e^{-s_0 t} \cL^t dt, 
\]
we may write the coefficient $a_n$ as 
\[
a_n=\frac{(-1)^{n-1}}{n}\cdot \flatTr (\cR(s_0)^n)\quad \mbox{for $n\ge 1$.}
\]
We are going to relate the asymptotic behavior of  flat trace $\flatTr (\cR(s_0)^n)$ as $n\to \infty$ with the spectrum of the generator $A$. Precisely we prove 
\begin{Proposition}\label{pp:zeta} The spectral set of the generator $A$ of  $\cL^t:\wK^{r}(K_0)\to \wK^{r}(K_0)$ in the disk $D(s_0,r_0)$ consists of finitely many eigenvalues $\chi_i\in \complex$, $1\le i\le m$,
counted with multiplicity. We have the asymptotic formula 
\[
\flatTr \left(\cR(s_0)^n\right)=\sum_{i=1}^m \frac{1}{(s_0-\chi_i)^n}+Q_n
\]
where the remainder term $Q_n$ satisfies
\begin{equation}\label{eq:claim_ppzeta}
|Q_n|\le C r_0^{-n}\quad \mbox{for $n\ge 0$}
\end{equation}
with a constant $C>0$ which may depend on $s_0$.
\end{Proposition}
Theorem \ref{th:zeta} is an immediate consequence of this proposition. In fact, we have
\begin{align*}
\log d(s_0+z)&=\log d(s_0)+\sum_{n=1}^{\infty}a_n z^n\\
&=\log d(s_0)+\sum_{i=1}^m \sum_{n=1}^{\infty}\frac{(-1)^{n-1} z^n}{n (s_0-\chi_i)^n}+\sum_{n=1}^{\infty}\frac
{(-1)^{n} Q_n}{n} z^n\\
&=\log d(s_0)+\sum_{i=1}^m \log\left(1+\frac{z}{s_0-\chi_i}\right)+\sum_{n=1}^{\infty}\frac{(-1)^{n} Q_n}{n}  z^n
\end{align*}
and hence 
\[
d(s_0+z)= d(s_0)\cdot \frac{\prod_{i=1}^m ((s_0+z)-\chi_i)}{\prod_{i=1}^m (s_0-\chi_i)}\cdot 
\exp\left(\sum_{n=1}^{\infty}\frac{(-1)^{n} Q_n}{n}  z^n\right)
\]
for $z\in \complex$ with sufficiently small absolute value.
The right-most factor on the right-hand side extends holomorphically to the disk $D(s_0,r_0)$ and has no zeros on it. So the dynamical Fredholm determinant  $d(s)$ extends to  the disk $D(s_0,r_0)$ as a holomorphic function and the zeros in $D(s_0,r_0)$ are exactly $\chi_i$, $1\le i\le m$, counted with multiplicity.
Note that this conclusion holds for any $s_0\in \complex$ with $\Re(s)>P_{top}$ so that the imaginary part of $s_0$ is arbitrary. Therefore, taking $r_*>0$ so large that $r_*\chi_0/4>c$, we obtain the conclusion of  Theorem \ref{th:zeta}.

\subsection{The flat trace of the lifted  transfer operators}
To proceed, we discuss about the flat trace of the lifted operators and averaging with respect to time.
Suppose that  $\bL:\bK^{r}\to\bK^{r}$ is a bounded operator, expressed as 
\begin{equation}\label{eq:matrix}
\bL(u_{\j})_{\j\in \cJ}=
\left(
\sum_{\j\in \cJ} \bL_{\j\to \j'} u_{\j}
\right)_{\j'\in \cJ}.
\end{equation}
If  the  diagonal components $\bL_{\j\to \j}:\bK^{r}_{\j}\to \bK^{r}_{\j}$ for $\j\in \cJ$ are trace class operators and if the sum of their traces converges absolutely, we set 
\[
\flatTr\; \bL :=\sum_{\j\in \cJ} \trace\; \bL_{\j\to \j}=\sum_{\j\in \cJ} \flatTr\; \bL_{\j\to \j}
\]
and call it the flat trace of the operator  $\bL:\bK^{r}\to\bK^{r}$. 
\begin{Remark} In this definition, we assume that each $\bL_{\j\to \j}$ is a trace class operator and hence\footnote{ If $\bL_{\j\to \j}$ is a trace class operator, it is expressed as $\bL_{\j\to \j}=\sum_{k} v_k\otimes v_k^*$ with $v_k\in \bK^r_{\j}$, $v_k^*\in (\bK^r_{\j})^*$ satisfying $\sum_k \|v_k\|_{\bK^r_{\j}}\|v_k^*\|_{(\bK^r_{\j})^*}<\infty$. Then the Schwartz kernel of $\bL_{\j\to \j}$ is $\sum_k v_k\otimes v_k^*$ and the flat trace equals $\sum_k (v_k)^* v_k=\trace \bL_{\j\to \j}$.} that its trace coincides with the flat trace.   
\end{Remark}
\begin{Definition} \label{Def:Trace_for_lift}
An operator $\bL$ as above is  {\em  upper triangular} (with respect to the index $\tilde{m}(\cdot)$)  if the components $\bL_{\j\to \j'}$ vanishes whenever $\tilde{m}(\j') \le \tilde{m}(\j)$. (Recall (\ref{def:tm}) for the definition of the index $\tilde{m}(\cdot)$.)
\end{Definition}
The next lemma is obvious from the definitions.
\begin{Lemma}\label{lm:monotone}
If $\bL$ is upper triangular, its flat trace vanishes. If $\bL$ and $\bL'$ are upper triangular, so are their linear combinations $\alpha\bL+\beta\bL'$ and their composition $\bL\circ \bL'$.
\end{Lemma}
Since the flat trace of  $\cL^t$ is a distribution as a function of $t$, it   takes values against smooth functions $\varphi(t)$ with compact support. Hence, rather than evaluating  the flat trace of $\cL^t$ itself, it is natural and convenient to consider the flat trace of the integration of $\cL^t$,
\[
\cL(\varphi):=\int \varphi(t) \cdot {\cL}^t dt
\]
against  a smooth function $\varphi:\real_+\to \real$ compactly supported on the positive part of the  real line. We will also consider the corresponding lifted operator 
\begin{equation}\label{eq:blphi}
\bL(\varphi):=\int_0^\infty \varphi(t) \cdot {\bL}^t dt=\bI\circ \cL(\varphi)\circ \bI^*.
\end{equation}
Recall the decomposition of the operator $\bL^t$, 
\[
\bL^t=\bL_{\mathrm{low}}^{t}+{\bL}^{t}_{\mathrm{hyp}}+\bL^{t}_{\mathrm{ctr}}=
\bL_{\mathrm{low}}^{t}+\bL^{t}_{\mathrm{hyp},\hookrightarrow}+\bL^{t}_{\mathrm{hyp},\not\hookrightarrow}+\bL^{t}_{\mathrm{ctr}},
\]
that we introduced in Section \ref{sec:proofs_of_props}. 
\begin{Lemma}\label{lm:bl_decomp} Suppose that $\mathcal{P}$ is a set of $C^\infty$ functions supported on $[0, 2]\subset \real$ and uniformly bounded in the $C^\infty$ sense.
Then there exists a constant $C>0$ such that the following holds true: For any $\varphi\in \mathcal{P}$,   the operator ${\bL}(\varphi)\circ \bL^t=\bL^t\circ {\bL}(\varphi)$ for $t_0\le t\le 2t_0$  is decomposed into two parts 
\begin{align*}
\widehat{\bL}&=\int \varphi(s-t) \cdot \bL^{s}_{\mathrm{hyp},\hookrightarrow} ds\quad \text{and}\\
\widecheck{\bL}&=\int \varphi(s-t) \cdot \left(\bL_{\mathrm{low}}^{s}+\bL^{s}_{\mathrm{hyp},\not\hookrightarrow}+\bL^{s}_{\mathrm{ctr}}\right) ds;
\end{align*}
The former $\widehat{\bL}$ is upper triangular and satisfies  $\|\widehat{\bL}\|\le C$, while the latter $\widecheck{\bL}$ is a trace class  operator and satisfies $\|\widecheck{\bL}\|_{\trace}\le C$; Further the operator
\[
(\bL^{t}-\bL^{t}_{\mathrm{hyp},\hookrightarrow})\circ \bL(\varphi)=(\bL_{\mathrm{low}}^{t}+\bL^{t}_{\mathrm{hyp},\not\hookrightarrow}+\bL^{t}_{\mathrm{ctr}})\circ \bL(\varphi)\quad \text{ for }t_0\le t\le 2t_0
\] is a trace class operator and we have $\|(\bL^{t}-\bL^{t}_{\mathrm{hyp},\hookrightarrow})\circ \bL(\varphi)\|_{\trace}\le C$. 
\end{Lemma}
\begin{proof}
The part $\widehat{\bL}$ is upper triangular and satisfies $\|\widehat{\bL}\|\le C$
by the definition of the relation $\hookrightarrow^t$ in Definition \ref{def:arrow} and Proposition \ref{prop:hyp}. 
From Lemma \ref{lm:low} and Proposition \ref{prop:hyp}, we know that the operators $\bL_{\mathrm{low}}^{s}$ and $\bL^{s}_{\mathrm{hyp},\not\hookrightarrow}$ are trace class operators and their trace norms are bounded uniformly for\footnote{Actually we proved this for $s\in [t_0,2t_0]$. But it is easy to see that the estimates remain true for $[t_0,2t_0+2]$.} $s\in [t_0, 2t_0+2]$. It remains to show that
$\int \varphi(s-t) \cdot \bL^{s}_{\mathrm{ctr}} ds$ and $
\bL^{t}_{\mathrm{ctr}}\circ \bL(\varphi)$
are trace class operators and that  their traces are uniformly bounded for $t_0\le t\le 2t_0$ and $\varphi\in \mathcal{P}$. 
These claims follow if we show   
\begin{equation}\label{eq:bLvarphi}
\|(\bL(\varphi)\circ \bL^t)_{\j\to \j'}:\bK^{r}_{\j}\to \bK^{r}_{\j'}\|_{\trace}\le C_{\nu}
\langle \omega(\j)\rangle^{-\nu} \langle \omega(\j)-\omega(\j')\rangle^{-\nu}\langle |m(\j)|\rangle^{-\nu}
\end{equation}
for $\j,\j'\in \cJ$ with $m(\j')=0$ and $t_0\le t\le 2t_0$.  
Note that, if we apply the argument in the proofs of Lemma \ref{lm:coarse} and Corollary \ref{cor:coarse} to  $(\bL(\varphi)\circ \bL^t)_{\j\to \j'}$, we obtain 
 the required estimate (\ref{eq:bLvarphi}) without the term $\langle \omega(\j)\rangle^{-\nu}$ on the right-hand side. 
To retain the term $\langle \omega(\j)\rangle^{-\nu}$, we make use of the additional integration with respect to time in $\bL(\varphi)$. Since $f^{t+t'}_{\j\to \j'}(w,z)=f^{t}_{\j\to \j'}(w,z)+(0,t')$ when $|t'|$ is sufficiently small, such integration will reduce the $\j'$-components of the image with large $\omega(\j')$. Indeed, if we (additionally) apply integration by parts to that integral with respect to time  using the differential operator $\cD=(1-i\xi'_z\partial_{t})(1+|\xi'_z|^2)$ for several times, we obtain the extra factor $\langle\omega(\j)\rangle^{-\nu}$.  
\end{proof}
\begin{Corollary}\label{Cor:trace_equality}
If $\varphi:[t_0, \infty)\to \real$ be a smooth function with compact support, we have
$\flatTr \bL(\varphi)=\flatTr \cL(\varphi)$.
\end{Corollary}

\begin{proof}
From the proof of Lemma \ref{lm:bl_decomp} above, we see that $\flatTr \bL(\varphi)$ is well-defined, that is, the sum over $\j\in \cJ$ in the definition converges absolutely. Hence we obtain 
\[
\flatTr \bL^t= \flatTr (\bI^{\sigma}\circ \cL^t\circ  (\bI^{\sigma})^*)
= \flatTr  (\cL^t\circ  (\bI^{\sigma})^* \circ \bI^{\sigma})=\flatTr  \cL^t
\]
by rotating the order of composition in the middle.
\end{proof}

\subsection{The flat trace of the iteration of the resolvent}
 Let us put
\begin{align}
&{\cR}^{(n)}=
\int_0^\infty (1-\chi(t/(2t_0)))\cdot  \frac{t^{n-1} e^{-ts_0}}{(n-1)!}\cdot \cL^t\, dt\\
\intertext{
and}
&{\mathbf{R}}^{(n)}=
\int_0^\infty (1-\chi(t/(2t_0)))\cdot  \frac{t^{n-1} e^{-ts_0}}{(n-1)!}\cdot \bL^t\, dt\label{eq:Rns}
\end{align}
where the function $\chi(\cdot)$ is that in (\ref{eq:def_chi}). 
\begin{Remark}\label{rem:wR} The operator ${\cR}^{(n)}$ above is defined as an approximation of  ${\cR}(s_0)^{n}$ and the difference is 
\begin{equation}\label{eq:wR}
\widetilde{\cR}^{(n)}:={\cR}(s_0)^{n}-{\cR}^{(n)}=
\int_0^\infty \chi(t/(2t_0))\cdot  \frac{t^{n-1}\cdot e^{-ts_0}}{(n-1)!}\cdot \cL^t dt.
\end{equation}
We put  the part $\widetilde{\cR}^{(n)}$ aside because
 we can not treat the operators $\cL^t$ with small $t>0$ in the same way as those with large $t>0$. Since the flat trace and also the operator norm of $\widetilde{\cR}^{(n)}$ on $\wK^{r}(K_0)$  converges to zero super-exponentially fast as $n\to \infty$, this does not cause any essential problem, though it introduces some complication  in a few places below.
\end{Remark}

We take constants $r'_0<r''_0$ such that 
\[
r_0=\Re(s_0)+(1/4)r\chi_0<r'_0<r''_0<\Re(s_0)+(1/2)r\chi_0.
\]
\begin{Lemma} \label{lm:R_decomp} There exists a constant $C>0$, independent of $n$,  such that the operator ${\mathbf{R}}^{(n)}$  is expressed as a sum $
{\mathbf{R}}^{(n)}=\widehat{\mathbf{R}}^{(n)}+\widecheck{\mathbf{R}}^{(n)}$
and 
\begin{enumerate}
\item $\widecheck{\mathbf{R}}^{(n)}:\bK^{r}\to \bK^{r}$ is a trace class operator, while 
\item $\widehat{\mathbf{R}}^{(n)}:\bK^{r}\to \bK^{r}$ is upper triangular and satisfies  
\[
\|\widehat{\mathbf{R}}^{(n)}\|\le C (r''_0)^{-n}.
\]
\end{enumerate}
\end{Lemma}
\begin{proof} Using the periodic partition of unity $\{q_\omega\}_{\omega\in \mathbb{Z}}$ defined in (\ref{eq:chi_omega}), we set\footnote{Note that the variable $\omega\in \integer$  
does not indicate the frequency as in the previous sections but  the range of time, now and henceforth. } 
\[
q_{\omega}^{(n)}(t)=q_{\omega}(t)\cdot (1-\chi(t/(2t_0))\cdot \frac{t^{n-1}e^{-s_0 t}}{(n-1)!},
\]
so that $
\sum_{\omega\ge [2t_0]-1} q_{\omega}^{(n)}(t)= (1-\chi(t/(2t_0))\cdot (t^{n-1}e^{-s_0 t})/(n-1)!$ and that
\[
{\mathbf{R}}^{(n)}=\sum_{\omega=[2t_0]-1}^\infty \bL(q_{\omega}^{(n)})
\]
from the definition (\ref{eq:blphi}). 
We deduce the claims of the lemma from  
\begin{Claim}\label{Claim4}
For arbitrarily small $\tau>0$, there exists a constant $C>0$, independent of $n$, such that  the operators $\bL(q_{\omega}^{(n)})$ for $n\ge 1$ are decomposed as 
\[
\bL(q_{\omega}^{(n)})=\widehat{\bL}(q_{\omega}^{(n)})+\widecheck{\bL}(q_{\omega}^{(n)})
\]
where $\widehat{\bL}(q_{\omega}^{(n)})$ are upper triangular and satisfy 
\begin{equation}\label{eq:Claim4-1}
\|\widehat{\bL}(q_{\omega}^{(n)})\|\le C \frac{\omega^{n-1}}{(n-1)!}\cdot
 e^{-(\Re(s_0) +(1/2)r \chi_0)\omega+\tau n},
\end{equation}
while $\widecheck{\bL}(q_{\omega}^{(n)})$ are trace class operators satisfying
\begin{equation}\label{eq:sum_trace}
\sum_{\omega=[2t_0]-1}^\infty 
\|\widecheck{\bL}(q_{\omega}^{(n)})\|_{\trace}<+\infty.
\end{equation}
\end{Claim}
From the claim above, we set
\[
\widehat{\mathbf{R}}^{(n)}=\sum_{\omega=[2t_0]-1}^\infty \widehat{\bL}(q_{\omega}^{(n)})\quad \mbox{and}\quad
\widecheck{\mathbf{R}}^{(n)}=\sum_{\omega=[2t_0]-1}^\infty \widecheck{\bL}(q_{\omega}^{(n)}).
\]
Then the first claim (1) of the lemma follows from (\ref{eq:sum_trace}). 
The second claim (2) also follows because  $\widehat{\mathbf{R}}^{(n)}$ is upper triangular from Lemma \ref{lm:monotone} and because
\begin{align*}
\|\widehat{\mathbf{R}}^{(n)}\|&\le 
\sum_{\omega=[2t_0]-1}^\infty \| \widehat{\bL}(q_{\omega}^{(n)})\|\le C \sum_{\omega=[2t_0]-1}^\infty\frac{\omega^{n-1}}{(n-1)!}\cdot
 e^{-(\Re(s_0) +(1/2)r \chi_0)\omega+\tau n}\\
&\le C\int_0^\infty \frac{t^{n-1}}{(n-1)!} \cdot e^{-(\Re(s_0) +(1/2)r \chi_0)t+\tau n} dt\\
&= C e^{\tau n}\cdot \left(\int_0^\infty e^{-(\Re(s_0) +(1/2)r \chi_0)t} dt\right)^n<C (r''_0)^{-n}
\end{align*} 
from (\ref{eq:Claim4-1}). We give the proof of Claim \ref{Claim4} below to complete the proof.
\end{proof}
\begin{proof}[Proof of Claim \ref{Claim4}.]
We first note that the family 
\begin{equation}\label{eq:cP}
\mathcal{P}=\left\{(n-1)!\cdot \omega^{-n+1}\cdot  e^{\Re(s_0) \omega-\tau n}\cdot {q}_{\omega}^{(n)}(t+\omega)\;\bigg|\;  n\ge 1, \omega \ge [2t_0]-1 \right\}
\end{equation}
satisfies the assumption in Lemma \ref{lm:bl_decomp}. 
To proceed, we write an integer $\omega\ge [2t_0]-1$ as a sum of real numbers in $[t_0,2t_0]$:
\[
\omega=\sum_{i=1}^{k(\omega)} t_i,\qquad t_0\le t_i\le 2t_0.
\]
Then we decompose ${\bL}(q_{\omega}^{(n)})$ as follows: First we write   
\begin{align*}
{\bL}(q_{\omega}^{(n)})={\bL}(\tilde{q}_{\omega}^{(n)})\circ \bL^{\omega}
&=\bL^{t_{k(\omega)}}_{\mathrm{hyp},\hookrightarrow}\circ \bL(\tilde{q}_{\omega}^{(n)})\circ {\bL}^{t_{k(\omega)-1}}\circ \cdots\circ {\bL}^{t_{2}}\circ  \bL^{t_1}\\
&\qquad + ({\bL}^{t_{k(\omega)}}-\bL^{t_{k(\omega)}}_{\mathrm{hyp},\hookrightarrow})\circ {\bL}^{t_{k(\omega)-1}}\circ\cdots\circ {\bL}^{t_{2}}\circ \bL(\tilde{q}_{\omega}^{(n)})\circ \bL^{t_1}
\end{align*}
where we set $\tilde{q}_{\omega}^{(n)}(t)=q_{\omega}^{(n)}(t+\omega)$ for brevity. 
Since the first term on the right-hand side other than the first factor is of the same form as ${\bL}(\tilde{q}_{\omega}^{(n)})\circ \bL^{\omega}$, we apply the parallel operation to it. 
If we continue this procedure, we can express  $\bL(q_{\omega}^{(n)})$ as 
\begin{align}\label{eq:Lqw0}
&\bL^{t_{k(\omega)}}_{\mathrm{hyp},\hookrightarrow}\circ\cdots\circ \bL^{t_{2}}_{\mathrm{hyp},\hookrightarrow}\circ \bL(\tilde{q}_{\omega}^{(n)})\circ \bL^{t_1}\\
& +\sum_{j=2}^{k(\omega)} \bL^{t_{k(\omega)}}_{\mathrm{hyp},\hookrightarrow}\circ\cdots\circ \bL^{t_{j+1}}_{\mathrm{hyp},\hookrightarrow}
\circ ({\bL}^{t_{j}}-\bL^{t_{j}}_{\mathrm{hyp}, \hookrightarrow})\circ {\bL}(\tilde{q}_{\omega}^{(n)})\circ {\bL}^{t_{j-1}}\circ \cdots \circ {\bL}^{t_{1}}.\notag
\end{align}
Note that, from Lemma \ref{lm:bl_decomp} and the estimate noted in the beginning, we see  
\[
(n-1)! \cdot \omega^{-n+1}\cdot e^{\Re(s_0) \omega-\tau n}\cdot\bL(\tilde{q}_{\omega}^{(n)})\circ \bL^{t_1}=\widehat{\bL}+\widecheck{\bL}
\]
where $\widehat{\bL}$ is upper triangular and $\|\widehat{\bL}\|\le C$ while $\widecheck{\bL}$ is in the trace class and $\|\widecheck{\bL}\|_{\trace}\le C$, with $C>0$ a constant independent of $n$ and $\omega$. 
So we may rewrite the first term on the right-hand side of (\ref{eq:Lqw0}) as  
\begin{align}\label{eq:Lqw}
&\frac{\omega^{n-1}}{(n-1)!}\cdot  e^{-\Re(s_0) \omega+\tau n}\cdot\bL^{t_{k(\omega)}}_{\mathrm{hyp},\hookrightarrow}\circ\cdots\circ \bL^{t_{2}}_{\mathrm{hyp},\hookrightarrow}\circ \widehat{\bL}\\
&\qquad +\frac{\omega^{n-1}}{(n-1)!}\cdot  e^{-\Re(s_0) \omega+\tau n}\cdot
\bL^{t_{k(\omega)}}_{\mathrm{hyp},\hookrightarrow}\circ\cdots\circ \bL^{t_{2}}_{\mathrm{hyp},\hookrightarrow}\circ \widecheck{\bL}.\notag
\end{align}

Let $\widehat{\bL}(q_{\omega}^{(n)})$ in Claim \ref{Claim4} be the first term of (\ref{eq:Lqw}) above and $\widecheck{\bL}(q_{\omega}^{(n)})$ be the remainder, that is, the sum of the second term in (\ref{eq:Lqw}) and the sum on the second line of (\ref{eq:Lqw0}).  
By Lemma \ref{lm:monotone}, $\widehat{\bL}(q_{\omega}^{(n)})$ is  upper triangular. Since  
\begin{equation}
\label{eq:bLhypnorm}
\| \bL^{t}_{\mathrm{hyp},\hookrightarrow}\|\le  e^{-r\chi_0 t/2 }\quad \mbox{for $t_0\le t\le 2t_0$}
\end{equation}
from Remark \ref{Rem:const}, we obtain the estimate (\ref{eq:Claim4-1}) immediately.
Also, for the second term in (\ref{eq:Lqw}), we have
\begin{align*}
&\left \|\frac{\omega^{n-1}  e^{-\Re(s_0) \omega+\tau n}}{(n-1)!}\cdot \bL^{t_{k(\omega)}}_{\mathrm{hyp},\hookrightarrow}\circ\cdots\circ \bL^{t_{2}}_{\mathrm{hyp},\hookrightarrow}\circ \widecheck{\bL}\right\|_{\trace}\\
&\le  \frac{\omega^{n-1}  e^{-\Re(s_0) \omega+\tau n}}{(n-1)!} \|\bL^{t_{k(\omega)}}_{\mathrm{hyp},\hookrightarrow}\|\cdots\| \bL^{t_{2}}_{\mathrm{hyp},\hookrightarrow}\|  \|\widecheck{\bL}\|_{\trace}\le    \frac{C\omega^{n-1}  e^{-(\Re(s_0) +(1/2)r\chi_0)\omega+\tau n}}{(n-1)!}
\end{align*}
and this bound is summable with respect to $\omega$. 
To look into the sum in (\ref{eq:Lqw0}), note that the operator
 $({\bL}^{t_{j}}-\bL^{t_{j}}_{\mathrm{hyp},\hookrightarrow})\circ {\bL}(\tilde{q}_{\omega}^{(n)})$ satisfies 
\[
\|({\bL}^{t_{j}}-\bL^{t_{j}}_{\mathrm{hyp},\hookrightarrow})\circ {\bL}(\tilde{q}_{\omega}^{(n)})\|_{\trace}
\le C \cdot  \frac{\omega^{n-1}\cdot e^{-\Re(s_0) \omega+\tau n}}{(n-1)!}
\]
for a constant $C>0$ independent of $n$ and $\omega$, from the latter claim of Lemma \ref{lm:bl_decomp} and uniform boundedness of $\mathcal{P}$.  
Hence we have
\begin{align*}
&\sum_{j=2}^{k(\omega)}\| \bL^{t_{k(\omega)}}_{\mathrm{hyp},\hookrightarrow}\circ\cdots\circ \bL^{t_{j+1}}_{\mathrm{hyp},\hookrightarrow}
\circ ({\bL}^{t_{j}}-\bL^{t_{j}}_{\mathrm{hyp},\hookrightarrow})\circ {\bL}(\tilde{q}_{\omega}^{(n)})\circ {\bL}^{t_{j-1}}\circ \cdots \circ {\bL}^{t_{1}}\|_{\trace}\\
&\le 
\sum_{j=2}^{k(\omega)}\| \bL^{t_{k(\omega)}}_{\mathrm{hyp},\hookrightarrow}\|\cdots\|\bL^{t_{j+1}}_{\mathrm{hyp},\hookrightarrow}\|
\cdot \| ({\bL}^{t_{j}}-\bL^{t_{j}}_{\mathrm{hyp},\hookrightarrow})\circ {\bL}(\tilde{q}_{\omega}^{(n)})\|_{\trace}\cdot \|{\bL}^{t_{j-1}} \circ \cdots\circ  {\bL}^{t_{1}}\|\\
&\le C  \cdot \frac{\omega^{n-1}\cdot  e^{-(\Re(s_0)-P)\omega+\tau n}}{(n-1)!}\quad \mbox{by (\ref{eq:bLnorm}) and (\ref{eq:bLhypnorm}).}
\end{align*}
This bound is again summable with respect to $\omega$, provided $\tau$ is sufficiently small. Therefore we obtain the estimate (\ref{eq:sum_trace}). 
\end{proof}

\begin{Corollary}\label{cor:ess}
The essential spectral radius of  $\cR(s_0):\widetilde{\cK}^{r}(K_0)\to \widetilde{\cK}^{r}(K_0)$ is bounded by $(r''_0)^{-1}$. 
\end{Corollary}
\begin{proof} 
We consider  the decomposition of $\cR(s_0)^n:\wK^{r}(K_0)\to \wK^{r}(K_0)$ into  
\[
\widetilde{\cR}^{(n)},\quad \widehat{\cR}^{(n)}=\bI^*\circ \widehat{\mathbf{R}}^{(n)}\circ \bI\quad \mbox{and} \quad \widecheck{\cR}^{(n)}=\bI^*\circ \widecheck{\mathbf{R}}^{(n)}\circ\bI
\] 
where $\widetilde{\cR}^{(n)}$ is that in (\ref{eq:wR}). 
From the last lemma,  the operator norm of  $\widehat{\cR}^{(n)}$ on $\cK^r(K_0)$ is bounded by $C (r''_0)^{-n}$ with $C>0$ independent of $n$, and  $\widecheck{\cR}^{(n)}$ on $\cK^r(K_0)$ is a trace class operator. 
Further it is not difficult to see that these remain true when we regard $\widehat{\cR}^{(n)}$ and $\widecheck{\cR}^{(n)}$ as operators on $\wK^{r}(K_0)$, because $\cL^{t_0}:\wK^{r}(K_0)\to \cK^r(K_0)$ is bounded from Proposition~\ref{pp:continuity}. 
Therefore, recalling Remark \ref{rem:wR} for $\widetilde{\cR}^{(n)}$, we obtain that the essential spectral radius of $\cR(s_0)^n$ is bounded by $C (r''_0)^{-n}$ and hence by  $(r''_0)^{-n}$ from the multiplicative property of essential spectral radius.  
\end{proof}

Corollary \ref{cor:ess} implies that the spectral set of  $\cR(s_0):\wK^{r}(K_0)\to \wK^{r}(K_0)$ on the outside of the disk $|z|\le (r'_0)^{-1}$ consists of discrete eigenvalues $\mu_i$, $1\le i\le m$, counted with multiplicity. Since $A \cR(s_0)=s_0 \cR(s_0)+1$, we have
\[
\mu-\cR(s_0)=\mu\cdot \left( (s_0-\mu^{-1}) -A\right) \cdot \cR(s_0).
\]
This implies that $\mu_i$'s are in one-to-one correspondence to the eigenvalues $\chi_i$, $1\le i\le m$, of the generator $A$ in the disk $D(s_0,r'_0)$ by the relation  
\[
\mu_i=\frac{1}{s_0-\chi_i}=\int_0^\infty e^{-s_0 t}e^{\chi_i t}dt.
\]
\begin{Remark}\label{rem:meromorphy}
Since the argument above holds for any $s_0$ satisfying $\Re(s_0)>P$, the spectrum of the generator $A$ on the half-plane
$\Re(s)>-(1/4)r\chi_0$ consists of discrete eigenvalues with finite multiplicity and the resolvent $\cR(s)$ is meromorphic on that  half-plane.
\end{Remark}

Let $\bP:\widetilde{\cK}^{r}(K_0)\to \widetilde{\cK}^{r}(K_0)$ be the spectral projector of  $\cR(s_0)$ for the spectral set $\{\mu_i\}_{i=1}^m$ on the outside of the disk $|z|\le r_0^{-1}$. This is also the spectral projector of the generator $A$ for the spectral set $\{\chi_i\}_{i=1}^m$ and its image is contained in  $\cK^{r}(K_0)$ from Proposition~\ref{pp:continuity}. We set $
\mathcal{F}(s_0)=\bP\circ \cR(s_0)$, 
so that 
\[
\trace\,\mathcal{F}(s_0)^n=\flatTr\,\mathcal{F}(s_0)^n=\sum_{i=1}^m \mu_i^n=\sum_{i=1}^m \frac{1}{(s_0-\chi_i)^n}.
\]
Our task is to prove (\ref{eq:claim_ppzeta}) in Proposition \ref{pp:zeta} for 
\[
Q_n=\flatTr \left(\cR(s_0)^n-\mathcal{F}(s_0)^n\right)
=\flatTr \left((1-\bP)\circ \cR(s_0)^n\right).
\]
To continue, let $N_0>0$ be a large integer constant which will be specified later in the course of the argument. 
Consider large integer $n$  and write
\begin{align*}
\cR(s_0)^n-\mathcal{F}(s_0)^n&
=(\mathrm{Id}-{\bP})\circ 
\cR(s_0)^{n(m)}\circ\
\cdots \circ\cR(s_0)^{n(1)}.
\end{align*}
where $n=n(1)+n(2)+\cdots +n(I)$ with  $N_0\le n(i)\le 2N_0$.
We decompose each term $\cR(s_0)^{n(i)}$ on the right-hand side as
\[
\cR(s_0)^{n(i)}=\cR^{(n(i))}+\widetilde{\cR}^{(n(i))}
\]
in the same manner as  (\ref{eq:wR}). 
From Remark \ref{rem:wR}, the part $\widetilde{\cR}^{(n(i))}$ is very small in the operator norm if we let the constant $N_0$ be sufficiently large. (And note that the following argument is much simpler if we ignore this part.)

Since the operators $\cR^{(n(i))}$ and $\widetilde{\cR}^{(n(i))}$ for $1\le i\le I$ commute each other and also  with the projection operator $\bP$, we can express $\cR(s_0)^n-\mathcal{F}(s_0)^n$ as the sum of the $2^I$ terms of the form
\begin{equation}\label{eq:prod_N_R}
(1-{\bP})\circ 
\left(\prod_{i=1}^{I''} \widetilde{\cR}^{(n''(i))}\right)\circ  
\left(\prod_{i=1}^{I'} \cR^{(n'(i))}\right)
\end{equation}
where $\{n''(1),\cdots, n''(I''), n'(1), \cdots, n'(I')\}$ with $I=I'+I''$ ranges over all the rearrangements of $\{n(1), n(2), \cdots, n(I)\}$. 
Hence our task is reduced to show  
\begin{Claim}\label{Claim5} There exists a constant $C>0$ such that  
\begin{equation}\label{eq:dynTr_negligible}
\left|\flatTr\left((1-{\bP})\circ 
\left(\prod_{i=1}^{I''}\widetilde{\cR}^{(n''(i))}\right)\circ  
\left(\prod_{i=1}^{I'} \cR^{(n'(i))}\right)\right)\right|\le C r_0^{-n}\cdot 2^{-I}.
\end{equation}
\end{Claim}
\begin{proof}[Proof of Claim \ref{Claim5}]
Below we prove the claim in the case  $I'\ge I''$ because, otherwise, we can get  the conclusion by a similar but much easier argument using Remark~\ref{eq:wR}.
We translate the claim to that on the lifted operators.  
Let us put
\[
\bP^{\lift}:=\bI\circ \bP\circ \bI^*:\bK^r\to \bK^r.
\]
We write
\begin{equation}\label{def:wtR}
\widetilde{\mathbf{R}}^{(n)}= 
\int_0^\infty \chi(t/(2t_0))\cdot  \frac{t^{n-1} e^{-ts_0}}{(n-1)!}\cdot \bL^t\, dt,
\end{equation}
so that
\[
{\mathbf{R}}^{(n)}+
\widetilde{\mathbf{R}}^{(n)}=\bI\circ \cR(s_0)^n\circ \bI^*
=\int_0^\infty \frac{t^{n-1} e^{-ts_0}}{(n-1)!}\cdot \bL^t\, dt.
\]
Then, from Corollary \ref{Cor:trace_equality},   the inequality in Claim \ref{Claim5} is equivalent to 
\begin{equation}\label{eq:main_estimate_trace}
\left|\flatTr\,\left((1-\bP^{\lift})\circ 
\left(\prod_{i=1}^{I''} \widetilde{\mathbf{R}}^{(n''(i))}\right)\circ  
\left(\prod_{i=1}^{I'} \mathbf{R}^{(n'(i))}\right)\right)\right|\le Cr_0^{-n}\cdot 2^{-I}.
\end{equation}
Since we are assuming $I'\ge I''$ and since the operators on the left hand side above commute, we may write the operator on the left-hand side above as 
\begin{equation}\label{eq:bPR}
\prod_{i=1}^{I'} ((1-\bP^{\lift})\circ{\mathbf{R}}_i)
\quad\text{
setting }
{\mathbf{R}}_i=
\begin{cases}
\widetilde{\mathbf{R}}^{(n''(i))}\circ \mathbf{R}^{(n'(i))},&\quad\mbox{if $1\le i\le I''$};\\
\mathbf{R}^{(n'(i))},&\quad \mbox{if $I''<i\le I'$.}
\end{cases}
\end{equation}
From the choice of the spectral projector $\bP$ and the fact that $\widetilde{\mathbf{R}}^{(n)}$ has small operator norm by the same reason as noted in Remark \ref{rem:wR}, we get the estimate  
\begin{equation}\label{eq:norm_1Pi}
\|(1-\bP^{\lift})\circ \mathbf{R}_i:\bK^{r}\to \bK^{r}\|\le (r_0)^{-\tilde{n}(i)}/4\quad \mbox{for $1\le i\le I'$},
\end{equation}
provided that the constant $N_0$ is  sufficiently large, where we set
\[
\tilde{n}(i)=\begin{cases}
n'(i)+n''(i),&\quad \mbox{if $1\le i\le I''$};\\
n'(i),&\quad \mbox{if $I''< i\le I'$}.
\end{cases}
\]
By Lemma \ref{lm:R_decomp},  the operators $\mathbf{R}_i$ are decomposed as $
\mathbf{R}_i=\widehat{\mathbf{R}}_i+\widecheck{\mathbf{R}}_i$
where  
\begin{enumerate}
\item $\widecheck{\mathbf{R}}_i$ is a trace class operator and $\|\widecheck{\mathbf{R}}_i\|_{\trace}<C$, and 
\item $\widehat{\mathbf{R}}_i$ is upper triangular and satisfies $
\|\widehat{\mathbf{R}}_i\|\le  
 r_0^{-\tilde{n}(i)}/4$ 
\end{enumerate}
provided that the constant $N_0$ is  sufficiently large.
(For the case $1\le  i\le I''$, we need a slight modification of Lemma \ref{lm:R_decomp} but the proof goes as well.) 
In (\ref{eq:bPR}), we consider the decomposition 
\[
(1-\bP^{\lift})\circ \mathbf{R}_i=\widehat{\mathbf{R}}_i+(\widecheck{\mathbf{R}}_i-\bP^{\lift} \circ \mathbf{R}_i)
\]
and apply the development  it in the parallel manner as we used to obtain (\ref{eq:Lqw0}).
Then, noting that $\flatTr\, (\widehat{\mathbf{R}}_1\circ \cdots \circ\widehat{\mathbf{R}}_{I'})=0$ from Lemma \ref{lm:monotone}, we obtain
\begin{align*}
&\flatTr\, \left((1-\bP^{\lift})\circ \mathbf{R}_1\circ \cdots \circ \mathbf{R}_{I'}\right)\\
& =
\sum_{j=1}^{I'} \flatTr\, \left(\widehat{\mathbf{R}}_1\circ \cdots \circ \widehat{\mathbf{R}}_{j-1}\circ(\widecheck{\mathbf{R}}_j-\bP^{\lift} \circ \mathbf{R}_j)
\circ ((1-\bP^{\lift})\circ\mathbf{R}_{j+1}\circ\cdots \circ \mathbf{R}_{I'})\right).
\end{align*}
This is bounded in absolute value by 
\begin{align*}
&\sum_{j=1}^{I'} \|\widehat{\mathbf{R}}_1\| \cdots  \|\widehat{\mathbf{R}}_{j-1}\|\cdot
\|(\widecheck{\mathbf{R}}_j-\bP^{\lift} \circ \mathbf{R}_j)\|_{\trace}
\cdot \|(1-\bP^{\lift})\circ\mathbf{R}_{j+1}\circ \cdots \circ\mathbf{R}_{I'}\|. 
\end{align*}
The trace norm $\|(\widecheck{\mathbf{R}}_j-\bP^{\lift} \circ \mathbf{R}_j)\|_{\trace}$ is bounded by a constant $C$ independent of $j$ and $n$.
Therefore, using the condition on $\widehat{\mathbf{R}}_i$ and  (\ref{eq:norm_1Pi}), we conclude (\ref{eq:main_estimate_trace}). 
This completes the proof of  Claim \ref{Claim5} and hence that of   Proposition \ref{pp:zeta}. 
 \end{proof}

\bibliographystyle{amsplain}
\bibliography{mybib}

\end{document}